\documentclass[a4paper,10pt]{amsart}
\usepackage{mathrsfs}
\usepackage{appendix}
\usepackage{amssymb} 
\usepackage{fancyhdr}
\usepackage[a4paper,width=15.5cm,top=3.5cm,bottom=3.5cm,footskip=1cm]{geometry}

\makeatletter
      \def\@setcopyright{}
      \def\serieslogo@{}
      \makeatother

\newcommand{\Complex}{\mathbb C}
\newcommand{\Real}{\mathbb R}
\newcommand{\N}{\mathbb N}
\newcommand{\ddbar}{\overline\partial}
\newcommand{\pr}{\partial}
\newcommand{\ol}{\overline}
\newcommand{\Td}{\widetilde}
\newcommand{\norm}[1]{\left\Vert#1\right\Vert}
\newcommand{\abs}[1]{\left\vert#1\right\vert}

\newcommand{\set}[1]{\left\{#1\right\}}
\newcommand{\To}{\rightarrow}

\newcommand{\mU}{\mathcal{U}}
\newcommand{\mV}{\mathcal{V}}

\theoremstyle{plain}
\newtheorem{thm}{Theorem}[section]

\newtheorem{lem}[thm]{Lemma}
\newtheorem{prop}[thm]{Proposition}

\theoremstyle{definition}
\newtheorem{defn}[thm]{Definition}

\theoremstyle{remark}
\newtheorem{rem}[thm]{Remark}

\newtheorem{que}[thm]{Question}

\numberwithin{equation}{section}

\begin{document}
\title[]{Szeg\"{o} kernel asymptotics for high power of CR line bundles and Kodaira embedding theorems on 
CR manifolds}
\author[]{Chin-Yu Hsiao}
\address{Institute of Mathematics, Academia Sinica, 6F, Astronomy-Mathematics Building, No.1, Sec.4, Roosevelt Road, Taipei 10617, Taiwan}
\thanks{The author was partially supported by Taiwan Ministry of Science of Technology project 
103-2115-M-001-001, the Golden-Jade fellowship of Kenda Foundation and the DFG funded project MA 2469/2-1}
\email{chsiao@math.sinica.edu.tw or chinyu.hsiao@gmail.com} 

\begin{abstract}
Let $X$ be an abstract not necessarily compact orientable CR manifold of dimension $2n-1$, $n\geqslant2$, and let $L^k$ be the $k$-th tensor power of a CR complex line bundle $L$ over $X$. Given $q\in\set{0,1,\ldots,n-1}$, let $\Box^{(q)}_{b,k}$ be the Gaffney extension of Kohn Laplacian for $(0,q)$ forms with values in $L^k$. For $\lambda\geq0$, let $\Pi^{(q)}_{k,\leq\lambda}:=E((-\infty,\lambda])$, where $E$ denotes the spectral measure of $\Box^{(q)}_{b,k}$.  
In this work, we prove that $\Pi^{(q)}_{k,\leq k^{-N_0}}F^*_k$,  $F_k\Pi^{(q)}_{k,\leq k^{-N_0}}F^*_k$, $N_0\geq1$, admit asymptotic expansions with respect to $k$ on the non-degenerate part of the characteristic manifold of $\Box^{(q)}_{b,k}$, where $F_k$ is some kind of microlocal cut-off function. Moreover, we show that $F_k\Pi^{(q)}_{k,\leq0}F^*_k$ admits a full asymptotic expansion with respect to $k$ if $\Box^{(q)}_{b,k}$ has small spectral gap property with respect to $F_k$ and $\Pi^{(q)}_{k,\leq0}$ is $k$-negligible away the diagonal with respect to $F_k$. By using these asymptotics, we establish almost Kodaira embedding theorems on CR manifolds and Kodaira embedding theorems on CR manifolds with transversal CR $S^1$ action. 
\end{abstract}

\maketitle \tableofcontents 

\section{Introduction and statement of the main results} \label{s-intro}

The problem of local and global embedding CR manifolds is prominent in areas such as complex analysis, partial differential equations and differential geometry.
Consider $X$ a compact CR manifold of dimension $2n-1$, $n\geq 2$. When $X$ is strongly pseudoconvex and dimension of $X$ is greater than five, a classical theorem of 
L. Boutet de Monvel~\cite{BdM1:74b} asserts that $X$ can be globally CR embedded into $\Complex^N$, for some $N\in\mathbb N$. When the Levi form of $X$ has mixed signature, then the space of global CR functions is finite dimensional (could be even trivial) and moreover many interesting examples live in the projective space (e.\,g.\ the quadric $\{[z]\in\Complex\mathbb{P}^{N-1};\, |z_1|^2+\ldots+|z_q|^2-|z_{q+1}|^2-\ldots-|z_N|^2=0\}$). 
It is thus natural to consider a setting analogue to the Kodaira embedding theorem and ask if $X$ can be embedded into the projective space by means of CR sections of a CR line bundle of positive curvature. For this purpose it is important to study the asymptotic behaviour of the associated Szeg\"o kernel and study if there are a lot of CR sections in high powers of the line bundle. This was initiated in~\cite{HM09}, ~\cite{Hsiao12}(see also Marinescu~\cite{Ma96}). 

Other developments recently concerned the Bergman kernel for a high power of a holomorphic line bundle.
The Bergman kernel is the smooth kernel of the orthogonal projection onto the space of $L^2$-integrable holomorphic sections. The study of the asymptotic behaviour of the Bergman kernel is an active research subject in complex geometry and is closely related to topics like the structure of algebraic manifolds, the existence of canonical K\"ahler metrics , Berezin-Toeplitz quantization and equidistribution of zeros of holomorphic sections(see~\cite{MM07}).  It is quite interesting to consider CR analogue of the Bergman kernel asymptotic expansion and to study the influence of the asymptotics in CR geometry as in the complex case. This direction could become a research area in CR geometry.

The purpose of this work is to completely study the asymptotic behaviour of the Szeg\"o kernel associated to a hypoelliptic operator $\Box^{(q)}_{b,k}$ with respect to a high power of a CR line bundle. The difficulty of this problem comes from the presence of positive eigenvalues of the curvature of the line bundle and negative eigenvalues of the Levi form of $X$ and hence the semi-classical characteristic manifold of $\Box^{(q)}_{b,k}$ is always degenerate at some point. 
This difficulty is also closely related to the fact that in the global $L^2$-estimates for the $\ddbar_b$-operator of Kohn-H\"{o}rmander there is a curvature term from the line bundle as well from the Levi form and, in general, it is very difficult to control the sign of the total curvature contribution.  In this work, we introduce some kind of microlocal cut-off function $F_k$ and we prove that $\Pi^{(q)}_{k,\leq k^{-N_0}}F^*_k$,  $F_k\Pi^{(q)}_{k,\leq k^{-N_0}}F^*_k$, $N_0\geq1$, admit asymptotic expansions on the non-degenerate part of the characteristic manifold of $\Box^{(q)}_{b,k}$, where for $\lambda\geq0$, $\Pi^{(q)}_{k,\leq\lambda}:=E((-\infty,\lambda])$, $E$ is the spectral measure of $\Box^{(q)}_{b,k}$. Moreover, we show that $F_k\Pi^{(q)}_{k,\leq0}F^*_k$ admits a full asymptotic expansion if $\Box^{(q)}_{b,k}$ has small spectral gap property with respect to $F_k$ and $\Pi^{(q)}_{k,\leq0}$ is $k$-negligible away the diagonal with respect to $F_k$. By using these asymptotics, we establish almost Kodaira embedding theorems on CR manifolds and  Kodaira embedding theorems on CR manifolds with transversal CR $S^1$ action. 
From the analytic view point, this work can be seen as a completely semi-classical study of some kind of hypoelliptic operators.

We now formulate the main results. We refer to section~\ref{s:prelim} for some standard notations and terminology used here. 

Let $(X,T^{1,0}X)$ be a paracompact orientable not necessarily compact CR manifold of dimension $2n-1$, $n\geqslant2$, with a Hermitian metric $\langle\,\cdot\,|\,\cdot\,\rangle$ on $\Complex TX$ such that $T^{1,0}X$
is orthogonal to $T^{0,1}X$ and $\langle\,u\,|\,v\,\rangle$ is real if $u$, $v$ are real tangent vectors. For every $q=0,1,\ldots,n-1$, the Hermitian metric $\langle\,\cdot\,|\,\cdot\,\rangle$ on $\Complex TX$ induces a Hermitian metric $\langle\,\cdot\,|\,\cdot\,\rangle$ on $T^{*0,q}X$ the bundle of $(0,q)$ forms of $X$.
Let $(L,h^L)$ be a CR Hermitian line bundle over $X$, where
the Hermitian fiber metric on $L$ is denoted by $h^L$. We will denote by
$\phi$ the local weights of the Hermitian metric(see \eqref{e-suV}). For $k>0$, let $(L^k,h^{L^k})$ be the $k$-th tensor power of the line bundle $(L,h^L)$. 
We denote by $dv_X=dv_X(x)$ the volume form on $X$ induced by the fixed 
Hermitian metric $\langle\,\cdot\,|\,\cdot\,\rangle$ on $\Complex TX$. Then we get natural global $L^2$ inner products $(\,\cdot\,|\,\cdot\,)_{h^{L^k}}$, $(\,\cdot\,|\,\cdot\,)$
on $\Omega^{0,q}_0(X, L^k)$ and $\Omega^{0,q}_0(X)$ respectively. We denote by $L^2_{(0,q)}(X,L^k)$ and $L^2_{(0,q)}(X)$ the completions of $\Omega^{0,q}_0(X, L^k)$ and $\Omega^{0,q}_0(X)$ with respect to $(\,\cdot\,|\,\cdot\,)_{h^{L^k}}$ and $(\,\cdot\,|\,\cdot\,)$ respectively. 
Let 
\[\Box^{(q)}_{b,k}:{\rm Dom\,}\Box^{(q)}_{b,k}\subset L^2_{(0,q)}(X,L^k)\To L^2_{(0,q)}(X,L^k)\] 
be the Gaffney extension of the Kohn Laplacian(see \eqref{e-suIX}). 
By a result of Gaffney, for every $q=0,1,\ldots,n-1$, $\Box^{(q)}_{b,k}$ is a positive self-adjoint operator (see Proposition 3.1.2 in Ma-Marinescu~\cite{MM07}). That is, $\Box^{(q)}_{b,k}$ is self-adjoint and the spectrum of $\Box^{(q)}_{b,k}$ is contained in $\ol\Real_+$, $q=0,1,\ldots,n-1$. For a Borel set $B\subset\Real$ we denote by $E(B)$ the spectral projection of $\Box^{(q)}_{b,k}$ corresponding to the set $B$, where $E$ is the spectral measure of $\Box^{(q)}_{b,k}$ (see section 2 in Davies~\cite{Dav95}, for the precise meanings of spectral projection and spectral measure). For $\lambda\geq0$, we set 
\begin{equation} \label{e-suX}
\begin{split}
&H^q_{b,\leq\lambda}(X, L^k):={\rm Ran\,}E\bigr((-\infty,\lambda]\bigr)\subset L^2_{(0,q)}(X,L^k)\,,\\
&H^q_{b,>\lambda}(X,L^k):={\rm Ran\,}E\bigr((\lambda,\infty)\bigr)\subset L^2_{(0,q)}(X,L^k).
\end{split}
\end{equation}  
For $\lambda=0$, we denote 
\begin{equation} \label{e-suXI}
H_b^q(X,L^k):=H^q_{b,\leq0}(X,L^k)={\rm Ker\,}\Box^{(q)}_{b,k}.
\end{equation}
For $\lambda\geq0$, let 
\begin{equation}\label{e-suXI-I}
\begin{split}
&\Pi^{(q)}_{k,\leq\lambda}:L^2_{(0,q)}(X,L^k)\To H^q_{b,\leq\lambda}(X,L^k),\\
&\Pi^{(q)}_{k,>\lambda}:L^2_{(0,q)}(X,L^k)\To H^q_{b,>\lambda}(X,L^k),
\end{split}
\end{equation}
be the orthogonal projections with respect to $(\,\cdot\,|\,\cdot\,)_{h^{L^k}}$ and let $\Pi^{(q)}_{k,\leq\lambda}(x,y)\in\mathscr D'(X\times X,(T^{*0,q}_yX\otimes L^k_y)\boxtimes(T^{*0,q}_xX\otimes L^k_x))$ and $\Pi^{(q)}_{k,>\lambda}(x,y)\in\mathscr D'(X\times X,(T^{*0,q}_yX\otimes L^k_y)\boxtimes(T^{*0,q}_xX\otimes L^k_x))$ denote the distribution kernels of $\Pi^{(q)}_{k,\leq\lambda}$ and $\Pi^{(q)}_{k,>\lambda}$ respectively. For $\lambda=0$, we denote $\Pi^{(q)}_k:=\Pi^{(q)}_{k,\leq0}$, $\Pi^{(q)}_k(x,y):=\Pi^{(q)}_{k,\leq0}(x,y)$.

Let $s$ be a local trivialization of $L$ on an open set $D\Subset X$, $\abs{s}^2_{h^L}=e^{-2\phi}$. We assume that $Y(q)$ holds at each point of $D$(see Definition~\ref{d-suII}, for the precise meaning of condition $Y(q)$). By the $L^2$ estimates of Kohn (see Folland-Kohn~\cite{FK72} and Chen-Shaw~\cite{CS01}), we see that for every $\lambda\geq0$, 
\begin{equation} \label{e-suXIII-I}
\Pi^{(q)}_{k,\leq\lambda}(x,y)\in C^\infty(D\times D,(T^{*0,q}_yX\otimes L^k_y)\boxtimes(T^{*0,q}_xX\otimes L^k_x)). 
\end{equation} 
Let $\Sigma$ be the semi-classical characteristic manifold of $\Box^{(q)}_{b,k}$ on $D$(see \eqref{e-crmiII}). We have 
\[\Sigma=\set{(x, \xi)\in T^*D;\, \xi=\lambda\omega_0(x)-2{\rm Im\,}\ddbar_b\phi(x),\lambda\in\Real}\]
(see Proposition~\ref{s2-pcrmiI}), where $\omega_0\in C^\infty(X,T^*X)$ is the uniquely determined global $1$-form(see the discussion before \eqref{e-suI}). For $x\in D$, let $M^\phi_x$ be the Hermitian quadratic form on $T^{1,0}_xX$ given by Definition~\ref{d-suIII-I} and let $\mathcal{L}_x$ be the Levi form at $x$ with respect to $\omega_0$(see Definition~\ref{d-suI}).  Let $\sigma$ denote the canonical two form on $T^*D$. We can show that $\sigma$ is non-degenerate at $\rho=(p, \lambda_0\omega_0(p)-2{\rm Im\,}\ddbar_b\phi(p))\in\Sigma$, $p\in D$, if and only if the Hermitian quadratic form $M^\phi_p-2\lambda_0\mathcal{L}_p$ is non-degenerate(see Theorem~\ref{t-dhI}). From this, it is easy to see that if $M^\phi_p$ has $q$ negative and $n-1-q$ positive eigenvalues, then $\sigma$ is degenerate at $(p, \lambda\omega_0(p)-2{\rm Im\,}\ddbar_b\phi(p))\in\Sigma$, for some $\lambda\in\Real$. Fix $(n_-,n_+)\in\mathbb N^2_0$, $n_-+n_+=n-1$. Put 
\begin{equation}\label{e-dhmpXIa}
\begin{split}
\Sigma'=&\{(x,\lambda\omega_0(x)-2{\rm Im\,}\ddbar_b\phi(x))\in T^*D\bigcap\Sigma;\,\\ &\quad\mbox{$M^\phi_x-2\lambda\mathcal{L}_x$ is non-degenerate of constant signature $(n_-,n_+)$}\}.
\end{split}
\end{equation} 

Let $s$ be a local trivializing section of $L$ on an open subset $D\Subset X$ and $\abs{s}^2_{h^L}=e^{-2\phi}$. 
Let $B_k:\Omega^{0,q}_0(D)\To\mathscr D'(D,T^{*0,q}X)$ be a $k$-dependent continuous operator. We write $B_k(x,y)$ to denote the distribution kernel of $B_k$. $B_k$ is called $k$-negligible (on $D$)
if $B_k$ is smoothing and the kernel $B_k(x, y)$ of $B_k$ satisfies
$\abs{\pr^\alpha_x\pr^\beta_yB_k(x, y)}=O(k^{-N})$ locally uniformly
on every compact set in $D\times D$, for all multi-indices $\alpha$,
$\beta$ and all $N\in\mathbb N$. Let $C_k:\Omega^{0,q}_0(D)\To\mathscr D'(D,T^{*0,q}X)$ be another $k$-dependent continuous operator. 
We write $B_k\equiv C_k\mod O(k^{-\infty})$ (on $D$) or $B_k(x,y)\equiv C_k(x,y)\mod O(k^{-\infty})$ (on $D$) if $B_k-C_k$ is $k$-negligible on $D$.
Let $A_k:L^2_{(0,q)}(X,L^k)\To L^2_{(0,q)}(X,L^k)$ be a continuous operator. We define the localized operator (with respect to the trivializing section $s$) of $A_k$ by
\begin{equation} \label{e-gue130820}
\begin{split}
\hat A_{k,s}:L^2_{(0,q)}(D)\bigcap\mathscr E'(D,T^{*0,q}X)&\To L^2_{(0,q)}(D),\\
u&\To e^{-k\phi}s^{-k}A_k(s^ke^{k\phi}u).
\end{split}
\end{equation} 
We write $A_k\equiv0\mod O(k^{-\infty})$ on $D$ if $\hat A_{k,s}\equiv0\mod O(k^{-\infty})$ on $D$. For $\lambda\geq0$, we write $\hat\Pi^{(q)}_{k,\leq\lambda,s}$ to denote the localized operator of $\Pi^{(q)}_{k,\leq\lambda}$. We denote $\hat\Pi^{(q)}_{k,s}:=\hat\Pi^{(q)}_{k,\leq0,s}$.


\subsection{Main results: Szeg\"o kernel asymptotics for lower energy forms and almost Kodaira embedding Theorems on CR manifolds} \label{s-gue140123}

One of the main results of this work is the following 

\begin{thm}\label{t-gue140121}
With the notations and assumptions used above, let $s$ be a local trivializing section of $L$ on an open subset $D\Subset X$ and $\abs{s}^2_{h^L}=e^{-2\phi}$. Fix $q\in\set{0,1,\ldots,n-1}$ and suppose that $Y(q)$ holds on $D$. We assume that there exist a $\lambda_0\in\Real$ and $x_0\in D$ such that $M^\phi_{x_0}-2\lambda_0\mathcal{L}_{x_0}$ is non-degenerate of constant signature $(n_-,n_+)$. 
We fix $D_0\Subset D$, $D_0$ open, and let $V$ be any bounded open set of $T^*D$ with  $\ol V\subset T^*D$, $\ol V\bigcap\Sigma\subset\Sigma'$, where $\Sigma'$ is given by \eqref{e-dhmpXIa}. 
Let 
\begin{equation}\label{e-gue151209}
\mbox{$\hat{\mathcal{I}}_k\equiv\frac{k^{2n-1}}{(2\pi)^{2n-1}}\int e^{ik<x-y,\eta>}\alpha(x,\eta,k)d\eta\mod O(k^{-\infty})$ at $T^*D_0\bigcap\Sigma$}
\end{equation}
be a properly supported classical semi-classical pseudodifferential operator on $D$ of order $0$ from sections of $T^{*0,q}X$ to sections of $T^{*0,q}X$
with $\alpha(x,\eta,k)=0$ if $\abs{\eta}>M$, for some large $M>0$ and ${\rm Supp\,}\alpha(x,\eta,k)\bigcap T^*D_0\Subset V$. Let $\hat{\mathcal{I}}^*_k$ be the adjoint of $\hat{\mathcal{I}}_k$ with respect to $(\,\cdot\,|\,\cdot\,)$. Then for every $N_0\geq1$ and every $D'\Subset D_0$, $\alpha, \beta\in\mathbb N^{2n-1}_0$, there is a constant $C_{D',\alpha,\beta,N_0}>0$ independent of $k$, such that 
\begin{equation} \label{e-gue130819IIm}
\begin{split}
&\abs{\pr^\alpha_x\pr^\beta_y\big((\hat\Pi^{(q)}_{k,\leq k^{-N_0},s}\hat{\mathcal{I}}^*_k)(x,y)-\int e^{ik\varphi(x,y,s)}a(x,y,s,k)ds\big)}\\
&\leq C_{D',\alpha,\beta,N_0}k^{3n+2\abs{\beta}+\abs{\alpha}-N_0-2}\ \ \mbox{on $D'\times D'$},\\
&\abs{\pr^\alpha_x\pr^\beta_y\big((\hat{\mathcal{I}}_k\hat\Pi^{(q)}_{k,\leq k^{-N_0},s}\hat{\mathcal{I}}^*_k)(x,y)-(\int e^{ik\varphi(x,y,s)}g(x,y,s,k)ds\big)}\\
&\leq C_{D',\alpha,\beta,N_0}k^{3n+2\abs{\beta}+\abs{\alpha}-N_0-2}\ \ \mbox{on $D'\times D'$},
\end{split}
\end{equation}
where $a(x,y,s,k)=g(x,y,s,k)=0$ if $q\neq n_-$, $a(x,y,s,k),\ g(x,y,s,k)\in S^{n}_{{\rm loc\,}}\big(1;\Omega,T^{*0,q}X\boxtimes T^{*0,q}X\big)\bigcap C^\infty_0\big(\Omega,T^{*0,q}X\boxtimes T^{*0,q}X\big)$ are as in \eqref{e-gue130819} and \eqref{e-gue130819I} if $q=n_-$ and 
\begin{equation}\label{e-gue140121I}
\begin{split}
&\varphi(x,y,s)\in C^\infty(\Omega),\ \ {\rm Im\,}\varphi(x,y,s)\geq0,\ \ \forall (x,y,s)\in\Omega,\\
&d_x\varphi(x,y,s)|_{x=y}=-2{\rm Im\,}\ddbar_b\phi(x)+s\omega_0(x),\ \ d_y\varphi(x,y,s)|_{x=y}=2{\rm Im\,}\ddbar_b\phi(x)-s\omega_0(x),\\
&\mbox{${\rm Im\,}\varphi(x,y,s)+\abs{\frac{\pr\varphi}{\pr s}(x,y,s)}\geq c\abs{x-y}^2$, $c>0$ is a constant, $\forall (x,y,s)\in\Omega$},\\
&\mbox{$\varphi(x,y,s)=0$ and $\frac{\pr\varphi}{\pr s}(x,y,s)=0$ if and only if $x=y$}.
\end{split}
\end{equation}
Here 
\begin{equation}\label{e-gue140121}
\begin{split}
\Omega:=&\{(x,y,s)\in D\times D\times\Real;\, (x,-2{\rm Im\,}\ddbar_b\phi(x)+s\omega_0(x))\in V\bigcap\Sigma,\\
&\quad\mbox{$(y,-2{\rm Im\,}\ddbar_b\phi(y)+s\omega_0(y))\in V\bigcap\Sigma$, $\abs{x-y}<\varepsilon$, for some $\varepsilon>0$}\}.
\end{split}\end{equation}
\end{thm}

We refer the reader to Definition~\ref{d-gue13628I}, Definition~\ref{d-gue130816} and Definition~\ref{d-gue13628} for the definition of classical semi-classical pseudodifferential operators and the precise meanings of \eqref{e-gue151209} and the H\"ormander symbol spaces $S^0_{{\rm loc\,},{\rm cl\,}}$ and $S^0_{{\rm loc\,}}$.

For more properties for the phase $\varphi$, see  Theorem~\ref{t-gue140121I}, Theorem~\ref{t-gue140121II} and Theorem~\ref{t-gue140121III}. 

Basically, Theorem~\ref{t-gue140121} says that $(\hat\Pi^{(q)}_{k,\leq k^{-N_0},s}\hat{\mathcal{I}}^*_k)(x,y)$, $(\hat{\mathcal{I}}_k\hat\Pi^{(q)}_{k,\leq k^{-N_0},s}\hat{\mathcal{I}}^*_k)(x,y)$ are asymptotically close to the complex Fourier integral operators $\int e^{ik\varphi(x,y,s)}a(x,y,s,k)ds$, $\int e^{ik\varphi(x,y,s)}g(x,y,s,k)ds$ if $N_0$ large. We will show in section~\ref{s-gue130820} that under certain conditions,  $(\hat{\mathcal{I}}_k\hat\Pi^{(q)}_{k,s}\hat{\mathcal{I}}^*_k)(x,y)$ is a complex Fourier integral operator $\int e^{ik\varphi(x,y,s)}g(x,y,s,k)ds$(see Theorem~\ref{t-gue140123I}). 

Since ${\rm Im\,}\varphi(x,y,s)+\abs{\frac{\pr\varphi}{\pr s}(x,y,s)}\geq c\abs{x-y}^2$, $c>0$ is a constant, we can integrate by parts with respect to $s$ and conclude that the integral $\int e^{ik\varphi(x,y,s)}b(x,y,s,k)ds$, $b(x,y,s,k)\in S^{n}_{{\rm loc\,}}\big(1;\Omega,T^{*0,q}X\boxtimes T^{*0,q}X\big)\bigcap C^\infty_0\big(\Omega,T^{*0,q}X\boxtimes T^{*0,q}X\big)$, is $k$-negligible away $x=y$. Thus, we can take $\varepsilon>0$ in \eqref{e-gue140121} to be any small positive constant.

\begin{defn}\label{d-gue130918}
We say that $L$ is positive if for every $x\in X$ there is a $\eta\in\Real$ such that the Hermitian quadratic form $M^\phi_x-2\eta\mathcal{L}_x$ is positive definite. 
\end{defn}

In view of Proposition~\ref{p-suI}, we see that Definition~\ref{d-gue130918} is well-defined.

\begin{defn}\label{d-gue131109}
Let $(X,T^{1,0}X)$ be a compact orientable CR manifold of dimension $2n-1$, $n\geqslant2$. 
We say that $X$ can be almost CR embedded into projective space 
if for every $m\in\mathbb N_0$ and $\varepsilon>0$  there is an embedding $\Phi_{\varepsilon,m}:X\To\Complex\mathbb P^N$, for some $N\in\mathbb N$, $N\geq2$, such that 
\[\norm{d\Phi_{\varepsilon,m}(T^{1,0}X)-\Complex T\Phi_{\varepsilon,m}(X)\bigcap T^{1,0}\Complex\mathbb P^N}_{C^m(\Complex\mathbb P^N,\Complex T\Complex\mathbb P^N)}<\varepsilon.\]
\end{defn}

We recall that a smooth map $\Phi:X\To\Complex\mathbb P^N$, $N\in\mathbb N$, $N\geq2$, is a CR embedding if $\Phi$ is an embedding and $d\Phi(T^{1,0}X)=\Complex T\Phi(X)\bigcap T^{1,0}\Complex\mathbb P^N$.

By using Theorem~\ref{t-gue140121}, we establish in section~\ref{s-aket} the almost Kodaira embedding theorems on CR manifolds (see Theorem~\ref{t-gue131108})

\begin{thm}\label{t-gue131109}
Let $(X,T^{1,0}X)$ be a compact orientable CR manifold of dimension $2n-1$, $n\geqslant2$. 
If $X$ admits a positive CR line bundle $L$ over $X$, then 
$X$ can be almost CR embedded into projective space. 
\end{thm} 

It should be mentioned that Theorem~\ref{t-gue131109} is in the spirit of the almost
symplectic and almost isometric embedding of compact symplectic manifolds by Borthwick-Uribe~\cite{BU}, Ma-Marinescu~\cite{MM08} and Shiffman-Zelditch~\cite{SS02}. Especially, in Ma-Marinescu~\cite{MM08} the spectral spaces of the Bochner Laplacian are used to obtain the embedding.

\subsection{Main results: Szeg\"o kernel asymptotics}\label{s-gue140123I}

In view of Theorem~\ref{t-gue140121}, we see that if $\Box^{(q)}_{b,k}$ has spectral gap $\geq k^{-M}$, for some $M>0$, the operator $(\hat{\mathcal{I}}_k\hat\Pi^{(q)}_{k,s}\hat{\mathcal{I}}^*_k)(x,y)$ admits a full asymptotic expansion. But in general, it is very difficult to see that if $\Box^{(q)}_{b,k}$ has spectral gap. We then impose some mild semi-classical local conditions and we show that a certain conjugation of the Szeg\"o projection by some kind of pseudodifferential operator 
is a Fourier integral operator under these semi-classical local conditions. More precisely, we have 



\begin{thm}\label{t-gue140123I}
With the notations and assumptions used above, let $s$ be a local trivializing section of $L$ on an open subset $D\Subset X$ and $\abs{s}^2_{h^L}=e^{-2\phi}$. Fix $q\in\set{0,1,\ldots,n-1}$ and suppose that $Y(q)$ holds on $D$. We assume that there exist a $\lambda_0\in\Real$ and $x_0\in D$ such that $M^\phi_{x_0}-2\lambda_0\mathcal{L}_{x_0}$ is non-degenerate of constant signature $(n_-,n_+)$.  Let $F_k:L^2_{(0,q)}(X,L^k)\To L^2_{(0,q)}(X,L^k)$ be a continuous operator and let $F^*_k:L^2_{(0,q)}(X,L^k)\To L^2_{(0,q)}(X,L^k)$ be the Hilbert space adjoint of $F_k$ with respect to $(\,\cdot\,|\,\cdot\,)_{h^{L^k}}$. Let
$\hat F_{k,s}$ and $\hat F^*_{k,s}$ be the localized operators of $F_{k}$ and $F^*_{k}$ respectively.
We fix $D_0\Subset D$, $D_0$ open and let $V$ be any bounded open set of $T^*D$ with $\ol V\subset T^*D$, $\ol V\bigcap\Sigma\subset\Sigma'$, where $\Sigma'$ is as in \eqref{e-dhmpXIa}. Assume that 
\[\hat F_{k,s}-A_k=O(k^{-\infty}):H^s_{{\rm comp\,}}(D,T^{*0,q}X)\To H^s_{{\rm loc\,}}(D,T^{*0,q}X),\ \ \forall s\in\mathbb N_0,\]
where
$\mbox{$A_k\equiv\frac{k^{2n-1}}{(2\pi)^{2n-1}}\int e^{ik<x-y,\eta>}\alpha(x,\eta,k)d\eta\mod O(k^{-\infty})$ at $T^*D_0\bigcap\Sigma$}$
is a classical semi-classical pseudodifferential operator on $D$ of order $0$ from sections of $T^{*0,q}X$ to sections of $T^{*0,q}X$ with $\alpha(x,\eta,k)=0$ if $\abs{\eta}>M$, for some large $M>0$ and ${\rm Supp\,}\alpha(x,\eta,k)\bigcap T^*D_0\Subset V$. Put $P_k:=F_k\Pi^{(q)}_kF^*_k$ and let $\hat P_{k,s}$ be the localized operator of $P_k$. Assume that $\Box^{(q)}_{b,k}$ has $O(k^{-n_0})$ small spectral gap on $D$ with respect to $F_k$ and $\Pi^{(q)}_k$ is $k$-negligible away the diagonal with respect to $F_k$ on $D$. If $q\neq n_-$, then 
\[\mbox{$\hat P_{k,s}(x,y)\equiv0\mod O(k^{-\infty})$ on $D_0$}.\]
If $q=n_-$, then 
\[
\mbox{$\hat P_{k,s}(x,y)\equiv\int e^{ik\varphi(x,y,s)}g(x,y,s,k)ds\mod O(k^{-\infty})$ on $D_0$},\] 
where $\varphi(x,y,s)\in C^\infty(\Omega)$ is as in Theorem~\ref{t-gue140121} and 
\[g(x,y,s,k)\in S^{n}_{{\rm loc\,}}\big(1;\Omega,T^{*0,q}X\boxtimes T^{*0,q}X\big)\bigcap C^\infty_0\big(\Omega,T^{*0,q}X\boxtimes T^{*0,q}X\big)\] 
is as in \eqref{e-gue130819} and \eqref{e-gue130819I}. Here $\Omega$ is given by \eqref{e-gue140121}. 
\end{thm}

We refer the reader to Definition~\ref{d-gue130820} and Definition~\ref{d-gue131205} for the precise meanings of "$O(k^{-n_0})$ small spectral gap on $D$ with respect to $F_k$" and "$k$-negligible away the diagonal with respect to $F_k$ on $D$". 

\begin{rem}\label{r-gue140123}
With the assumptions and notations used in Theorem~\ref{t-gue140123I}, since we only assume some local properties of $F_k$ on $D$, we don't know what is $F_k$ outside $D$. In order to get an asymptotic expansion for the Szeg\"o kernel, we need to assume that $\Pi^{(q)}_k$ is $k$-negligible away the diagonal with respect to $F_k$ on $D$. If $X$ is compact and $F_k$ is a global classical semi-classical pseudodifferential operator on $X$(see Definition~\ref{d-gue131205I}), we can show that(see Proposition~\ref{p-gue131205}) $\Pi^{(q)}_k$ is $k$-negligible away the diagonal with respect to $F_k$ on every local trivialization $D\Subset X$. Furthermore, if $X$ is non-compact and $F_k$ is properly supported on $D\Subset X$, then $\Pi^{(q)}_k$ is $k$-negligible away the diagonal with respect to $F_k$ on $D$(see also Proposition~\ref{p-gue131205}). 
\end{rem} 

\begin{rem}\label{r-gue140126}
With the assumptions and notations used in Theorem~\ref{t-gue140123I}, let
\[\set{f_1\in H^0_{b}(X,L^k),\ldots,f_{d_k}\in H^0_{b}(X,L^k)}\] 
be an orthonormal frame of the space $H^0_{b}(X,L^k)$, where $d_k\in\mathbb N_0\bigcup\set{\infty}$. It is easy to see that 
\[\hat P_{k,s}(x,x)=\sum^{d_k}_{j=1}\abs{(F_kf_j)(x)}^2_{h^{L^k}},\ \ \forall x\in D.\] 
Theorem~\ref{t-gue140123I} implies that if $\Box^{(q)}_{b,k}$ has $O(k^{-n_0})$ small spectral gap on $D$ with respect to $F_k$ and $\Pi^{(q)}_k$ is $k$-negligible away the diagonal with respect to $F_k$ on $D$, then 
\[\mbox{$\sum^{d_k}_{j=1}\abs{(F_kf_j)(x)}^2_{h^{L^k}}\equiv0\mod O(k^{-\infty})$ on $D_0$ when $q\neq n_-$}\]
and
\[\mbox{$\sum^{d_k}_{j=1}\abs{(F_kf_j)(x)}^2_{h^{L^k}}\equiv\sum^{\infty}_{j=1}k^{n-j}b_j(x)\mod O(k^{-\infty})$ on $D_0$ when $q=n_-$},\]
where $b_j(x)\in C^\infty_0(D)$, $j=0,1,\ldots$, and for every $x\in D_0$, $b_0(x)=\int g_0(x,x,s)ds$, $g_0(x,y,s)$ is given by \eqref{e-gue130819I}. 
\end{rem} 

After proving Theorem~\ref{t-gue131109} and Theorem~\ref{t-gue140123I}, we asked the following two questions: When we can get  "true" Kodaira embedding Theorems on CR manifolds? Can we find some non-trivial examples for Theorem~\ref{t-gue140123I}? In order to answer these questions, let's study carefully some CR submanifolds of projective space. We consider $\Complex\mathbb P^{N-1}$, $N\geq4$. Let $[z]=[z_1,\ldots,z_N]$ be the homogeneous coordinates of $\Complex\mathbb P^{N-1}$. Put 
\[X:=\set{[z_1,\ldots,z_N]\in\Complex\mathbb P^{N-1};\, \lambda_1\abs{z_1}^2+\cdots+\lambda_m\abs{z_m}^2+\lambda_{m+1}\abs{z_{m+1}}^2+\cdots+\lambda_N\abs{z_N}^2=0},\]
where $m\in\mathbb N$ and $\lambda_j\in\Real$, $j=1,\ldots,N$. Then, $X$ is a compact CR manifold of dimension $2(N-1)-1$ with CR structure $T^{1,0}X:=T^{1,0}\Complex\mathbb P^{N-1}\bigcap\Complex TX$. Now, we assume that $\lambda_1<0,\lambda_2<0,\ldots,\lambda_m<0$, $\lambda_{m+1}>0,\lambda_{m+2}>0,\ldots,\lambda_N>0$, where $m\geq2$, $N-m\geq2$. Then, it is easy to see that the Levi form has at least one negative and one positive eigenvalues at each point of $X$. Thus, $Y(0)$ holds at each point of $X$. $X$ admits a $S^1$ action: 
\begin{equation}\label{e-gue131218IIIm}
\begin{split}
S^1\times X&\To X,\\
e^{i\theta}\circ[z_1,\ldots,z_m,z_{m+1},\ldots,z_N]&\To[e^{i\theta}z_1,\ldots,e^{i\theta}z_m,z_{m+1},\ldots,z_N],\ \ \theta\in[-\pi,\pi[.
\end{split}
\end{equation}
Since $(z_1,\ldots,z_m)\neq0$ on $X$, this $S^1$ action is well-defined. Let $T\in C^\infty(X,TX)$ be the real vector field given by 
\begin{equation}\label{e-gue131205Im}
Tu=\frac{\pr}{\pr\theta}(u(e^{i\theta}x))|_{\theta=0},\ \ u\in C^\infty(X).
\end{equation}
It is easy to see that $[T,C^\infty(X,T^{1,0}X)]\subset C^\infty(X,T^{1,0}X)$ and $T(x)\oplus T^{1,0}_xX\oplus T^{0,1}_xX=\Complex T_xX$(we say that the $S^1$ action is CR and transversal, see Definition~\ref{d-gue131205II}). 

Let $E\To\Complex\mathbb P^{N-1}$ be the canonical line bundle with respect to the standard Fubini-Study metric. 
Consider $L:=E|_X$. Then, it is easy to see that(see section~\ref{s-cmips}) $L$ is a $T$-rigid positive CR line bundle over $(X,T^{1,0}X)$(see Definition~\ref{d-gue131209}, for the meaning of "positive $T$-rigid CR line bundle"). Thus, we ask the following question

\begin{que}\label{q-gue140123}
Let $(X, T^{1,0}X)$ be a compact CR manifold with a transversal CR $S^1$ action and let $T$ be the global vector field induced by the $S^1$ action. If there is a $T$-rigid positive CR line bundle over $X$, then can $X$ be CR embedded into $\mathbb C\mathbb P^N$, for some $N\in\mathbb N$? 
\end{que}

In section~\ref{s-saak}, we study "CR manifolds with transversal CR $S^1$ actions" and fortunately, we found that if $Y(q)$ holds and $M^\phi_x$ is non-degenerate of constant signature, then we can always find a continuous operator $F_k:L^2_{(0,q)}(X,L^k)\To L^2_{(0,q)}(X,L^k)$ such that the assumptions in Theorem~\ref{t-gue140123I} hold and by using  Theorem~\ref{t-gue140123I}, we solve Question~\ref{q-gue140123} completely.

\subsection{Main results: Sezg\"o kernel asymptotics and Kodairan embedding Theorems on CR manifolds with transversal CR $S^1$ actions}

Let $(X, T^{1,0}X)$ be a CR manifold. We assume that $X$ admits a transversal CR $S^1$ action: $S^1\times X\To X$ (see Definition~\ref{d-gue131205II}). We write $e^{i\theta}$, $0\leq\theta<2\pi$, to denote the $S^1$ action and we let $T$ be the global vector field induced by the $S^1$ action(see \eqref{e-gue131205I}). Note that we don't assume that the $S^1$ action is globally free.

We show in Theorem~\ref{t-gue131206I} that there is a $T$-rigid Hermitian metric $\langle\,\cdot\,|\,\cdot\,\rangle$ on $\Complex TX$  (see Definition~\ref{d-gue131205III} and Definition~\ref{d-gue131206}) such that $T^{1,0}X\perp T^{0,1}X$, $T\perp (T^{1,0}X\oplus T^{0,1}X)$, $\langle\,T\,|\,T\,\rangle=1$ and $\langle\,u\,|v\,\rangle$ is real if $u, v$ are real tangent vectors. Until further notice, we fix a $T$-rigid Hermitian metric $\langle\,\cdot\,|\,\cdot\,\rangle$ on $\Complex TX$ such that $T^{1,0}X\perp T^{0,1}X$, $T\perp (T^{1,0}X\oplus T^{0,1}X)$, $\langle\,T\,|\,T\,\rangle=1$ and $\langle\,u\,|v\,\rangle$ is real if $u, v$ are real tangent vectors. 

Let $L$ be a $T$-rigid CR line bundle over $(X,T^{1,0}X)$ with a $T$-rigid Hermitian fiber metric $h^L$ on $L$(see Definition~\ref{d-gue131206II} and Definition~\ref{d-gue131207}). For $k>0$, as before, we shall consider $(L^k,h^{L^k})$ and we will use the same notations as before. Since $L$ is $T$-rigid, $Tu$ is well-defined, for every $u\in\Omega^{0,q}(X,L^k)$. For $m\in\mathbb Z$, put 
\begin{equation}\label{e-gue131207Im}
A^{0,q}_m(X,L^k):=\set{u\in\Omega^{0,q}(X,L^k);\, Tu=imu}
\end{equation}
and let $\mathcal{A}^{0,q}_m(X,L^k)\subset L^2_{(0,q)}(X,L^k)$ be the completion of $A^{0,q}_m(X,L^k)$ with respect to $(\,\cdot\,|\,\cdot\,)_{h^{L^k}}$. 
For $m\in\mathbb Z$, let 
\begin{equation}\label{e-gue131207IIIm}
Q^{(q)}_{m,k}:L^2_{(0,q)}(X,L^k)\To\mathcal{A}^{0,q}_m(X,L^k)
\end{equation}
be the orthogonal projection with respect to $(\,\cdot\,|\,\cdot\,)_{h^{L^k}}$. Fix $\delta>0$. Take $\tau_\delta(x)\in C^\infty_0(]-\delta,\delta[)$, 
$0\leq\tau_\delta\leq1$ and $\tau_\delta=1$ on $[-\frac{\delta}{2},\frac{\delta}{2}]$. Let $F^{(q)}_{\delta,k}:L^2_{(0,q)}(X,L^k)\To L^2_{(0,q)}(X,L^k)$ be the continuous map given by 
\begin{equation}\label{e-gue131207IVm}
\begin{split}
F^{(q)}_{\delta,k}:L^2_{(0,q)}(X,L^k)&\To L^2_{(0,q)}(X,L^k),\\
u&\To\sum_{m\in\mathbb Z}\tau_\delta(\frac{m}{k})(Q^{(q)}_{m,k}u).
\end{split}
\end{equation}

One of the main results of this work is the following

\begin{thm}\label{t-gue140124}
Let $(X, T^{1,0}X)$ be a compact CR manifold with a transversal CR $S^1$ action and let $T$ be the global vector field induced by the $S^1$ action. Let $L$ be a $T$-rigid  CR line bundle over $X$ with a $T$-rigid Hermitian fiber metric $h^L$. We assume that $Y(q)$ holds at each point of $X$ and $M^{\phi}_x$ is non-degenerate of constant signature $(n_-,n_+)$, for every $x\in X$.   
Let $s$ be a local trivializing section of $L$ on an  open set $D\subset X$, $\abs{s}^2_{h^L}=e^{-2\phi}$. Fix $D_0\Subset D$. Let $F^{(q)}_{\delta,k}:L^2_{(0,q)}(X,L^k)\To L^2_{(0,q)}(X,L^k)$ be the continuous operator given by \eqref{e-gue131207IVm} and let $F^{(q),*}_{\delta,k}:L^2_{(0,q)}(X,L^k)\To L^2_{(0,q)}(X,L^k)$ be the adjoint of $F^{(q)}_{\delta,k}$ with respect to $(\,\cdot\,|\,\cdot\,)_{h^{L^k}}$. Put $P_k:=F^{(q)}_{\delta,k}\Pi^{(q)}_kF^{(q),*}_{\delta,k}$ and let $\hat P_{k,s}$ be the localized operator of $P_k$. Assume that $\delta>0$ is small. If $q\neq n_-$, then $\hat P_{k,s}\equiv0\mod O(k^{-\infty})$ on $D_0$. 
If $q=n_-$, then 
\[
\mbox{$\hat P_{k,s}(x,y)\equiv\int e^{ik\varphi(x,y,s)}g(x,y,s,k)ds\mod O(k^{-\infty})$ on $D_0$},\] 
where $\varphi(x,y,s)\in C^\infty(\Omega)$ is as in Theorem~\ref{t-gue140121} and 
\[
\begin{split}
&g(x,y,s,k)\in S^{n}_{{\rm loc\,}}\big(1;\Omega,T^{*0,q}X\boxtimes T^{*0,q}X\big)\bigcap C^\infty_0\big(\Omega,T^{*0,q}X\boxtimes T^{*0,q}X\big),\\
&g(x,y,s,k)\sim\sum^\infty_{j=0}g_j(x,y,s)k^{n-j}\text{ in }S^{n}_{{\rm loc\,}}
\big(1;\Omega,T^{*0,q}X\boxtimes T^{*0,q}X\big), \\
&g_j(x,y,s)\in C^\infty_0\big(\Omega,T^{*0,q}X\boxtimes T^{*0,q}X\big),\ \ j=0,1,2,\ldots,
\end{split}\]
and for every $(x,x,s)\in\Omega$, $x\in D_0$, $g_0(x,x,s)=(2\pi)^{-n}\abs{\det\bigr(M^\phi_x-2s\mathcal{L}_x\bigr)}\abs{\tau_\delta(s)}^2\mathcal{\pi}_{(x,s,n_-)}$. 
Here $\Omega$ is given by \eqref{e-gue140121}, $\mathcal{\pi}_{(x,s,n_-)}$ is as in Theorem~\ref{t-gue130816} and $\tau_\delta$ is as in the discussion after \eqref{e-gue131207IIIm}. 
\end{thm}

\begin{defn}\label{d-gue131209}
Let $(X, T^{1,0}X)$ be a CR manifold with a transversal CR $S^1$ action and let $T$ be the global vector field induced by the $S^1$ action. We say that $L$ is a $T$-rigid positive CR line bundle over $X$ 
if $L$ is a $T$-rigid CR line bundle over $X$ and there is a $T$-rigid Hermitian fiber metric $h^L$ on $L$ such that $M^\phi_x$ is positive, for every $x\in X$, where $\phi$ is the local weights of $h^L$.
\end{defn} 

By using Theorem~\ref{t-gue140124}, we can repeat the proof of Theorem~\ref{t-gue131109} and establish Kodaira embedding theorems on CR manifolds with transversal CR $S^1$ actions

\begin{thm}\label{t-gue140124I}
Let $(X, T^{1,0}X)$ be a compact CR manifold with a transversal CR $S^1$ action and let $T$ be the global vector field induced by the $S^1$ action. If there is a $T$-rigid positive CR line bundle over $X$, then $X$ can be CR embedded into $\mathbb C\mathbb P^N$, for some $N\in\mathbb N$. 
\end{thm} 

In section~\ref{s-skaohg}, we also establish Szeg\"o kernel asymptotis on some non-compact CR manifolds (see Theorem~\ref{t-gue131227}).

The layout of this paper is as follows. In section~\ref{s-gue151208}, we collect some properties for the phase $\varphi(x,y,s)$. In section~\ref{s:prelim}, we collect some notations, definitions and statements we use throughout. In section~\ref{s-scb}, we write down $\Box^{(q)}_{b,k}$ in a local trivialization and we get the formula for the characteristic manifold for $\Box^{(q)}_{b,k}$ and we prove that the canonical two form $\sigma$ is non-degenerate at $\rho=(p, \lambda_0\omega_0(p)-2{\rm Im\,}\ddbar_b\phi(p))\in\Sigma$ if and only if the Hermitian quadratic form $M^\phi_p-2\lambda_0\mathcal{L}_p$ is non-degenerate(see Theorem~\ref{t-dhI}). In section~\ref{s-thef}, by using the identity $k=e^{-ikx_{2n}}\bigr(-i\frac{\pr}{\pr x_{2n}}(e^{ikx_{2n}})\bigr)$, we introduce the local operator $\Box^{(q)}_s$ defined on some open set of $\Real^{2n}$ and by using the heat equation method, we establish microlocal Hodge decomposition theorems for $\Box^{(q)}_s$. 
Moreover, we study the phase $\varphi$ carefully and we prove Theorem~\ref{t-gue140121I} and Theorem~\ref{t-gue140121II}. In section~\ref{s-sch}, we reduce the semi-classical analysis of $\Box^{(q)}_{b,k}$ to the microlocal analysis of $\Box^{(q)}_s$ and we establish semi-classical Hodge decomposition theorems for $\Box^{(q)}_{b,k}$ in non-degenerate part of $\Sigma$ and we calculate the leading terms of the asymptotic expansions in \eqref{e-gue130819}. We notice that it is possible to consider semi-classical heat equation for $\Box^{(q)}_{b,k}$ and  establish semi-classical Hodge decomposition theorems for $\Box^{(q)}_{b,k}$ directly. In a further publication, we will study Kohn Laplacian on a CR manifold with codimension$\geq2$ and the local operator $\Box^{(q)}_s$ is similar to some kind of Kohn Laplacian on a CR manifold with codimension $2$. Hence, we decide to reduce the semi-classical analysis of $\Box^{(q)}_{b,k}$ to the microlocal analysis of $\Box^{(q)}_s$. In section~\ref{s-safle}, by using the scaling technique introduced in~\cite{HM09} and the semi-classical Hodge decomposition theorems in section~\ref{s-sch}, we prove Theorem~\ref{t-gue140121}. In section~\ref{s-aket}, by using Theorem~\ref{t-gue140121}, we establish almost Kodaira embedding theorems on CR manifolds and therefore we prove Theorem~\ref{t-gue131109}. It should be mentioned that the method in section~\ref{s-aket} works well in complex case(the complex case is much more simpler than CR case) and we can give a pure analytic proof of classical Kodaira embedding theorem.
In section~\ref{s-gue130820}, we prove Theorem~\ref{t-gue140123I}.  
In section~\ref{s-saak}, we introduce and study CR manifolds with transversal CR $S^1$ action. We show that there is a $T$-rigid Hermitian metric $\langle\,\cdot\,|\,\cdot\,\rangle$ on $\Complex TX$(see  Theorem~\ref{t-gue131206I}) and we prove Theorem~\ref{t-gue140124}. Moreover, by using Theorem~\ref{t-gue140124}, we can repeat the proof of Theorem~\ref{t-gue131109} and establish Kodaira embedding theorems on CR manifolds with transversal CR $S^1$ action. For simplicity, we only give the outline of the proof of Theorem~\ref{t-gue140124I}. In section~\ref{s-skaohg}, by using H\"ormander's $L^2$ estimates, we establish the small spectral gap property for $\Box^{(0)}_{b,k}$ with respect to $F^{(0)}_{\delta,k}$ and we prove Theorem~\ref{t-gue131227}. Finally, in section~\ref{sub-pf}, by using global theory of complex Fourier integral operators of Melin-Sj\"ostrand~\cite{MS74}, we prove Theorem~\ref{t-gue140121III}. 

{\emph{
\textbf{Acknowledgements.}
Some part of this paper has been carried out when the author was a research fellow at Universit{\"a}t zu K{\"o}ln,  Mathematisches Institut.  I am grateful to Universit{\"a}t zu K{\"o}ln,  Mathematisches Institut, for offering excellent working conditions. The methods of microlocal analysis used in this work are marked by the influence of Professor Johannes Sj\"{o}strand.  This paper is dedicated to him for his retirement. 
}} 

\section{More properties of the phase $\varphi(x,y,s)$}\label{s-gue151208}

In this section, we collect some properties of the phase $\varphi(x,y,s)$ for the convenience of the reader. 

We can estimate ${\rm Im\,}\varphi(x,y,s)$ in some local coordinates

\begin{thm}\label{t-gue140121I}
With the assumptions and notations used in Theorem~\ref{t-gue140121}, fix $p\in D$. We take local coordinates $x=(x_1,\ldots,x_{2n-1})$ defined in a small neighbourhood of $p$ so that $\omega_0(p)=dx_{2n-1}$, $T^{1,0}_pX\oplus T^{0,1}_pX=\set{\sum^{2n-2}_{j=1}a_j\frac{\pr}{\pr x_j};\, a_j\in\Real, j=1,\ldots,2n-2}$. If $D$ is small enough, then there is a constant $c>0$ such that 
\begin{equation}\label{e-gue140121II}\begin{split}
&{\rm Im\,}\varphi(x,y,s)\geq c\abs{x'-y'}^2,\ \ \forall (x,y,s)\in\Omega,\\
&{\rm Im\,}\varphi(x,y,s)+\abs{\frac{\pr\varphi}{\pr s}(x,y,s)}\geq c\bigr(\abs{x_{2n-1}-y_{2n-1}}+\abs{x'-y'}^2\bigr),\ \ \forall (x,y,s)\in\Omega,\end{split}\end{equation}
where $x'=(x_1,\ldots,x_{2n-2})$, $y'=(y_1,\ldots,y_{2n-2})$, $\abs{x'-y'}^2=\sum^{2n-2}_{j=1}\abs{x_j-y_j}^2$. 
\end{thm}

In section~\ref{s-cth}, we determine the tangential Hessian of $\varphi(x,y,s)$ 

\begin{thm}\label{t-gue140121II}
With the assumptions and notations used in Theorem~\ref{t-gue140121}, fix $(p,p,s_0)\in\Omega$, and let $\ol Z_{1,s_0},\ldots,\ol Z_{n-1,s_0}$ be an orthonormal frame of $T^{1,0}_xX$ varying smoothly with $x$ in a neighbourhood of $p$, for which the Hermitian quadratic form $M^\phi_x-2s_0\mathcal{L}_x$ is diagonalized at $p$. 
That is, 
\[M^\phi_p\bigr(\ol Z_{j,s_0}(p),Z_{t,s_0}(p)\bigr)-2s_0\mathcal{L}_p\bigr(\ol Z_{j,s_0}(p),Z_{t,s_0}(p)\bigr)
=\lambda_j(s_0)\delta_{j,t},\ \ j,t=1,\ldots,n-1.\]
Assume that $\lambda_j(s_0)<0$, $j=1,\ldots,n_-$, 
$\lambda_j(s_0)>0$, $j=n_-+1,\ldots,n-1$. Let $x=(x_1,\ldots,x_{2n-1})$, $z_j=x_{2j-1}+ix_{2j}$, $j=1,\ldots,n-1$, be local coordinates of $X$ defined in some small neighbourhood of $p$ such that 
\begin{equation}\label{e-geusw13623}
\begin{split}
&x(p)=0,\ \ \omega_0(p)=dx_{2n-1},\ \ T(p)=-\frac{\pr}{\pr x_{2n-1}}(p),\\
&\langle\,\frac{\pr}{\pr x_j}(p)\,|\,\frac{\pr}{\pr x_t}(p)\,\rangle=2\delta_{j,t},\ \ j,t=1,\ldots,2n-2,\\
&\ol Z_{j,s_0}(p)=\frac{\pr}{\pr z_j}+i\sum^{n-1}_{t=1}\tau_{j,t}\ol z_t\frac{\pr}{\pr x_{2n-1}}+c_jx_{2n-1}\frac{\pr}{\pr x_{2n-1}}+O(\abs{x}^2),\ \ j=1,\ldots,n-1,\\
&\phi(x)=\beta x_{2n-1}+\sum^{n-1}_{j=1}\bigr(\alpha_jz_j+\ol\alpha_j\ol z_j\bigr)+\frac{1}{2}\sum^{n-1}_{l,t=1}\mu_{t,l}z_t\ol z_l+\sum^{n-1}_{l,t=1}\bigr(a_{l,t}z_lz_t+\ol a_{l,t}\ol z_l\ol z_t\bigr)\\
&\quad+\sum^{n-1}_{j=1}\bigr(d_jz_jx_{2n-1}+\ol d_j\ol z_jx_{2n-1}\bigr)+O(\abs{x_{2n-1}}^2)+O(\abs{x}^3),
\end{split}
\end{equation}
where $\beta\in\Real$, $\tau_{j,t}, c_j, \alpha_j,  \mu_{j,t}, a_{j,t}, d_j\in\Complex$, $\mu_{j,t}=\ol\mu_{t,j}$, $j, t=1,\ldots,n-1$. (This is always possible, see \cite[p.\,157--160]{BG88}).
We also write
$y=(y_1,\ldots,y_{2n-1})$, $w_j=y_{2j-1}+iy_{2j}$, $j=1,\ldots,n-1$. Then, in some small neighbourhood of $(p,p)$, 
\begin{equation}\label{e-guew13627}
\begin{split}
&\varphi(x,y,s_0)\\
&=-i\sum^{n-1}_{j=1}\alpha_j(z_j-w_j)+i\sum^{n-1}_{j=1}\ol\alpha_j(\ol z_j-\ol w_j)+s_0(x_{2n-1}-y_{2n-1})
-\frac{i}{2}\sum^{n-1}_{j,l=1}(a_{l,j}+a_{j,l})(z_jz_l-w_jw_l)\\
&\quad+\frac{i}{2}\sum^{n-1}_{j,l=1}(\ol a_{l,j}+\ol a_{j,l})(\ol z_j\ol z_l-\ol w_j\ol w_l)+\frac{1}{2}\sum^{n-1}_{j,l=1}\Bigr(is_0(\ol\tau_{l,j}-\tau_{j,l})+(\ol\tau_{l,j}+\tau_{j,l})\beta\Bigr)(z_j\ol z_l-w_j\ol w_l)\\
&\quad+\sum^{n-1}_{j=1}(-ic_j\beta-s_0c_j-id_j)(z_jx_{2n-1}-w_jy_{2n-1})+\sum^{n-1}_{j=1}(i\ol c_j\beta-s_0\ol c_j+i\ol d_j)(\ol z_jx_{2n-1}-\ol w_jy_{2n-1})\\
&\quad-\frac{i}{2}\sum^{n-1}_{j=1}\lambda_j(s_0)(z_j\ol w_j-\ol z_jw_j)+\frac{i}{2}\sum^{n-1}_{j=1}\abs{\lambda_j(s_0)}\abs{z_j-w_j}^2+(x_{2n-1}-y_{2n-1})f(x,y,s_0)+O(\abs{(x, y)}^3),\\ 
&\quad\quad\quad f\in C^\infty,\ \  f(0,0)=0.
\end{split}
\end{equation}
\end{thm}

\begin{defn}\label{d-gue140121}
With the assumptions and notations used in Theorem~\ref{t-gue140121}, let $\varphi_1(x,y,s), \varphi_2(x,y,s)\in C^\infty(\Omega)$. We assume that $\varphi_1(x,y,s)$ and $\varphi_2(x,y,s)$ satisfy 
\eqref{e-gue140121I} and \eqref{e-gue140121II}. We say that $\varphi_1(x,y,s)$ and $\varphi_2(x,y,s)$ are equivalent on $\Omega$ if for any  
\[b_1(x,y,s,k)\in S^{m}_{{\rm loc\,}}\big(1;\Omega,T^{*0,q}X\boxtimes T^{*0,q}X\big)\bigcap C^\infty_0\big(\Omega,T^{*0,q}X\boxtimes T^{*0,q}X\big),\ \ m\in\mathbb Z,\] 
we can find 
\[b_2(x,y,s,k)\in S^{m}_{{\rm loc\,}}\big(1;\Omega,T^{*0,q}X\boxtimes T^{*0,q}X\big)\bigcap C^\infty_0\big(\Omega,T^{*0,q}X\boxtimes T^{*0,q}X\big)\] 
such that 
\[\int e^{ik\varphi_1(x,y,s)}b_1(x,y,s,k)ds\equiv \int e^{ik\varphi_1(x,y,s)}b_2(x,y,s,k)ds\mod O(k^{-\infty})\ \ \mbox{on $D$}\]
and vise versa. 
\end{defn}

Let $\varphi_1(x,y,s)\in C^\infty(\Omega)$. We assume that $\varphi_1(x,y,s)$ satisfies 
\eqref{e-gue140121I} and \eqref{e-gue140121II}.
Fix $p\in D$. We take local coordinates $x=(x_1,\ldots,x_{2n-1})$ defined in a small neighbourhood of $p$ so that $\omega_0(p)=dx_{2n-1}$, $T^{1,0}_pX\oplus T^{0,1}_pX=\set{\sum^{2n-2}_{j=1}a_j\frac{\pr}{\pr x_j};\, a_j\in\Real, j=1,\ldots,2n-2}$. Put 
\[-2{\rm Im\,}\ddbar_b\phi(x)=\sum^{2n-1}_{j=1}\alpha_j(x)dx_j,\ \ \omega_0(x)=\sum^{2n-1}_{j=1}\beta_j(x)dx_j.\] 
We assume that $D$ is a small open neighbourhood of $p$. In Lemma~\ref{l-gue130804}, we will show that if $D$ is small enough then we can find $\hat\varphi_1(x,y,s)\in C^\infty(\Omega)$ such that $\hat\varphi_1(x,y,s)$ satisfies \eqref{e-gue140121I}, \eqref{e-gue140121II}, $\frac{\pr\hat\varphi_1}{\pr y_{2n-1}}(x,y,s)-(\alpha_{2n-1}(y)+s\beta_{2n-1}(y))$ vanishes to infinite order at $x=y$ and $\hat\varphi_1(x,y,s)$ and $\varphi_1(x,y,s)$ are equivalent on $\Omega$ in the sense of Definition~\ref{d-gue140121}. We characterize the phase $\varphi$(see section~\ref{s-mhdt} and section~\ref{sub-pf})

\begin{thm}\label{t-gue140121III}
With the assumptions and notations used in Theorem~\ref{t-gue140121}, put $\Td\varphi(x,y,s):=-\ol\varphi(y,x,s)$. Then, $\Td\varphi(x,y,s)$ and $\varphi(x,y,s)$ are equivalent on $\Omega$ in the sense of Definition~\ref{d-gue140121}. Moreover, let $\varphi_1(x,y,s)\in C^\infty(\Omega)$. We assume that $\varphi_1(x,y,s)$ satisfies \eqref{e-gue140121I} and \eqref{e-gue140121II}. If $D$ is small enough, then $\varphi$ and $\varphi_1$ are equivalent on $\Omega$ in the sense of Definition~\ref{d-gue140121} if and only if there are functions $f\in C^\infty(\Omega)$, $g_j\in C^\infty(\Omega)$, $j=0,1,\ldots,2n-1$, $p_j\in C^\infty(\Omega)$, $j=1,\ldots,2n-1$, such that
\[
\begin{split}
&\frac{\pr\hat\varphi}{\pr s}(x,y,s)-f(x,y,s)\frac{\pr\hat\varphi_1}{\pr s}(x,y,s),\\
&\hat\varphi(x,y,s)-\hat\varphi_1(x,y,s)=g_0(x,y,s)\frac{\pr\hat\varphi}{\pr s}(x,y,s),\\
&\frac{\pr\hat\varphi}{\pr x_j}(x,y,s)-\frac{\pr\hat\varphi_1}{\pr x_j}(x,y,s)=g_j(x,y,s)\frac{\pr\hat\varphi}{\pr s}(x,y,s),\ \ j=1,2,\ldots,2n-1, \\
&\frac{\pr\hat\varphi}{\pr y_j}(x,y,s)-\frac{\pr\hat\varphi_1}{\pr y_j}(x,y,s)=p_j(x,y,s)\frac{\pr\hat\varphi}{\pr s}(x,y,s),\ \ j=1,2,\ldots,2n-1,
\end{split}\]
vanish to infinite order on $x=y$, for every $(x,y,s)\in\Omega$, where $\hat\varphi(x,y,s)\in C^\infty(\Omega)$, $\hat\varphi_1(x,y,s)\in C^\infty(\Omega)$ are as in the discussion after Definition~\ref{d-gue140121}. 
\end{thm}
\section{Preliminaries}\label{s:prelim}

\subsection{Some standard notations} \label{s-ssn}

We shall use the following notations: $\Real$ is the set of real numbers, $\ol\Real_+:=\set{x\in\Real;\, x\geq0}$, $\mathbb N=\set{1,2,\ldots}$, $\mathbb N_0=\mathbb N\bigcup\set{0}$. An element $\alpha=(\alpha_1,\ldots,\alpha_n)$ of $\mathbb N_0^n$ will be called a multiindex, the size of $\alpha$ is: $\abs{\alpha}=\alpha_1+\cdots+\alpha_n$ and the length of $\alpha$ is $l(\alpha)=n$. For $m\in\mathbb N$, we write $\alpha\in\set{1,\ldots,m}^n$ if $\alpha_j\in\set{1,\ldots,m}$, $j=1,\ldots,n$. We say that $\alpha$ is strictly increasing if $\alpha_1<\alpha_2<\cdots<\alpha_n$. We write $x^\alpha=x_1^{\alpha_1}\cdots x^{\alpha_n}_n$, 
$x=(x_1,\ldots,x_n)$,
$\pr^\alpha_x=\pr^{\alpha_1}_{x_1}\cdots\pr^{\alpha_n}_{x_n}$, $\pr_{x_j}=\frac{\pr}{\pr x_j}$, $\pr^\alpha_x=\frac{\pr^{\abs{\alpha}}}{\pr x^\alpha}$, $D^\alpha_x=D^{\alpha_1}_{x_1}\cdots D^{\alpha_n}_{x_n}$, $D_x=\frac{1}{i}\pr_x$, $D_{x_j}=\frac{1}{i}\pr_{x_j}$. 
Let $z=(z_1,\ldots,z_n)$, $z_j=x_{2j-1}+ix_{2j}$, $j=1,\ldots,n$, be coordinates of $\Complex^n$.  
We write $z^\alpha=z_1^{\alpha_1}\cdots z^{\alpha_n}_n$, $\ol z^\alpha=\ol z_1^{\alpha_1}\cdots\ol z^{\alpha_n}_n$,
$\frac{\pr^{\abs{\alpha}}}{\pr z^\alpha}=\pr^\alpha_z=\pr^{\alpha_1}_{z_1}\cdots\pr^{\alpha_n}_{z_n}$, $\pr_{z_j}=
\frac{\pr}{\pr z_j}=\frac{1}{2}(\frac{\pr}{\pr x_{2j-1}}-i\frac{\pr}{\pr x_{2j}})$, $j=1,\ldots,n$. 
$\frac{\pr^{\abs{\alpha}}}{\pr\ol z^\alpha}=\pr^\alpha_{\ol z}=\pr^{\alpha_1}_{\ol z_1}\cdots\pr^{\alpha_n}_{\ol z_n}$, $\pr_{\ol z_j}=
\frac{\pr}{\pr\ol z_j}=\frac{1}{2}(\frac{\pr}{\pr x_{2j-1}}+i\frac{\pr}{\pr x_{2j}})$, $j=1,\ldots,n$.

Let $M$ be a $C^\infty$ paracompact manifold. 
We let $TM$ or $T(M)$ and $T^*M$ or $T^*(M)$ denote the tangent bundle of $M$ and the cotangent bundle of $M$ respectively.
The complexified tangent bundle of $M$ and the complexified cotangent bundle of $M$ will be denoted by $\Complex TM$ and $\Complex T^*M$ respectively. We write $\langle\,\cdot\,,\cdot\,\rangle$ to denote the pointwise duality between $TM$ and $T^*M$.
We extend $\langle\,\cdot\,,\cdot\,\rangle$ bilinearly to $\Complex TM\times\Complex T^*M$.
Let $E$ be a $C^\infty$ vector bundle over $M$. The fiber of $E$ at $x\in M$ will be denoted by $E_x$.
Let $F$ be another vector bundle over $M$. We write 
$E\boxtimes F$ to denote the vector bundle over $M\times M$ with fiber over $(x, y)\in M\times M$ 
consisting of the linear maps from $E_x$ to $F_y$.  Let $Y\subset M$ be an open set. From now on, the spaces of
smooth sections of $E$ over $Y$ and distribution sections of $E$ over $Y$ will be denoted by $C^\infty(Y, E)$ and $\mathscr D'(Y, E)$ respectively.
Let $\mathscr E'(Y, E)$ be the subspace of $\mathscr D'(Y, E)$ whose elements have compact support in $Y$.
For $m\in\Real$, we let $H^m(Y, E)$ denote the Sobolev space
of order $m$ of sections of $E$ over $Y$. Put
\begin{gather*}
H^m_{\rm loc\,}(Y, E)=\big\{u\in\mathscr D'(Y, E);\, \varphi u\in H^m(Y, E),
      \, \forall\varphi\in C^\infty_0(Y)\big\}\,,\\
      H^m_{\rm comp\,}(Y, E)=H^m_{\rm loc}(Y, E)\cap\mathscr E'(Y, E)\,.
\end{gather*} 

Let $E$ and $F$ be $C^\infty$ vector
bundles over a paracompact $C^\infty$ manifold $M$ equipped with a smooth density of integration. If
$A: C^\infty_0(M,E)\To \mathscr D'(M,F)$
is continuous, we write $K_A(x, y)$ or $A(x, y)$ to denote the distribution kernel of $A$.
The following two statements are equivalent
\begin{enumerate}
\item $A$ is continuous: $\mathscr E'(M,E)\To C^\infty(M,F)$,
\item $K_A\in C^\infty(M\times M,E_y\boxtimes F_x)$.
\end{enumerate}
If $A$ satisfies (a) or (b), we say that $A$ is smoothing. Let
$B: C^\infty_0(M,E)\to \mathscr D'(M,F)$ be a continuous operator. 
We write $A\equiv B$ if $A-B$ is a smoothing operator. We say that $A$ is properly supported if ${\rm Supp\,}K_A\subset M\times M$ is proper. That is, the two projections: $t_x:(x,y)\in{\rm Supp\,}K_A\To x\in M$, $t_y:(x,y)\in{\rm Supp\,}K_A\To y\in M$ are proper (i.e. the inverse images of $t_x$ and $t_y$ of all compact subsets of $M$ are compact). 

Let $H(x,y)\in\mathscr D'(M\times M,E_y\boxtimes F_x)$. We write $H$ to denote the unique continuous operator $C^\infty_0(M,E)\To\mathscr D'(M,F)$ with distribution kernel $H(x,y)$. In this work, we identify $H$ with $H(x,y)$.

\subsection{Set up and Terminology} \label{s-su}
Let $(X,T^{1,0}X)$ be a paracompact orientable not necessarily compact CR manifold of dimension $2n-1$, $n\geqslant2$, 
where $T^{1,0}X$ is a 
CR structure of $X$. That is, $T^{1,0}X$ is a complex $n-1$ dimensional subbundle of the complexified tangent bundle
$\Complex TX$, satisfying $T^{1,0}X\bigcap T^{0,1}X=\set{0}$, where $T^{0,1}X=\ol{T^{1,0}X}$, 
and $[\mathcal{V},\mathcal{V}]\subset\mathcal{V}$, where $\mathcal{V}=C^\infty(X,T^{1,0}X)$. 

Fix a smooth Hermitian metric $\langle\,\cdot\,|\,\cdot\,\rangle$ on $\Complex TX$ so that $T^{1,0}X$
is orthogonal to $T^{0,1}X:=\ol{T^{1,0}X}$ and $\langle\,u\,|\,v\,\rangle$ is real if $u$, $v$ are real tangent vectors. 
Then locally there is a real non-vanishing vector field $T$ of length one which is pointwise orthogonal to
$T^{1, 0}X\oplus T^{0, 1}X$. $T$ is unique up to the choice of sign. For $u\in\Complex TX$, we write $\abs{u}^2:=\langle\,u\,|\,u\,\rangle$. 
Denote by $T^{*1,0}X$ and $T^{*0,1}X$ the dual bundles of $T^{1,0}X$ and $T^{0,1}X$, respectively.
They can be identified with subbundles of the complexified cotangent bundle $\Complex T^*X$.
Define the vector bundle of $(0, q)$ forms by $T^{*0,q}X:=\Lambda^{q}T^{*0,1}X$. 
The Hermitian metric $\langle\,\cdot\,|\,\cdot\,\rangle$ on $\Complex TX$ induces,
by duality, a Hermitian metric on $\Complex T^*X$ and also on the bundles of $(0,q)$ forms $T^{*0,q}X$, $q=0,1,\ldots,n-1$. We shall also denote all these induced metrics by $\langle\,\cdot\,|\,\cdot\,\rangle$. For $v\in T^{*0,q}X$, we write $\abs{v}^2:=\langle\,v\,|\,v\,\rangle$, and for any $p=0,1,2,\ldots,n-1$, let $v^{\wedge,*}:T^{*0,q+p}X\To T^{*0,p}X$ be the adjoint of $v^{\wedge}:T^{*0,p}X\To T^{*0,p+q}X$ with respect to $\langle\,\cdot\,|\,\cdot\,\rangle$. That is, $\langle\,v\wedge u\,|\,g\,\rangle=\langle\,u\,|\,v^{\wedge,*}g\,\rangle$, $\forall u\in T^{*0,p}X$, $g\in T^{*0,p+q}X$. Let $D\subset X$ be an open set. Let $\Omega^{0,q}(D)$ denote the space of smooth sections of $T^{*0,q}X$ over $D$ and let $\Omega^{0,q}_0(D)$ be the subspace of
$\Omega^{0,q}(D)$ whose elements have compact support in $D$. Similarly, if $E$ is a vector bundle over $D$, then we let $\Omega^{0,q}(D, E)$
denote the space of smooth sections of $T^{*0,q}X\otimes E$ over $D$ and let $\Omega^{0,q}_0(D, E)$ be the subspace of $\Omega^{0,q}(D, E)$ whose elements have compact support in $D$. 

Locally we can choose an orthonormal frame $\omega_1,\ldots,\omega_{n-1}$
of the bundle $T^{*1,0}X$. Then $\ol\omega_1,\ldots,\ol\omega_{n-1}$
is an orthonormal frame of the bundle $T^{*0,1}X$. The real $(2n-2)$ form
$\omega=i^{n-1}\omega_1\wedge\ol\omega_1\wedge\cdots\wedge\omega_{n-1}\wedge\ol\omega_{n-1}$
is independent of the choice of the orthonormal frame. Thus $\omega$ is globally
defined. Locally there is a real $1$-form $\omega_0$ of length one which is orthogonal to
$T^{*1,0}X\oplus T^{*0,1}X$. The form $\omega_0$ is unique up to the choice of sign.
Since $X$ is orientable, there is a nowhere vanishing $(2n-1)$ form $Q$ on $X$.
Thus, $\omega_0$ can be specified uniquely by requiring that $\omega\wedge\omega_0=fQ$,
where $f$ is a positive function. Therefore $\omega_0$, so chosen, is globally defined.
We call $\omega_0$
the uniquely determined global real $1$-form. 
We choose a vector field $T$ so that
\begin{equation}\label{e-suI}
\abs{T}=1\,,\quad \langle\,T\,,\,\omega_0\,\rangle=-1\,.
\end{equation}
Therefore $T$ is uniquely determined. We call $T$ the uniquely determined global real vector field. We have the
pointwise orthogonal decompositions:
\begin{equation} \label{e-suII}\begin{split}
\Complex T^*X&=T^{*1,0}X\oplus T^{*0,1}X\oplus\set{\lambda\omega_0;\,
\lambda\in\Complex},  \\
\Complex TX&=T^{1,0}X\oplus T^{0,1}X\oplus\set{\lambda T;\,\lambda\in\Complex}.
\end{split}\end{equation} 

We recall  

\begin{defn} \label{d-suI}
For $p\in X$, the Levi form $\mathcal{L}_p$ is the Hermitian quadratic form on $T^{1,0}_pX$ defined as follows. For any $U,\ V\in T^{1,0}_pX$, pick $\mathcal{U},\mathcal{V}\in
C^\infty(X,T^{1,0}X)$ such that
$\mathcal{U}(p)=U$, $\mathcal{V}(p)=V$. Set
\begin{equation} \label{e-suIII}
\mathcal{L}_p(U,\ol V)=\frac{1}{2i}\big\langle\big[\mathcal{U}\ ,\ol{\mathcal{V}}\,\big](p)\ ,\omega_0(p)\big\rangle\,,
\end{equation}
where $\big[\mathcal{U}\ ,\ol{\mathcal{V}}\,\big]=\mathcal{U}\ \ol{\mathcal{V}}-\ol{\mathcal{V}}\ \mathcal{U}$ denotes the commutator of $\mathcal{U}$ and $\ol{\mathcal{V}}$.
Note that $\mathcal{L}_p$ does not depend of the choices of $\mathcal{U}$ and $\mathcal{V}$.
\end{defn} 

Locally there is an orthonormal basis $\{\mathcal{U}_1,\ldots,\mathcal{U}_{n-1}\}$
of $T^{1,0}X$ with respect to $\langle\,\cdot\,|\,\cdot\,\rangle$ such that $\mathcal{L}_p$ is diagonal in this basis, $\mathcal{L}_p(\mathcal{U}_j,\ol{\mathcal{U}}_l)=\delta_{j,l}\lambda_j(p)$, where $\delta_{j,l}=1$ if $j=l$, $\delta_{j,l}=0$ if $j\neq l$. 
The entries $\{\lambda_1(p),\ldots,\lambda_{n-1}(p)\}$ are called the eigenvalues of the Levi form
at $p\in X$ with respect to $\langle\,\cdot\,|\,\cdot\,\rangle$.

\begin{defn} \label{d-suII}
Given $q\in\{0,\ldots,n-1\}$, the Levi form is said to satisfy condition $Y(q)$ at $p\in X$, if $\mathcal{L}_p$ has at least either $\max{(q+1, n-q)}$ eigenvalues of the same sign or $\min{(q+1,n-q)}$ pairs of eigenvalues with opposite signs. Note that the sign of the eigenvalues does not depend on the choice of the metric $\langle\,\cdot\,|\,\cdot\,\rangle$. 
\end{defn} 

Let
\begin{equation} \label{e-suIV}
\ddbar_b:\Omega^{0,q}(X)\To\Omega^{0,q+1}(X)
\end{equation}
be the tangential Cauchy-Riemann operator. We say that a function $u\in C^\infty (X)$ is Cauchy-Riemann (CR for short) if $\ddbar_b u=0$.

\begin{defn} \label{d-suIII}
Let $L$ be a complex line bundle over $X$. We say that $L$ is a Cauchy-Riemann (CR) complex line bundle over $X$
if its transition functions are CR.
\end{defn}

From now on, we let $(L,h^L)$ be a CR Hermitian line bundle over $X$, where
the Hermitian fiber metric on $L$ is denoted by $h^L$. We will denote by
$\phi$ the local weights of the Hermitian metric. More precisely, if
$s$ is a local trivializing
section of $L$ on an open subset $D\subset X$, then the local weight of $h^L$ with respect to $s$ is the function $\phi\in C^\infty(D,\Real)$ for which
\begin{equation} \label{e-suV}
\abs{s(x)}^2_{h^L}=e^{-2\phi(x)}\,,\quad x\in D.
\end{equation}

Let $L^k$, $k>0$, be the $k$-th tensor power of the line bundle $L$. The Hermitian fiber metric on $L$ induces a Hermitian fiber metric on $L^k$ that we shall denote by $h^{L^k}$. If $s$ is a local trivializing section
of $L$ then $s^k$ is a local trivializing section of $L^k$. The Hermitian metrics $\langle\,\cdot\,|\,\cdot\,\rangle$ on $T^{*0,q}X$ and $h^{L^k}$ induce
Hermitian metrics on $T^{*0,q}X\otimes L^k$, $q=0,1,\ldots,n-1$. We shall denote these induced metrics by $\langle\,\cdot\,|\,\cdot\,\rangle_{h^{L^k}}$.
For $f\in\Omega^{0,q}(X,L^k)$, we denote the pointwise norm $\abs{f(x)}^2_{h^{L^k}}:=\langle\,f(x)\,|\,f(x)\rangle_{h^{L^k}}$. We write $\ddbar_{b,k}$ to denote the tangential
Cauchy-Riemann operator acting on forms with values in $L^k$, defined locally by:
\begin{equation} \label{e-suVI}
\ddbar_{b,k}:\Omega^{0,q}(X, L^k)\To\Omega^{0,q+1}(X, L^k)\,,\quad \ddbar_{b,k}(s^ku):=s^k\ddbar_bu,
\end{equation}
where $s$ is a local trivialization of $L$ on an open subset $D\subset X$ and $u\in\Omega^{0,q}(D)$. 
We denote by $dv_X=dv_X(x)$ the volume form on $X$ induced by the fixed 
Hermitian metric $\langle\,\cdot\,|\,\cdot\,\rangle$ on $\Complex TX$. Then we get natural global $L^2$ inner products $(\,\cdot\,|\,\cdot\,)_{h^{L^k}}$, $(\,\cdot\,|\,\cdot\,)$
on $\Omega^{0,q}_0(X, L^k)$ and $\Omega^{0,q}_0(X)$ respectively. We denote by $L^2_{(0,q)}(X,L^k)$ and $L^2_{(0,q)}(X)$ the completions of $\Omega^{0,q}_0(X, L^k)$ and $\Omega^{0,q}_0(X)$ with respect to $(\,\cdot\,|\,\cdot\,)_{h^{L^k}}$ and $(\,\cdot\,|\,\cdot\,)$ respectively. We extend $(\,\cdot\,|\,\cdot\,)_{h^{L^k}}$ and $(\,\cdot\,|\,\cdot\,)$ to $L^2_{(0,q)}(X,L^k)$ and $L^2_{(0,q)}(X)$ in the standard way respectively. For $f\in\Omega^{0,q}(X,L^k)$, we denote $\norm{f}^2_{h^{L^k}}:=(\,f\,|\,f\,)_{h^{L^k}}$. Similarly, for $f\in\Omega^{0,q}(X)$, we denote $\norm{f}^2:=(\,f\,|\,f\,)$.
We extend 
$\ddbar_{b,k}$ to $L^2_{(0,r)}(X,L^k)$, $r=0,1,\ldots,n-1$, by 
\begin{equation}\label{e-suVII}
\ddbar_{b,k}:{\rm Dom\,}\ddbar_{b,k}\subset L^2_{(0,r)}(X,L^k)\To L^2_{(0,r+1)}(X,L^k)\,,
\end{equation}
where ${\rm Dom\,}\ddbar_{b,k}:=\{u\in L^2_{(0,r)}(X, L^k);\, \ddbar_{b,k}u\in L^2_{(0,r+1)}(X, L^k)\}$, where for any $u\in L^2_{(0,r)}(X,L^k)$, $\ddbar_{b,k} u$ is defined in the sense of distribution. 
We also write 
\begin{equation}\label{e-suVIII}
\ol{\pr}^{*}_{b,k}:{\rm Dom\,}\ol{\pr}^{*}_{b,k}\subset L^2_{(0,r+1)}(X, L^k)\To L^2_{(0,r)}(X, L^k)
\end{equation}
to denote the Hilbert space adjoint of $\ddbar_{b,k}$ in the $L^2$ space with respect to $(\,\cdot\,|\,\cdot\, )_{h^{L^k}}$.
Let $\Box^{(q)}_{b,k}$ denote the (Gaffney extension) of the Kohn Laplacian given by 
\begin{equation}\label{e-suIX}
\begin{split}
{\rm Dom\,}\Box^{(q)}_{b,k}=\{s\in L^2_{(0,q)}(X,L^k);\, s\in{\rm Dom\,}\ddbar_{b,k}\cap{\rm Dom\,}\ol{\pr}^{*}_{b,k},\ \ddbar_{b,k}s\in{\rm Dom\,}\ol{\pr}^{*}_{b,k},\ \ol{\pr}^{*}_{b,k}s\in{\rm Dom\,}\ddbar_{b,k}\}\,,
 \end{split}
\end{equation}
and $\Box^{(q)}_{b,k}s=\ddbar_{b,k}\ol{\pr}^{*}_{b,k}s+\ol{\pr}^{*}_{b,k}\ddbar_{b,k}s$ for $s\in {\rm Dom\,}\Box^{(q)}_{b,k}$. 

We need 

\begin{defn} \label{d-suIII-I}
Let $s$ be a local trivializing section of $L$ on an open subset $D\subset X$ and $\phi$ the corresponding local weight as in \eqref{e-suV}.
For $p\in D$, we define the Hermitian quadratic form $M^\phi_p$ on $T^{1,0}_pX$ by
\begin{equation} \label{e-suXIII-II}
M^\phi_p(U, \ol V)=\Big\langle U\wedge\ol V, d\big(\ddbar_b\phi-\pr_b\phi\big)(p)\Big\rangle,\ \ U, V\in T^{1,0}_pX,
\end{equation}
where $d$ is the usual exterior derivative and $\ol{\pr_b\phi}=\ddbar_b\ol\phi$.
\end{defn}

The definition of $M^\phi_p$ depends on the choice of local trivializations.
The following is well-known (see Proposition 4.2 of~\cite{HM09})

\begin{prop} \label{p-suI}
Let $\Td D$ be another local trivialization with
$D\cap\Td D\neq\emptyset$. Let $\Td s$ be a local trivializing section of $L$ on $\Td D$, $\abs{\Td s(x)}_{h^L}^2=e^{-2\Td\phi(x)}$, $\Td\phi\in C^\infty(\Td D,\Real)$, and $\Td s=gs$ on $D\cap\Td D$, for some non-zero CR function $g$. For $p\in D\bigcap\Td D$, we have
\begin{equation} \label{e-suXIII-III}
M^\phi_p=M^{\Td\phi}_p+i\Big(\,\frac{Tg}{g}-\frac{T\,\ol g}{\ol g}\,\Bigr)(p)\mathcal{L}_p.
\end{equation}
\end{prop}

\section{Semi-classical $\Box^{(q)}_{b,k}$ and the characteristic manifold for $\Box^{(q)}_{b,k}$} \label{s-scb}

As before, let $s(x)$ be a local trivializing of $L$ on some open subset $D\subset X$. We have the unitary identification
\begin{equation} \label{s2-emsmilkI} 
\left\{\begin{aligned}  
C^\infty_0(D,T^{*0,q}X)&\longleftrightarrow C^\infty_0(D,L^k\otimes T^{*0,q}X) \\
u&\longleftrightarrow \Td u=(e^\phi s)^ku, \ \ u=e^{-k\phi}s^{-k}\Td u\\
\ddbar_{s,k}&\longleftrightarrow\ddbar_{b,k},\ \ \ddbar_{s,k}u=e^{-k\phi}s^{-k}\ddbar_{b,k}((e^\phi s)^ku),\\
\ol{\pr}^*_{s,k}&\longleftrightarrow\ol{\pr}^*_{b,k},\ \ \ol{\pr}^*_{s,k}u=e^{-k\phi}s^{-k}\ol{\pr}^*_{b,k}((e^\phi s)^ku),\\
\Box^{(q)}_{s,k}&\longleftrightarrow\Box^{(q)}_{b,k},\ \ \Box^{(q)}_{s,k}u=e^{-k\phi}s^{-k}\Box^{(q)}_{b,k}((e^\phi s)^ku).
\end{aligned}
\right.
\end{equation} 
We can check that 
\begin{equation} \label{s2-emsmilkII}
\ddbar_{s,k}=\ddbar_b+k(\ddbar_b\phi)^\wedge, 
\end{equation} 
\begin{equation} \label{s2-emsmilkIII} 
\ol{\pr}^*_{s,k}=\ol{\pr}^*_b+k(\ddbar_b\phi)^{\wedge,*}
\end{equation} 
and 
\begin{equation} \label{e-msmilkVI}
\Box^{(q)}_{s,k}=\ddbar_{s,k}\ol{\pr}^*_{s,k}+\ol{\pr}^*_{s,k}\ddbar_{s,k}.
\end{equation}
Here $\ol{\pr}^*_b:C^\infty(X,T^{*0,q+1}X)\To C^\infty(X,T^{*0,q}X)$ is the formal 
adjoint of $\ddbar_b$ with respect to $(\,\cdot\,|\,\cdot\,)$.

For $p\in X$, we can choose an orthonormal frame
$e_1(x),\ldots,e_{n-1}(x)$
for $T^{*0,1}_xX$ varying smoothly with $x$ in a neighbourhood of $p$.
Let $Z_j(x)$, $j=1,\ldots,n-1$, denote the basis of $T^{0,1}_xX$,
which is dual to $e_j(x)$, $j=1,\ldots,n-1$. Let $Z^*_j$ be the formal adjoint of $Z_j$ with respect to $(\,\cdot\,|\,\cdot\,)$, $j=1,\ldots,n-1$. That is,
$(Z_jf\ |\ h)=(f\ |\ Z_j^*h), f, h\in C^\infty(X)$. On scalar functions, we have 
\[\ddbar_{s,k}=\sum^{n-1}_{j=1}e_j^\wedge\circ (Z_j+kZ_j(\phi)).\] 
If $f(x)e_{j_1}\wedge\cdots\wedge e_{j_q}$ is a typical term in a general $(0, q)$ form, we get 
\begin{equation*} 
\begin{split} 
&\ddbar_{s,k}(f(x)e_{j_1}\wedge\cdots\wedge e_{j_q}) \\
&=\sum^{n-1}_{j=1}\bigr(Z_j(f)+kZ_j(\phi)\bigr)e^\wedge_je_{j_1}\wedge\cdots\wedge e_{j_q} \\
&+\sum^q_{t=1}(-1)^{t-1}f(z)e_{j_1}\wedge\cdots\wedge(\ddbar_b e_{j_t})\wedge\cdots\wedge e_{j_q} \\
&=\Bigr(\sum^{n-1}_{j=1}\bigr(Z_j(f)+kZ_j(\phi)\bigr)e^\wedge_j\Bigr)e_{j_1}\wedge\cdots\wedge e_{j_q}\\
&+\Bigr(\sum^{n-1}_{j=1}(\ddbar_be_j)^\wedge e^{\wedge,*}_j\Bigr)(f(z)e_{j_1}\wedge\cdots\wedge e_{j_q}).
\end{split} 
\end{equation*} 
So for the given orthonormal frame we have the identification 
\begin{equation} \label{e-msmilkIV} 
\ddbar_{s,k}=\sum^{n-1}_{j=1}\Bigr(e^\wedge_j\circ(Z_j+kZ_j(\phi))+
(\ddbar_b e_j)^\wedge e^{\wedge,*}_j\Bigr)
\end{equation} 
and correspondingly 
\begin{equation} \label{e-msmilkV} 
\ol{\pr}^*_{s,k}=\sum^{n-1}_{j=1}\Bigr(e^{\wedge,*}_j\circ(Z^*_j+k\ol Z_j(\phi))+
e^\wedge_j(\ddbar_b e_j)^{\wedge,*}\Bigr).
\end{equation} 
 
We have 

\begin{equation} \label{e-msmilkVII} 
\begin{split} 
\Box^{(q)}_{s,k}&=\ddbar_{s,k}\ol{\pr}^*_{s,k}+\ol{\pr}^*_{s,k}\ddbar_{s,k}\\
&=\sum^{n-1}_{j,t=1}\Bigr[\bigr(e^\wedge_j\circ(Z_j+kZ_j(\phi))+
(\ddbar_be_j)^\wedge e^{\wedge,*}_j\bigr)\bigr(e^{\wedge,*}_t\circ(Z^*_t+k\ol Z_t(\phi))+
e^\wedge_t(\ddbar_be_t)^{\wedge,*}\bigr)\\
&\quad+\bigr(e^{\wedge,*}_t\circ(Z^*_t+k\ol Z_t(\phi))+
e^\wedge_t(\ddbar_be_t)^{\wedge,*}\bigr)\bigr(e^\wedge_j
\circ(Z_j+kZ_j(\phi))+
(\ddbar_be_j)^\wedge e^{\wedge,*}_j\bigr)\Bigr]\\
&=\sum^{n-1}_{j,t=1}\Bigr[\bigr(e^\wedge_j\circ(Z_j+kZ_j(\phi))\bigr)
\bigr(e^{\wedge,*}_t\circ(Z^*_t+k\ol Z_t(\phi))\bigr)
+\bigr(e^{\wedge,*}_t\circ(Z^*_t+k\ol Z_t(\phi))\bigr)\\
&\bigr(e^\wedge_j\circ(Z_j+kZ_j(\phi)))\Bigr]+\varepsilon(Z+kZ(\phi))+\varepsilon(Z^*+k\ol Z(\phi))+f\\
&=\sum^{n-1}_{j,t=1}\Bigr[e^\wedge_je^{\wedge,*}_t\circ(Z_j+kZ_j(\phi))(Z^*_t+k\ol Z_t(\phi))+e^{\wedge,*}_te^\wedge_j\circ(Z^*_t+k\ol Z_t(\phi))\\
&(Z_j+kZ_j(\phi))\Bigr]+\varepsilon(Z+kZ(\phi))+\varepsilon(Z^*+k\ol Z(\phi))+f\\
&=\sum^{n-1}_{j,t=1}(e^\wedge_je^{\wedge,*}_t+e^{\wedge,*}_te^\wedge_j)
\circ(Z^*_t+k\ol Z_t(\phi))(Z_j+kZ_j(\phi))\\
&+\sum^{n-1}_{j,t=1}e^\wedge_je^{\wedge,*}_t\circ[Z_j+kZ_j(\phi),
Z^*_t+k\ol Z_t(\phi)]\\
&+\varepsilon(Z+kZ(\phi))+\varepsilon(Z^*+k\ol Z(\phi))+f,
\end{split}
\end{equation}
where $\varepsilon(Z+kZ(\phi))$ denotes remainder terms of the form $\sum a_j(Z_j+k Z_j\phi)$ with $a_j$ smooth, matrix-valued and independent of $k$, for all $j$, and similarly for $\varepsilon(Z^*+k\ol Z(\phi))$ and $f$ is a smooth function independent of $k$.

Note that 
\begin{equation} \label{e-msmilkVIII} 
e^\wedge_je^{\wedge,*}_t+e^{\wedge,*}_te^\wedge_j=\delta_{j,t}.
\end{equation} 
Combining \eqref{e-msmilkVII} with \eqref{e-msmilkVIII}, we get the following 

\begin{prop} \label{s2-pmsmilkI} 
With the notations used before, using the identification \eqref{s2-emsmilkI}, we can identify the Kohn Laplacian $\Box^{(q)}_{b,k}$ with
\begin{equation} \label{e-msmilkIX} 
\begin{split}
\Box^{(q)}_{s,k}&=\ddbar_{s,k}\ol{\pr}^*_{s,k}+\ol{\pr}^*_{s,k}\ddbar_{s,k} \\
       &=\sum^{n-1}_{j=1}(Z^*_j+k\ol Z_j(\phi))(Z_j+kZ_j(\phi))\\
       &+\sum^{n-1}_{j,t=1}e^\wedge_je^{\wedge,*}_t\circ[Z_j+kZ_j(\phi), Z^*_t+k\ol Z_t(\phi)]\\
&+\varepsilon(Z+kZ(\phi))+\varepsilon(Z^*+k\ol Z(\phi))+f,
\end{split}
\end{equation}
where $\varepsilon(Z+kZ(\phi))$ denotes remainder terms of the form $\sum a_j(Z_j+k Z_j(\phi))$ with $a_j$ smooth, matrix-valued and independent of $k$, for all $j$, and similarly for $\varepsilon(Z^*+k\ol Z(\phi))$ and $f$ is a smooth function independent of $k$. 
\end{prop} 

We work with some real local coordinates $x=(x_1,\ldots,x_{2n-1})$ defined on $D$. 
Let $\xi=(\xi_1,\ldots,\xi_{2n-1})$ denote the dual variables of $x$. Then $(x, \xi)$ are local coordinates of the cotangent bundle $T^*D$.
Let $q_j(x, \xi)$ be the semi-classical principal symbol of $Z_j+kZ_j(\phi)$, $j=1,\ldots,n-1$. Then the semi-classical principal symbol of $\Box^{(q)}_{s,k}$ is given by 
\begin{equation} \label{e-crmiI}
p_0=\sum^{n-1}_{j=1}\ol q_jq_j.
\end{equation}
The characteristic manifold $\Sigma$ of $\Box^{(q)}_{s,k}$ is 
\begin{equation} \label{e-crmiII}
\begin{split}
\Sigma&=\set{(x, \xi)\in T^*D;\,  p_0(x, \xi)=0}\\
&=\set{(x, \xi)\in T^*D;\, q_1(x, \xi)=\cdots=q_{n-1}(x,\xi)=\ol q_1(x, \xi)=\cdots=\ol q_{n-1}(x, \xi)=0}.
\end{split}
\end{equation}
From \eqref{e-crmiII}, we see that $p_0$ vanishes to second order at $\Sigma$. We have the following

\begin{prop} \label{s2-pcrmiI} 
We have 
\begin{equation} \label{e-crmiIII} 
\Sigma=\set{(x, \xi)\in T^*D;\, \xi=\lambda\omega_0(x)-2{\rm Im\,}\ddbar_b\phi(x),\lambda\in\Real}.
\end{equation}
\end{prop} 

\begin{proof} 
For $p\in D$, we may choose local coordinates
$x=(x_1,\ldots,x_{2n-1})$,
such that $x(p)=0$, $\omega_0(p)=\sqrt{2}dx_{2n-1}$,
$(\frac{\pr}{\pr x_j}(p)\ |\ \frac{\pr}{\pr x_k}(p))=2\delta_{j,k}$, $j, k=1,\ldots,2n-1$
and $Z_j=\frac{1}{2}(\frac{\pr}{\pr x_{2j-1}}+i\frac{\pr}{\pr x_{2j}})$ at $p$, $j=1,\ldots,n-1$,
where $Z_j$, $j=1,\ldots,n-1$, are as in \eqref{e-msmilkIX}.
Then the principal symbol of $Z_j$ is 
$\frac{1}{2}i(\xi_{2j-1}+i\xi_{2j})$ at $p$, $j=1,\ldots,n-1$. Thus 
\begin{equation} \label{e-patI} 
q_j(p, \xi)=\frac{1}{2}i(\xi_{2j-1}+i\xi_{2j})+\frac{1}{2}(\frac{\pr\phi}{\pr x_{2j-1}}(p)+i\frac{\pr\phi}{\pr x_{2j}}(p)),
\end{equation} 
$j=1,\ldots,n-1$. From \eqref{e-patI}, we see that $(p, \xi)\in\Sigma$ if and only if 
\begin{equation} \label{e-patII}
\xi_{2j-1}=-\frac{\pr\phi}{\pr x_{2j}}(p),\ \ \xi_{2j}=\frac{\pr\phi}{\pr x_{2j-1}}(p),\ \ j=1,\ldots,n-1. 
\end{equation}
Note that 
\begin{equation*} 
(\ddbar_b\phi)(p)=\sum^{n-1}_{j=1}\Bigr(\frac{1}{2}\bigr(\frac{\pr\phi}{\pr x_{2j-1}}(p)+i\frac{\pr\phi}{\pr x_{2j}}(p)\bigr)(dx_{2j-1}-idx_{2j})\Bigr).
\end{equation*} 
Hence 
\begin{equation} \label{e-patIII} 
{\rm Im\,}(\ddbar_b\phi)(p)=\sum^{n-1}_{j=1}\Bigr(\frac{1}{2}\bigr(-\frac{\pr\phi}{\pr x_{2j-1}}(p)dx_{2j}+\frac{\pr\phi}{\pr x_{2j}}(p)dx_{2j-1}\bigr)\Bigr).
\end{equation} 
From \eqref{e-patII} and \eqref{e-patIII}, the proposition follows.
\end{proof}

Let $\sigma=d\xi\wedge dx$ denote the canonical two form on $T^*D$. We are 
therefore interested in whether $\sigma$ is non-degenerate at $\rho\in\Sigma$. 
We recall that $\sigma$ is non-degenerate at $\rho\in\Sigma$ if $\sigma(u,v)=0$ 
for all $v\in\Complex T_\rho\Sigma$, where $u\in\Complex T_\rho\Sigma$, then $u=0$. 

From now on, for any $f\in C^\infty(T^*D,\Complex)$, we write $H_f$ to denote the Hamilton field of $f$. That is, in local symplectic coordinates $(x, \xi)$, 
\[H_f=\sum^{2n-1}_{j=1}\Bigr(\frac{\pr f}{\pr\xi_j}\frac{\pr}{\pr x_j}-\frac{\pr f}{\pr x_j}\frac{\pr}{\pr\xi_j}\Bigr).\] 
For $f, g\in C^\infty(T^*D,\Complex)$, 
$\set{f, g}$ denotes the Poisson bracket of $f$ and $g$. We recall that
$\set{f, g}=\sum^{2n-1}_{s=1}(\frac{\pr f}{\pr \xi_s}\frac{\pr g}{\pr x_s}-\frac{\pr f}{\pr x_s}\frac{\pr g}{\pr\xi_s})$. 

First, we need 
 
\begin{lem} \label{l-patI}
For $\rho=(p, \lambda_0\omega_0(p)-2{\rm Im\,}\ddbar_b\phi(p))\in\Sigma$, we have 
\begin{equation} \label{e-patIV} 
\sigma(H_{q_j}, H_{q_t})|_\rho=0,\ \ j, t=1,\ldots,n-1,
\end{equation} 
\begin{equation} \label{e-patV} 
\sigma(H_{\ol q_j}, H_{\ol q_t})|_\rho=0,\ \ j, t=1,\ldots,n-1,
\end{equation}
and 
\begin{equation} \label{e-patVI} 
\begin{split}
\sigma(H_{\ol q_j}, H_{q_t})|_\rho&=2i\lambda_0\mathcal{L}_p(\ol Z_j, Z_t)+i<[\ol Z_j, Z_t](p), \ddbar_b\phi(p)-\pr_b\phi(p)>\\
&-i(\ol Z_j Z_t+Z_t\ol Z_j)\phi(p),\ \ j, t=1,\ldots,n-1,
\end{split}
\end{equation} 
where $Z_j$, $j=1,\ldots,n-1$, are as in \eqref{e-msmilkIX} and $q_j$ is the semi-classical principal symbol of $Z_j+kZ_j(\phi)$, $j=1,\ldots,n-1$.
\end{lem} 

\begin{proof} 
We write $\rho=(p, \xi_0)$. It is straightforward to see that 
\begin{equation} \label{s1-e33pat} 
\sigma(H_{q_j}, H_{q_t})|_\rho=\set{q_j, q_t}(\rho)=-<[Z_j, Z_t](p), \xi_0>
+i[Z_j,Z_t]\phi(p).
\end{equation} 
We have 
\begin{equation} \label{s1-e34pat}
\begin{split}
&<[Z_j, Z_t](p), \xi_0>=<[Z_j, Z_t](p), \lambda_0\omega_0(p)-2{\rm Im\,}\ddbar_b\phi(p)>\\
&=\lambda_0<[Z_j, Z_t](p), \omega_0(p)>+i<[Z_j, Z_t](p), \ddbar_b\phi(p)-\pr_b\phi(p)>.
\end{split}
\end{equation} 
Since $[Z_j, Z_t](p)\in T^{0,1}_pX$, we have 
\begin{equation} \label{s1-e35pat}
<[Z_j, Z_t](p), \omega_0(p)>=0
\end{equation} 
and 
\begin{equation} \label{s1-e36pat} 
<[Z_j, Z_t](p), \pr_b\phi(p)>=0.
\end{equation}
Thus, 
\begin{equation} \label{s1-e37pat}
<[Z_j, Z_t](p), \ddbar_b\phi(p)-\pr_b\phi(p)>=<[Z_j, Z_t](p), \ddbar_b\phi(p)>=[Z_j, Z_t]\phi(p). 
\end{equation} 
From \eqref{s1-e34pat}, \eqref{s1-e35pat} and \eqref{s1-e37pat}, we get 
\[<[Z_j, Z_t](p), \xi_0>=i[Z_j, Z_t]\phi(p).\]
Combining this with \eqref{s1-e33pat}, we get \eqref{e-patIV}. The proof of \eqref{e-patV} is the same.

As \eqref{s1-e33pat}, it is straightforward to see that
\begin{equation} \label{s1-e38pat} 
\sigma(H_{\ol q_j}, H_{q_t})|_\rho=\set{\ol q_j, q_t}(\rho)=<[\ol Z_j, Z_t](p), \xi_0>
-i(\ol Z_jZ_t+Z_t\ol Z_j)\phi(p),
\end{equation} 
where $j, t=1,\ldots,n-1$. We have 
\begin{equation} \label{s1-e39pat} 
\begin{split}
&<[\ol Z_j, Z_t](p), \xi_0>=<[\ol Z_j, Z_t](p), \lambda_0\omega_0(p)-2{\rm Im\,}\ddbar_b\phi(p)>\\
&=\lambda_0<[\ol Z_j, Z_t](p), \omega_0(p)>+i<[\ol Z_j, Z_t](p), \ddbar_b\phi(p)-\pr_b\phi(p)>\\
&=2i\lambda_0\mathcal{L}_p(\ol Z_j, Z_t)+
i<[\ol Z_j, Z_t](p), \ddbar_b\phi(p)-\pr_b\phi(p)>.
\end{split}
\end{equation} 
Combining \eqref{s1-e39pat} with \eqref{s1-e38pat}, \eqref{e-patVI} follows.
\end{proof}

The following is well-known (see Lemma 4.1 in~\cite{HM09}) 

\begin{lem}\label{l-dhI}
For any $U, V\in T^{1, 0}_pX$, pick $\mU, \mV\in C^\infty(D,T^{1, 0}X)$ that satisfy $\mU(p)=U$,
$\mV(p)=V$. Then,
\begin{equation} \label{e-dhI}
M^\phi_p(U, \ol V)=-\big\langle\big[\,\mU, \ol{\mV}\,\big](p), \ddbar_b\phi(p)-\pr_b\phi(p)\big\rangle
+\big(\mU\ol{\mV}+\ol{\mV}\mU\big)\phi(p).
\end{equation}
\end{lem}

Now, we can prove 

\begin{thm} \label{t-dhI} 
$\sigma$ is non-degenerate at $\rho=(p, \lambda_0\omega_0(p)-2{\rm Im\,}\ddbar_b\phi(p))\in\Sigma$ if and only if the Hermitian quadratic form $M^\phi_p-2\lambda_0\mathcal{L}_p$ is non-degenerate.
\end{thm}

\begin{proof} 
Note that
\[\Sigma=\set{(x, \xi)\in T^*D;\, q_j(x, \xi)=\ol q_j(x, \xi)=0,\ \ j=1,\ldots,n-1}.\]
Let $\Complex T_\rho\Sigma$ and $\Complex T_\rho(T^*D)$ be the
complexifications of $T_\rho\Sigma$ and $T_\rho(T^*D)$ respectively. 
Let $T_\rho\Sigma^\bot$ be the orthogonal to
$\Complex T_\rho\Sigma$ in $\Complex T_\rho(T^*D)$ with
respect to the canonical two form $\sigma$. We notice that ${\rm dim}_\Complex T_\rho\Sigma^\bot=2n-2$. It is easy to check that 
\[\sigma(v,H_{q_j})|_\rho=<dq_j(\rho), v>,\ \ \sigma(v,H_{\ol q_j})|_\rho=<d\ol q_j(\rho), v>,\] 
$j=1,\ldots,n-1$, $v\in\Complex T_\rho(T^*D)$. Thus, if $v\in \Complex T_\rho\Sigma$, we get 
$\sigma(H_{q_j}, v)|_\rho=0$, $\sigma(H_{\ol q_j}, v)|_\rho=0$, $j=1,\ldots,n-1$. We conclude that
$H_{q_1},\ldots,H_{q_{n-1}},H_{\ol q_1},\ldots,H_{\ol q_{n-1}}$
is a basis for $T_\rho\Sigma^\bot$. 

Now, we assume that $M^\phi_p-2\lambda_0\mathcal{L}_p$ is non-degenerate. Let $\nu\in \Complex T_\rho\Sigma\bigcap T_\rho\Sigma^\bot$. We write $\nu=\sum^{n-1}_{j=1}(\alpha_jH_{q_j}(\rho)+\beta_jH_{\ol q_j}(\rho))$. Since $\nu\in\Complex T_\rho\Sigma$, we have 
\[\sigma(\nu, H_{q_t})|_\rho=\sigma(\nu, H_{\ol q_t})|_\rho=0,\] 
$t=1,\ldots,n-1$. In view of \eqref{e-patIV}, \eqref{e-patV}, \eqref{e-patVI} and \eqref{e-dhI}, we see that 
\begin{equation} \label{s1-e45dh} 
\begin{split}
\sigma(\nu, H_{q_t})|_\rho&=\sum^{n-1}_{j=1}\beta_j\Bigr(2i\lambda_0\mathcal{L}_p(\ol Z_j, Z_t)-iM^\phi_p(\ol Z_j, Z_t)\Bigr)\\
&=2i\lambda_0\mathcal{L}_p(Y,Z_t)-iM^\phi_p(Y,Z_t)=0, 
\end{split}
\end{equation} 
for all $t=1,\ldots,n-1$, where $Y=\sum^{n-1}_{j=1}\beta_j\ol Z_j(p)\in T^{1,0}_pX$. Since $-M^\phi_p+2\lambda_0\mathcal{L}_p$ is non-degenerate, we get 
$Y=0$. Thus, $\beta_j=0$, $j=1,\ldots,n-1$. Similarly, we can repeat the process above to show that $\alpha_j=0$, $j=1,\ldots,n-1$. We conclude that $\Complex T_\rho\Sigma\bigcap T_\rho\Sigma^\bot=0$. Hence $\sigma$ is non-degenerate at 
$\rho$.

Conversely, we assume that $\sigma$ is non-degenerate at $\rho$. If for some $Y\in T^{1,0}_pX$, we have $M^\phi_p(Y, \ol Z)-2\lambda_0\mathcal{L}_p(Y,\ol Z)=0$ for all 
$Z\in T^{1,0}_pX$. We write $Y=\sum^{n-1}_{j=1}\beta_j\ol Z_j(p)$.
As before, we can show that $\sigma(\sum^{n-1}_{j=1}\beta_jH_{\ol q_j}, H_{q_t})|_\rho=0$
and $\sigma(\sum^{n-1}_{j=1}\beta_jH_{\ol q_j}, H_{\ol q_t})|_\rho=0$
for all $t=1,\ldots,n-1$. Thus, $\sum^{n-1}_{j=1}\beta_jH_{\ol q_j}\in(T_\rho\Sigma^{\bot})^{\bot}=\Complex T_\rho\Sigma$. Hence, $\sum^{n-1}_{j=1}\beta_jH_{\ol q_j}\in\Complex T_\rho\Sigma\bigcap T_\rho\Sigma^\bot$. 
Since $\sigma$ is non-degenerate at $\rho$, we get $\beta_j=0$, $j=1,\ldots,n-1$. 
Thus, $M^\phi_p-2\lambda_0\mathcal{L}_p$ is non-degenerate. The theorem follows.
\end{proof} 

\section{The heat equation for the local operatot $\Box^{(q)}_s$} \label{s-thef}

In this section, we will introduce the local operator $\Box^{(q)}_s$. 
The goal of this section is to find a microlocal partial inverse and an approximate kernel for $\Box^{(q)}_s$ in some non-degenerate part of the characteristic manifold of $\Box^{(q)}_s$. In the next section, we will reduce the semi-classical analysis of the Kohn Laplacian $\Box^{(q)}_{b,k}$ to the microlocal analysis of the local operator $\Box^{(q)}_s$. 

\subsection{$\Box^{(q)}_s$ and the eikonal equation for $\Box^{(q)}_s$}\label{s-bat}

We first introduce some notations. Let $\Omega$ be an open set in $\Real^N$ and let $f$, $g$ be positive continuous functions on $\Omega$. We write $f\asymp g$ if for every compact set $K\subset\Omega$ there is a constant $c_K>0$ such that $f\leq c_Kg$ and $g\leq c_Kf$ on $K$. 

Let $s$ be a local trivializing section of $L$ on an open subset $D\Subset X$ and $\abs{s}^2_{h^L}=e^{-2\phi}$. In this section, we work with some real local coordinates $x=(x_1,\ldots,x_{2n-1})$ defined on $D$. We write $\xi=(\xi_1,\ldots,\xi_{2n-1})$ or $\eta=(\eta_1,\ldots,\eta_{2n-1})$ to denote the dual coordinates of $x$. We consider the domain $\hat D:=D\times\Real$. We write $\hat x:=(x,x_{2n})=(x_1,x_2,\ldots,x_{2n-1},x_{2n})$ to denote the coordinates of $D\times\Real$, where $x_{2n}$ is the coordinate of $\Real$. We write $\hat\xi:=(\xi,\xi_{2n})$ or $\hat\eta:=(\eta,\eta_{2n})$ to denote the dual coordinates of $\hat x$, where $\xi_{2n}$ and $\eta_{2n}$ denote the dual coordinate of $x_{2n}$. We shall use the following notations: $<x,\eta>:=\sum^{2n-1}_{j=1}x_j\eta_j$, $<x,\xi>:=\sum^{2n-1}_{j=1}x_j\xi_j$, 
$<\hat x,\hat\eta>:=\sum^{2n}_{j=1}x_j\eta_j$, $<\hat x,\hat\xi>:=\sum^{2n}_{j=1}x_j\xi_j$. 

Let $T^{*0,q}\hat D$ be the bundle with fiber 
\[T^{*0,q}_{\hat x}\hat D:=\set{u\in T^{*0,q}_xD, \hat x=(x,x_{2n})}\] 
at $\hat x\in\hat D$. From now on, for every point $\hat x=(x,x_{2n})\in\hat D$, we identify $T^{*0,q}_{\hat x}\hat D$ with $T^{*0,q}_xX$. 
Let $\langle\,\cdot\,|\,\cdot\,\rangle$ be the Hermitian metric on $\Complex T^*\hat D$ given by $\langle\,\hat\xi\,|\,\hat\eta\,\rangle=\langle\,\xi\,|\,\eta\,\rangle+\xi_{2n}\ol{\eta_{2n}}$, $(\hat x,\hat\xi), (\hat x,\hat\eta)\in\Complex T^*\hat D$. Let $\Omega^{0,q}(\hat D)$ denote the space of smooth sections of $T^{*0,q}\hat D$ over $\hat D$ and put $\Omega^{0,q}_0(\hat D):=\Omega^{0,q}(\hat D)\bigcap\mathscr E'(\hat D,T^{*0,q}\hat D)$. Using the identification 
\[ku(x)=e^{-ikx_{2n}}(-i\frac{\pr}{\pr x_{2n}}(e^{ikx_{2n}}u(x)),\ \ \forall u\in\Omega^{0,q}(D),\]
we consider the following operators
\begin{equation}\label{s3-lkmidhI}
\begin{split}
&\ddbar_s:\Omega^{0,r}(\hat D)\To\Omega^{0,r+1}(\hat D),\ \ \ddbar_{s,k}u=e^{-ikx_{2n+1}}\ddbar_s(ue^{ikx_{2n}}),\ \ \forall u\in\Omega^{0,r}(D),\\
&\ol{\pr}^*_s:\Omega^{0,r+1}(\hat D)\To\Omega^{0,r}(\hat D),\ \ \ol{\pr}^*_{s,k}u=e^{-ikx_{2n+1}}\ol{\pr}^*_s(ue^{ikx_{2n}}),\ \ \forall u\in\Omega^{0,r+1}(D),
\end{split}
\end{equation}
where $r=0,1,\ldots,n-1$ and $\ddbar_{s,k}$, $\ol{\pr}^*_{s,k}$ are given by \eqref{s2-emsmilkI}. From \eqref{e-msmilkIV} and \eqref{e-msmilkV}, it is easy to see that 
\begin{equation}\label{e-dhmpI}
\begin{split}
&\ddbar_s=\sum^{n-1}_{j=1}\Bigr(e^\wedge_j\circ(Z_j-iZ_j(\phi)\frac{\pr}{\pr x_{2n}})+
(\ddbar_b e_j)^\wedge e^{\wedge,*}_j\Bigr),\\
&\ol{\pr}^*_{s}=\sum^{n-1}_{j=1}\Bigr(e^{\wedge,*}_j\circ(Z^*_j-i\ol Z_j(\phi)\frac{\pr}{\pr x_{2n}})+
e^\wedge_j(\ddbar_b e_j)^{\wedge,*}\Bigr),
\end{split}
\end{equation}
where $Z_1,\ldots,Z_{n-1}$, $Z^*_1,\ldots,Z^*_{n-1}$ and $e_1,\ldots,e_{n-1}$ are as in Proposition~\ref{s2-pmsmilkI}. Put 
\begin{equation}\label{e-dhmpII}
\Box^{(q)}_s:=\ddbar_s\ol{\pr}^*_s+\ol{\pr}^*_s\ddbar_s:\Omega^{0,q}(\hat D)\To\Omega^{0,q}(\hat D).
\end{equation}
From \eqref{s3-lkmidhI}, we have 
\begin{equation}\label{e-dhmpIII}
\Box^{(q)}_{s,k}u=e^{-ikx_{2n+1}}\Box^{(q)}_s(ue^{ikx_{2n}}),\ \ \forall u\in\Omega^{0,q}(D),
\end{equation}
where $\Box^{(q)}_{s,k}$ is given by \eqref{e-msmilkVI}. 

Let $u(x)\in\Omega^{0,q}_0(\hat D)$.
Note that
\[k\int e^{-ikx_{2n}}u(x)dx_{2n}=\int i\frac{\pr}{\pr x_{2n}}(e^{-ikx_{2n}})u(x)dx_{2n}=\int e^{-ikx_{2n}}\bigr(-i\frac{\pr u}{\pr x_{2n}}(x)\bigr)dx_{2n}.\]
From this observation and explicit formulas of $\ddbar_{s,k}$,
$\ol\pr^*_{s,k}$, $\ddbar_s$ and $\ol\pr^*_s$ (see \eqref{e-msmilkIV},
\eqref{e-msmilkV} and \eqref{e-dhmpI}), we conclude that
\begin{equation} \label{s3-e9bis}
\Box^{(q)}_{s,k}\bigr(\int e^{-ikx_{2n}}u(x)dx_{2n}\bigr)=\int e^{-ikx_{2n}}(\Box^{(q)}_su)(x)dx_{2n},
\end{equation}
for all $u(x)\in\Omega^{0,q}_0(\hat D)$.

From \eqref{e-dhmpI} and \eqref{e-dhmpII}, we can repeat the proof 
of Proposition~\ref{s2-pmsmilkI} and conclude that

\begin{prop} \label{p-dhmpI} 
With the notations used before, we have
\begin{equation} \label{e-dhmpIV} 
\begin{split}
\Box^{(q)}_{s}&=\ddbar_{s}\ol{\pr}^*_{s}+\ol{\pr}^*_{s}\ddbar_{s} \\
       &=\sum^{n-1}_{j=1}(Z^*_j-i\ol Z_j(\phi)\frac{\pr}{\pr x_{2n}})(Z_j-iZ_j(\phi)\frac{\pr}{\pr x_{2n}})\\
       &+\sum^{n-1}_{j,t=1}e^\wedge_je^{\wedge,*}_t\circ[Z_j-iZ_j(\phi)\frac{\pr}{\pr x_{2n}}, Z^*_t-i\ol Z_t(\phi)\frac{\pr}{\pr x_{2n}}]\\
&+\varepsilon(Z-iZ(\phi)\frac{\pr}{\pr x_{2n}})+\varepsilon(Z^*-i\ol Z(\phi)\frac{\pr}{\pr x_{2n}})+\mbox{zero order terms},
\end{split}
\end{equation}
where $\varepsilon(Z-iZ(\phi)\frac{\pr}{\pr x_{2n}})$ denotes remainder terms of the form $\sum a_j(Z_j-iZ_j(\phi)\frac{\pr}{\pr x_{2n}})$ with $a_j$ smooth, matrix-valued, for all $j$, and similarly for $\varepsilon(Z^*-i\ol Z(\phi)\frac{\pr}{\pr x_{2n}})$.
\end{prop}  

In this section, we will study the heat equation $\pr_t+\Box^{(q)}_s$. Until further notice, we fix $q\in\set{0,1,\ldots,n-1}$. First, we consider the problem 
\begin{equation} \label{e-dhmpV}
\left\{ \begin{array}{ll}
(\pr_t+\Box^{(q)}_s)u(t,\hat x)=0  & \text{ in }\Real_+\times\hat D,  \\
u(0,\hat x)=v(\hat x). \end{array}\right.
\end{equation}

We need

\begin{defn} \label{d-dhikbmiI}
Let $0\leq q_1\leq n-1$, $q_1\in\mathbb N_0$. We say that $a(t,\hat x,\hat \eta)\in C^\infty(\ol\Real_+\times T^*\hat D,T^{*0,q_1}\hat D\boxtimes T^{*0,q}\hat D)$
is quasi-homogeneous of
degree $j$ if $a(t,\hat x,\lambda\hat\eta)=\lambda^ja(\lambda t,\hat x,\hat\eta)$ for all $\lambda>0$.
\end{defn} 

\begin{defn} \label{d-dhikbmiII}
Let $0\leq q_1\leq n-1$, $q_1\in\mathbb N_0$.
We say that $b(\hat x,\hat \eta)\in C^\infty(T^*\hat D,T^{*0,q_1}\hat D\boxtimes T^{*0,q}\hat D)$
is positively homogeneous of degree $j$ if $b(\hat x,\lambda\hat\eta)=\lambda^jb(\hat x,\hat\eta)$ for all $\lambda>0$.
\end{defn} 

Let $0\leq q_1\leq n-1$, $q_1\in\mathbb N_0$.
We look for an approximate solution of \eqref{e-dhmpV} of the form $u(t,\hat x)=A(t)v(\hat x)$,
\begin{equation} \label{e-dhmpVI}
A(t)v(\hat x)=\frac{1}{(2\pi)^{2n}}\iint\!\! e^{i(\psi(t,\hat x,\hat\eta)-<\hat y,\hat \eta>)}a(t,\hat x,\hat\eta)v(\hat y)d\hat yd\hat\eta
\end{equation}
where formally $a(t,\hat x,\hat \eta)\sim\sum^\infty_{j=0}a_j(t,\hat x,\hat \eta)$, 
$a_j(t, \hat x, \hat\eta)\in C^\infty(\ol\Real_+\times T^*\hat D,T^{*0,q_1}\hat D\boxtimes T^{*0,q}\hat D)$, $a_j(t,\hat x,\hat \eta)$ is a quasi-homogeneous function of degree $-j$.
The phase $\psi(t,\hat x,\hat\eta)$ should solve the eikonal equation 
\begin{equation} \label{e-dhmpVIII}
\left\{ \begin{array}{ll}
 \frac{\displaystyle\pr\psi}{\displaystyle\pr t}-i\hat p_0(\hat x,\psi'_{\hat x})=
   O(\abs{{\rm Im\,}\psi}^N), & \forall N\geq0,   \\
 \psi|_{t=0}=<\hat x,\hat\eta> \end{array}\right.
\end{equation}
with ${\rm Im\,}\psi\geq0$, where $\hat p_0$ denotes the principal symbol of $\Box^{(q)}_s$. From \eqref{e-dhmpIV}, we have 
\begin{equation}\label{e-dhmpIX}
\hat p_0=\sum^{n-1}_{j=1}\ol{\hat q}_j\hat q_j,
\end{equation}
where $\hat q_j$ is the principal sympol of $Z_j-iZ_j(\phi)\frac{\pr}{\pr x_{2n}}$, $j=1,\ldots,n-1$. The characteristic 
manifold $\hat\Sigma$ of $\Box^{(q)}_s$ is given by 
\begin{equation}\label{e-dhmpX}
\hat\Sigma=\set{(\hat x,\hat\xi)\in T^*\hat D;\, \hat q_1(\hat x,\hat\xi)=\cdots=\hat q_{n-1}(\hat x,\hat\xi)=\ol{\hat q}_1(\hat x,\hat\xi)=\cdots=\ol{\hat q}_{n-1}(\hat x,\hat\xi)=0}.
\end{equation}
From \eqref{e-dhmpX}, we see that $\hat p_0$ vanishes to second order at $\hat\Sigma$. Let $\hat\sigma$ denote the canonical two form on $T^*\hat D$. We can repeat the proofs of Proposition~\ref{s2-pcrmiI} and Theorem~\ref{t-dhI} with minor changes and conclude that 

\begin{thm}\label{t-dhlkmp}
We have 
\begin{equation} \label{e-dhmpXI} 
\hat\Sigma=\set{(\hat x, \hat\xi)\in T^*\hat D;\, \hat\xi=(\lambda\omega_0(x)-2{\rm Im\,}\ddbar_b\phi(x)\xi_{2n},\xi_{2n}),\lambda\in\Real}.
\end{equation}

Moreover, $\hat\sigma$ is non-degenerate at $\hat\rho=((p,x_{2n}), (\lambda_0\omega_0(p)-2{\rm Im\,}\ddbar_b\phi(p),\xi_{2n}))\in\hat\Sigma$ if and only if the Hermitian quadratic form $\xi_{2n}M^\phi_p-2\lambda_0\mathcal{L}_p$ is non-degenerate.
\end{thm} 

Until further notice, we assume that 
\begin{equation}\label{e-batlkI}
\begin{split}
&\mbox{there exist $x_0\in D$ and $\lambda_0\in\Real$ such that $M^\phi_{x_0}-2\lambda_0\mathcal{L}_{x_0}$ is non-degenerate}\\
&\mbox{of constant signature $(n_-,n_+)$ at each point of $D$}. 
\end{split}
\end{equation}

Let $V$ be a bounded open set of $T^*D$ with  $\ol V\subset T^*D$ and
\begin{equation}\label{e-dhmpXII}
\ol V\bigcap\Sigma\subset\Sigma',
\end{equation}
where $\Sigma'$ is given by \eqref{e-dhmpXIa}. 
Put 
\begin{equation}\label{e-dhmpXIII}
U=\set{(\hat x,\hat\xi)\in T^*\hat D;\, \hat\xi=(\xi_{2n}\xi,\xi_{2n}), (x,\xi)\in V, \xi_{2n}>0}.
\end{equation}
$U$ is a conic open set of $T^*\hat D$ and 
\begin{equation}\label{e-dhlkmimI}
\begin{split}
U\bigcap\hat\Sigma\subset&\{(\hat x,(\lambda\omega_0(x)-2{\rm Im\,}\ddbar_b\phi(x)\xi_{2n},\xi_{2n}));\, \mbox{$\xi_{2n}M^\phi_x-2\lambda\mathcal{L}_x$ is non-degenerate}\\
&\quad\mbox{of constant signature $(n_-,n_+)$}\}. 
\end{split}
\end{equation}
Since $\ol V\bigcap\Sigma\Subset\Sigma'$, it is not difficult to see that there is a constant $\mu>0$ such that
\begin{equation}\label{e-dhlkmimII}
\inf\{\mbox{$\abs{\lambda}$;\, $\lambda$: eigenvalue of $\xi_{2n}M^\phi_x-2\lambda\mathcal{L}_x$, $(\hat x,\hat\xi)\in U\bigcap\hat\Sigma$}\}\geq\mu\xi_{2n}.
\end{equation}

Until further notice, we work in $U$. Since $\hat\sigma$ is non-degenerate at each point of $U\bigcap\hat\Sigma$, \eqref{e-dhmpVIII} can be solved with ${\rm Im\,}\psi\geq0$ on $U$. More precisely, we have the following 

\begin{thm} \label{t-dhlkmimI}
There exists $\psi(t,\hat x,\hat\eta)\in C^\infty(\ol\Real_+\times U)$ such
that $\psi(t,\hat x,\hat\eta)$ is quasi-homogeneous of degree $1$ and ${\rm Im\,}\psi\geq 0$ and such that \eqref{e-dhmpVIII} holds where the error term is uniform on every set of the form $[0,T]\times K$ with $T>0$ and $K\subset U$ compact.
Furthermore, $\psi$ is unique up to a term which is
$O(\abs{{\rm Im\,}\psi}^N)$ locally uniformly for every $N$ and
\begin{equation}\label{e-dhlkmimIII}
\begin{split}
\mbox{$\psi(t,\hat x,\hat \eta)=<\hat x,\hat \eta>$ on $\hat\Sigma\bigcap U$},\\
\mbox{$d_{\hat x,\hat\eta}(\psi-<\hat x,\hat\eta>)=0$ on $\hat\Sigma\bigcap U$}.
\end{split}
\end{equation}
Moreover, we have
\begin{equation} \label{e-dhlkmimIIIa}
{\rm Im\,}\psi(t, \hat x,\hat\eta)\asymp\Big(\abs{\hat\eta}\frac{t\abs{\hat\eta}}{1+t\abs{\eta}}\Big)\Big({\rm dist\,}
\big((\hat x, \frac{\hat\eta}{\abs{\hat\eta}}),\hat\Sigma\big)\Big)^2,\ \ t\geq0,\ \ (\hat x,\hat\eta)\in U.
\end{equation}
\end{thm}

\begin{thm} \label{t-dhlkmimII}
There exists a function $\psi(\infty,\hat x,\hat\eta)\in C^\infty(U)$ with a
uniquely determined Taylor expansion at each point of\, $\hat\Sigma\bigcap U$ such that 
$\psi(\infty,\hat x,\hat\eta)$ is positively homogeneous of degree $1$ and
for every compact set $K\subset U$ there is a $c_K>0$ such that   
${\rm Im\,}\psi(\infty,\hat x,\hat \eta)\geq c_K\abs{\hat\eta}\Big({\rm dist\,}\big((\hat x,\frac{\hat\eta}{\abs{\hat\eta}}),\hat\Sigma\big)\Big)^2$,  
$d_{\hat x,\hat\eta}(\psi(\infty, \hat x, \hat\eta)-<\hat x,\hat\eta>)=0\text{ on } U\bigcap\hat\Sigma$.
If $\lambda\in C(U)$, $\lambda>0$ and $\lambda(\hat x,\hat\xi)<\min\lambda_j(\hat x,\hat\xi)$, for all $(\hat x,\hat\xi)=(\hat x,(\lambda\omega_0(x)-2{\rm Im\,}\ddbar_b\phi(x)\xi_{2n},\xi_{2n}))\in\hat\Sigma\bigcap U$, where $\lambda_j(\hat x,\hat\xi)$ are the eigenvalues of the Hermitian quadratic form $\xi_{2n}M^\phi_x-2\lambda\mathcal{L}_x$, then the solution $\psi(t,\hat x,\hat\eta)$ of \eqref{e-dhmpVIII} can be chosen so that for every compact set $K\subset U$ and all indices $\alpha$, $\beta$, $\gamma$,
there is a constant $c_{\alpha,\beta,\gamma,K}>0$ such that
\begin{equation} \label{e-dhlkmimIV}
\abs{\pr^\alpha_{\hat x}\pr^\beta_{\hat\eta}\pr^\gamma_t(\psi(t,\hat x,\hat\eta)-\psi(\infty,\hat x,\hat\eta))}
\leq c_{\alpha,\beta,\gamma,K}e^{-\lambda(\hat x,\hat \eta)t} \text{ on }\ol\Real_+\times K.
\end{equation}
\end{thm} 

For the proofs of Theorem~\ref{t-dhlkmimI} and Theorem~\ref{t-dhlkmimII}, we refer the reader to Menikoff-Sj\"{o}strand~\cite{MS78} and~\cite{Hsiao08}. Put 
\begin{equation}\label{e-gedhanmpI}
{\rm diag\,}'\Bigr((U\bigcap\hat\Sigma)\times(U\bigcap\hat\Sigma)\Bigr):=\set{(\hat x,\hat x,\hat\xi,-\hat\xi);\, (\hat x,\hat\xi)\in U\bigcap\hat\Sigma}.
\end{equation} 
We also need the following which is also well-known (see Proposition 3.5 of part {\rm I} in~\cite{Hsiao08}) 

\begin{thm}\label{t-dhlkmimIII}
The two phases $\psi(\infty, \hat x, \hat \eta)-<\hat y, \hat \eta>$,
$-\ol\psi(\infty, \hat y, \hat\eta)+<\hat x, \hat\eta>$
are equivalent for classical symbols at every point of ${\rm diag\,}'\Bigr((U\bigcap\hat\Sigma)\times(U\bigcap\hat\Sigma)\Bigr)$ in the sense of Melin-Sj\"{o}strand~\cite{MS74}.
\end{thm}

From \eqref{e-dhmpIV}, we can check that the principal symbol $\hat p_0$ has the following property:
\begin{equation}\label{e-dhlkmimV}
\hat p_0((x,x_{2n}+\alpha),\hat\eta)=\hat p_0(\hat x,\hat\eta),\ \ \forall \alpha\in\Real.
\end{equation}
Fix $\alpha\in\Real$, we consider 
\[\Td\psi(t,\hat x,\hat\eta):=\psi(t,(x,x_{2n}+\alpha),\hat\eta)-\alpha\eta_{2n}.\] 
From \eqref{e-dhlkmimV}, it is not difficult to see that $\Td\psi(t,\hat x,\hat\eta)$ solves \eqref{e-dhmpVIII}. 
From Theorem~\ref{t-dhlkmimI}, we see that 
\[\Td\psi(t,\hat x,\hat\eta)-\psi(t,\hat x,\hat\eta)=\bigr(\psi(t,(x,x_{2n}+\alpha),\eta)-(x_{2n}+\alpha)\eta_{2n}\bigr)
-\bigr(\psi(t,\hat x,\hat\eta)-x_{2n}\eta_{2n}\bigr)\]
vanishes to infinite order at $\hat\Sigma\bigcap U$, for all $\alpha\in\Real$. This means that the Taylor expansions of 
$\psi(t,\hat x,\hat\eta)-x_{2n}\eta_{2n}$ at $\hat\Sigma\bigcap U$ do not depend on $x_{2n}$. Thus, $\frac{\pr\psi}{\pr x_{2n}}(t,\hat x,\hat\eta)-\hat\eta_{2n}$ vanishes to infinite order at $\hat\Sigma\bigcap U$. We conclude that $\psi(t,\hat x,\hat\eta)-\bigr(\psi(t,(x,0),\hat\eta)+x_{2n}\eta_{2n}\bigr)$ vanishes to infinite order at $\hat\Sigma\bigcap U$. Since we only need to consider Taylor expansions at $\hat\Sigma\bigcap U$, from now on, we assume that 
\begin{equation}\label{e-dhlkmimVI}
\psi(t,\hat x,\hat\eta)=\psi(t,(x,0),\hat\eta)+x_{2n}\eta_{2n}.
\end{equation}
Thus, 
\begin{equation}\label{e-dhlkmimVII}
\psi(\infty,\hat x,\hat\eta)=\psi(\infty,(x,0),\hat\eta)+x_{2n}\eta_{2n}.
\end{equation}

\subsection{The transport equations for $\Box^{(q)}_s$} \label{s-ttef}

We let the full symbol of $\Box^{(q)}_s$ be:
\[\text{full symbol of } \Box^{(q)}_s=\sum^2_{j=0}\hat p_j(\hat x,\hat\xi),\]
where $\hat p_j(\hat x,\hat\xi)$ is positively homogeneous of order $2-j$. We apply $\pr_t+\Box^{(q)}_s$ formally
under the integral in \eqref{e-dhmpVI} and then introduce the asymptotic expansion of
$\Box^{(q)}_s(ae^{i\psi})$~(see page 148 of \cite{MS74}). Setting $(\pr_t+\Box^{(q)}_s)(ae^{i\psi})\sim 0$ and regrouping
the terms according to the degree of quasi-homogeneity. We obtain the transport equations
\begin{equation} \label{e-dhlkmimVIII}
\left\{ \begin{array}{ll}
 T(t,\hat x,\hat\eta,\pr_t,\pr_{\hat x})a_0=O(\abs{{\rm Im\,}\psi}^N),\; \forall N,   \\
 T(t,\hat x,\hat\eta,\pr_t,\pr_{\hat x})a_j+R_j(t,\hat x,\hat \eta,a_0,\ldots,a_{j-1})= O(\abs{{\rm Im\,}\psi}^N),\; \forall N.
 \end{array}\right.
\end{equation}
Here
\[T(t,\hat x,\hat\eta,\pr_t,\pr_{\hat x})=\pr_t-i\sum^{2n}_{j=1}\frac{\pr\hat p_0}{\pr\xi_j}(\hat x,\psi'_{\hat x})\frac{\pr}{\pr x_j}+q(t,\hat x,\hat\eta)\]
where
\[q(t,\hat x,\hat\eta)=\hat p_1(\hat x,\psi'_{\hat x})+\frac{1}{2i}\sum^{2n}_{j,t=1}\frac{\pr^2\hat p_0(\hat x,\psi'_{\hat x})}
    {\pr\xi_j\pr\xi_t}\frac{\pr^2\psi(t,\hat x,\hat\eta)}{\pr x_j\pr x_t}\]
and $R_j$ is a linear differential operator acting on $a_0,a_1,\ldots,a_{j-1}$. 
We note that $q(t,\hat x,\hat \eta)\To q(\infty,\hat x,\hat\eta)$ exponentially fast in the sense of \eqref{e-dhlkmimIV} 
and the same is true for the coefficients of $R_j$, for all $j$.

We pause and introduce some notations. The subprincipal symbol of $\Box^{(q)}_s$ at $(\hat x,\hat\xi)\in\hat\Sigma$ is given by 
\begin{equation}\label{e-dhlkmimIX}
\hat p^s_0(\hat x,\hat\xi)=\hat p_1(\hat x,\hat\xi)+\frac{i}{2}\sum^{2n}_{j=1}\frac{\pr^2\hat p_0(\hat x,\hat\xi)}{\pr \hat x_j\pr\hat\xi_j}\in T^{*0,q}_{\hat x}\hat D\boxtimes T^{*0,q}_{\hat x}\hat D.
\end{equation}
Since $\hat\Sigma$ is doubly characteristic, it is well-known that the subprincipal symbol of $\Box^{(q)}_s$ is invariantly defined on $\hat\Sigma$ (see page 83 in H\"{o}rmander~\cite{Hor85}). The fundamental matrix of $\hat p_0$ at 
$\hat\rho\in\hat\Sigma$ is the linear map $F(\hat\rho)$ on $T_{\hat\rho}(T^*\hat D)$ defined by 
\begin{equation}\label{e-dhlkmimX}
\hat\sigma(t,F(\rho)s)=<t,\hat p_0''(\hat\rho)s>,\ \ t, s\in T_{\hat\rho}(T^*\hat D),
\end{equation}
where $\hat p''_0(\hat\rho)=\left(\begin{array}[c]{cc}
  \frac{\pr^2\hat p_0}{\pr\hat x\pr\hat x}(\hat\rho)& \frac{\pr^2\hat p_0}{\pr\hat\xi\pr\hat x}(\hat\rho) \\
  \frac{\pr^2\hat p_0}{\pr\hat x\pr\hat\xi}(\hat\rho)& \frac{\pr^2\hat p_0}{\pr\hat\xi\pr\hat\xi}(\hat\rho)
\end{array}\right)$.
For $\hat\rho\in\hat\Sigma$, let $\Td{\rm tr\,}F(\hat\rho):=\sum\abs{\mu_j}$, where $\pm i\mu_j$ are non-vanishing eigenvalues of $F(\hat\rho)$. For $\hat\rho=(\hat x,\hat\xi)\in\hat\Sigma\bigcap U$, put 
\begin{equation}\label{e-dhlkmimXI}
\inf(\hat p^s_0(\hat\rho)+\frac{1}{2}\Td{\rm tr\,}F(\hat\rho))=\inf\set{\lambda;\, \mbox{$\lambda$: eigenvalue of $\hat p^s_0(\hat\rho)+\frac{1}{2}\Td{\rm tr\,}F(\hat\rho)$}}
\end{equation}
and set 
\begin{equation}\label{e-gue1373}
N(\hat p^s_0(\hat\rho)+\frac{1}{2}\Td{\rm tr\,}F(\hat\rho))=\set{u\in T^{*0,q}_{\hat x}\hat D;\, (\hat p^s_0(\hat\rho)+\frac{1}{2}\Td{\rm tr\,}F(\hat\rho))u=0}.
\end{equation}

We return to our situation. We can repeat the proof of Proposition 4.3 in part {\rm I} of \cite{Hsiao08} with minor changes and obtain the following

\begin{thm} \label{t-dhplkmiI}
Let $0\leq q_1\leq n-1$, $q_1\in\mathbb N_0$. Let $c_j(\hat x, \hat\eta)\in C^\infty(U,T^{*0,q_1}\hat D\boxtimes T^{*0,q}\hat D)$, $j=0, 1,\ldots$, be positively homogeneous functions of degree $m-j$, $m\in\mathbb Z$. Then, we can find solutions $a_j(t, \hat x, \hat\eta)\in C^\infty(\ol\Real_+\times T^*\hat D,T^{*0,q_1}\hat D\boxtimes T^{*0,q}\hat D)$, $j=0, 1,\ldots$, 
of the system \eqref{e-dhlkmimVIII} with $a_j(0, \hat x, \hat\eta)=c_j(\hat x, \hat\eta)$, $j=0, 1,\ldots$,
where $a_j(t,\hat x,\hat\eta)$ is a quasi-homogeneous function of degree $m-j$ such that $a_j(t,\hat x,\hat\eta)$ has unique Taylor expansions
on $\hat\Sigma$, for all $j$. Furthermore, let $\lambda(\hat x,\hat\eta)\in C(U)$ and $\lambda(\hat x,\hat\eta)<\inf(\hat p^s_0(\hat x,\hat\eta)+\frac{1}{2}\Td{\rm tr\,}F(\hat x,\hat\eta))$, for all $(\hat x,\hat\eta)\in\hat\Sigma\bigcap U$. Then for all indices $\alpha,\beta,\gamma,j$ and every compact set $K\Subset\hat\Sigma\bigcap U$ there exists a
constant $c>0$ such that
\begin{equation} \label{e-dhlkmimXII}
\abs{\pr^\gamma_t\pr^\alpha_{\hat x}\pr^\beta_{\hat\eta}a_j(t,\hat x,\hat\eta)}\leq
ce^{-t\lambda(\hat x,\hat\eta)} \text{ on }\ol\Real_+\times K.
\end{equation}

Moreover, for every $\hat\rho_0=(\hat x_0,\hat\eta_0)\in\hat\Sigma\bigcap U$, 
\begin{equation} \label{e-gue1373I}
\begin{split}
&\mbox{$\lim\limits_{t\To\infty}a_0(t,\hat x_0,\hat\eta_0)$ exists},\\
&\bigr(\lim_{t\To\infty}a_0(t,\hat x_0,\hat\eta_0)\bigr)u\in N(\hat p^s_0(\hat\rho)+\frac{1}{2}\Td{\rm tr\,}F(\hat\rho)),\ \ \forall u\in T^{*0,q}_{\hat x}\hat D. 
\end{split}
\end{equation}
\end{thm} 

We are therefore interested in whether $\inf(\hat p^s_0(\hat\rho)+\frac{1}{2}\Td{\rm tr\,}F(\hat\rho))>0$, $\hat\rho\in U\bigcap\hat\Sigma$. We have the following 

\begin{thm} \label{t-dhplkmiII}
If $q=n_-$, then for all $(\hat x, \hat\xi)\in\hat\Sigma\bigcap U$, we have 
\begin{equation} \label{e-dhlkmimXIII}
\inf(\hat p^s_0(\hat x,\hat\xi)+\frac{1}{2}\Td{\rm tr\,}F(\hat x,\hat\xi))=0.
\end{equation}
If $q\neq n_-$, then there is a constant $\mu_0>0$ such that for all $(\hat x, \hat\xi)\in\hat\Sigma\bigcap U$, we have 
\begin{equation} \label{e-dhlkmimXIIIbis}
\inf(\hat p^s_0(\hat x,\hat\xi)+\frac{1}{2}\Td{\rm tr\,}F(\hat x,\hat\xi))>\mu_0\xi_{2n}.
\end{equation}
\end{thm} 

\begin{proof} 
First, we compute the subprincipal symbol $\hat p^s_0(\hat\rho)$, $\hat\rho\in\hat\Sigma$. 
For an operator of the form $(Z^*_j-i\ol Z_j(\phi)\frac{\pr}{\pr x_{2n}})(Z_j-iZ_j(\phi)\frac{\pr}{\pr x_{2n}})$ this subprincipal symbol is given by $-\frac{1}{2i}\{\hat q_j,\ol{\hat q}_j\}$
and the contribution from the double sum in
\eqref{e-dhmpIV} to the subprincipal symbol of $\Box^{(q)}_s$ is $\frac{1}{i}\sum^{n-1}_{j,t=1}e^\wedge_je^{\wedge, *}_t\circ\{\hat q_j,\ol{\hat q}_t\}$, where $\hat q_j$ is the principal symbol of $Z_j-iZ_j(\phi)\frac{\pr}{\pr x_{2n}}$ and $\{\hat q_j, \ol{\hat q}_t\}$ denotes the Poisson bracket of $\hat q_j$ and $\ol{\hat q}_t$. We recall that
$\{\hat q_j, \ol{\hat q}_t\}=\sum^{2n}_{s=1}(\frac{\pr\hat q_j}{\pr \xi_s}\frac{\pr\ol{\hat q}_t}{\pr x_s}-\frac{\pr\hat q_j}{\pr x_s}\frac{\pr\ol{\hat  q}_t}{\pr\xi_s})$.
We get the subprincipal symbol of $\Box^{(q)}_s$ on $\hat\Sigma$,
$\hat p^s_0=\sum^{n-1}_{j=1}-\frac{1}{2i}\{\hat q_j,\ol{\hat q}_j\}+\sum^{n-1}_{j,t=1}
e^\wedge_je^{\wedge, *}_t\frac{1}{i}\{\hat q_j,\ol{\hat q}_t\}$. 
For $\hat\rho=(\hat x,(\lambda\omega_0(x)-2{\rm Im\,}\ddbar_b\phi(x)\xi_{2n},\xi_{2n}))\in\hat\Sigma$, from the proof of Lemma~\ref{l-patI}, we see that 
\begin{equation}\label{e-mslknaIb}
\{\hat q_j,\ol{\hat q_t}\}(\hat\rho)=iM^\phi_x(\ol Z_j,Z_t)\xi_{2n}-2i\lambda\mathcal{L}_x(\ol Z_j,Z_t).
\end{equation} 
Thus, 
\begin{equation} \label{e-mslknaI}
\begin{split}
\hat p^s_0(\hat\rho)=&\sum^{n-1}_{j=1}-\frac{1}{2}\Bigr(M^\phi_x(\ol Z_j,Z_j)\xi_{2n}-2\lambda\mathcal{L}_x(\ol Z_j,Z_j)\Bigr)\\
&\quad+\sum^{n-1}_{j,t=1}
e^\wedge_je^{\wedge, *}_t\Bigr(M^\phi_x(\ol Z_j,Z_t)\xi_{2n}-2\lambda\mathcal{L}_x(\ol Z_j,Z_t)\Bigr),
\end{split}
\end{equation}
for all $\hat\rho=(\hat x,\hat\xi)=(\hat x,(\lambda\omega_0(x)-2{\rm Im\,}\ddbar_b\phi(x)\xi_{2n},\xi_{2n}))\in\hat\Sigma$. 

Now, we compute the fundamental matrix $F$ of $\hat p_0$ at $\hat\rho\in\hat\Sigma$. From now on, for any $f\in C^\infty(T^*\hat D)$, we write $H_f$ to denote the Hamilton field of $f$. We can choose the basis
$H_{\hat q_1},\ldots,H_{\hat q_{n-1}},H_{\ol{\hat q}_1},\ldots,H_{\ol{\hat q}_{n-1}}$
for $T_\rho\hat\Sigma^\bot$, where $T_\rho\hat\Sigma^{\bot}$ is the orthogonal to
$\Complex T_\rho\hat\Sigma$ in $\Complex T_\rho(T^*\hat D)$ with
respect to canonical two form $\hat\sigma$. Since $\hat p_0=\sum^{n-1}_{j=1}\ol{\hat q}_j\hat q_j$, we have $H_{\hat p_0}=\sum^{n-1}_{j=1}\Bigr(\ol{\hat q}_jH_{\hat q_j}+\hat q_jH_{\ol{\hat q}_j}\Bigr)$.
We compute the linearization of $H_{\hat p_0}$ at $\hat\rho$
\[H_{\hat p_0}\Big(\hat\rho+\sum(t_kH_{\hat q_k}+s_kH_{\ol{\hat q}_k})\Big)= O(\abs{t,s}^2)+\sum_{j,k}t_k\{\hat q_k,\ol {\hat q}_j\}H_{\hat q_j} 
+\sum_{j,k}s_k\{\ol{\hat q}_k,\hat q_j\}H_{\ol{\hat q}_j}.\]
So $F(\hat\rho)$ is expressed in the basis $H_{\hat q_1},\ldots,H_{\hat q_{n-1}},H_{\ol{\hat q}_1},\ldots,H_{\ol{\hat q}_{n-1}}$ by
\begin{equation} \label{e-mslknaII}
F(\hat\rho)=\left(
 \begin{array}{cc}
   \{\hat q_t,\ol{\hat q}_j\}(\hat\rho)  &  0  \\
     0     &  \{\ol{\hat q}_t,\hat q_j\}(\hat\rho)
 \end{array}\right).
\end{equation}
Again, from \eqref{e-mslknaIb}, we see that the non-vanishing eigenvalues of $F(\hat\rho)$ are
\begin{equation} \label{e-mslknaIII}
\pm i\lambda_1(x,\lambda,\xi_{2n}),\ldots,\pm i\lambda_{n-1}(x,\lambda,\xi_{2n}),
\end{equation}
where $\hat\rho=(\hat x,(\lambda\omega_0(x)-2{\rm Im\,}\ddbar_b\phi(x)\xi_{2n},\xi_{2n}))$ and $\lambda_j(x,\lambda,\xi_{2n})$, $j=1,\ldots,n-1$, are the eigenvalues of the Hermitian quadratic form $M^\phi_x\xi_{2n}-2\lambda\mathcal{L}_x$. 

To compute further, fix a point $\hat\rho_0=((p,x_{2n}),(\lambda\omega_0(p)-2{\rm Im\,}\ddbar_b\phi(p)\xi_{2n},\xi_{2n}))\in U\bigcap\hat\Sigma$ and we may assume that the Hermitian quadratic form $\xi_{2n}M^\phi_p-2\lambda\mathcal{L}_p$ is diagonalized with respect to $\ol Z_j(p)$, $j=1,\ldots,n-1$. Thus, 
\[\begin{split}
\sum^{n-1}_{j,t=1}e^\wedge_je^{\wedge, *}_t\Bigr(M^\phi_p(\ol Z_j,Z_t)\xi_{2n}-2\lambda\mathcal{L}_p(\ol Z_j,Z_t)\Bigr)
=\sum^{n-1}_{j=1}e^\wedge_je^{\wedge, *}_j\Bigr(M^\phi_p(\ol Z_j,Z_j)\xi_{2n}-2\lambda\mathcal{L}_p(\ol Z_j,Z_j)\Bigr).\end{split}\]
From this, \eqref{e-mslknaI} and \eqref{e-mslknaIII}, we see that on $\hat\Sigma\bigcap U$ and on the space of $(0,q)$ forms, $\hat p^s_0(\hat\rho)+\frac{1}{2}\Td{\rm tr\,}F(\hat\rho)$, $\hat\rho=(\hat x, (\lambda\omega_0(x)-2{\rm Im\,}\ddbar_b\phi(x)\xi_{2n},\xi_{2n}))\in\hat\Sigma\bigcap U$ has the eigenvalues
\begin{equation} \label{e-mslknaIV}\begin{split}
&\frac{1}{2}\sum^{n-1}_{j=1}\abs{\lambda_j(x,\lambda,\xi_{2n})}-\frac{1}{2}\sum_{j\notin J}\lambda_j(x,\lambda,\xi_{2n})+\frac{1}{2}\sum_{j\in J}
  \lambda_j(x,\lambda,\xi_{2n}), \; \abs{J}=q,  \\
& J=(j_1,j_2,\ldots,j_q),\; 1\leq j_1<j_2<\cdots<j_q\leq n-1,
\end{split}\end{equation}
where $\lambda_j(x,\lambda,\xi_{2n})$, $j=1,\ldots,n-1$, are the eigenvalues of the Hermitian quadratic form $M^\phi_x\xi_{2n}-2\lambda\mathcal{L}_x$. 

Note that $\xi_{2n}M^\phi_x-2\lambda\mathcal{L}_x$ is non-degenerate of constant signature $(n_-,n_+)$, for every $(\hat x, (\lambda\omega_0(x)-2{\rm Im\,}\ddbar_b\phi(x)\xi_{2n},\xi_{2n}))\in\hat\Sigma\bigcap U$ and there is a constant $\mu>0$ such that $\abs{\lambda_j(x,\lambda,\xi_{2n})}>\mu\xi_{2n}$, $j=1,\ldots,n-1$, for all $(\hat x, (\lambda\omega_0(x)-2{\rm Im\,}\ddbar_b\phi(x)\xi_{2n},\xi_{2n}))\in\hat\Sigma\bigcap U$ (see \eqref{e-dhlkmimII}). Combining this observation with \eqref{e-mslknaIV}, it is straightforward to see that \eqref{e-dhlkmimXIII} and \eqref{e-dhlkmimXIIIbis} hold. 
\end{proof} 

\begin{rem}\label{r-gue1373}
With the notations and assumptions above, let $q=n_-$ and let
\[\rho_0=((x_0,(s_0\omega_0(x_0)-2{\rm Im\,}\ddbar_b\phi(x_0))\in V\bigcap\Sigma.\] 
Let $\ol Z_{1,s_0},\ldots,\ol Z_{n-1,s_0}$ be an orthonormal frame of $T^{1,0}_xX$ varying smoothly with $x$ in a neighbourhood of $p$, for which the Hermitian quadratic form $M^\phi_x-2s_0\mathcal{L}_x$ is diagonalized at $x_0$. 
That is, 
\[M^\phi_p\bigr(\ol Z_{j,s_0}(x_0),Z_{t,s_0}(x_0)\bigr)-2s_0\mathcal{L}_{x_0}\bigr(\ol Z_{j,s_0}(x_0),Z_{t,s_0}(x_0)\bigr)
=\lambda_j(s_0)\delta_{j,t},\ \ j,t=1,\ldots,n-1.\]
Assume that $\lambda_j(s_0)<0$, $j=1,\ldots,n_-$. Let $e_{1,s_0},\ldots,e_{n-1,s_0}$ denote the basis of $T^{*0,1}X$, which is dual to $Z_{1,s_0},\ldots,Z_{n-1,s_0}$. Put 
\begin{equation}\label{e-gue1373III}
\mathcal{N}(x_0,s_0,n_-):=\set{ce_{1,s_0}(x_0)\wedge\cdots\wedge e_{n_-,s_0}(x_0)\in T^{*0,q}_{x_0}X;\, c\in\Complex}. 
\end{equation}
From the proof of Theorem~\ref{t-dhplkmiII}, it is not difficult to see that for every 
$\hat\rho=((x,x_{2n}),(\xi_{2n}\xi,\xi_{2n}))\in U\bigcap\hat\Sigma$, $\xi=s_0\omega_0(x)-2{\rm Im\,}\ddbar_b\phi(x)$, we have 
\begin{equation}\label{e-gue1373II}
N(\hat p^s_0(\hat\rho)+\frac{1}{2}\Td{\rm tr\,}F(\hat\rho))=\mathcal{N}(x,s_0,n_-).
\end{equation}
\end{rem}

Put 
\begin{equation}\label{e-mslknaVa}
\pi(U)=\set{\hat x\in\hat D;\, (\hat x,\hat\xi)\in U, \mbox{for some $\hat\xi\in\Real^{2n}$}}.
\end{equation}

From Theroem~\ref{t-dhplkmiI} and Theorem~\ref{t-dhplkmiII}, we get the following

\begin{thm}\label{t-dhplkmiIII}
Let $0\leq q_1\leq n-1$, $q_1\in\mathbb N_0$. Let $c_j(\hat x, \hat\eta)\in C^\infty(U,T^{*0,q_1}\hat D\boxtimes T^{*0,q}\hat D)$, $j=0, 1,\ldots$, be positively homogeneous functions of degree $m-j$, $m\in\mathbb Z$. Then, we can find solutions $a_j(t, \hat x, \hat\eta)\in C^\infty(\ol\Real_+\times T^*\hat D,T^{*0,q_1}\hat D\boxtimes T^{*0,q}\hat D)$, $j=0, 1,\ldots$, 
of the system \eqref{e-dhlkmimVIII} with $a_j(0, \hat x, \hat\eta)=c_j(\hat x, \hat\eta)$, $j=0, 1,\ldots$,
where $a_j(t,\hat x,\hat\eta)$ is a quasi-homogeneous function of degree $m-j$, such that $a_0(t,\hat x,\hat\eta)$ satisfies \eqref{e-gue1373I} and for all $\alpha, \beta\in\mathbb N_0^{2n}$, $\gamma, j\in\mathbb N_0$, every $\varepsilon_0>0$ and compact set $K\Subset\pi(U)$, there is a constant $c>0$ such that
\begin{equation}\label{e-mslknaV}
\abs{\pr^\gamma_t\pr^\alpha_{\hat x}\pr^\beta_{\hat\eta}a_j(t,\hat x,\hat\eta)}\leq ce^{\varepsilon_0t\abs{\eta_{2n}}}(1+\abs{\hat\eta})^{m-j-\abs{\beta}+\gamma}\ \ \mbox{on $\ol\Real_+\times\bigr(K\times\Real^{2n})\bigcap(U\bigcap\hat\Sigma)\bigr)$}.
\end{equation}
Furthermore, if $q\neq n_-$, then for all $\alpha, \beta\in\mathbb N_0^{2n}$, $\gamma, j\in\mathbb N_0$, and every compact set $K\Subset\pi(U)$, there is a constant $c>0$ such that
\begin{equation}\label{e-mslknaVI}
\abs{\pr^\gamma_t\pr^\alpha_{\hat x}\pr^\beta_{\hat\eta}a_j(t,\hat x,\hat\eta)}\leq ce^{-\mu_0t\abs{\eta_{2n}}}(1+\abs{\hat\eta})^{m-j-\abs{\beta}+\gamma}\ \ \mbox{on $\ol\Real_+\times\bigr(K\times\Real^{2n})\bigcap(U\bigcap\hat\Sigma)\bigr)$},
\end{equation}
where $\mu_0>0$ is a constant as in \eqref{e-dhlkmimXIIIbis}.
\end{thm} 

We introduce some symbol classes 

\begin{defn} \label{d-d-msmapamo}
Let $\mu\geq0$ be a non-negative constant. For $0\leq q_1, q_2\leq n-1$, $q_1, q_2\in\mathbb N_0$ and $m\in\Real$, we say that $a\in\hat S^m_\mu(\ol\Real_+\times U,T^{*0,q_1}\hat D\boxtimes T^{*0,q_2}\hat D)$
if $a\in C^\infty(\ol\Real_+\times U,T^{*0,q_1}\hat D\boxtimes T^{*0,q_2}\hat D)$
and for all indices $\alpha, \beta\in\mathbb N^{2n}_0$, $\gamma\in\mathbb N_0$, every compact set $K\Subset\pi(U)$ and every
$\varepsilon>0$, there exists a constant $c>0$ such that
\[\abs{\pr^\gamma_t\pr^\alpha_{\hat x}\pr^\beta_{\hat\eta}a(t,\hat x,\hat\eta)}\leq
 ce^{t(-\mu\abs{\eta_{2n}}+\varepsilon\abs{\eta_{2n}})}(1+\abs{\eta})^{m+\gamma-\abs{\beta}},\; \hat x\in K, (\hat x,\hat\eta)\in U.\] 
\end{defn}

\begin{rem} \label{r:ss-heatclass}
It is easy to see that we have the following properties:
 \begin{enumerate}
  \item If $a\in\hat S^m_{\mu_1}$, $b\in\hat S^l_{\mu_2}$ then $ab\in\hat S^{m+l}_{\mu_1+\mu_2}$,
        $a+b\in\hat S^{\max(m,l)}_{\min(\mu_1,\mu_2)}$.
  \item If $a\in\hat S^m_\mu$ then $\pr^\gamma_t\pr^\alpha_{\hat x}\pr^\beta_{\hat\eta}a\in
        \hat S^{m-\abs{\beta}+\gamma}_\mu$.
  \item If $a_j\in\hat S^{m_j}_\mu$, $j=0,1,2,\ldots$ and $m_j\searrow -\infty$ as $j\To\infty$,
        then there exists $a\in\hat S^{m_0}_\mu$ such that $a-\sum^{v-1}_0a_j\in
        \hat S^{m_v}_\mu$, for all $v=1,2,\ldots$. Moreover, if $\hat S^{-\infty}_\mu$ denotes $\bigcap_{m\in\Real}\hat    S^m_\mu$ then $a$ is unique modulo $\hat S^{-\infty}_\mu$.
 \end{enumerate}

If $a$ and $a_j$ have the properties of (c), we write $a\sim\sum^\infty_{j=0}a_j$ in $\hat S^{m_0}_\mu$.
\end{rem} 


 



From Theorem~\ref{t-dhplkmiIII} and the standard Borel construction,
we get the following

\begin{thm} \label{t-aldhmpI}
Let $0\leq q_1\leq n-1$, $q_1\in\mathbb N_0$. Let $c_j(\hat x, \hat\eta)\in C^\infty(U,T^{*0,q_1}\hat D\boxtimes T^{*0,q}\hat D)$, $j=0, 1,\ldots$, be positively homogeneous functions of degree $m-j$, $m\in\mathbb Z$. We can find solutions
$a_j(t,\hat x,\hat\eta)\in C^\infty(\ol\Real_+\times U,T^{*0,q_1}\hat D\boxtimes T^{*0,q}\hat D)$, $j=0, 1,\ldots$ of the system \eqref{e-dhlkmimVIII} with $a_j(0, \hat x, \hat\eta)=c_j(\hat x, \hat\eta)$, $j=0, 1,\ldots$,
where $a_j(t,\hat x,\hat\eta)$ is a quasi-homogeneous function of degree $m-j$, such that $a_0(t,\hat x,\hat\eta)$ satisfies \eqref{e-gue1373I} and $a_j\in\hat S^{m-j}_\mu(\ol\Real_+\times U,T^{*0,q_1}\hat D,\boxtimes T^{*0,q}\hat D)$, $j=0,1,\ldots$, 
for some $\mu$ with $\mu>0$ if $q\neq n_-$ and $\mu=0$ if $q=n_-$.
\end{thm}

For $0\leq q_1\leq n-1$, $q_1\in\mathbb N_0$, let $a_j(t, \hat x, \hat\eta)\in\hat S^{m-j}_\mu(\ol\Real_+\times U,T^{*0,q_1}\hat D\boxtimes T^{*0,q}\hat D)$, $j=0,1,\ldots$, 
be quasi-homogeneous functions of degree $m-j$, $m\in\mathbb Z$. Assume that
$a_j(t, \hat x, \hat\eta)$, $j=0,1,\dots$,
are the solutions of the system \eqref{e-dhlkmimVIII}. 
Let $a(t,\hat x,\hat\eta)\sim\sum^\infty_{j=0}a_j(t,\hat x,\hat\eta)$ in $\hat S^m_\mu(\ol\Real_+\times U,T^{*0,q_1}\hat D\boxtimes T^{*0,q}\hat D)$. Put
\[(\pr_t+\Box^{(q)}_s)(e^{i\psi(t,\hat x,\hat\eta)}a(t,\hat x,\hat\eta))=e^{i\psi(t,\hat x,\hat\eta)}b(t,\hat x,\hat\eta),\]
where 
\[\mbox{$b(t,\hat x,\hat\eta)\sim\sum^\infty_{j=0}b_j(t,\hat x,\hat\eta)$ in $\hat S^{m+2}_\mu(\ol\Real_+\times U,T^{*0,q_1}\hat D\boxtimes T^{*0,q}\hat D)$},\] 
$b_j\in \hat S^{m+2-j}_\mu(\ol\Real_+\times U,T^{*0,q_1}\hat D, T^{*0,q}\hat D)$, 
$b_j$ is a quasi-homogeneous function of degree $m+2-j$, $j=0,1,\ldots$. 

Since $a_j(t,\hat x,\hat\eta)$, $j=0,1,\ldots$, solve the transport equations \eqref{e-dhlkmimVIII}, we have that for all $N\in\mathbb N$, every compact set $K\Subset\pi(U)$, $\varepsilon>0$, and all indices $\alpha, \beta\in\mathbb N^{2n}_0$, there exists $c>0$ such that 
\begin{equation}\label{e-mslknaVII}
\abs{\pr^\alpha_{\hat x}\pr^\beta_{\hat\eta}b}\leq ce^{\varepsilon t\abs{\eta_{2n}}}\bigr(\abs{\hat\eta}^{-N}+\abs{\hat\eta}^{m+2-N}({\rm Im\,}\psi(t,\hat x,\hat\eta)\bigr)^N\bigr)\ \ \mbox{on $\ol\Real_+\times\bigr(K\times\Real^{2n})\bigcap(U\bigcap\hat\Sigma)\bigr)$}.
\end{equation} 

Conversely, if
$(\pr_t+\Box^{(q)}_s)(e^{i\psi(t,\hat x,\hat\eta)}a(t,\hat x,\hat\eta))=e^{i\psi(t,\hat x,\hat\eta)}b(t,\hat x,\hat\eta)$
and $b$ satisfies the same kind of estimates as \eqref{e-mslknaVII},
then $a_j(t,\hat x,\hat\eta)$, $j=0,1,\ldots$, solve the system \eqref{e-dhlkmimVIII} to infinite order at $\hat\Sigma\bigcap U$. From this observation and the particular structure of the problem, we will next show 

\begin{thm} \label{t-aldhmpII}
Let $q=n_-$. Let $c_j(\hat x, \hat\eta)\in C^\infty(U,T^{*0,q}\hat D\boxtimes T^{*0,q}\hat D)$, $j=0, 1,\ldots$, be positively homogeneous functions of degree $m-j$, $m\in\mathbb Z$. We can find solutions
$a_j(t,\hat x,\hat\eta)\in\hat S^{m-j}_0(\ol\Real_+\times U,T^{*0,q}\hat D\boxtimes T^{*0,q}\hat D)$, $j=0, 1,\ldots$ of the system \eqref{e-dhlkmimVIII}, where $a_j(t,\hat x,\hat\eta)$ is a quasi-homogeneous function of degree $m-j$, for each $j$, with $a_j(0, \hat x, \hat\eta)=c_j(\hat x, \hat\eta)$, $j=0, 1,\ldots$, 
\[a_j(t,\hat x,\hat\eta)-a_j(\infty,\hat x,\hat\eta)\in\hat S^{m-j}_\mu(\ol\Real_+\times U,T^{*0,q}\hat D\boxtimes T^{*0,q}\hat D),\ \ j=0,1,2,\ldots,\]
and for every $(\hat x,\hat\eta)=((x,x_{2n}),(\eta_{2n}\eta,\eta_{2n}))\in U\bigcap\hat\Sigma$, $\eta=s_0\omega_0(x)-2{\rm Im\,}\ddbar_b\phi(x)$, we have 
\begin{equation}\label{e-gue1373IV}
a_0(\infty,\hat x,\hat\eta)u\in\mathcal{N}(x,s_0,n_-),\ \ \forall u\in T^{*0,q}_{\hat x}\hat D,
\end{equation}
where $\mu>0$ is a constant and $a_j(\infty,\hat x,\hat\eta)\in C^\infty(U,T^{*0,q}\hat D\boxtimes T^{*0,q}\hat D)$, $j=0,1,\ldots$, $a_j(\infty,\hat x,\hat\eta)$ is a positively homogeneous function of degree $m-j$, for each $j$. 
\end{thm}

\begin{proof}
Let $\Td a_j(t, \hat x, \hat\eta)\in\hat S^{m-j}_0(\ol\Real_+\times U,T^{*0,q}\hat D\boxtimes T^{*0,q}\hat D)$, $j=0,1,\ldots$, be any solutions of the system \eqref{e-dhlkmimVIII} with 
$\Td a_j(0, \hat x, \hat\eta)=c_j(\hat x,\hat\eta)$, $j=0,1,\dots$, where $\Td a_j(t, \hat x, \hat\eta)$ is a quasi-homogeneous function of degree $m-j$, $j=0,1,\ldots$. Set $\Td a(t,\hat x,\hat\eta)\sim\sum^\infty_{j=0}\Td a_j(t,\hat x,\hat\eta)$ in $\hat S^m_0(\ol\Real_+\times U,T^{*0,q}\hat D\boxtimes T^{*0,q}\hat D)$ and put 
\[(\pr_t+\Box^{(q)}_s)(e^{i\psi(t,\hat x,\hat\eta)}\Td a(t,\hat x,\hat\eta))=e^{i\psi(t,\hat x,\hat\eta)}\Td b(t,\hat x,\hat\eta),\] 
where 
\[\mbox{$\Td b(t,\hat x,\hat\eta)\sim\sum^\infty_{j=0}\Td b_j(t,\hat x,\hat\eta)$ in $\hat S^{m+2}_\mu(\ol\Real_+\times U,T^{*0,q}\hat D\boxtimes T^{*0,q}\hat D)$},\] 
$\Td b_j\in \hat S^{m+2-j}_\mu(\ol\Real_+\times U,T^{*0,q}\hat D\boxtimes T^{*0,q}\hat D)$, 
$\Td b_j$ is a quasi-homogeneous function of degree $m+2-j$, $j=0,1,\ldots$. Since $\Td a_j(t,\hat x,\hat\eta)$, $j=0,1,\ldots$, solve the transport equations \eqref{e-dhlkmimVIII}, $\Td b(t,\hat x,\hat\eta)$ satisfies \eqref{e-mslknaVII}.
Note that we have the interwing properties
\begin{equation}\label{e-mslknaVIII}
\ddbar_s\Box^{(q)}_s=\Box^{(q+1)}_s\ddbar_s,\ \ \ol{\pr}^*_s\Box^{(q)}_s=\Box^{(q-1)}_s\ol{\pr}^*_s.
\end{equation}
Now, 
\[\ddbar_s(e^{i\psi}\Td a)=e^{i\psi}c,\ \ \ol{\pr}^*_s(e^{i\psi}\Td a)=e^{i\psi}d,\]
$c\sim\sum_{j=0}^\infty c_j(t,\hat x,\hat \eta)$ in $\hat S^{m+1}_0(\ol\Real_+\times U,T^{*0,q}\hat D\boxtimes T^{*0,q+1}\hat D)$, $d\sim\sum_{j=0}^\infty d_j(t,\hat x,\hat \eta)$ in $\hat S^{m+1}_0(\ol\Real_+\times U,T^{*0,q}\hat D\boxtimes T^{*0,q-1}\hat D)$, where $c_j(t,\hat x,\hat\eta)$ and $d_j(t,\hat x,\hat\eta)$ are quasi-homogeneous functions of degree $m+1-j$, $j=0,1,\ldots$. From \eqref{e-mslknaVIII}, we have $(\pr_t+\Box^{(q+1)}_s)(e^{i\psi}c)=e^{i\psi}e$, 
$(\pr_t+\Box^{(q-1)}_s)(e^{i\psi}d)=e^{i\psi}f$, 
where $e$ and $f$ satisfy \eqref{e-mslknaVII}.  Since $e$ and $f$ satisfy \eqref{e-mslknaVII}, we deduce that 
$c_j$, $j=0,1,\ldots$, solve the system \eqref{e-dhlkmimVIII} and $d_j$, $j=0,1,\ldots$, solve the system \eqref{e-dhlkmimVIII} too. From Theorem~\ref{t-dhplkmiI} and Theorem~\ref{t-dhplkmiIII}, we see that $c_j(t,\hat x,\hat\eta)$, $d_j(t,\hat x,\hat\eta)$, $j=0,1,\ldots$, satisfy the same kind of estimates as \eqref{e-mslknaVI}. 
Now $\Box^{(q)}_s=\ol{\pr_s}^*\ddbar_s+\ol{\pr}^*_s\ddbar_s$, so $\Box^{(q)}_s(e^{i\psi}\Td a)=e^{i\psi}g$,
where $g$ satisfies the same kind of estimates as \eqref{e-mslknaVI}. From this we see that
$\pr_t(e^{i\psi}\Td a)=e^{i\psi}h$, where $h$ has the same properties as $g$. Since
$h=i(\pr_t\psi)\Td a+\pr_t\Td a$ and $\pr_t\psi$ satisfy the same kind of estimates as
\eqref{e-mslknaVI}, $\pr _t\Td a$ satisfies the same kind of estimates as
\eqref{e-mslknaVI}. From the standard Borel construction, we can find $a_j(t,\hat x,\hat\eta)\in C^\infty(\ol\Real_+\times U,T^{*0,q}\hat D\boxtimes T^{*0,q}\hat D)$, $j=0,1,\ldots$, such that $a_j(t,\hat x,\hat\eta)-\Td a_j(t,\hat x,\hat\eta)$ vanishes to infinite order at each point of $\hat\Sigma\bigcap U$, $a_j(t,\hat x,\hat\eta)$ is a quasi-homogeneous function of degree $m-j$ and there is a $\mu>0$ such that $\pr_ta_j(t,\hat x,\hat\eta)\in\hat S^{m-j}_{\mu}(\ol\Real_+\times U,T^{*0,q}\hat D\boxtimes T^{*0,q}\hat D)$, $j=0,1,\ldots$. We conclude that we can find
$a_j(\infty,\hat x,\hat\eta)\in C^\infty(U,T^{*0,q}\hat D\boxtimes T^{*0,q}\hat D)$,
where $a_j(\infty,\hat x,\hat\eta)$ is a positively homogeneous function of degree $m-j$, $j=0,1,\ldots$, 
such that $a_j(t,\hat x,\hat\eta)-a_j(\infty,\hat x,\hat\eta)\in\hat S^{m-j}_\mu(\ol\Real_+\times U,T^{*0,q}\hat D\boxtimes T^{*0,q}\hat D)$, $\mu>0$, $j=0,1,2,\ldots$. 

Finally, from Theorem~\ref{t-dhplkmiI}, \eqref{e-gue1373I} and \eqref{e-gue1373II}, we obtain \eqref{e-gue1373IV}. The theorem follows.
\end{proof} 

\subsection{Microlocal Hodge decomposition theorems for $\Box^{(q)}_s$ in $U$} \label{s-mhdt}

We use the same notations and assumptions as before. Fix $D_0\Subset D$, where $D_0$ is an open set of $D$. As before, we put $\hat D_0:=D_0\times\Real$.
We need the following which is essentially well-known (see Chapter 5 in part {\rm I\,} of~\cite{Hsiao08})

\begin{prop} \label{p-mimpzI}
Let $\mu>0$ and let $b(t,\hat x,\hat\eta)\in\hat S^m_\mu(\ol\Real_+\times U,T^{*0,q}\hat D\boxtimes T^{*0,q}\hat D)$, $m\in\Real$. 
We assume that $b(t,\hat x,\hat\eta)=0$ when $\abs{\hat\eta}\leq 1$ and for every $t\in\ol\Real_+$, ${\rm Supp\,}b(t,\hat x,\hat\eta)\bigcap T^*\hat D_0\subset\ol W$, where $W\subset U$ is a conic open set with $\ol W\subset U$. Take $\tau(\hat x,\hat\eta)\in C^\infty(T^*\hat D)$, $\tau=1$ on $\ol W$, $\tau=0$ outside $U$ and $\tau$ is positively homogeneous of degree $0$. Let $\chi\in C^\infty_0(\Real^{2n})$ be equal to $1$ near the origin. Put
\[B_{\epsilon}(\hat x, \hat y)=\int(\int^\infty_0e^{i(\psi(t,\hat x,\hat\eta)-<\hat y,\hat\eta>)}b(t,\hat x,\hat\eta)dt)\chi(\epsilon\eta)\tau(\hat x,\hat\eta)d\hat\eta.\] 
For $u\in\Omega^{0,q}_0(\hat D)$, we have
\[\lim_{\epsilon\To0}\big(\int B_{\epsilon}(\hat x, \hat y)u(\hat y)d\hat y\big)\in \Omega^{0,q}(\hat D)\]
and the operator
\begin{equation}\label{e-mslknaIX}\begin{split}
B: \Omega^{0,q}_0(\hat D) &\To\Omega^{0,q}(\hat D) \\
u&\To\lim_{\epsilon\To0}\big(\int B_{\epsilon}(\hat x, \hat y)u(\hat y)dy\big)
\end{split}\end{equation}
is continuous and $B$ has a unique continuous extension:
\[B: \mathcal E'(\hat D,T^{*0,q}\hat D)\To\mathcal D'(\hat D,T^{*0,q}\hat D)\]
and
$B(x,y)\in C^\infty\big(\hat D\times\hat D\setminus{\rm diag\,}(\hat D\times\hat D),T^{*0,q}\hat D\boxtimes T^{*0,q}\hat D)$, where $B(\hat x,\hat y)$ denotes the distribution kernel of $B$. 
\end{prop}

Let $b(t,\hat x,\hat\eta)\in\hat S^m_\mu(\ol\Real_+\times U,T^{*0,q}\hat D\boxtimes T^{*0,q}\hat D)$, $\mu>0$. $m\in\Real$. We assume that $b(t,\hat x,\hat\eta)=0$ when $\abs{\hat\eta}\leq 1$ and for every $t\in\ol\Real_+$, ${\rm Supp\,}b(t,\hat x,\hat\eta)\bigcap T^*\hat D_0\subset\ol W$, where $W\subset U$ is a conic open set with $\ol W\subset U$. Let 
\[B:\Omega^{0,q}_0(\hat D)\To\Omega^{0,q}(\hat D),\ \ \mathcal E'(\hat D,T^{*0,q}\hat D)\To\mathcal D'(\hat D,T^{*0,q}\hat D)\] 
be the continuous operator given by \eqref{e-mslknaIX}. We formally write 
\[B=B(\hat x,\hat y)=\int\Bigr(\int^\infty_0e^{i(\psi(t,\hat x,\hat\eta)-<\hat y,\hat\eta>)}b(t,\hat x,\hat\eta)dt\Bigr)\tau(\hat x,\hat\eta)d\hat\eta.\]
From now on, we identify $B$ with $B(\hat x,\hat y)$. 

\begin{rem} \label{r-mimpzI}
Let $a(t,\hat x,\hat\eta)\in \hat S^m_0(\ol\Real_+\times U,T^{*0,q}\hat D\boxtimes T^{*0,q}\hat D)$, $m\in\Real$. 
We assume that $a(t, \hat x, \hat\eta)=0$ if $\abs{\eta}\leq 1$ and for every $t\in\ol\Real_+$, ${\rm Supp\,}a(t,\hat x,\hat\eta)\bigcap T^*\hat D_0\subset\ol W$, where $W\subset U$ is a conic open set with $\ol W\subset U$ and $a(t, \hat x, \hat\eta)-a(\infty,\hat x,\hat\eta)\in \hat S^m_\mu(\ol\Real_+\times U,T^{*0,q}\hat D\boxtimes T^{*0,q}\hat D)$, $\mu>0$, where
$a(\infty,\hat x, \hat\eta)\in C^\infty(U,T^{*0,q}\hat D\boxtimes T^{*0,q}\hat D)$ and ${\rm Supp\,}a(\infty,\hat x,\hat\eta)\bigcap T^*\hat D_0\subset\ol W$. 
Then we can also define
\begin{equation} \label{e-mslknaX} 
A(\hat x, \hat y)
=\int\Bigr(\int^{\infty}_0\bigr(e^{i(\psi(t,\hat x,\hat\eta)-<\hat y,\hat\eta>)}a(t,\hat x,\eta)-e^{i(\psi(\infty,\hat x,\hat\eta)-<\hat y,\hat\eta>)}
a(\infty,\hat x,\hat\eta)\bigr)dt\Bigr)\tau(\hat x,\hat\eta)d\hat\eta
\end{equation}
as an oscillatory integral by the following formula: 
\[A(\hat x, \hat y)
=\int\Bigr(\int^{\infty}_0e^{i(\psi(t,\hat x,\hat\eta)-
<\hat y,\hat\eta>)}(-t)\big(i\psi^{'}_t(t,\hat x,\hat\eta)a(t,\hat x,\hat\eta)+a'_t(t,\hat x,\hat\eta)\big)dt\Bigr)\tau(\hat x,\hat\eta)d\hat\eta.\]
We notice that
$(-t)\big(i\psi'_t(t,\hat x,\hat\eta)a(t,\hat x,\hat\eta)+a'_t(t,\hat x,\hat\eta)\big)\in\hat S^{m+1}_\mu$, $\mu>0$.

The oscillatory integral $A(\hat x,\hat y)$ defines a continuous operator 
\[A:\Omega^{0,q}_0(\hat D)\To\Omega^{0,q}(\hat D),\ \ \mathcal E'(\hat D,T^{*0,q}\hat D)\To\mathcal D'(\hat D,T^{*0,q}\hat D).\]
We formally write 
\[\begin{split}
A&=A(\hat x, \hat y)\\
&=\int\Bigr(\int^{\infty}_0\bigr(e^{i(\psi(t,\hat x,\hat\eta)-<\hat y,\hat\eta>)}a(t,\hat x,\eta)-e^{i(\psi(\infty,\hat x,\hat\eta)-<\hat y,\hat\eta>)}
a(\infty,\hat x,\hat\eta)\bigr)dt\Bigr)\tau(\hat x,\hat\eta)d\hat\eta.\end{split}\]
\end{rem} 

Let $m\in\Real$, $0\leq\rho,\delta\leq1$. Let $\Gamma$ be a conic open set of $T^*\hat D$. Let $S^m_{\rho,\delta}(\Gamma,T^{*0,q}\hat D\boxtimes T^{*0,q}\hat D)$ denote the H\"{o}rmander symbol space on $\Gamma$ with values in $T^{*0,q}\hat D\boxtimes T^{*0,q}\hat D$ of order $m$ type $(\rho,\delta)$(see Definition 1.1 of Grigis-Sj\"{o}strand~\cite{GS94})
and let $S^m_{{\rm cl\,}}(\Gamma,T^{*0,q}\hat D\boxtimes T^{*0,q}\hat D)$
denote the space of classical symbols on $\Gamma$ with values in $T^{*0,q}\hat D\boxtimes T^{*0,q}\hat D$ of order $m$ (see page 35 of Grigis-Sj\"{o}strand~\cite{GS94}). Let $B\subset D$ be an open set. Let 
$L^m_{\frac{1}{2},\frac{1}{2}}(B,T^{*0,q}\hat D\boxtimes T^{*0,q}\hat D)$ and $L^m_{{\rm cl\,}}(B,T^{*0,q}\hat D\boxtimes T^{*0,q}\hat D)$ denote the space of
pseudodifferential operators on $B$ of order $m$ type $(\frac{1}{2},\frac{1}{2})$ from sections of
$T^{*0,q}\hat D$ to sections of $T^{*0,q}\hat D$ and the space of classical 
pseudodifferential operators on $B$ of order $m$ from sections of
$T^{*0,q}\hat D$ to sections of $T^{*0,q}\hat D$. The classical result of Calderon and Vaillancourt tells us that for any $A\in L^m_{\frac{1}{2},\frac{1}{2}}(B,T^{*0,q}\hat D\boxtimes T^{*0,q}\hat D)$, 
\begin{equation}\label{e-mslknaXI}
\mbox{$A:H^s_{\rm comp}(B,T^{*0,q}\hat D)\To H^{s-m}_{\rm loc}(B,T^{*0,q}\hat D)$ is continuous, for every $s\in\Real$}. 
\end{equation}
(See H\"{o}rmander~\cite{Hor85}, for a proof).

We can repeat the proofs of Lemma 5.14, Lemma 5.16 in~\cite{Hsiao08} and obtain the following 

\begin{prop} \label{p-mimpzII}
Let $a(t,\hat x,\hat\eta)\in \hat S^m_0(\ol\Real_+\times U,T^{*0,q}\hat D\boxtimes T^{*0,q}\hat D)$, $m\in\Real$. 
We assume $a(t, \hat x, \hat\eta)=0$ if $\abs{\eta}\leq 1$ and  for every $t\in\ol\Real_+$, ${\rm Supp\,}a(t,\hat x,\hat\eta)\bigcap T^*\hat D_0\subset\ol W$, where $W\subset U$ is a conic open set with $\ol W\subset U$ and $a(t, \hat x, \hat\eta)-a(\infty, \hat x, \hat \eta)\in \hat S^m_\mu(\ol\Real_+\times U,T^{*0,q}\hat D\boxtimes T^{*0,q}\hat D)$, $\mu>0$, where
$a(\infty,\hat x, \hat\eta)\in C^\infty(U,T^{*0,q}\hat D\boxtimes T^{*0,q}\hat D)$ and ${\rm Supp\,}a(\infty,\hat x,\hat\eta)\bigcap T^*\hat D_0\subset\ol W$. Take $\tau(\hat x,\hat\eta)\in C^\infty(T^*\hat D)$, $\tau=1$ on $\ol W$, $\tau=0$ outside $U$ and $\tau$ is positively homogeneous of degree $0$.
Let
\[\begin{split}
&A(\hat x, \hat y)\\
&=\frac{1}{(2\pi)^{2n}}\int\Bigr(\int^{\infty}_0\bigr(e^{i(\psi(t,\hat x,\hat\eta)-<\hat y,\hat\eta>)}a(t,\hat x,\hat\eta)-e^{i(\psi(\infty,\hat x,\hat\eta)-<\hat y,\hat\eta>)}
a(\infty,\hat x,\hat\eta)\bigr)dt\Bigr)\tau(\hat x,\hat\eta)d\hat\eta
\end{split}\]
be the oscillatory integral as in \eqref{e-mslknaX}. Then
$A\in L^{m-1}_{\frac{1}{2},\frac{1}{2}}(\hat D, T^{*0,q}\hat D\boxtimes T^{*0,q}\hat D)$
with symbol
\[q(\hat x,\hat\eta)=\int^{\infty}_0\Bigr(e^{i(\psi(t,\hat x,\hat\eta)-<\hat x,\hat\eta>)}a(t,\hat x,\hat\eta)-
e^{i(\psi(\infty,\hat x,\hat\eta)-<\hat x,\hat\eta>)}a(\infty,\hat x,\hat\eta)\Bigr)dt\tau(\hat x,\hat\eta)\]
in $S^{m-1}_{\frac{1}{2},\frac{1}{2}}(T^*\hat D,T^{*0,q}\hat D\boxtimes T^{*0,q}\hat D)$.
\end{prop}

We can repeat the proof Proposition 5.18 in~\cite{Hsiao08} and conclude that

\begin{prop} \label{p-gelkmibI}
Let 
$a(\infty,\hat x,\hat\eta)\in C^\infty(T^*\hat D,T^{*0,q}\hat D\boxtimes T^{*0,q}\hat D)$, ${\rm Supp\,}a(\infty,\hat x,\hat\eta)\bigcap T^*\hat D_0\subset \ol W$, be a classical symbol of order $m$, where $W\subset U$ is a conic open set with $\ol W\subset U$. Take $\tau(\hat x,\hat\eta)\in C^\infty(T^*\hat D)$, $\tau=1$ on $\ol W$, $\tau=0$ outside $U$ and $\tau$ is positively homogeneous of degree $0$. Then 
\[a(\hat x,\hat\eta)=e^{i(\psi(\infty,x,\eta)-<x,\eta>)}a(\infty,\hat x,\hat\eta)\tau(\hat x,\hat\eta)
\in S^m_{\frac{1}{2},\frac{1}{2}}(T^*\hat D,T^{*0,q}\hat D\boxtimes T^{*0,q}\hat D).\]
\end{prop}

We assume that $q\neq n_-$. Let $\Td I=(2\pi)^{-2n}\int e^{i<\hat x-\hat y,\hat\eta>}c(\hat x,\hat\eta)d\hat\eta$
be a classical pseudodifferential operator on $\hat D$ of order $0$ from sections of $T^{*0,q}\hat D$ to sections of $T^{*0,q}\hat D$ with $c(\hat x,\hat\eta)\in S^0_{{\rm cl\,}}(T^*\hat D,T^{*0,q}\hat D\boxtimes T^{*0,q}\hat D)$, ${\rm Supp\,}c(\hat x,\hat\eta)\bigcap T^*\hat D_0\subset W$, where $W\subset U$ is a conic open set with $\ol W\subset U$. 
We have 
\[c(\hat x,\hat\eta)\sim\sum^\infty_{j=0}c_j(\hat x,\hat\eta)\] 
in the H\"{o}rmander symbol space $S^0_{1,0}(T^*\hat D,T^{*0,q}\hat D\boxtimes T^{*0,q}\hat D)$,  $c_j(\hat x, \hat\eta)\in C^\infty(U,T^{*0,q}\hat D\boxtimes T^{*0,q}\hat D)$, ${\rm Supp\,}c_j(\hat x,\hat\eta)\subset\ol W$, $j=0, 1,\ldots$, are positively homogeneous functions of degree $-j$. Let
\[a_j(t,\hat x,\hat\eta)\in\hat S^{-j}_{\mu_0}(\ol\Real_+\times U,T^{*0,q}\hat D\boxtimes T^{*0,q}\hat D), \ \ j=0,1,\ldots,\]
where $\mu_0>0$ is as in Theorem~\ref{t-aldhmpI}, with $a_j(0,\hat x,\hat\eta)=c_j(\hat x,\hat\eta)$, $j=0,1,\ldots$. Let
$a(t,\hat x,\hat\eta)\sim\sum^\infty_{j=0}a_j(t,\hat x,\hat\eta)$
in $\hat S^{0}_{\mu_0}(\ol\Real_+\times U,T^{*0,q}\hat D\boxtimes T^{*0,q}\hat D)$. Choose $\chi\in C^\infty_0
(\Real^{2n})$ so that $\chi(\hat\eta)=1$ when $\abs{\hat\eta}<1$ and $\chi(\hat\eta)=0$ when $\abs{\hat\eta}>2$. 
Take $\tau(\hat x,\hat\eta)\in C^\infty(T^*\hat D)$, $\tau=1$ on $\ol W$, $\tau=0$ outside $U$ and $\tau$ is positively homogeneous of degree $0$.
Set
\begin{equation} \label{e-dcmIV}
A(\hat x,\hat y)=\frac{1}{(2\pi)^{2n}}\int\bigr(\int^{\infty}_0e^{i(\psi(t,\hat x,\hat\eta)-<\hat y,\hat\eta>)}a(t,\hat x,\hat\eta)(1-\chi(\hat\eta))\tau(\hat x,\hat\eta)dt\bigr)d\hat\eta.
\end{equation}
We can repeat the proof of Proposition 6.3 in~\cite{Hsiao08} with minor changes and conclude that

\begin{thm} \label{t-dcmimpI}
Assume that $q\neq n_-$. Let $\Td I=(2\pi)^{-2n}\int e^{i<\hat x-\hat y,\hat\eta>}c(\hat x,\hat\eta)d\hat\eta$
be a classical pseudodifferential operator on $\hat D$ of order $0$ from sections of $T^{*0,q}\hat D$ to sections of $T^{*0,q}\hat D$ with $c(\hat x,\hat\eta)\in S^0_{{\rm cl\,}}(\hat D,T^{*0,q}\hat D\boxtimes T^{*0,q}\hat D)$, ${\rm Supp\,}c(\hat x,\hat\eta)\bigcap T^*\hat D_0\subset\ol W$, where $W\subset U$ is a conic open set with $\ol W\subset U$. Let 
$A=A(x,y)\in L^{-1}_{\frac{1}{2},\frac{1}{2}}(\hat D,T^{*0,q}\hat D\boxtimes T^{*0,q}\hat D)$ be as in \eqref{e-dcmIV}. Then, on $\hat D_0$, 
\[\Box^{(q)}_s\circ A\equiv\Td I.\] 
\end{thm}

We assume that $q=n_-$. Let $\Td I$ be the classical pseudodifferential operator as in Theorem~\ref{t-dcmimpI}. Let
\[a_j(t,\hat x,\hat\eta)\in\hat S^{-j}_0(\ol\Real_+\times U,T^{*0,q}\hat D\boxtimes T^{*0,q}\hat D),\ \ j=0,1,\ldots,\]
and
$a_j(\infty,\hat x,\hat\eta)\in C^\infty(U,T^{*0,q}\hat D\boxtimes T^{*0,q}\hat D)$, $j=0,1,\ldots$,
be as in Theorem~\ref{t-aldhmpII}. We recall that for some $\mu>0$,
\[a_j(t,\hat x,\hat\eta)-a_j(\infty,\hat x,\hat\eta)\in\hat S^{-j}_{\mu}(\ol\Real_+\times U,T^{*0,q}\hat D\boxtimes T^{*0,q}\hat D),\ \ j=0,1,\ldots,\] 
and for every $(\hat x,\hat\eta)=((x,x_{2n}),(\eta_{2n}\eta,\eta_{2n}))\in U\bigcap\hat\Sigma$, $\eta=s_0\omega_0(x)-2{\rm Im\,}\ddbar_b\phi(x)$, we have 
\begin{equation}\label{e-gue1373V}
a_0(\infty,\hat x,\hat\eta)u\in\mathcal{N}(x,s_0,n_-),\ \ \forall u\in T^{*0,q}_{\hat x}\hat D.
\end{equation}
Let
\begin{equation} \label{e-dcmV}
a(\infty,\hat x,\hat\eta)\sim\sum^\infty_{j=0}a_j(\infty,\hat x, \hat\eta)
\end{equation}
in $S^{0}_{1,0}(U,T^{*0,q}\hat D\boxtimes T^{*0,q}\hat D)$. Let
\begin{equation} \label{e-dcmVI}
a(t,\hat x,\hat \eta)\sim\sum^\infty_{j=0}a_j(t,\hat x,\hat\eta)
\end{equation}
in $\hat S^{0}_0(\ol\Real_+\times U,T^{*0,q}\hat D\boxtimes T^{*0,q}\hat D)$.
We take $a(t,\hat x,\hat\eta)$ so that for every compact set $K\subset\pi(U)$ and all indices $\alpha, \beta\in\mathbb N^{2n}_0$, $\gamma, l\in\mathbb N_0$, there exists $c>0$ independent of $t$, such that
\begin{equation} \label{e-dcmVII}
\abs{\pr^\gamma_t\pr^\alpha_{\hat x}\pr^\beta_{\hat\eta}\Bigr(a(t,\hat x,\hat\eta)-\sum^l_{j=0}a_j(t,\hat x, \hat\eta)\Bigr)}
\leq c(1+\abs{\hat\eta})^{-l-1+\gamma-\abs{\beta}},
\end{equation}
where $t\in\ol\Real_+$, $\hat x\in K$, $(\hat x,\hat\eta)\in U$, $\abs{\hat\eta}\geq1$, and
\[a(t,\hat x,\hat\eta)-a(\infty,\hat x,\hat\eta)\in\hat S^0_{\mu}(\ol\Real_+\times U,T^{*0,q}\hat D\boxtimes T^{*0,q}\hat D)\ \ \mbox{with $\mu>0$}.\]
Choose $\chi\in C^\infty_0
(\Real^{2n})$ so that $\chi(\hat\eta)=1$ when $\abs{\hat\eta}<1$ and $\chi(\hat\eta)=0$ when $\abs{\hat\eta}>2$. 
Take $\tau(\hat x,\hat\eta)\in C^\infty(T^*\hat D)$, $\tau=1$ on $\ol W$, $\tau=0$ outside $U$ and $\tau$ is positively homogeneous of degree $0$.
Set
\begin{equation} \label{e-dcmVIII}\begin{split}
G(\hat x,\hat y) &= \frac{1}{(2\pi)^{2n}}\int\Bigr(\int^{\infty}_0\!\!\bigr(e^{i(\psi(t,\hat x,\hat\eta)-<\hat y,\hat\eta>)}a(t,\hat x,\hat\eta)  \\
  &\quad -e^{i(\psi(\infty,\hat x,\hat\eta)-<\hat y,\hat\eta>)}a(\infty,\hat x,\hat\eta)\bigr)(1-\chi(\hat\eta))\tau(\hat x,\hat\eta)dt\Bigr)d\hat\eta.
\end{split}\end{equation}
Put
\begin{equation} \label{e-dcmapaI}
S(\hat x,\hat y)=\frac{1}{(2\pi)^{2n}}\int e^{i(\psi(\infty,\hat x,\hat\eta)-<\hat y,\hat\eta>)}a(\infty,\hat x,\hat\eta)(1-\chi(\hat\eta))\tau(\hat x,\hat\eta)d\hat\eta.
\end{equation} 

We can repeat the proof of Proposition 6.5 in~\cite{Hsiao08} with minor changes and obtain 

\begin{thm} \label{t-dcmimpII}
We assume that $q=n_-$. Let 
$\Td I=(2\pi)^{-2n}\int e^{i<\hat x-\hat y,\hat\eta>}c(\hat x,\hat\eta)d\hat\eta$
be a classical pseudodifferential operator on $\hat D$ of order $0$ from sections of $T^{*0,q}\hat D$ to sections of $T^{*0,q}\hat D$ with $c(\hat x,\hat\eta)\in S^0_{{\rm cl\,}}(T^*\hat D,T^{*0,q}\hat D\boxtimes T^{*0,q}\hat D)$, ${\rm Supp\,}c(\hat x,\hat\eta)\bigcap T^*\hat D_0\subset\ol W$, where $W\subset U$ is a conic open set with $\ol W\subset U$. Let 
$G=G(\hat x,\hat y)\in L^{-1}_{\frac{1}{2},\frac{1}{2}}(\hat D,T^{*0,q}\hat D\boxtimes T^{*0,q}\hat D)$ be as in \eqref{e-dcmVIII} and 
let $S=S(\hat x,\hat y)\in L^{0}_{\frac{1}{2},\frac{1}{2}}(\hat D,T^{*0,q}\hat D\boxtimes T^{*0,q}\hat D)$ be as in \eqref{e-dcmapaI}. Then, 
\[S+\Box^{(q)}_s\circ G\equiv\Td I\ \ \mbox{on $\hat D_0$},\ \ \Box^{(q)}_s\circ S\equiv 0\ \ \mbox{on $\hat D$}.\] 
\end{thm} 

Now, we study the distribution kernel $S(\hat x,\hat y)$ of $S$. Fix $p\in D$ and assume that $D$ is a small open neighbourhood of $p$. We take local coordinates $x=(x_1,\ldots,x_{2n-1})$ so that $x(p)=0$, $\omega_0(p)=(0,0,\ldots,1)\in\Real^{2n-1}$ and 
\[-2{\rm Im\,}\ddbar_b\phi(p)=(\alpha_1,\ldots,\alpha_{2n-2},0):=(\alpha,0)\in\Real^{2n-1}.\] 
Thus, $((p,x_{2n}),\hat\xi)\in\hat\Sigma$ if and only if $\hat\xi=(\xi_{2n}\alpha_1,\xi_{2n}\alpha_2,\ldots,\xi_{2n}\alpha_{2n-2},\xi_{2n}\lambda,\xi_{2n})$, $\lambda\in\Real$. We need 

\begin{lem}\label{l-dcageI}
We have 
\[\det\left(\frac{\pr^2\psi}{\pr\eta_j\pr\eta_t}(\infty,(p,x_{2n}),(\alpha_1,\ldots,\alpha_{2n-2},\lambda,1))\right)^{2n-2}_{j,t=1}\neq 0,\]
for every $((p,x_{2n}),(\alpha_1,\ldots,\alpha_{2n-2},\lambda,1))\in U$. 
\end{lem}

\begin{proof}
We first claim that 
\begin{equation} \label{e-dcpamigeI}
\mbox{$\left(\frac{\pr^2{\rm Im\,}\psi}{\pr\eta_j\pr\eta_t}(\infty,(p,x_{2n}),(\alpha_1,\ldots,\alpha_{2n-2},\lambda,1))\right)^{2n-2}_{j,t=1}$ is positive definite}, 
\end{equation}
for every $((p,x_{2n}),(\alpha_1,\ldots,\alpha_{2n-2},\lambda,1))\in U$. We consider Taylor expansion of 
\[{\rm Im\,}\psi(\infty,(p,x_{2n}),(\eta_1,\ldots,\eta_{2n-2},\lambda,1))\] 
at $((p,x_{2n}),(\alpha_1,\ldots,\alpha_{2n-2},\lambda,1))$: 
\begin{equation}\label{e-dcpamigeII}
\begin{split}
&{\rm Im\,}\psi(\infty,(p,x_{2n}),(\eta_1,\ldots,\eta_{2n-2},\lambda,1))\\
&=\frac{1}{2}\sum^{2n-2}_{j,t=1}\frac{\pr^2{\rm Im\,}\psi}{\pr\eta_j\pr\eta_t}(\infty,(p,x_{2n}),(\alpha_1,\ldots,\alpha_{2n-2},\lambda,1))(\eta_j-\alpha_j)(\eta_t-\alpha_t)\\
&\quad+O(\sum^{2n-2}_{j=1}\abs{\eta_j-\alpha_j}^3).
\end{split}
\end{equation}
Here we use the fact that $({\rm Im\,}d_{\eta}\psi(\infty,(p,x_{2n}),(\alpha_1,\ldots,\alpha_{2n-2},\lambda,1))=0$ (see \eqref{e-dhlkmimIII}). From Theorem~\ref{t-dhlkmimII}, it is straightforward to see that 
\[{\rm Im\,}\psi(\infty,(p,x_{2n}),(\eta_1,\ldots,\eta_{2n-2},\lambda,1))\asymp\abs{\eta-\alpha}^2,\]
for $(\eta_1,\ldots,\eta_{2n-2})$ is in some small neighbourhood of $(\alpha_1,\ldots,\alpha_{2n-2})$. From this and \eqref{e-dcpamigeII}, we conclude that
\[\left(\frac{\pr^2{\rm Im\,}\psi}{\pr\eta_j\pr\eta_t}(\infty,(p,x_{2n}),(\alpha_1,\ldots,\alpha_{2n-2},\lambda,1))\right)^{2n-2}_{j,t=1}\] 
is positive definite. The claim \eqref{e-dcpamigeI} follows. 

Put $A=\left(\frac{\pr^2{\rm Im\,}\psi}{\pr\eta_j\pr\eta_t}(\infty,(p,x_{2n}),(\alpha_1,\ldots,\alpha_{2n-2},\lambda,1))\right)^{2n-2}_{j,t=1}={\rm Re\,}A+i{\rm Im\,}A$. Let $u\in\Complex^{2n-2}$. If $Au=0$, then $<({\rm Re\,}A+i{\rm Im\,}A)u, \ol u>=<({\rm Re\,}A)u, \ol u>+i<({\rm Im\,}A)u, \ol u>=0$. Thus, $<({\rm Re\,}A)u, \ol u>=<({\rm Im\,}A)u, \ol u>=0$. Since ${\rm Im\,}A$ is positive definite, we conclude that $u=0$. The lemma follows. 
\end{proof} 

Put 
\begin{equation}\label{e-dgugeI-I}
-2{\rm Im\,}\ddbar_b\phi(x)=(\alpha_1(x),\ldots,\alpha_{2n-1}(x)),\ \ \omega_0(x)=(\beta_1(x),\ldots,\beta_{2n-1}(x)),\ \ x\in D.
\end{equation}
From Lemma~\ref{l-dcageI}, we may take $V$ and $D$ small enough so that 
\[\det\left(\frac{\pr^2\psi}{\pr\eta_j\pr\eta_k}(\infty, \hat x,\hat\eta)\right)^{2n-2}_{j,k=1}\neq 0,\ \  \forall (\hat x,\hat\eta)\in U\] 
and 
\begin{equation}\label{e-dgugeI}
\beta_{2n-1}(x)\geq\frac{1}{2},\ \ \forall x\in D.
\end{equation}
Set 
\[\Td a(\hat x,\hat\eta):=a(\infty,\hat x,\hat\eta)(1-\chi(\hat\eta))\tau(\hat x,\hat\eta).\] 
Since $\tau(\hat x,\hat\eta)=0$ outside $U$, $\Td a(\hat x,\hat\eta)=0$ if $\eta_{2n}\leq0$. We have 
\begin{equation} \label{e-dgugeII}
\begin{split}
S(\hat x,\hat y)&=\frac{1}{(2\pi)^{2n}}\int e^{i(\psi(\infty,\hat x,\hat\eta)-<\hat y,\hat\eta>)}\Td a(\hat x,\hat\eta)d\hat\eta\\
&=\frac{1}{(2\pi)^{2n}}\int_{t>0}e^{it(\psi(\infty,\hat x,(w,1))-<\hat y,(w,1)>)}t^{2n-1}\Td a(\hat x,(tw,t))dwdt,
\end{split}
\end{equation}
where $\eta=(\eta_1,\ldots,\eta_{2n-1})=tw$, $\eta_{2n}=t$, $w=(w_1,\ldots,w_{2n-1})\in\Real^{2n-1}$. Let $w_{2n-1}=\alpha_{2n-1}(x)+s\beta_{2n-1}(x)$ and put $w'=(w_1,\ldots,w_{2n-2})$ in \eqref{e-dgugeII}, we get 
\begin{equation} \label{e-dgugeIII}
\begin{split}
&S(\hat x,\hat y)\\
&=\frac{1}{(2\pi)^{2n}}\int_{t>0}e^{it(\psi(\infty,\hat x,(w',\alpha_{2n-1}(x)+s\beta_{2n-1}(x),1))-<\hat y,(w',\alpha_{2n-1}(x)+s\beta_{2n-1}(x),1)>)}\times\\
&\quad\quad\quad t^{2n-1}\beta_{2n-1}(x)\Td a(\hat x,(tw',t\alpha_{2n-1}(x)+ts\beta_{2n-1}(x),t))dw'dsdt.
\end{split}
\end{equation}
Note that $\beta_{2n-1}(x)\geq\frac{1}{2}$, for every $x\in D$. 
The stationary phase method of Melin and Sj\"{o}strand~(see page 148 of~\cite{MS74})
then permits us to carry out the $w'$ integration in \eqref{e-dgugeIII}, to get
\begin{equation} \label{e-dgugeIV}
S(\hat x, \hat y)\equiv\int e^{it\Phi(\hat x, \hat y, s)}b(\hat x, \hat y, s, t)dsdt
\end{equation}
with
\begin{equation} \label{e-dgugeV}
b(\hat x, \hat y, s, t)\sim\sum^\infty_{j=0}b_j(\hat x,\hat y, s)t^{n-j}
\end{equation}
in $S^{n}_{1,0}(\hat\Omega\times]0, \infty[,T^{*0,q}_{\hat y}\hat D\boxtimes T^{*0,q}_{\hat x}\hat D)$, ${\rm Supp\,}b(\hat x,\hat y,s,t)\subset\hat\Omega\times\Real_+$, 
where 
\begin{equation}\label{e-dgugeVI}
\begin{split}
\hat\Omega:=&\{(\hat x,\hat y,s)\in\hat D\times\hat D\times\Real;\, (\hat x,(-2{\rm Im\,}\ddbar_b\phi(x)+s\omega_0(x),1))\in U\bigcap\hat\Sigma,\\
&\quad(\mbox{$\hat y,(-2{\rm Im\,}\ddbar_b\phi(y)+s\omega_0(y),1))\in U\bigcap\hat\Sigma$, 
$\abs{\hat x-\hat y}<\varepsilon$, for some $\varepsilon>0$}\}
\end{split}
\end{equation} 
and
\[{\rm Supp\,}b_j(\hat x,\hat y,s)\subset\hat\Omega,\ \ b_j(\hat x, \hat y,s)\in C^\infty(\hat\Omega,T^{*0,q}_{\hat y}\hat D\boxtimes T^{*0,q}_{\hat x}\hat D),\ \ j=0, 1,\ldots,\]
and $\Phi(\hat x, \hat y,s)\in C^\infty(\hat\Omega)$ is the corresponding critical value. Since $V$ is bounded, there is a constant $M>0$ so that $\abs{s}<M$, for every $(\hat x,\hat y,s)\in\hat\Omega$. Since $S$ is a pseudodifferential operator, $S(\hat x,\hat y)$ is smoothing away the diagonal $\hat x=\hat y$. We can take $\varepsilon>0$ in \eqref{e-dgugeVI} to be any small positive constant. That is, we may assume that $\Phi(\hat x, \hat y,s)$ and $b_j(\hat x, \hat y,s)$, $j=0, 1,\ldots$, are supported in some small neighbourhood of $\hat x=\hat y$. 

From \eqref{e-dhlkmimVII}, it is straightforward to see that we can take $\Phi(\hat x,\hat y,s)$ so that 
\begin{equation}\label{e-dgugeVII}
\begin{split}
\Phi(\hat x,\hat y,s)=x_{2n}-y_{2n}+\varphi(x,y,s),\ \ \varphi(x,y,s)\in C^\infty(\Omega),
\end{split}
\end{equation} 
where 
\begin{equation}\label{e-dgugeVIII}
\begin{split}
\Omega:=&\{(x,y,s)\in D\times D\times\Real;\, (x,-2{\rm Im\,}\ddbar_b\phi(x)+s\omega_0(x))\in V\bigcap\Sigma,\\
&\quad\mbox{$(y,-2{\rm Im\,}\ddbar_b\phi(y)+s\omega_0(y))\in V\bigcap\Sigma$, $\abs{x-y}<\varepsilon$, for some $\varepsilon>0$}\}.
\end{split}
\end{equation} 
Since 
\[d_{w'}(\psi(\infty,\hat x,(w',\alpha_{2n-1}(x)+s\beta_{2n-1}(x),1))-<\hat y,(w',\alpha_{2n-1}(x)+s\beta_{2n-1}(x),1))=0\] 
at $\hat x=\hat y$, $w'=(\alpha_1(x)+s\beta_1(x),\ldots,\alpha_{2n-2}(x)+s\beta_{2n-2}(x))$,
it follows that when $\hat x=\hat y$, the corresponding critical point is $w'=(\alpha_1(x)+s\beta_1(x),\ldots,\alpha_{2n-2}(x)+s\beta_{2n-2}(x))$ and consequently for every $(\hat x,\hat x,s)\in\hat\Omega$ and every $(x,x,s)\in\Omega$, 

\begin{gather}
\varphi(x, x,s)=0,\label{e-dgugeIX}\\
\varphi'_x(x, x,s)=(\alpha_1(x)+s\beta_1(x),\ldots,\alpha_{2n-1}(x)+s\beta_{2n-1}(x))=-2{\rm Im\,}\ddbar_b\phi(x)+s\omega_0(x),\label{e-dugeX}\\ 
\varphi'_y(x, x,s)=2{\rm Im\,}\ddbar_b\phi(x)-s\omega_0(x). \label{e-dugeXI}
\end{gather} 

Moreover, from the process above and \eqref{e-gue1373V}, it is easy to see that
\begin{equation} \label{e-gue1373VI}
b_0(\hat x, \hat x,s):T^{*0,q}_{\hat x}\hat D\To \mathcal{N}(x,s,n_-),\ \ \forall (\hat x,\hat x,s)\in\hat\Omega,
\end{equation}
where $b_0(\hat x,\hat y,s)$ is as in \eqref{e-dgugeV}. 

The following is essentially well-known~(see page 147 of~\cite{MS74} or Proposition~B.14 of paper~I in \cite{Hsiao08th}).

\begin{prop} \label{p-dgudmgeI}
With the notations used above, if $D$ and $V$ are small enough, then there is a constant $c>0$ such that
\begin{equation} \label{e-dgugeX} 
\begin{split}
&{\rm Im\,}\varphi(x, y,s)\geq c\inf_{w'\in\Lambda}\Big({\rm Im\,}\psi(\infty,\hat x,(w',\alpha_{2n-1}(x)+s\beta_{2n-1}(x),1))\\
&+|d_{w'}\bigr(\psi(\infty, \hat x, (w',\alpha_{2n-1}(x)+s\beta_{2n-1}(x),1))-<\hat y, (w',\alpha_{2n-1}(x)+s\beta_{2n-1}(x),1)>)|^2\Big),
\end{split}
\end{equation}
for all $(x,y,s)\in\Omega$, where $\Lambda$ is some open set of the origin in $\Real^{2n-2}$.
\end{prop}

From now on, we take $D$ and $V$ small enough so that \eqref{e-dgugeX} holds. We need

\begin{thm}\label{t-dgugeI} 
With the notations used above, there is a constant $c>0$ such that
\begin{equation} \label{e-dgugeXI}
{\rm Im\,}\varphi(x,y,s)\geq c\abs{x'-y'}^2,\ \ \forall (x,y,s)\in\Omega,
\end{equation}
where $x'=(x_1,\ldots,x_{2n-2})$, $y'=(y_1,\ldots,y_{2n-2})$. 

Moreover,  if $\varepsilon>0$ is small enough ($\varepsilon$ is as in \eqref{e-dgugeVI}) then 
there is a constant $c_1>0$ such that
\begin{equation}\label{e-gue1373VIIa}
{\rm Im\,}\varphi(x,y,s)+\abs{\frac{\pr\varphi}{\pr s}(x,y,s)}\geq c_1\bigr(\abs{x_{2n-1}-y_{2n-1}}+\abs{x'-y'}^2\bigr),\ \ \forall (x,y,s)\in\Omega,
\end{equation}
and
\begin{equation} \label{e-dgugeXII}
\mbox{$\Phi(\hat x,\hat y,s)=0$ and $\frac{\pr\Phi}{\pr s}(\hat x,\hat y,s)=\frac{\pr\varphi}{\pr s}(x,y,s)=0$ if and only if $x=y$.}
\end{equation}
\end{thm}

\begin{proof}
From 
\[\begin{split}
&\psi(\infty,\hat x, (w',\alpha_{2n-1}(x)+s\beta_{2n-1}(x),1))-<\hat y, (w',\alpha_{2n-1}(x)+s\beta_{2n-1}(x),1)> \\
&=<\hat x-\hat y, (w',\alpha_{2n-1}(x)+s\beta_{2n-1}(x),1)>\\
&\quad+O(\abs{w'-(\alpha_1(x)+s\beta_1(x),\ldots,\alpha_{2n-2}(x)+s\beta_{2n-2}(x))}^2)\end{split}\]
we can check that
\[\begin{split}
&d_{w'}\big(\psi(\infty,\hat x, (w',\alpha_{2n-1}(x)+s\beta_{2n-1}(x),1))-<\hat y, (w',\alpha_{2n-1}(x)+s\beta_{2n-1}(x),1)>\big)\\
&=<x'-y', dw'>+O(\abs{w'-(\alpha_1(x)+s\beta_1(x),\ldots,\alpha_{2n-2}(x)+s\beta_{2n-2}(x))}).\end{split}\]
Thus, there are
constants $c_1$, $c_2>0$ such that
\[\begin{split}
&\abs{d_{w'}\bigr(\psi(\infty, \hat x, (w',\alpha_{2n-1}(x)+s\beta_{2n-1}(x),1))-<\hat y, (w',\alpha_{2n-1}(x)+s\beta_{2n-1}(x),1)>)}^2\\
&\geq c_1\abs{x'-y'}^2-c_2\abs{w'-(\alpha_1(x)+s\beta_1(x),\ldots,\alpha_{2n-2}(x)+s\beta_{2n-2}(x))}^2.\end{split}\]
If $\frac{c_1}{2}\abs{(x'-y')}^2\geq c_2\abs{w'-(\alpha_1(x)+s\beta_1(x),\ldots,\alpha_{2n-2}(x)+s\beta_{2n-2}(x))}^2$, then
\begin{equation} \label{e-dgugeXII-I}
\begin{split}
&\abs{d_{w'}\bigr(\psi(\infty, \hat x, (w',\alpha_{2n-1}(x)+s\beta_{2n-1}(x),1))-<\hat y, (w',\alpha_{2n-1}(x)+s\beta_{2n-1}(x),1)>)}^2\\
&\geq\frac{c_1}{2}\abs{(x'-y')}^2.
\end{split}
\end{equation}
Now, we assume that 
\[\abs{(x'-y')}^2\leq\frac{2c_2}{c_1}\abs{w'-(\alpha_1(x)+s\beta_1(x),\ldots,\alpha_{2n-2}(x)+s\beta_{2n-2}(x))}^2.\] 
From Theorem~\ref{t-dhlkmimII}, we have
\begin{equation} \label{e-dgugeXIII}
\begin{split}
&{\rm Im\,}\psi(\infty,\hat x,(w',\alpha_{2n-1}(x)+s\beta_{2n-1}(x),1))\\
&\geq c_3\abs{w'-(\alpha_1(x)+s\beta_1(x),\ldots,\alpha_{2n-2}(x)+s\beta_{2n-2}(x))}^2\geq\frac{c_1c_3}{2c_2}\abs{(x'-y')}^2,
\end{split}
\end{equation}
where $c_3$ is a positive constant. From \eqref{e-dgugeXII-I}, \eqref{e-dgugeXIII} and Proposition~\ref{p-dgudmgeI}, \eqref{e-dgugeXI} follows. 

Now, we prove \eqref{e-gue1373VIIa}. In view of \eqref{e-dugeX} and \eqref{e-dugeXI}, we see that $\varphi(x,y,s)=<-2{\rm Im\,}\ddbar_b\phi(x)+s\omega_0(x),x-y>+O(\abs{x-y}^2)$. Thus, 
\begin{equation}\label{e-gmedmdhI}
\frac{\pr\varphi}{\pr s}(x,y,s)=<\omega_0(x),x-y>+O(\abs{x-y}^2)=\sum^{2n-1}_{j=1}\beta_j(x)(x_j-y_j)+O(\abs{x-y}^2).
\end{equation}
Since $\beta_{2n-1}(x)\geq\frac{1}{2}$ for every $x\in D$, we conclude that if $\varepsilon>0$ is small then 
there are constants $c_2>0$, $c_3>0$, such that 
\[\abs{\frac{\pr\varphi}{\pr s}(x,y,s)}\geq c_2\abs{x_{2n-1}-y_{2n-1}}-c_3\abs{x'-y'}^2,\ \ \forall (x,y,s)\in\Omega.\]
Combining this with \eqref{e-dgugeXI}, we obtain \eqref{e-gue1373VIIa}. 

Finally, from \eqref{e-dgugeXI} and \eqref{e-gue1373VIIa}, it is easy to see that \eqref{e-dgugeXII} holds and the theorem follows. 
\end{proof}

From now on, we take $\varepsilon>0$ small enough so that \eqref{e-gue1373VIIa} and \eqref{e-dgugeXII} hold. 

\begin{rem} \label{r-gudhanmpI}
The phase $t\Phi(\hat x,\hat y,s)$ is not positively homogeneous with respect to $(s,t)$. Since $t>0$ on $\hat\Omega$, we put $\Phi_0(\hat x,\hat y,s,t):=t\Phi(\hat x,\hat y,\frac{s}{t})$, $\Phi_0(\hat x,\hat y,s,t)\in C^\infty(\hat\Omega_0)$, 
where
\[\begin{split}
\hat\Omega_0&=\{(\hat x,\hat y,s,t)\in\hat D\times\hat D\times\Real\times\Real_+;\, (\hat x,(-2{\rm Im\,}\ddbar_b\phi(x)t+s\omega_0(x),t))\in U\bigcap\hat\Sigma,\\
&\quad(\hat y,(-2{\rm Im\,}\ddbar_b\phi(y)t+s\omega_0(y),t))\in U\bigcap\hat\Sigma,\ \ \abs{\hat x-\hat y}<\varepsilon\}.
\end{split}\]
It is easy to see that $\Phi_0$ is a complex valued phase function in the sense of Melin-Sj\"ostrand(see Definition 3.5 in~\cite{MS74}). By using the identification $t\Phi\leftrightarrow\Phi_0$, the framework of complex Fourier integral operators in~\cite{MS74} works well in this non-homogeneous case. 
\end{rem} 

We pause and introduce some notations. Let $g(x,y,s)\in C^\infty(\Omega)$ be a complex valued smooth function. Assume that $g(x,y,s)=0$ if and only if $x=y$. For $p\in D$, put 
\begin{equation}\label{e-guesw13622I}
\begin{split}
&T_{(p,p,s)}H_g\\
&=\set{(a_1,\ldots,a_{2n-1},b_1,\ldots,b_{2n-1})\in\Complex^{2n-1}\times\Complex^{2n-1};\, \sum^{2n-1}_{j=1}\bigr(a_j\frac{\pr^2g}{\pr x_j\pr s}(p,p,s)+b_j\frac{\pr^2g}{\pr y_j\pr s}(p,p,s)\bigr)=0}.
\end{split}
\end{equation}
The tangential Hessian of $g(x,y,s)$ at $(p, p,s)\in\Omega$ is the bilinear map ${\rm Hess\,}(g,{T_{(p,p,s)}H_g}):T_{(p, p,s)}H_g\times T_{(p,p,s)}H_g\To\Complex$ given by
\begin{equation} \label{e-guesw13622II}\begin{split}
{\rm Hess\,}(g,T_{(p,p,s)}H_g):T_{(p, p,s)}H_g\times T_{(p,p,s)}H_g&\To\Complex , \\
(u, v) &\To<g''(p, p)u, v>,\ \ u, v\in T_{(p,p,s)}H_g,
\end{split}\end{equation}
where $g''=\left[
\begin{array}[c]{cc}
  g''_{xx} & g''_{xy} \\
  g''_{yx} & g''_{yy}
\end{array}\right]$.
More precisely, if we put $u=(a_1,\ldots,a_{2n-1},b_1,\ldots,b_{2n-1})\in\Complex^{2n-1}\times\Complex^{2n-1}$, $v=(c_1,\ldots,c_{2n-1},d_1,\ldots,d_{2n-1})\in\Complex^{2n-1}\times\Complex^{2n-1}$, then 
\[\begin{split}
&<g''(p, p)u, v>\\
&=\sum^{2n-1}_{s,t=1}\Bigr(c_sa_t\frac{\pr^2 g}{\pr x_s\pr x_t}(p,p,s)+d_sa_t\frac{\pr^2 g}{\pr y_s\pr x_t}(p,p,s)+c_sb_t\frac{\pr^2 g}{\pr x_s\pr y_t}(p,p,s)+d_sb_t\frac{\pr^2 g}{\pr y_s\pr y_t}(p,p,s).\end{split}\]

In view of \eqref{e-dugeX} and \eqref{e-dugeXI}, it is easy to see that $T_{(p,p,s)}H_\varphi$ is spanned by
\begin{equation} \label{e-guesw13622III}
(u, v),\ \ \big(T(p), T(p)\big),\ \ u, v\in T^{1,0}_pX\oplus T^{0,1}_pX.
\end{equation} 

Let $\mathcal{U}$ be an open set in $\Real^N$. We let $\mathcal{U}^\Complex$ be an almost comlexification of $\mathcal{U}$. That is, $\mathcal{U}^\Complex$ is an open set in $\Complex^N$ with $\mathcal{U}^\Complex\bigcap\Real^N=\mathcal{U}$. For any smooth function $f\in C^\infty(\mathcal{U})$, 
we write $\Td f\in C^\infty(\mathcal{U}^\Complex)$ to denote an almost analytic extension of $f$. (See Chapter 1 of Melin-Sj\"ostrand~\cite{MS74}, for the precise meaning of "almost analytic extension ").  We need 

\begin{lem}\label{l-gue130804}
Let $\upsilon(x,y,s)\in C^\infty(\Omega)$. We assume that $\upsilon(x,y,s)$ satisfies
\eqref{e-dgugeIX}, \eqref{e-dugeX}, \eqref{e-dugeXI}, \eqref{e-dgugeXI} and \eqref{e-gue1373VIIa}. 
If $D$ is small enough then for 
every $(x_0,x_0,s_0)\in\Omega$, we can find a function $\hat\upsilon(x,y,s)\in C^\infty(\Lambda)$, where $\Lambda\subset\Omega$ is a small neighbourhood of $(x_0,x_0,s_0)$, such that $\hat\upsilon(x,y,s)$ satisfies \eqref{e-dgugeIX}, \eqref{e-dugeX}, \eqref{e-dugeXI}, \eqref{e-dgugeXI} and \eqref{e-gue1373VIIa} and $\frac{\pr\hat\upsilon}{\pr y_{2n-1}}(x,y,s)-(\alpha_{2n-1}(y)+s\beta_{2n-1}(y))$ vanishes to infinite order at $x=y$, ${\rm Hess\,}(\upsilon,T_{(x,x,s)}H_\upsilon)={\rm Hess\,}(\hat\upsilon,T_{(x,x,s)}H_{\hat\upsilon})$, $\forall (x,x,s)\in\Lambda$, and 
$t\Upsilon(\hat x,\hat y,s):=t(x_{2n}-y_{2n}+\upsilon(x,y,s))$ and $t\hat\Upsilon(\hat x,\hat y,s):=t(x_{2n}-y_{2n}+\hat\upsilon(x,y,s))$ are equivalent for classical symbols at every point of 
\[{\rm diag\,}'\Bigr((U\bigcap\hat\Sigma)\times(U\bigcap\hat\Sigma)\Bigr)\bigcap\set{(\hat x,\hat x,td_{\hat x}\Upsilon(\hat x,\hat x,s),-td_{\hat x}\Upsilon(\hat x,\hat x,s))\in T^*\hat D;\, (x,x,s)\in\Lambda, t>0}\]  in the sense of Melin-Sj\"{o}strand~\cite{MS74}.(Remind that ${\rm diag\,}'\Bigr((U\bigcap\hat\Sigma)\times(U\bigcap\hat\Sigma)\Bigr)$ is given by \eqref{e-gedhanmpI}.)
\end{lem} 

\begin{proof}
We first claim that we can find $g(x,y,s)\in C^\infty(\Lambda)$ with $g(x,x,s)=s$, where $\Lambda\subset\Omega$ is a small neighbourhood of $(x_0,x_0,s_0)$, such that if we put 
\[\upsilon_1(x,y,s):=\Td\upsilon(x,y,g(x,y,s))\] 
then $\frac{\pr\upsilon_1}{\pr y_{2n-1}}(x,y,s)-(\alpha_{2n-1}(y)+s\beta_{2n-1}(y))$ vanishes to infinite order at $x=y$. We formally set $g(y,y,s)=s$ and $\frac{\pr\upsilon_1}{\pr y_{2n-1}}(x,y,s)=\alpha_{2n-1}(y)+s\beta_{2n-1}(y)\mod O(\abs{x-y}^{\infty})$. Then, 
\[\begin{split}
&\frac{\pr\Td\upsilon}{\pr y_{2n-1}}(x,y,g(x,y,s))+\frac{\pr\Td\upsilon}{\pr s}(x,y,g(x,y,s))\frac{\pr g}{\pr y_{2n-1}}(x,y,s)\\
&=\alpha_{2n-1}(y)+s\beta_{2n-1}(y)\mod O(\abs{x-y}^{\infty}).\end{split}\]
Thus, 
\begin{equation}\label{e-gue130804}\begin{split}
&\frac{\pr^2\Td\upsilon}{\pr y_{2n-1}^2}(x,y,g(x,y,s))+2\frac{\pr^2\Td\upsilon}{\pr s\pr y_{2n-1}}(x,y,g(x,y,s))\frac{\pr g}{\pr y_{2n-1}}(x,y,s)\\
&+\frac{\pr^2\Td\upsilon}{\pr s^2}(x,y,g(x,y,s))(\frac{\pr g}{\pr y_{2n-1}}(x,y,s))^2
+\frac{\pr\Td\upsilon}{\pr s}(x,y,g(x,y,s))\frac{\pr^2g}{\pr y^2_{2n-1}}(x,y,s)\\
&=\frac{\pr\alpha_{2n-1}}{\pr y_{2n-1}}(y)+s\frac{\pr\beta_{2n-1}}{\pr y_{2n-1}}(y)\mod O(\abs{x-y}^{\infty}).\end{split}\end{equation}
Note that $\frac{\pr^2\Td\upsilon}{\pr s\pr y_{2n-1}}(y,y,g(y,y,s))=\frac{\pr^2\Td\upsilon}{\pr s\pr y_{2n-1}}(y,y,s)\neq0$(see \eqref{e-gmedmdhI}) and 
\[\frac{\pr\Td\upsilon}{\pr s}(y,y,g(y,y,s))=\frac{\pr\Td\upsilon}{\pr s^2}(y,y,g(y,y,s))=0.\] 
From this observation and \eqref{e-gue130804}, we can determine $\frac{\pr g}{\pr y_{2n-1}}(x,y,s)|_{x=y}$. Continuing in this way, we can determine $\frac{\pr^{\abs{\alpha}}g}{\pr y^\alpha}(x,y,s)|_{x=y}$, for every multiindex $\alpha=(\alpha_1,\ldots,\alpha_{2n-1})\in\mathbb N_0^{2n-1}$. By using Borel construction, the claim follows. 

Since $\alpha_{2n-1}(y)+s\beta_{2n-1}(y)$ is real, we have 
\begin{equation}\label{e-gue131412}
\frac{\pr^{\abs{\alpha}+1}{\rm Im\,}\upsilon_1}{\pr y^\alpha\pr y_{2n-1}}(x,y,s)|_{x=y}=0,
\end{equation} 
for every multiindex $\alpha=(\alpha_1,\ldots,\alpha_{2n-1})\in\mathbb N_0^{2n-1}$. Moreover, from \eqref{e-dgugeXI}, it is straightforward to see that if $D$ is small enough then, 
\begin{equation}\label{e-gue131412I}
\mbox{$\left(\frac{\pr^2{\rm Im\,}\upsilon_1}{\pr y_j\pr y_t}(x,y,s)|_{x=y}\right)^{2n-2}_{j,t=1}$ is positive definite at each point of $(x,x,s)\in\Lambda$}.
\end{equation}
From \eqref{e-gue131412} and \eqref{e-gue131412I}, we deduce that 
for every $N>0$, there is a $C_N>0$, such that 
\begin{equation}\label{e-gue130807I}
{\rm Im\,}\upsilon_1(x,y,s)+\frac{1}{C_N}\abs{x-y}^N\geq C_N\abs{x'-y'}^2,\ \ \forall (x,y,s)\in \Lambda.
\end{equation}
From \eqref{e-gue131412}, \eqref{e-gue130807I} and the standard Borel construction, we can find $\hat\upsilon(x,y,s)\in C^\infty(\Lambda)$ such that $\hat\upsilon(x,y,s)-\upsilon_1(x,y,s)$ vanishes to infinite order at $x=y$ and ${\rm Im\,}\hat\upsilon(x,y,s)\geq C_0\abs{x'-y'}^2$, $\forall (x,y,s)\in \Lambda$, where $C_0>0$ is a constant. Since  $\hat\upsilon(x,y,s)-\upsilon_1(x,y,s)$ vanishes to infinite order at $x=y$, $\upsilon(x,y,s)$ satisfies \eqref{e-dgugeIX}, \eqref{e-dugeX}, \eqref{e-dugeXI}, \eqref{e-gue1373VIIa} and $\frac{\pr\hat\upsilon}{\pr y_{2n-1}}(x,y,s)-(\alpha_{2n-1}(y)+s\beta_{2n-1}(y))$ vanishes to infinite order at $x=y$.

Now, we prove that $t\Upsilon(\hat x,\hat y,s):=t(x_{2n}-y_{2n}+\upsilon(x,y,s))$ and $t\hat\Upsilon(\hat x,\hat y,s):=t(x_{2n}-y_{2n}+\hat\upsilon(x,y,s))$ are equivalent for classical symbols at every point of 
\[{\rm diag\,}'\Bigr((U\bigcap\hat\Sigma)\times(U\bigcap\hat\Sigma)\Bigr)\bigcap\set{(\hat x,\hat x,td_{\hat x}\Upsilon(\hat x,\hat x,s),-td_{\hat x}\Upsilon(\hat x,\hat x,s))\in T^*\hat D;\, (x,x,s)\in\Lambda, t>0}.\] 
Fix $(\hat x_0,\hat x_0,s_0)\in\hat\Omega$, $(x,x,s)\in\Lambda$, $t_0>0$ and set
\[\begin{split}
&\hat x_0=(x_0,x_{n,0}),\ \ x_0\in\Real^{2n-1},\ \ (x_0,x_0,s_0)\in\Omega,\\
&(\hat x_0,\hat\xi_0)=(\hat x_0,t_0\frac{\pr\hat\Upsilon}{\pr\hat x}(\hat x_0,\hat x_0,s_0))=(\hat x_0,t_0\frac{\pr\Upsilon}{\pr\hat x}(\hat x_0,\hat x_0,s_0))\in\Bigr(U\bigcap\hat\Sigma\Bigr)\bigcap T^*\hat D.
\end{split}\] 
Let $\hat W$ be a small neighbourhood of $(\hat x_0,\hat x_0,s_0)$ and let $I_0$ be a neighbourhood of $t_0$ in $\Real_+$. Put 
\begin{equation}\label{e-gue140103}
\begin{split}
\Lambda_{\Td{t\hat\Upsilon}}:=&\{(\Td{\hat x},\Td{\hat y},\Td t\frac{\pr\Td{\hat\Upsilon}}{\pr\Td x}(\Td{\hat x},\Td{\hat y},\Td s),\Td t\frac{\pr\Td{\hat\Upsilon}}{\pr\Td y}(\Td{\hat x},\Td{\hat y},\Td s))\in\Complex^{2n}\times\Complex^{2n}\times\Complex^{2n}\times\Complex^{2n};\, \\
&\quad\Td{\hat\Upsilon}(\Td{\hat x},\Td{\hat y},\Td s)=0, \frac{\pr\Td{\hat\Upsilon}}{\pr\Td s}(\Td{\hat x},\Td{\hat y},\Td s)=0, (\Td{\hat x},\Td{\hat y},\Td s)\in\hat W^\Complex,  \Td t\in I_0^\Complex\},\\
\Lambda_{\Td{t\Upsilon}}:=&\{(\Td{\hat x},\Td{\hat y},\Td t\frac{\pr\Td{\Upsilon}}{\pr\Td x}(\Td{\hat x},\Td{\hat y},\Td s),\Td t\frac{\pr\Td{\Upsilon}}{\pr\Td y}(\Td{\hat x},\Td{\hat y},\Td s))\in\Complex^{2n}\times\Complex^{2n}\times\Complex^{2n}\times\Complex^{2n};\, \\
&\quad\Td{\Upsilon}(\Td{\hat x},\Td{\hat y},\Td s)=0, \frac{\pr\Td{\Upsilon}}{\pr\Td s}(\Td{\hat x},\Td{\hat y},\Td s)=0, (\Td{\hat x},\Td{\hat y},\Td s)\in\hat W^\Complex,  \Td t\in I_0^\Complex\}.\\
\end{split}
\end{equation} 
From global theory of complex Fourier integral operators of Melin-S\"ostrand~\cite{MS74}, we know that $t\hat\Upsilon(\hat x,\hat y,s)$ and $t\Upsilon(\hat x,\hat y,s)$ are equivalent for classical symbols at $(\hat x_0,\hat x_0,\hat\xi_0,-\hat\xi_0)\in{\rm diag\,}'\Bigr((U\bigcap\hat\Sigma)\times(U\bigcap\hat\Sigma)\Bigr)$ in the sense of Melin-Sj\"{o}strand~\cite{MS74} if and only if $\Lambda_{\Td{t\hat\Upsilon}}$ and $\Lambda_{\Td{t\Upsilon}}$ are equivalent in the sense that there is a neighbourhood $Q$ of $(\hat x_0,\hat x_0,\hat\xi_0,-\hat\xi_0)$ in $\Complex^{2n}\times\Complex^{2n}\times\Complex^{2n}\times\Complex^{2n}$, such that for every $N>0$, we have 
\begin{equation}\label{e-gue140103I}
\begin{split}
&{\rm dist\,}(z,\Lambda_{\Td{t\hat\Upsilon}})\leq C_N\abs{{\rm Im\,}z}^N,\ \ \forall z\in Q\bigcap\Lambda_{\Td{t\Upsilon}},\\
&{\rm dist\,}(z_1,\Lambda_{\Td{t\Upsilon}})\leq C_N\abs{{\rm Im\,}z_1}^N,\ \ \forall z_1\in Q\bigcap\Lambda_{\Td{t\hat\Upsilon}},
\end{split}
\end{equation}
where $C_N>0$ is independent of $z$ and $z_1$. Put $t\Upsilon_1(\hat x,\hat y,s)=t(x_{2n}-y_{2n}+\upsilon_1(x,y,s))$. We take almost analytic extensions of $t\hat\Upsilon$, $t\Upsilon$ and $t\Upsilon_1$ such that
\begin{equation}\label{e-gue140103II}
\begin{split}
&\Td{t\hat\Upsilon}(\Td{\hat x},\Td{\hat y},\Td s)=\Td t\Td{\hat\Upsilon}(\Td{\hat x},\Td{\hat y},\Td s)=\Td t(\Td{x}_{2n}-\Td{y}_{2n})+\Td t\Td{\hat\upsilon}(\Td x,\Td y,\Td s),\\
&\Td{t\Upsilon}(\Td{\hat x},\Td{\hat y},\Td s)=\Td t\,\Td{\Upsilon}(\Td{\hat x},\Td{\hat y},\Td s)=\Td t(\Td{x}_{2n}-\Td{y}_{2n})+\Td t\Td{\upsilon}(\Td x,\Td y,\Td s),\\
&\Td{t\Upsilon_1}(\Td{\hat x},\Td{\hat y},\Td s)=\Td t\,\Td{\Upsilon_1}(\Td{\hat x},\Td{\hat y},\Td s)=\Td t(\Td{x}_{2n}-\Td{y}_{2n})+\Td t\Td{\upsilon_1}(\Td x,\Td y,\Td s),
\end{split}
\end{equation} 
and near $(\hat x_0,\hat x_0,\hat\xi_0,-\hat\xi_0)$, we have
\begin{equation}\label{e-gue140103III}
\Lambda_{\Td{t\Upsilon}}=\Lambda_{\Td{t\Upsilon_1}},
\end{equation}
where $\Lambda_{\Td{t\Upsilon_1}}$ is defined as in \eqref{e-gue140103}, $(\Td{\hat x},\Td{\hat y},\Td s)\in\hat W^\Complex$, $\Td t\in I_0^\Complex$. Thus, we only need to prove that $\Lambda_{\Td{t\hat\Upsilon}}$ and $\Lambda_{\Td{t\Upsilon_1}}$ are equivalent in the sense of \eqref{e-gue140103I}. 

Since $\hat\upsilon(x,y,s)-\upsilon_1(x,y,s)$ vanishes to infinite order at $x=y$, it is straightforward to see that (see section~\ref{sub-pf}) there is a neighbourhood $Q$ of $(\hat x_0,\hat x_0,\hat\xi_0,-\hat\xi_0)$ in $\Complex^{2n}\times\Complex^{2n}\times\Complex^{2n}\times\Complex^{2n}$, such that for every $N>0$ and every $z=(\Td{\hat x},\Td{\hat y},\Td t\frac{\pr\Td{\Upsilon}}{\pr\Td x}(\Td{\hat x},\Td{\hat y},\Td s),\Td t\frac{\pr\Td{\Upsilon}}{\pr\Td y}(\Td{\hat x},\Td{\hat y},\Td s))\in Q\bigcap\Lambda_{\Td{t\Upsilon}}$, we have 
\begin{equation}\label{e-gue140103IV}
{\rm dist\,}(z,\Lambda_{\Td{t\hat\Upsilon}})\leq C_N\Bigr(\abs{{\rm Im\,}(\Td x',\Td y,\Td s)}^N+\abs{{\rm Re\,}\Td x'-{\rm Re\,}\Td y}^N\Bigr),
\end{equation}
where $C_N>0$ is independent of $z\in Q$. Moreover, we can repeat the process in section~\ref{sub-pf} and conclude that if $Q$ is small enough then there is a constant $C_1>0$ independent of $z\in Q$ such that 
\begin{equation}\label{e-gue140103V}
\abs{{\rm Im\,}(\Td x',\Td y,\Td s)}+\abs{{\rm Re\,}\Td x'-{\rm Re\,}\Td y}\leq C_1\abs{{\rm Im\,}z},
\end{equation}
for every $z=(\Td{\hat x},\Td{\hat y},\Td t\frac{\pr\Td{\Upsilon}}{\pr\Td x}(\Td{\hat x},\Td{\hat y},\Td s),\Td t\frac{\pr\Td{\Upsilon}}{\pr\Td y}(\Td{\hat x},\Td{\hat y},\Td s))\in Q\bigcap\Lambda_{\Td{t\Upsilon}}$. From \eqref{e-gue140103IV} and \eqref{e-gue140103V}, we get the first formula in \eqref{e-gue140103I}. Similarly, we can repeat the process above and conclude the second formula in \eqref{e-gue140103I}. Moreover, from the construction above, it is easy to see that
\[{\rm Hess\,}(\upsilon,T_{(x,x,s)}H_\upsilon)={\rm Hess\,}(\hat\upsilon,T_{(x,x,s)}H_{\hat\upsilon}),\ \ \forall (x,x,s)\in\Lambda.\] 
The lemma follows. 
\end{proof}

\begin{defn}\label{d-geudhdc13619}
Let $\Phi_1(\hat x,\hat y,s)=x_{2n}-y_{2n}+\varphi_1(x,y,s)\in C^\infty(\hat\Omega)$, $\Phi_2(\hat x,\hat y,s)=x_{2n}-y_{2n}+\varphi_2(x,y,s)\in C^\infty(\hat\Omega)$, $, \varphi_1(x,y,s), \varphi_2(x,y,s)\in C^\infty(\Omega)$. We assume that $\varphi_1$ and $\varphi_2$ satisfy 
\eqref{e-dgugeIX}, \eqref{e-dugeX}, \eqref{e-dugeXI}, \eqref{e-dgugeXI} and \eqref{e-gue1373VIIa}. Let $(x_0,x_0,s_0)\in\Omega$. From Lemma~\ref{l-gue130804}, we know that in some small neighbourhood $\Lambda\subset\Omega$ of $(x_0,x_0,s_0)$, there are $\hat\varphi_1(x,y,s), \hat\varphi_2(x,y,s)\in C^\infty(\Lambda)$ such that $\varphi_1$ and $\varphi_2$ satisfy 
\eqref{e-dgugeIX}, \eqref{e-dugeX}, \eqref{e-dugeXI}, \eqref{e-dgugeXI},  \eqref{e-gue1373VIIa} and $\frac{\pr\hat\varphi_1}{\pr y_{2n-1}}(x,y,s)-(\alpha_{2n-1}(y)+s\beta_{2n-1}(y))$ and $\frac{\pr\hat\varphi_2}{\pr y_{2n-1}}(x,y,s)-(\alpha_{2n-1}(y)+s\beta_{2n-1}(y))$  vanish to infinite order at $x=y$, ${\rm Hess\,}(\varphi_1,T_{(x,x,s)}H_{\varphi_1})={\rm Hess\,}(\hat\varphi_1,T_{(x,x,s)}H_{\hat\varphi_1})$, $\forall (x,x,s)\in\Lambda$, ${\rm Hess\,}(\varphi_2,T_{(x,x,s)}H_{\varphi_2})={\rm Hess\,}(\hat\varphi_2,T_{(x,x,s)}H_{\hat\varphi_2})$, $\forall (x,x,s)\in\Lambda$ and $t\hat\Phi_1(\hat x,\hat y,s):=t(x_{2n}-y_{2n}+\hat\varphi_1(x,y,s))$ and $t\Phi_1(\hat x,\hat y,s)$ are equivalent for classical symbols at every point of 
\[{\rm diag\,}'\Bigr((U\bigcap\hat\Sigma)\times(U\bigcap\hat\Sigma)\Bigr)\bigcap\set{(\hat x,\hat x,td_{\hat x}\Phi_1(\hat x,\hat x,s),-td_{\hat x}\Phi_1(\hat x,\hat x,s))\in T^*\hat D;\, (x,x,s)\in\Lambda, t>0}\]  in the sense of Melin-Sj\"{o}strand~\cite{MS74}, $t\hat\Phi_2(\hat x,\hat y,s):=t(x_{2n}-y_{2n}+\hat\varphi_2(x,y,s))$ and $t\Phi_2(\hat x,\hat y,s)$ are equivalent for classical symbols at every point of 
\[{\rm diag\,}'\Bigr((U\bigcap\hat\Sigma)\times(U\bigcap\hat\Sigma)\Bigr)\bigcap\set{(\hat x,\hat x,td_{\hat x}\Phi_2(\hat x,\hat x,s),-td_{\hat x}\Phi_2(\hat x,\hat x,s))\in T^*\hat D;\, (x,x,s)\in\Lambda, t>0}\]  in the sense of Melin-Sj\"{o}strand~\cite{MS74}. 
We say that $\varphi_1(x,y,s)$ and $\varphi_2(x,y,s)$ are equivalent at $(x_0,x_0,s)$ if there are functions $f\in C^\infty(\Lambda')$, $g_j\in C^\infty(\Lambda')$, $j=0,1,\ldots,2n-1$, $p_j\in C^\infty(\Lambda')$, $j=1,\ldots,2n-1$, such that
\[
\begin{split}
&\frac{\pr\hat\varphi_1}{\pr s}(x,y,s)-f(x,y,s)\frac{\pr\hat\varphi_2}{\pr s}(x,y,s),\\
&\hat\varphi_1(x,y,s)-\hat\varphi_2(x,y,s)=g_0(x,y,s)\frac{\pr\hat\varphi_1}{\pr s}(x,y,s),\\
&\frac{\pr\hat\varphi_1}{\pr x_j}(x,y,s)-\frac{\pr\hat\varphi_2}{\pr x_j}(x,y,s)=g_j(x,y,s)\frac{\pr\hat\varphi_1}{\pr s}(x,y,s),\ \ j=1,2,\ldots,2n-1, \\
&\frac{\pr\hat\varphi_1}{\pr y_j}(x,y,s)-\frac{\pr\hat\varphi_2}{\pr y_j}(x,y,s)=p_j(x,y,s)\frac{\pr\hat\varphi_1}{\pr s}(x,y,s),\ \ j=1,2,\ldots,2n-1,
\end{split}\]
vanish to infinite order on $x=y$, for every $(x,y,s)\in\Lambda'$, where $\Lambda'\subset\Lambda$ is a small neighbourhood of $(x_0,x_0,s_0)$. 
\end{defn}


We have

\begin{thm} \label{t-dgudmgeaI}
Let $\Phi_1(\hat x,\hat y,s)=x_{2n}-y_{2n}+\varphi_1(x,y,s)\in C^\infty(\hat\Omega)$, $\varphi_1(x,y,s)\in C^\infty(\Omega)$. Assume that $\varphi_1$ satisfis
\eqref{e-dgugeIX}, \eqref{e-dugeX}, \eqref{e-dugeXI}, \eqref{e-dgugeXI} and \eqref{e-gue1373VIIa}. 
Then $t\Phi(\hat x,\hat y,s)$ and $t\Phi_1(\hat x,\hat y,s)$ are equivalent for classical symbols at every point of 
\[{\rm diag\,}'\Bigr((U\bigcap\hat\Sigma)\times(U\bigcap\hat\Sigma)\Bigr)\bigcap\set{(\hat x,\hat x,\hat\xi,-\hat\xi);\, (\hat x,\hat\xi)\in T^*\hat D}.\] (remind that ${\rm diag\,}'\Bigr((U\bigcap\hat\Sigma)\times(U\bigcap\hat\Sigma)\Bigr)$ is given by \eqref{e-gedhanmpI}) in the sense of Melin-Sj\"{o}strand~\cite{MS74} if and only if $\varphi(x,y,s)$ and $\varphi_1(x,y,s)$ are equivalent at each point of $\Omega$ in the sense of Definition~\ref{d-geudhdc13619}.
\end{thm}

The proof is straightforward and follows from global theory of complex Fourier integral operators of Melin-Sj\"ostrand~\cite{MS74}. We put the proof in section~\ref{sub-pf}.

We notice that $\psi(\infty,\hat x,\hat\eta)-<\hat y,\hat\eta>$ and $t\Phi(\hat x, \hat y,s)$ are equivalent
for classical symbols at every point of 
${\rm diag\,}'\Bigr((U\bigcap\hat\Sigma)\times(U\bigcap\hat\Sigma)\Bigr)\bigcap\set{(\hat x,\hat x,\hat\xi,-\hat\xi);\, (\hat x,\hat\xi)\in T^*\hat D}$ in the sense of Melin-Sj\"{o}strand~\cite{MS74}.  Consider 
\[-\ol\Phi(\hat y,\hat x,s)=x_{2n}-y_{2n}-\ol\varphi(y,x,s).\]
From Theorem~\ref{t-dhlkmimIII}, we see that $t\Phi(\hat x,\hat y,s)$ and $-t\ol\Phi(\hat y,\hat x,s)$ are equivalent
for classical symbols at every point of ${\rm diag\,}'\Bigr((U\bigcap\hat\Sigma)\times(U\bigcap\hat\Sigma)\Bigr)$ in the sense of Melin-Sj\"{o}strand~\cite{MS74}. Note that $-\ol\varphi(y,x,s)$ satisfies \eqref{e-dgugeIX}, \eqref{e-dugeX}, \eqref{e-dugeXI}, \eqref{e-dgugeXI} and \eqref{e-gue1373VIIa}. From Theorem~\ref{t-dgudmgeaI}, we see that $\varphi(x,y,s)$ and $-\ol\varphi(y,x,s)$ are equivalent at each point of $\Omega$ in the sense of Definition~\ref{d-geudhdc13619}.

Summing up, we obtain the main result of this section 

\begin{thm} \label{t-dcgewI}
With the notations and assumptions above. 
Let $S=S(\hat x,\hat y)\in L^0_{\frac{1}{2},\frac{1}{2}}(\hat D,T^{*0,q}\hat D\boxtimes T^{*0,q}\hat D)$ be as in Theorem~\ref{t-dcmimpII}. Then, on $\hat D$, we have 
\begin{equation}\label{e-dcgewI}
S(\hat x, \hat y)\equiv\int e^{it\Phi(\hat x, \hat y, s)}b(\hat x, \hat y, s, t)dsdt
\end{equation}
with
\begin{equation} \label{e-dcgewII}
b(\hat x, \hat y, s, t)\sim\sum^\infty_{j=0}b_j(\hat x,\hat y, s)t^{n-j}
\end{equation}
in $S^{n}_{1,0}(\hat\Omega\times]0, \infty[,T^{*0,q}_{\hat y}\hat D\boxtimes T^{*0,q}_{\hat x}\hat D)$, ${\rm Supp\,}b(\hat x, \hat y, s, t)\subset\hat\Omega\times\Real_+$,  
\begin{equation}\label{e-gue1373VII}
b_0(\hat x, \hat x,s):T^{*0,q}_{\hat x}\hat D\To \mathcal{N}(x,s,n_-),\ \ \forall (\hat x,\hat x,s)\in\hat\Omega,
\end{equation}
where $\mathcal{N}(x,s,n_-)$ is given by \eqref{e-gue1373III}, 
\begin{equation}\label{e-dcgewIIa}
\begin{split}
\hat\Omega:=&\{(\hat x,\hat y,s)\in\hat D\times\hat D\times\Real;\, (\hat x,(-2{\rm Im\,}\ddbar_b\phi(x)+s\omega_0(x),1))\in U\bigcap\hat\Sigma,\\
&\quad(\mbox{$\hat y,(-2{\rm Im\,}\ddbar_b\phi(y)+s\omega_0(y),1))\in U\bigcap\hat\Sigma$, 
$\abs{\hat x-\hat y}<\varepsilon$, for some $\varepsilon>0$}\},
\end{split}
\end{equation} 
\[{\rm Supp\,}b_j(\hat x, \hat y, s)\subset\hat\Omega,\ \ b_j(\hat x, \hat y,s)\in C^\infty(\hat\Omega,T^{*0,q}_{\hat y}\hat D\boxtimes T^{*0,q}_{\hat x}\hat D),\ \ j=0, 1,\ldots,\]
\[\begin{split}
&\Phi(\hat x,\hat y,s)=x_{2n}-y_{2n}+\varphi(x,y,s),\\ 
&\varphi(x,y,s)\in C^\infty(\Omega),\ \ \Omega=\set{(x,y,s)\in D\times D\times\Real;\, (\hat x,\hat y,s)\in\hat\Omega},\end{split}\]
and $\varphi(x,y,s)$ satisfies \eqref{e-dgugeIX}, \eqref{e-dugeX}, \eqref{e-dugeXI}, \eqref{e-dgugeXI} and \eqref{e-gue1373VIIa}. Furthermore, $\varphi(x,y,s)$ and $-\ol\varphi(y,x,s)$ are equivalent at each point of $\Omega$ in the sense of Definition~\ref{d-geudhdc13619}. 

Moreover, the phase $t\Phi(\hat x,\hat y,s)$ can be characterized as follows: Let $\Phi_1(\hat x,\hat y,s)=x_{2n}-y_{2n}+\varphi_1(x,y,s)\in C^\infty(\hat\Omega)$, $\varphi_1(x,y,s)\in C^\infty(\Omega)$. We assume that $\varphi_1$ satisfies 
\eqref{e-dgugeIX}, \eqref{e-dugeX}, \eqref{e-dugeXI}, \eqref{e-dgugeXI} and \eqref{e-gue1373VIIa}.
Then $t\Phi(\hat x,\hat y,s)$ and $t\Phi_1(\hat x,\hat y,s)$ are equivalent for classical symbols at every point of 
\[{\rm diag\,}'\Bigr((U\bigcap\hat\Sigma)\times(U\bigcap\hat\Sigma)\Bigr)\bigcap\set{(\hat x,\hat x,\hat\xi,-\hat\xi);\, (\hat x,\hat\xi)\in T^*\hat D}\] 
in the sense of Melin-Sj\"{o}strand~\cite{MS74} if and only if  $\varphi(x,y,s)$ and $\varphi_1(x,y,s)$ are equivalent at each point of $\Omega$ in the sense of Definition~\ref{d-geudhdc13619}. 
\end{thm}

\subsection{The tangential Hessian of $\varphi(x,y,s)$} \label{s-cth} 

In this section, we will calculate the tangential Hessian of $\varphi(x,y,s)$ and we will use the same notations as before. Let $x=(x_1,\ldots,x_{2n-1})$ be local coordinates on $D$. 
The following is straightforward. We omit the proof.

\begin{prop}\label{p-geusw13622I}
Let $\varphi_1(x,y,s), \varphi_2(x,y,s)\in C^\infty(\hat\Omega)$. We assume that $\varphi_1$ and $\varphi_2$ satisfy 
\eqref{e-dgugeIX}, \eqref{e-dugeX}, \eqref{e-dugeXI}, \eqref{e-dgugeXI} and \eqref{e-gue1373VIIa}.
Then, $T_{(x,x,s)}H_{\varphi_1}=T_{(x,x,s)}H_{\varphi_2}$, for every $(x,x,s)\in\Omega$. Assume further that 
$\varphi_1(x,y,s)$ and $\varphi_2(x,y,s)$ are equivalent at each point of $\Omega$ in the sense of Definition~\ref{d-geudhdc13619}. Then, ${\rm Hess\,}(\varphi_1,{T_{(x,x,s)}H_{\varphi_1}})={\rm Hess\,}(\varphi_2,{T_{(x,x,s)}H_{\varphi_2}})$, $\forall (x,x,s)\in\Omega$.

In particular, if we put $\varphi_1(x,y,s)=-\ol\varphi(y,x,s)$ then
${\rm Hess\,}(\varphi,{T_{(x,x,s)}H_\varphi})={\rm Hess\,}(\varphi_1,{T_{(x,x,s)}H_{\varphi_1}})$, 
$\forall (x,x,s)\in\Omega$. 
\end{prop}

From Proposition~\ref{p-geusw13622I}, we know that the tangential Hessian of $\varphi$ at $(x,x,s)\in\Omega$ is uniquely determined in the equivalence class of $\varphi$ in the sense of Definition~\ref{d-geudhdc13619}. In the rest of this section, we will determine the tangential Hessian of $\varphi(x,y,s)$ at $(x,x,s)\in\Omega$. Until further notice, we fix $(p,p,s_0)\in\Omega$. Recall that (see \eqref{e-dhmpXII} and \eqref{e-dgugeVIII}) $M^\phi_p-2s_0\mathcal{L}_p$ is non-degenerate of constant signature $(n_-,n_+)$. We can repeat the proof of Lemma 8.1 in~\cite{Hsiao08} with minor change and conclude that 

\begin{prop} \label{p-geusw13622II}
Let $\ol Z_{1,s_0},\ldots,\ol Z_{n-1,s_0}$ be an orthonormal frame of $T^{1,0}_xX$ varying smoothly with $x$ in a neighbourhood of $p$, for which the Hermitian quadratic form $M^\phi_x-2s_0\mathcal{L}_x$ is diagonalized at $p$. That is, 
\begin{equation}\label{e-guesw13622IV}
M^\phi_p\bigr(\ol Z_{j,s_0}(p),Z_{t,s_0}(p)\bigr)-2s_0\mathcal{L}_p\bigr(\ol Z_{j,s_0}(p),Z_{t,s_0}(p)\bigr)
=\lambda_j(s_0)\delta_{j,t},\ \ j,t=1,\ldots,n-1.
\end{equation}
Assume that $\lambda_j(s_0)<0$, $j=1,\ldots,n_-$, 
$\lambda_j(s_0)>0$, $j=n_-+1,\ldots,n-1$. Let $x=(x_1,\ldots,x_{2n-1})$ be local coordinates of $X$ defined in some small neighbourhood of $p$ such that $x(p)=0$. Let $h_j(x,\xi)$ be the principal symbol of $Z_{j,s_0}$, $j=1,\ldots,n-1$. 
Then in some open neighbourhood $W\subset\Omega$ of $(p,p,s_0)$, there exist $g_j(x,y,s)\in C^\infty(W)$, $j=1,\ldots,n-1$, such that 
\begin{equation}\label{e-guesw13622V}
\begin{split}
&\ol h_j(x,\varphi'_x(x,y,s_0))+(\ol Z_{j,s_0}\phi)(x)=g_j(x,y,s_0)\frac{\pr\varphi}{\pr s}(x,y,s_0)+O(\abs{(x,y)}^2),\ \ j=1,\ldots,n_-,\\
&h_j(x,\varphi'_x(x,y,s_0))+(Z_{j,s_0}\phi)(x)=g_j(x,y,s_0)\frac{\pr\varphi}{\pr s}(x,y,s_0)+O(\abs{(x,y)}^2),\ \ j=n_-+1,\ldots,n-1.
\end{split}
\end{equation}
\end{prop}

Let $\ol Z_{1,s_0},\ldots,\ol Z_{n-1,s_0}$ be as in Proposition~\ref{p-geusw13622II}. We take local coordinates $x=(x_1,\ldots,x_{2n-1})$, $z_j=x_{2j-1}+ix_{2j}$, $j=1,\ldots,n-1$, defined in some small neighbourhood of $p$ so that \eqref{e-geusw13623} hold. 
From \eqref{e-geusw13623}, \eqref{e-dugeX} and \eqref{e-dugeXI} it is not difficult to see that 
\begin{equation}\label{e-geusw13623b}
\begin{split}
&\frac{\pr^2\varphi}{\pr s\pr x_j}(p,p,s_0)=\frac{\pr^2\varphi}{\pr s\pr y_j}(p,p,s_0)=0,\ \ j=1,\ldots,2n-2,\\
&\frac{\pr^2\varphi}{\pr s\pr x_{2n-1}}(p,p,s_0)=1,\ \ \frac{\pr^2\varphi}{\pr s\pr y_{2n-1}}(p,p,s_0)=-1.
\end{split}
\end{equation}
From \eqref{e-geusw13623b}, it is easy to see that to determine the tangential Hessian of $\varphi(x,y,s)$ at $(p,p,s_0)$ is equivalent to determine
\begin{equation}\label{e-geusw13623a}\begin{split}
&\frac{\pr^2\varphi}{\pr x_j\pr x_l}(p,p,s_0),\ \ \frac{\pr^2\varphi}{\pr x_j\pr y_l}(p,p,s_0),\ \ \frac{\pr^2\varphi}{\pr y_j\pr y_l}(p,p,s_0),\ \  j, l=1,\ldots,2n-2,\\
&\bigr(\frac{\pr^2\varphi}{\pr x_j\pr x_{2n-1}}+\frac{\pr^2\varphi}{\pr x_j\pr y_{2n-1}})(p,p,s_0),\ \ (\frac{\pr^2\varphi}{\pr y_j\pr x_{2n-1}}+\frac{\pr^2\varphi}{\pr y_j\pr y_{2n-1}}\bigr)(p,p,s_0),\ \ j=1,\ldots,2n-2,\\
&\bigr(\frac{\pr^2\varphi}{\pr x^2_{2n-1}}+2\frac{\pr^2\varphi}{\pr x_{2n-1}\pr y_{2n-1}}+\frac{\pr^2\varphi}{\pr y^2_{2n-1}}\bigr)(p,p,s_0).
 \end{split}\end{equation}
 
From \eqref{e-geusw13623}, \eqref{e-suIII} and \eqref{e-dhI}, it is straightforward to check that 
\begin{equation}\label{e-geusw13623I}
\begin{split}
&M^\phi_p\bigr(\ol Z_{j,s_0}(p),Z_{l,s_0}(p)\bigr)=(i\tau_{j,l}-i\ol\tau_{l,j})\beta+\mu_{j,l},\ \ j,l=1,\ldots,n-1,\\
&\mathcal{L}_p\bigr(\ol Z_{j,s_0}(p),Z_{l,s_0}(p)\bigr)=-\frac{1}{2}(\tau_{j,l}+\ol\tau_{l,j}),\ \ j,l=1,\ldots,n-1.
\end{split}
\end{equation}
Since $M^\phi_p-2s_0\mathcal{L}_p$ is diagonal in the basis $\set{\ol Z_{1,s_0},\ldots,\ol Z_{n-1,s_0}}$, we have 
\begin{equation}\label{e-geusw13623II}
(i\tau_{j,l}-i\ol\tau_{l,j})\beta+\mu_{j,l}+s_0(\tau_{j,l}+\ol\tau_{l,j})=\lambda_{j}(s_0)\delta_{j,l},\ \ j,l=1,\ldots,n-1. 
\end{equation}

We write $y=(y_1,\ldots,y_{2n-1})$, $w_j=y_{2j-1}+iy_{2j}$, $j=1,\ldots,n-1$, 
\[\frac{\pr}{\pr w_j}=\frac{1}{2}(\frac{\pr}{\pr y_{2j-1}}-i\frac{\pr}{\pr y_{2j}}),\ \ 
\frac{\pr}{\pr\ol w_j}=\frac{1}{2}(\frac{\pr}{\pr y_{2j-1}}+i\frac{\pr}{\pr y_{2j}}),\ \ j=1,\ldots,n-1.\]
From \eqref{e-guesw13622V} and \eqref{e-geusw13623}, we can check that 
\begin{equation}\label{e-geusw13623III}
\begin{split}
&-i\frac{\pr\varphi}{\pr z_j}(x,y,s_0)+\sum^{n-1}_{t=1}\tau_{j,t}\ol z_t\frac{\pr\varphi}{\pr x_{2n-1}}(x,y,s_0)-ic_jx_{2n-1}\frac{\pr\varphi}{\pr x_{2n-1}}(x,y,s)+(\ol Z_{j,s_0}\phi)(x)\\
&=g_j(x,y,s_0)\frac{\pr\varphi}{\pr s}(x,y,s_0)+O(\abs{(x,y)}^2),\ \ j=1,\ldots,n_-,\\
&i\frac{\pr\varphi}{\pr\ol z_j}(x,y,s_0)+\sum^{n-1}_{t=1}\ol\tau_{j,t}z_t\frac{\pr\varphi}{\pr x_{2n-1}}(x,y,s_0)+i\ol c_jx_{2n-1}\frac{\pr\varphi}{\pr x_{2n-1}}(x,y,s_0)+(Z_{j,s_0}\phi)(x)\\
&=g_j(x,y,s_0)\frac{\pr\varphi}{\pr s}(x,y,s_0)+O(\abs{(x,y)}^2),\ \ j=n_-+1,\ldots,n-1.\\
\end{split}
\end{equation}
From \eqref{e-geusw13623}, \eqref{e-geusw13623b}, \eqref{e-geusw13623III} and notice that $\frac{\pr\varphi}{\pr x_{2n-1}}(p,p,s_0)=s_0$, it is straightforward to see that
\begin{equation} \label{e-geusw13623IV}\begin{split}
&\frac{\pr^2\varphi}{\pr z_j\pr z_l}(p,p,s_0)=-i(a_{l,j}+a_{j,l}),\ \ 1\leq j\leq n_-,\ \ 1\leq l\leq n-1,\\
&\frac{\pr^2\varphi}{\pr z_j\pr w_l}(p,p,s_0)=\frac{\pr^2\varphi}{\pr z_j\pr \ol w_l}(p,p,s_0)=0,\ \ 1\leq j\leq n_-,\ \  1\leq l\leq n-1, \\ 
&\frac{\pr^2\varphi}{\pr z_j\pr\ol z_l}(p,p,s_0)=-is_0\tau_{j,l}+\tau_{j,l}\beta-\frac{i}{2}\mu_{j,l},\ \ 1\leq j\leq n_-,\ \  1\leq l\leq n-1,\\ 
&\bigr(\frac{\pr^2\varphi}{\pr z_j\pr x_{2n-1}}
  +\frac{\pr^2\varphi}{\pr z_j\pr y_{2n-1}}\bigr)(p,p,s_0)=-ic_j\beta-s_0c_j-id_j,\ \ 1\leq j\leq n_-,
\end{split}\end{equation}
and 
\begin{equation} \label{e-geusw13623V}\begin{split}
&\frac{\pr^2\varphi}{\pr\ol z_j\pr\ol z_l}(p,p,s_0)=i(\ol a_{l,j}+\ol a_{j,l}),\ \ n_-+1\leq j\leq n-1,\ \ 1\leq l\leq n-1,\\
&\frac{\pr^2\varphi}{\pr\ol z_j\pr w_l}(p,p,s_0)=\frac{\pr^2\varphi}{\pr\ol z_j\pr \ol w_l}(p,p,s_0)=0,\ \ n_-+1\leq j\leq n-1,\ \  1\leq l\leq n-1, \\ 
&\frac{\pr^2\varphi}{\pr\ol z_j\pr z_l}(p,p,s_0)=is_0\ol\tau_{j,l}+\ol\tau_{j,l}\beta+\frac{i}{2}\ol\mu_{j,l},\ \ n_-+1\leq j\leq n-1,\ \  1\leq l\leq n-1,\\ 
&\bigr(\frac{\pr^2\varphi}{\pr\ol z_j\pr x_{2n-1}}
  +\frac{\pr^2\varphi}{\pr\ol z_j\pr y_{2n-1}}\bigr)(p,p,s_0)=i\ol c_j\beta-s_0\ol c_j+i\ol d_j,\ \ n_-+1\leq j\leq n-1.
\end{split}\end{equation}

Put $\varphi_1(x,y,s)=-\ol\varphi(y,x,s)$. In view of Proposition~\ref{p-geusw13622II}, we know the ${\rm Hess\,}(\varphi,{T_{(p,p,s_0)}H_\varphi})={\rm Hess\,}(\varphi_1,{T_{(p,p,s_0)}H_{\varphi_1}})$. From this observation, 
\eqref{e-geusw13623IV} and \eqref{e-geusw13623V}, we can check that 
\begin{equation}\label{e-geusw13624a}
\begin{split}
&\frac{\pr^2\varphi}{\pr\ol w_j\pr\ol w_l}(p,p,s_0)=-i(\ol a_{l,j}+\ol a_{j,l}),\ \ 1\leq j\leq n_-,\ \ 1\leq l\leq n-1,\\
&\frac{\pr^2\varphi}{\pr\ol w_j\pr\ol z_l}(p,p,s_0)=\frac{\pr^2\varphi}{\pr\ol w_j\pr z_l}(p,p,s_0)=0,\ \ 1\leq j\leq n_-,\ \  1\leq l\leq n-1, \\ 
&\frac{\pr^2\varphi}{\pr\ol w_j\pr w_l}(p,p,s_0)=-is_0\ol\tau_{j,l}-\ol\tau_{j,l}\beta-\frac{i}{2}\ol\mu_{j,l},\ \ 1\leq j\leq n_-,\ \  1\leq l\leq n-1,\\ 
&\bigr(\frac{\pr^2\varphi}{\pr\ol w_j\pr x_{2n-1}}
  +\frac{\pr^2\varphi}{\pr\ol w_j\pr y_{2n-1}}\bigr)(p,p,s_0)=-i\ol c_j\beta+s_0\ol c_j-i\ol d_j,\ \ 1\leq j\leq n_-,\\
 &\frac{\pr^2\varphi}{\pr w_j\pr w_l}(p,p,s_0)=i(a_{l,j}+a_{j,l}),\ \ n_-+1\leq j\leq n-1,\ \ 1\leq l\leq n-1,\\
&\frac{\pr^2\varphi}{\pr w_j\pr\ol z_l}(p,p,s_0)=\frac{\pr^2\varphi}{\pr w_j\pr z_l}(p,p,s_0)=0,\ \ n_-+1\leq j\leq n-1,\ \  1\leq l\leq n-1, \\ 
&\frac{\pr^2\varphi}{\pr w_j\pr\ol w_l}(p,p,s_0)=is_0\tau_{j,l}-\tau_{j,l}\beta+\frac{i}{2}\mu_{j,l},\ \ n_-+1\leq j\leq n-1,\ \  1\leq l\leq n-1,\\ 
&\bigr(\frac{\pr^2\varphi}{\pr w_j\pr x_{2n-1}}
  +\frac{\pr^2\varphi}{\pr w_j\pr y_{2n-1}}\bigr)(p,p,s_0)=ic_j\beta+s_0c_j+id_j,\ \ n_-+1\leq j\leq n-1.
\end{split}
\end{equation}

Fix $n_-+1\leq j, l\leq n-1$. We determine $\frac{\pr^2\varphi}{\pr z_j\pr z_l}(p,p,s_0)$. From the fact $\varphi(z,z,s)=0$, we can check that 
\begin{equation} \label{e-geusw13623VI}
\frac{\pr^2\varphi}{\pr z_j\pr z_l}(p,p,s_0)+\frac{\pr^2\varphi}{\pr z_j\pr w_l}(p,p,s_0)+\frac{\pr^2\varphi}{\pr w_j\pr z_l}(p,p,s_0)+\frac{\pr^2\varphi}{\pr w_j\pr w_l}(p,p,s_0)=0.
\end{equation}
From \eqref{e-geusw13623VI} and \eqref{e-geusw13624a}, we conclude that
\begin{equation}\label{e-geusw13623VIII}
\frac{\pr^2\varphi}{\pr z_j\pr z_l}(p,p,s_0)=-i(a_{l,j}+a_{j,l}),\ \ n_-+1\leq j, l\leq n-1.
\end{equation}
We can repeat the procedure above several times and deduce (we omit the computations)
\begin{equation}\label{e-geusw13623VIIIa}
\begin{split}
&\frac{\pr^2\varphi}{\pr\ol z_j\pr\ol z_l}(p,p,s_0)=i(\ol a_{l,j}+\ol a_{j,l}),\ \ 1\leq j, l\leq n_-,\\
&\frac{\pr^2\varphi}{\pr\ol z_j\pr z_l}(p,p,s_0)=is_0\ol\tau_{j,l}+\ol\tau_{j,l}\beta+\frac{i}{2}\ol\mu_{j,l},\ \ 1\leq j\leq n_-,\ \ n_-+1\leq l\leq n-1,\\
&\frac{\pr^2\varphi}{\pr\ol z_j\pr\ol w_l}(p,p,s_0)=0,\ \ 1\leq j\leq n_-,\ \ n_-+1\leq l\leq n-1,\\
&\frac{\pr^2\varphi}{\pr z_j\pr w_l}(p,p,s_0)=0,\ \ n_-+1\leq j\leq n-1,\ \ 1\leq l\leq n_-,\\
&\frac{\pr^2\varphi}{\pr\ol z_j\pr w_l}(p,p,s_0)=is_0(\tau_{l,j}+\ol\tau_{j,l})+(\ol\tau_{j,l}-\tau_{l,j})\beta+i\mu_{l,j},\ \ 1\leq j, l\leq n_-,\\
&\frac{\pr^2\varphi}{\pr z_j\pr\ol w_l}(p,p,s_0)=-is_0(\ol\tau_{l,j}+\tau_{j,l})+(\tau_{j,l}-\ol\tau_{l,j})\beta-i\mu_{j,l},\ \ n_-+1\leq j, l\leq n-1,\\
&\bigr(\frac{\pr^2\varphi}{\pr z_j\pr x_{2n-1}}+\frac{\pr^2\varphi}{\pr z_j\pr y_{2n-1}}\bigr)(p,p,s_0)=-ic_j\beta-s_0c_j-id_j,\ \  n_-+1\leq j\leq n-1,\\
&\bigr(\frac{\pr^2\varphi}{\pr\ol z_j\pr x_{2n-1}}+\frac{\pr^2\varphi}{\pr\ol z_j\pr y_{2n-1}}\bigr)(p,p,s_0)=i\ol c_j\beta-s_0\ol c_j+i\ol d_j,\ \  1\leq j\leq n_-.
\end{split}
\end{equation}

Again, from the fact that ${\rm Hess\,}(\varphi,{T_{(p,p,s_0)}H_\varphi})={\rm Hess\,}(\varphi_1,{T_{(p,p,s_0)}H_{\varphi_1}})$, \eqref{e-geusw13623VIII} and \eqref{e-geusw13623VIIIa}, we can check that

\begin{equation}\label{e-geusw13624}
\begin{split}
&\frac{\pr^2\varphi}{\pr\ol w_j\pr\ol w_l}(p,p,s_0)=-i(\ol a_{l,j}+\ol a_{j,l}),\ \ n_-+1\leq j, l\leq n-1,\\
&\frac{\pr^2\varphi}{\pr w_j\pr w_l}(p,p,s_0)=i(a_{l,j}+a_{j,l}),\ \ 1\leq j, l\leq n_-,\\
&\frac{\pr^2\varphi}{\pr w_j\pr\ol w_l}(p,p,s_0)=is_0\tau_{j,l}-\tau_{j,l}\beta+\frac{i}{2}\mu_{j,l},\ \ 1\leq j\leq n_-,\ \ n_-+1\leq l\leq n-1,\\
&\bigr(\frac{\pr^2\varphi}{\pr\ol w_j\pr x_{2n-1}}+\frac{\pr^2\varphi}{\pr\ol w_j\pr y_{2n-1}}\bigr)(p,p,s_0)=-i\ol c_j\beta+s_0\ol c_j-i\ol d_j,\ \  n_-+1\leq j\leq n-1,\\
&\bigr(\frac{\pr^2\varphi}{\pr w_j\pr x_{2n-1}}+\frac{\pr^2\varphi}{\pr w_j\pr y_{2n-1}}\bigr)(p,p,s_0)=ic_j\beta+s_0c_j+id_j,\ \  1\leq j\leq n_-.
\end{split}
\end{equation}

Moreover, from $\varphi(x,x,s)=0$, we conclude that 
\begin{equation}\label{e-geusw13624I}
\bigr(\frac{\pr^2\varphi}{\pr x^2_{2n-1}}+2\frac{\pr^2\varphi}{\pr x_{2n-1}\pr y_{2n-1}}+\frac{\pr^2\varphi}{\pr y^2_{2n-1}}\bigr)(p,p,s_0)=0. 
\end{equation} 

From \eqref{e-geusw13623IV}, \eqref{e-geusw13623V}, \eqref{e-geusw13624a}, \eqref{e-geusw13623VIII}, \eqref{e-geusw13623VIIIa}, \eqref{e-geusw13624}, \eqref{e-geusw13624I} and \eqref{e-geusw13623a}, we completely determine the tangential Hessian of $\varphi(x,y,s)$ at $(p,p,s_0)$. Summing up, we obtain Theorem~\ref{t-gue140121II}.

\section{Semi-classical Hodge decomposition theorems for $\Box^{(q)}_{s,k}$ in some non-degenerate part of $\Sigma$} \label{s-sch} 

In this section we apply the results about the Microlocal decomposition for $\Box^{(q)}_s$ previously in order to describe the semi-classical behaviour of $\Box^{(q)}_{s,k}$ in some non-degenerate part of $\Sigma$. We pause and introduce some notations and definitions. We first recall briefly the definition of semi-classical pseudodifferential operators. We need 

\begin{defn} \label{d-gue13628}
Let $W$ be an open set in $\Real^N$. Let $S(1;W)=S(1)$ be the set of
$a\in C^\infty(W)$ such that for every $\alpha\in\mathbb N^N_0$, there
exists $C_\alpha>0$, such that $\abs{\pr^\alpha_xa(x)}\leq
C_\alpha$ on $W$. If $a=a(x,k)$ depends on $k\in]1,\infty[$, we say that
$a(x,k)\in S_{{\rm loc\,}}(1;W)=S_{{\rm loc\,}}(1)$ if $\chi(x)a(x,k)$ uniformly bounded
in $S(1)$ when $k$ varies in $]1,\infty[$, for any $\chi\in
C^\infty_0(W)$. For $m\in\Real$, we put $S^m_{{\rm
loc}}(1;W)=S^m_{{\rm loc}}(1)=k^mS_{{\rm loc\,}}(1)$. If $a_j\in S^{m_j}_{{\rm
loc\,}}(1)$, $m_j\searrow-\infty$, we say that $a\sim
\sum^\infty_{j=0}a_j$ in $S^{m_0}_{{\rm loc\,}}(1)$ if
$a-\sum^{N_0}_{j=0}a_j\in S^{m_{N_0+1}}_{{\rm loc\,}}(1)$ for every
$N_0$.  For a given sequence $a_j$ as above, we can always find such an asymptotic sum $a$ and $a$ is unique up to an element in $S^{-\infty}_{{\rm loc\,}}(1)=S^{-\infty}_{{\rm loc\,}}(1;W):=\bigcap_mS^m_{{\rm loc\,}}(1)$. 
We say that $a(x,k)\in S^{m_0}_{{\rm loc\,}}(1)$ is a classical symbol on $W$ of order $m_0$ if 
\begin{equation} \label{e-gue13628I} 
\mbox{$a(x,k)\sim\sum^\infty_{j=0}k^{m_0-j}a_j(x)$ in $S^{m_0}_{{\rm loc\,}}(1)$},\ \ a_j(x)\in S_{{\rm loc\,}}(1),\ j=0,1\ldots.
\end{equation} 
The set of all classical symbols on $W$ of order $m_0$ is denoted by $S^{m_0}_{{\rm loc\,},{\rm cl\,}}(1)=S^{m_0}_{{\rm loc\,},{\rm cl\,}}(1;W)$. 

Let $E$ be a vector bundle over a smooth paracompact manifold $Y$. We extend the definitions above to the space of smooth sections of $E$ over $Y$ in the natural way and we write $S^m_{{\rm loc\,}}(1;Y,E)$ and $S^m_{{\rm loc\,},{\rm cl\,}}(1;Y,E)$ to denote the corresponding spaces. 
\end{defn} 

Let $W$ be an open set in $\Real^N$ and let $E$ and $F$ be complex vector bundles over $W$ with Hermitian metrics. 
For any $k$-dependent continuous function
\[F_k:H^s_{{\rm comp\,}}(W,E)\To H^{s'}_{{\rm loc\,}}(W,F),\ \ s, s'\in\Real,\]
we write
\[F_k=O(k^{n_0}):H^s_{{\rm comp\,}}(W,E)\To H^{s'}_{{\rm loc\,}}(W,F),\ \ n_0\in\mathbb Z,\]
if for any $\chi_0, \chi_1\in C^\infty_0(W)$, there is a positive constant $c>0$ independent of $k$, such that
\begin{equation} \label{e-gue13628II}
\norm{(\chi_0F_k\chi_1)u}_{s'}\leq ck^{n_0}\norm{u}_{s},\ \ \forall u\in H^s_{{\rm loc\,}}(W,E),
\end{equation}
where $\norm{u}_s$ is the usual Sobolev norm of order $s$. 

A $k$-dependent continuous operator
$A_k:C^\infty_0(W,E)\To\mathscr D'(W,F)$ is called $k$-negligible (on $W$)
if $A_k$ is smoothing and the kernel $A_k(x, y)$ of $A_k$ satisfies
$\abs{\pr^\alpha_x\pr^\beta_yA_k(x, y)}=O(k^{-N})$ locally uniformly
on every compact set in $W\times W$, for all multi-indices $\alpha$,
$\beta$ and all $N\in\mathbb N$. $A_k$ is $k$-negligible if and only if
\[A_k=O(k^{-N'}): H^s_{\rm comp\,}(W,E)\To H^{s+N}_{\rm loc\,}(W,F)\,,\]
for all $N, N'\geq0$ and $s\in\mathbb Z$. Let $C_k:C^\infty_0(W,E)\To\mathscr D'(W,F)$
be another $k$-dependent continuous operator. We write $A_k\equiv C_k\mod O(k^{-\infty})$ (on $W$) or 
$A_k(x,y)\equiv C_k(x,y)\mod O(k^{-\infty})$ (on $W$)
if $A_k-C_k$ is $k$-negligible on $W$. 

\begin{defn} \label{d-gue13628I}
Let $W$ be an open set in $\Real^N$ and let $E$ and $F$ be complex vector bundles over $W$. A classical semi-classical pseudodifferential operator on $W$ of order $m$ from sections of $E$ to sections of $F$ is a $k$-dependent continuous operator $A_k:C^\infty_0(W,E)\To C^\infty(W,F)$ such that the distribution kernel $A_k(x,y)$ is given by the oscillatory integral
\[\begin{split}
&A_k(x,y)\equiv\frac{k^N}{(2\pi)^N}\int e^{ik<x-y,\eta>}a(x,y,\eta,k)d\eta\mod O(k^{-\infty}),\\ &a(x,y,\eta,k)\in S^m_{{\rm loc\,},{\rm cl\,}}(1;W\times W\times\Real^N,E\boxtimes F).\end{split}\]

We shall identify $A_k$ with $A_k(x,y)$ and it is clearly that $A_k$ has a unique continuous extension $\mathscr E'(W,E)\To\mathscr D'(W,F)$.
\end{defn}

\begin{defn}\label{d-gue130816}
Let
\[\hat{\mathcal{I}}_k=\frac{k^{2n-1}}{(2\pi)^{2n-1}}\int e^{ik<x-y,\eta>}p(x,y,\eta,k)d\eta\] 
be a classical semi-classical pseudodifferential operator on $D$ of order $0$ from sections of $T^{*0,q}X$ to sections of $T^{*0,q}X$ with $p(x,y,\eta,k)\in S^0_{{\rm loc\,},{\rm cl\,}}(1;D\times D\times\Real^{2n-1},T^{*0,q}X\boxtimes T^{*0,q}X)$. 
Let $\Lambda$ be an open set of $T^*D$.
We write 
\[\mbox{$\hat{\mathcal{I}}_k\equiv\frac{k^{2n-1}}{(2\pi)^{2n-1}}\int e^{ik<x-y,\eta>}q(x,y,\eta,k)d\eta\mod O(k^{-\infty})$ at $\Lambda\bigcap\Sigma$},\]
where $q(x,y,\eta,k)\in S^0_{{\rm loc\,},{\rm cl\,}}(1;D\times D\times\Real^{2n-1},T^{*0,q}X\boxtimes T^{*0,q}X)$, if 
\[\hat{\mathcal{I}}_k\equiv\frac{k^{2n-1}}{(2\pi)^{2n-1}}\int e^{ik<x-y,\eta>}q(x,y,\eta,k)d\eta+\frac{k^{2n-1}}{(2\pi)^{2n-1}}\int e^{ik<x-y,\eta>}\beta(x,y,\eta,k)d\eta\mod O(k^{-\infty}),\]
where 
$\beta(x,y,\eta,k)\in S^0_{{\rm loc\,}}(1;D\times D\times\Real^{2n-1},T^{*0,q}X\boxtimes T^{*0,q}X)$ and there is a small neighbourhood $\Gamma$ of $\Lambda\bigcap\Sigma$ such that  $\beta(x,y,\eta,k)=0$ if $(x,\eta)\in\Gamma$.
\end{defn}

We return to our situation. Let $s$ be a local trivializing section of $L$ on an open subset $D\subset X$ and $\abs{s}^2_{h^L}=e^{-2\phi}$. From now on, we assume that there exist a $\lambda_0\in\Real$ and $x_0\in D$ such that $M^\phi_{x_0}-2\lambda_0\mathcal{L}_{x_0}$ is non-degenerate of constant signature $(n_-,n_+)$. We fix $D_0\Subset D$, $D_0$ open. We work with some real local coordinates $x=(x_1,\ldots,x_{2n-1})$ defined on $D$. We write $\xi=(\xi_1,\ldots,\xi_{2n-1})$ or $\eta=(\eta_1,\ldots,\eta_{2n-1})$ to denote the dual coordinates of $x$. We will use the same notations as in section~\ref{s-thef}. Note that we write $\hat x=(x_1,\ldots,x_{2n-1},x_{2n})$ to denote the local coordinates of $\hat D$ and we write $\hat\xi=(\xi_1,\ldots,\xi_{2n-1},\xi_{2n})$ or $\hat\eta=(\eta_1,\ldots,\eta_{2n-1},\eta_{2n})$ to denote the dual coordinates of $\hat x$.

Let $\chi(x_{2n}), \chi_1(x_{2n})\in C^\infty_0(\Real)$, $\chi, \chi_1\geq 0$. We assume that $\chi_1=1$ on ${\rm Supp\,}\chi$. We take $\chi$ so that  $\int\chi(x_{2n})dx_{2n}=1$. Put
\begin{equation} \label{e-gue13628IIa}
\chi_k(x_{2n})=e^{ikx_{2n}}\chi(x_{2n}).
\end{equation}
Let $V$ and $U$ be as in \eqref{e-dhmpXII} and \eqref{e-dhmpXIII} respectively. The following is straightforward and follows from the usual stationary phase formula and therefore we omit the proof. 

\begin{prop}\label{p-gue13628}
With the notations before, let $q\in\set{0,1,\ldots,n-1}$. Let 
\[\Td{\mathcal{I}}_k=\frac{k^{2n-1}}{(2\pi)^{2n-1}}\int e^{ik<x-y,\eta>}\alpha(x,\eta,k)d\eta\] 
be a classical semi-classical pseudodifferential operator on $D$ of order $0$ from sections of $T^{*0,q}X$ to sections of $T^{*0,q}X$ with $\alpha(x,\eta,k)\in S^0_{{\rm loc\,},{\rm cl\,}}(1;T^*D,T^{*0,q}X\boxtimes T^{*0,q}X)$, $\alpha(x,\eta,k)=0$ if $\abs{\eta}>M$, for some large $M>0$, and ${\rm Supp\,}\alpha(x,\eta,k)\bigcap T^*D_0\Subset V$.
Then, there is a classical pseudodifferential operator $\Td I=(2\pi)^{-2n}\int e^{i<\hat x-\hat y,\hat\eta>}c(\hat x,\hat\eta)d\hat\eta$ on $\hat D$ of order $0$ from sections of $T^{*0,q}\hat D$ to sections of $T^{*0,q}\hat D$ with $c(\hat x,\hat\eta)\in S^0_{{\rm cl\,}}(T^*\hat D,T^{*0,q}\hat D\boxtimes T^{*0,q}\hat D)$, ${\rm Supp\,}c(\hat x,\hat\eta)\bigcap T^*\hat D_0\subset\ol W$,
where $W\subset U$ is a conic open set with $\ol W\subset U$, such that 
\[\Td{\mathcal{I}}_k\equiv\Td I_k\mod O(k^{-\infty})\ \ \mbox{on $D$},\]
where $\Td I_k$ is the continuous operator $C^\infty_0(D,T^{*0,q}X)\To C^\infty(D,T^{*0,q}X)$ given by 
\[\begin{split}
\Td I_{k}:C^\infty_0(D,T^{*0,q}X)&\To C^\infty(D,T^{*0,q}X),\\
u&\To \int e^{-ikx_{2n}}\chi_1(x_{2n})\Td I(\chi_ku)(\hat x)dx_{2n}.
\end{split}\]
\end{prop}

Now, we assume that $q=n_-$. Let $\Td{\mathcal{I}}_k=\frac{k^{2n-1}}{(2\pi)^{2n-1}}\int e^{ik<x-y,\eta>}\alpha(x,\eta,k)d\eta$
be a classical semi-classical pseudodifferential operator on $D$ of order $0$ from sections of $T^{*0,q}X$ to sections of $T^{*0,q}X$ with $\alpha(x,\eta,k)\in S^0_{{\rm loc\,},{\rm cl\,}}(1;T^*D,T^{*0,q}X\boxtimes T^{*0,q}X)$, $\alpha(x,\eta,k)=0$ if $\abs{\eta}>M$, for some large $M>0$, and ${\rm Supp\,}\alpha(x,\eta,k)\bigcap T^*D_0\Subset V$. Let $\Td I$ be as in Proposition~\ref{p-gue13628} and let $S\in L^{0}_{\frac{1}{2},\frac{1}{2}}(\hat D,T^{*0,q}\hat D\boxtimes T^{*0,q}\hat D)$ and $G\in L^{-1}_{\frac{1}{2},\frac{1}{2}}(\hat D,T^{*0,q}\hat D\boxtimes T^{*0,q}\hat D)$ be as in Theorem~\ref{t-dcmimpII}. Then, we have 
\begin{equation}\label{e-gue13629}
S+\Box^{(q)}_s\circ G\equiv\Td I\ \ \mbox{on $\hat D_0$},\ \ \Box^{(q)}_s\circ S\equiv 0\ \ \mbox{on $\hat D$}. 
\end{equation}
Now, we assume that $S$ and $G$ are properly supported. 
Define
\begin{equation} \label{e-geu13629I}
\begin{split}
\mathcal{S}_{k}: H^s_{{\rm loc\,}}(D,T^{*0,q}X)&\To H^s_{{\rm loc\,}}(D,T^{*0,q}X),\ \ \forall s\in\mathbb N_0,\\
u&\To \int e^{-ikx_{2n}}\chi_1(x_{2n})S(\chi_ku)(\hat x)dx_{2n}.
\end{split}
\end{equation}
Let $u\in H^s_{{\rm loc\,}}(D,T^{*0,q}X)$, $s\in\mathbb N_0$. We have $\chi_ku\in H^s_{{\rm loc\,}}(\hat D,T^{*0,q}\hat D)$.
Since $S\in L^{0}_{\frac{1}{2},\frac{1}{2}}(\hat D,T^{*0,q}\hat D\boxtimes T^{*0,q}\hat D)$, we see that
$S(\chi_ku)\in H^{s}_{{\rm loc\,}}(\hat D,T^{*0,q}\hat D)$. 
From this, it is not difficult to see that 
\[\int e^{-ikx_{2n}
}\chi_1(x_{2n})S(\chi_ku)(\hat x)dx_{2n}\in H^s_{{\rm loc\,}}(D,T^{*0,q}X).\] 
Thus, $\mathcal{S}_{k}$ is well-defined. Since $S$ is properly supported, $\mathcal{S}_{k}$ is properly
supported, too. Moreover, from \eqref{e-geu13629I} and the fact that $S:H^s_{{\rm comp\,}}(\hat D,T^{*0,q}\hat D)\To H^{s}_{{\rm comp\,}}(\hat D,T^{*0,q}\hat D)$ is continuous, for every $s\in\Real$, it is straightforward to check that
\begin{equation} \label{e-geu13629II}
\mathcal{S}_{k}=O(k^s): H^s_{{\rm comp\,}}(D,T^{*0,q}X)\To H^{s}_{{\rm comp\,}}(D,T^{*0,q}X),
\end{equation}
for all $s\in\mathbb N_0$. 

Let $\mathcal{S}_{k}^*:\mathscr D'(D,T^{*0,q}X)\To\mathscr D'(D,T^{*0,q}X)$ be the formal adjoint of $\mathcal{S}_{k}$ with respect to $(\,\cdot\,|\,\cdot\,)$. Then $\mathcal{S}_{k}^*$ is also properly supported. 
It is not difficult to see that 
\[(\mathcal{S}_{k}^*v)(x)=\int\ol{\chi_k(x_{2n})}S^*(ve^{ix_{2n}k}\chi_1)(\hat x)dx_{2n}\in\Omega^{0,q}_0(D),\]
for all $v\in\Omega^{0,q}_0(D)$. From this observation, we can check that
\begin{equation} \label{e-gue13630}
\mathcal{S}_{k}^*=O(k^s):H^s_{{\rm comp\,}}(D,T^{*0,q}X)\To H^{s}_{{\rm comp\,}}(D,T^{*0,q}X)\,,\quad\forall s\in\mathbb N_0.
\end{equation} 

From \eqref{s3-e9bis}, we have
\begin{equation} \label{e-gue13630I}
\begin{split}
\Box^{(q)}_{s,k}\circ\bigr(\int e^{-ikx_{2n}}\chi_1(x_{2n})S(\chi_ku)(\hat x)dx_{2n}\bigr)&=\int e^{-ikx_{2n} }\bigr(\Box^{(q)}_s(\chi_1S)\bigr)(\chi_ku)(\hat x)dx_{2n}\\
&=\int e^{-ikx_{2n}}\bigr(\Box^{(q)}_s(\chi_1S\Td\chi)\bigr)(\chi_ku)(\hat x)dx_{2n},
\end{split}
\end{equation}
where $\Td\chi\in C^\infty_0(\Real)$, $\Td\chi=1$ on ${\rm Supp\,}\chi$ and $\chi_1=1$ on ${\rm Supp\,}\Td\chi$ and  $u\in\Omega^{0,q}_0(D_0)$. Note that $\Box^{(q)}_s(\chi_1S\Td\chi)=\Box^{(q)}_s(S\Td\chi)-\Box^{(q)}_s((1-\chi_1)S\Td\chi)$. From Theorem~\ref{t-dcmimpII}, we know that $\Box^{(q)}_sS$ is smoothing and the kernel of $S$ is smoothing away the diagonal. Thus,
$(1-\chi_1)S\Td\chi$ is smoothing. It follows that $\Box^{(q)}_s((1-\chi_1)S\Td\chi)$ is smoothing. We conclude that
$\Box^{(q)}_s(\chi_1S\Td\chi)$ is smoothing. Let $K(\hat x,\hat y)\in C^\infty$ be the distribution kernel of $\Box^{(q)}_s(\chi_1S\Td\chi)$. From \eqref{e-gue13630I} and recall the form $\chi_k$ (see \eqref{e-gue13628IIa}), we see that the distribution kernel of $\Box^{(q)}_{s,k}\mathcal{S}_{k}$ is given by
\begin{equation} \label{e-gue13630II}
(\Box^{(q)}_{s,k}\mathcal{S}_{k})(x,y)=\int e^{-i(x_{2n}-y_{2n})k}K(\hat x,\hat y)\chi(y_{2n})dx_{2n}dy_{2n}.
\end{equation}
For $N\in\mathbb N$, we have
\begin{equation} \label{e-gue13630III}
\begin{split}
\abs{k^N(\Box^{(q)}_{s,k}\mathcal{S}_{k})(x,y)}&=\abs{\int \bigr((-i\frac{\pr}{\pr y_{2n}})^Ne^{-i(x_{2n}-y_{2n})k})\bigr)K(\hat x,\hat y)\chi(y_{2n})dy_{2n}dx_{2n}}\\
&=\abs{\int e^{-i(x_{2n}-y_{2n})k}(i\frac{\pr}{\pr y_{2n}})^N\bigr(K(\hat x,\hat y)\chi(y_{2n})\bigr)dy_{2n}dx_{2n}}.
\end{split}
\end{equation} 
Thus, $(\Box^{(q)}_{s,k}\mathcal{S}_{k})(x,y)=O(k^{-N})$, locally uniformly for all $N\in\mathbb N$, and similarly for the derivatives. We deduce that
\begin{equation} \label{e-gue13630IV}
\Box^{(q)}_{s,k}\mathcal{S}_{k}\equiv 0\mod O(k^{-\infty})\ \ \mbox{on $D$}.
\end{equation}
Thus,
\begin{equation} \label{e-gue13630V}
\mathcal{S}_{k}^*\Box^{(q)}_{s,k}\equiv 0\mod O(k^{-\infty})\ \ \mbox{on $D$}.
\end{equation} 

Define 
\begin{equation} \label{e-gue13630VI}
\begin{split}
\mathcal{G}_{k}: H^s_{{\rm loc\,}}(D,T^{*0,q}X)&\To H^{s+1}_{{\rm loc\,}}(D,T^{*0,q}X),\ \ \forall s\in\mathbb N_0, \\
u&\To\int e^{-ikx_{2n}}\chi_1G(\chi_ku)(\hat x)dx_{2n}.
\end{split}
\end{equation}
As above, we can show that $\mathcal{G}_{k}$ is well-defined.
Since $G$ is properly supported, $\mathcal{G}_{k}$ is properly
supported, too. Moreover, from \eqref{e-gue13630VI} and the fact that $G:H^s_{{\rm comp\,}}(\hat D,T^{*0,q}\hat D)\To H^{s+1}_{{\rm comp\,}}(\hat D,T^{*0,q}\hat D)$ is continuous, for every $s\in\Real$, it is straightforward to check that
\begin{equation} \label{e-gue13630VIbis}
\mathcal{G}_{k}=O(k^s): H^s_{{\rm comp\,}}(D,T^{*0,q}X)\To H^{s+1}_{{\rm comp\,}}(D,T^{*0,q}X),
\end{equation}
for all $s\in\mathbb N_0$.

Let $\mathcal{G}_{k}^*:\mathscr D'(D,T^{*0,q}X)\To\mathscr D'(D,T^{*0,q}X)$ be the formal adjoint of
$\mathcal{G}_{k}$ with respect to $(\,\cdot\,|\,\cdot\,)$. We can check that
\[(\mathcal{G}_{k}^*v)(x)=\int\ol{\chi_k(x_{2n})}G^*(ve^{ix_{2n}k}\chi_1)(\hat x)dx_{2n}\in\Omega^{0,q}_0(D),\]
for all $v\in\Omega^{0,q}_0(D)$. Thus, $\mathcal{G}_{k}^*:\Omega^{0,q}_0(D)\To\Omega^{0,q}_0(D)$.
Moreover, as before, we can show that
\begin{equation} \label{e-gue13630VII}
\mathcal{G}_{k}^*=O(k^s):H^s_{{\rm comp\,}}(D,T^{*0,q}X)\To H^{s+1}_{{\rm comp\,}}(D,T^{*0,q}X)\,,\quad\forall s\in\mathbb N_0.
\end{equation}

Let $u\in\Omega^{0,q}_0(D_0)$. From \eqref{s3-e9bis}, we have
\begin{equation*}
\begin{split}
\Box^{(q)}_{s,k}(\mathcal{G}_{k}u)&=\Box^{(q)}_{s,k}\circ\bigr(\int e^{-ikx_{2n}}\chi_1G(\chi_ku)dx_{2n}\bigr)=\int e^{-ikx_{2n}}\bigr(\Box^{(q)}_s\chi_1G\Td\chi)(\chi_ku)(\hat x)dx_{2n},
\end{split}
\end{equation*}
where $\Td\chi$ is as in \eqref{e-gue13630I}. Note that $\Box^{(q)}_s(\chi_1G\Td\chi)=\Box^{(q)}_s(G\Td\chi)-\Box^{(q)}_s((1-\chi_1)G\Td\chi)$. From \eqref{e-gue13629} and Theorem~\ref{t-dcmimpII}, we know that $\Box^{(q)}_sG+S\equiv\Td I$ and the kernel of $G$ is smoothing away the diagonal. Thus,
$(1-\chi_1)G\Td\chi$ is smoothing. It follows that $\Box^{(q)}_s((1-\chi_1)G\Td\chi)$ is smoothing. We conclude that
$\Box^{(q)}_s(\chi_1G\Td\chi)\equiv(\Td I-S)\Td\chi$. From this, we get
\begin{equation} \label{e-gue13630VIII}
\begin{split}
\Box^{(q)}_{s,k}(\mathcal{G}_{k}u)&=\int e^{-ikx_{2n}}\bigr((\Td I-S)(\chi_ku)\bigr)(\hat x)dx_{2n}+\int e^{-ikx_{2n}}F(\chi_ku)(\hat x)dx_{2n}\\
&=\int e^{-ikx_{2n}}\chi_1\bigr((\Td I-S)(\chi_ku)\bigr)(\hat x)dx_{2n}+\int e^{-ikx_{2n}}(1-\chi_1)\bigr((\Td I-S)(\chi_ku)\bigr)(\hat x)dx_{2n}\\
&\quad+\int e^{-ikx_{2n}}F(\chi_ku)(\hat x)dx_{2n}\\
&=(\Td I_k-\mathcal{S}_{k})u+\int e^{-ikx_{2n}}(1-\chi_1)\bigr((\Td I-S)(\chi_ku)\bigr)(\hat x)dx_{2n}\\
&\quad+\int e^{-ikx_{2n}}F(\chi_ku)(\hat x)dx_{2n},
\end{split}
\end{equation}
where $F$ is a smoothing operator. We can repeat the procedure as in \eqref{e-gue13630I} and conclude that the operator
\[u\To \int e^{-ikx_{2n}}F(\chi_k u)(\hat x)dx_{2n},\ \ u\in\Omega^{0,q}_0(D_0),\]
is $k$-negligible. Similarly, since $(1-\chi_1)(\Td I-S)\chi$ is smoothing, the operator
\[u\To \int e^{-ikx_{2n}}(1-\chi_1)\bigr((\Td I-S)(\chi_k u)\bigr)(\hat x)dx_{2n},\ \ u\in\Omega^{0,q}_0(D_0),\]
is also $k$-negligible. From this observation and note that $\Td{\mathcal{I}}_k\equiv\Td I_k\mod O(k^{-\infty})$, we obtain
\begin{equation} \label{e-gue13630a}
\Box^{(q)}_{s,k}\mathcal{G}_{k}+\mathcal{S}_{k}\equiv\Td{\mathcal{I}}_k\mod O(k^{-\infty})\ \ \mbox{on $D_0$}
\end{equation}
and hence
\begin{equation} \label{e-gue1374III}
\mathcal{G}^*_{k}\Box^{(q)}_{s,k}+\mathcal{S}^*_{k}\equiv\Td{\mathcal{I}}^*_k\mod O(k^{-\infty})\ \ \mbox{on $D_0$},
\end{equation}
where $\Td{\mathcal{I}}^*_k$ is the formal adjoint of $\Td{\mathcal{I}}_k$ with respect to $(\,\cdot\,|\,\cdot\,)$. 

We now consider more general situations. We recall Definition~\ref{d-gue130816}. 
Let \[\hat{\mathcal{I}}_k=\frac{k^{2n-1}}{(2\pi)^{2n-1}}\int e^{ik<x-y,\eta>}p(x,y,\eta,k)d\eta\] 
be a classical semi-classical pseudodifferential operator on $D$ of order $0$ from sections of $T^{*0,q}X$ to sections of $T^{*0,q}X$ with $p(x,y,\eta,k)\in S^0_{{\rm loc\,},{\rm cl\,}}(1;D\times D\times\Real^{2n-1},T^{*0,q}X\boxtimes T^{*0,q}X)$. 
We assume that $\hat{\mathcal{I}}_k\equiv\Td{\mathcal{I}}_k\mod O(k^{-\infty})$ at $T^*D_0\bigcap\Sigma$, where 
\[\Td{\mathcal{I}}_k=\frac{k^{2n-1}}{(2\pi)^{2n-1}}\int e^{ik<x-y,\eta>}\alpha(x,\eta,k)d\eta\]
with $\alpha(x,\eta,k)\in S^0_{{\rm loc\,},{\rm cl\,}}(1;T^*D,T^{*0,q}X\boxtimes T^{*0,q}X)$, $\alpha(x,\eta,k)=0$ if $\abs{\eta}>M$, for some large $M>0$ and ${\rm Supp\,}\alpha(x,\eta,k)\bigcap T^*D_0\Subset V$. We write
\[\begin{split}
&\hat{\mathcal{I}}_k\equiv\Td{\mathcal{I}}_k+\Td{\mathcal{I}}_k^1\mod O(k^{-\infty}),\\
&\Td{\mathcal{I}}_k^1=\frac{k^{2n-1}}{(2\pi)^{2n-1}}\int e^{ik<x-y,\eta>}\beta(x,y,\eta,k)d\eta,
\end{split}\]
where $\beta(x,y,\eta,k)\in S^0_{{\rm loc\,},{\rm cl\,}}(1;D\times D\times\Real^{2n-1},T^{*0,q}X\boxtimes T^{*0,q}X)$ and there is a small neighbourhood $\Gamma$ of $T^*D_0\bigcap\Sigma$ such that $\beta(x,y,\eta,k)=0$ if $(x,\eta)\in\Gamma$. Let $\mathcal{G}_k$ and $\mathcal{S}_k$ be as in \eqref{e-gue13630VI} and \eqref{e-geu13629II} respectively. Then, \eqref{e-gue13630a} and 
\eqref{e-gue1374III} hold. Since $\beta(x,y,\eta,k)=0$ if $(x,\eta)$ is in some small neighbourhood of $T^*D_0\bigcap\Sigma$, it is clear that there is a properly supported continuous 
operator $\mathcal{G}^1_k=O(k^s):H^s_{{\rm comp\,}}(D,T^{*0,q}X)\To H^{s+1}_{{\rm comp\,}}(D,T^{*0,q}X)$, $\forall s\in\mathbb N_0$, such that $\Box^{(q)}_{s,k}\mathcal{G}^1_k\equiv\Td{\mathcal{I}}^1_k\mod O(k^{-\infty})$ on $D_0$. Put 
\begin{equation}\label{e-gue130814}
\mathcal{N}_k:=\mathcal{G}_k+\mathcal{G}^1_k=O(k^s):H^s_{{\rm comp\,}}(D,T^{*0,q}X)\To H^{s+1}_{{\rm comp\,}}(D,T^{*0,q}X),\ \ \forall s\in\mathbb N_0.
\end{equation}
Then, we have 
\begin{equation}\label{e-gue130814I}
\begin{split}
&\mathcal{N}_k=O(k^s):H^s_{{\rm comp\,}}(D,T^{*0,q}X)\To H^{s+1}_{{\rm comp\,}}(D,T^{*0,q}X),\ \ \forall s\in\mathbb N_0,\\
&\mathcal{N}^*_k=O(k^s):H^s_{{\rm comp\,}}(D,T^{*0,q}X)\To H^{s+1}_{{\rm comp\,}}(D,T^{*0,q}X),\ \ \forall s\in\mathbb N_0,
\end{split}
\end{equation}
and
\begin{equation} \label{e-gue130814II}
\begin{split}
&\Box^{(q)}_{s,k}\mathcal{N}_{k}+\mathcal{S}_{k}\equiv\hat{\mathcal{I}}_k\mod O(k^{-\infty})\ \ \mbox{on $D_0$},\\
&\mathcal{N}^*_{k}\Box^{(q)}_{s,k}+\mathcal{S}^*_{k}\equiv\hat{\mathcal{I}}^*_k\mod O(k^{-\infty})\ \ \mbox{on $D_0$},
\end{split}
\end{equation}
where $\mathcal{N}^*_k$ and $\hat{\mathcal{I}}^*_k$ are the formal adjoints of $\hat{\mathcal{I}}_k$ and $\mathcal{N}_k$ with respect to $(\,\cdot\,|\,\cdot\,)$ respectively.

From \eqref{e-geu13629II}, \eqref{e-gue13630}, \eqref{e-gue13630IV}, \eqref{e-gue13630V}, \eqref{e-gue130814I} and \eqref{e-gue130814II}, we get 

\begin{thm}\label{t-gue13630}
Let $s$ be a local trivializing section of $L$ on an open subset $D\subset X$ and $\abs{s}^2_{h^L}=e^{-2\phi}$. We assume that there exist a $\lambda_0\in\Real$ and $x_0\in D$ such that $M^\phi_{x_0}-2\lambda_0\mathcal{L}_{x_0}$ is non-degenerate of constant signature $(n_-,n_+)$. Let $q=n_-$. We fix $D_0\Subset D$, $D_0$ open. Let $V$ be as in \eqref{e-dhmpXII}. Let 
\[\hat{\mathcal{I}}_k=\frac{k^{2n-1}}{(2\pi)^{2n-1}}\int e^{ik<x-y,\eta>}p(x,y,\eta,k)d\eta\] 
be a classical semi-classical pseudodifferential operator on $D$ of order $0$ from sections of $T^{*0,q}X$ to sections of $T^{*0,q}X$ with $p(x,y,\eta,k)\in S^0_{{\rm loc\,},{\rm cl\,}}(1;D\times D\times\Real^{2n-1},T^{*0,q}X\boxtimes T^{*0,q}X)$. We assume that 
\begin{equation}\label{e-gue130814III}\begin{split}
&\mbox{$\hat{\mathcal{I}}_k\equiv\Td{\mathcal{I}}_k\mod O(k^{-\infty})$ at $T^*D_0\bigcap\Sigma$},\\
&\Td{\mathcal{I}}_k=\frac{k^{2n-1}}{(2\pi)^{2n-1}}\int e^{ik<x-y,\eta>}\alpha(x,\eta,k)d\eta,\\
&\mbox{$\alpha(x,\eta,k)\sim\sum_{j=0}\alpha_j(x,\eta)k^{-j}$ in $S^0_{{\rm loc\,}}(1;T^*D,T^{*0,q}X\boxtimes T^{*0,q}X)$},\\ 
&\alpha_j(x,\eta)\in C^\infty(T^*D,T^{*0,q}D\boxtimes T^{*0,q}D),\ \ j=0,1,\ldots,\\
\end{split}\end{equation}
where $\alpha(x,\eta,k)\in S^0_{{\rm loc\,},{\rm cl\,}}(1;T^*D,T^{*0,q}X\boxtimes T^{*0,q}X)$ with $\alpha(x,\eta,k)=0$ if $\abs{\eta}>M$, for some large $M>0$ and ${\rm Supp\,}\alpha(x,\eta,k)\bigcap T^*D_0\Subset V$. 
Let $\mathcal{S}_k$, $\mathcal{G}_k$ and $\mathcal{N}_k$ be as in \eqref{e-geu13629I}, \eqref{e-gue13630VI} and \eqref{e-gue130814} respectively. Then, 
\begin{equation} \label{e-gue13630Ia}
\begin{split}
&\mathcal{S}^*_k, \mathcal{S}_{k}=O(k^s): H^s_{{\rm comp\,}}(D,T^{*0,q}X)\To H^{s}_{{\rm comp\,}}(D,T^{*0,q}X),\ \ \forall s\in\mathbb N_0,\\
&\mathcal{G}^*_k, \mathcal{G}_{k}, \mathcal{N}^*_k, \mathcal{N}_{k}=O(k^s): H^s_{{\rm comp\,}}(D,T^{*0,q}X)\To H^{s+1}_{{\rm comp\,}}(D,T^{*0,q}X),\ \ \forall s\in\mathbb N_0,
\end{split}
\end{equation}
and we have
\begin{gather}
\Box^{(q)}_{s,k}\mathcal{S}_{k}\equiv 0\mod
O(k^{-\infty})\ \ \mbox{on $D$},\ \ \mathcal{S}_{k}^*\Box^{(q)}_{s,k}\equiv0\mod O(k^{-\infty})\ \ \mbox{on $D$}, \label{e-gue13630IIa} \\
\mathcal{S}_{k}+\Box^{(q)}_{s,k}\mathcal{G}_{k}\equiv\Td{\mathcal{I}}_k\mod O(k^{-\infty})\ \ \mbox{on $D_0$}, \label{e-gue13630IIIa}\\
\mathcal{G}^*_{k}\Box^{(q)}_{s,k}+\mathcal{S}^*_{k}\equiv\Td{\mathcal{I}}^*_k\mod O(k^{-\infty})\ \ \mbox{on $D_0$},\label{e-gue1374IV}\\
\mathcal{S}_{k}+\Box^{(q)}_{s,k}\mathcal{N}_{k}\equiv\hat{\mathcal{I}}_k\mod O(k^{-\infty})\ \ \mbox{on $D_0$}, \label{e-gue13630IIIabal}\\
\mathcal{N}^*_{k}\Box^{(q)}_{s,k}+\mathcal{S}^*_{k}\equiv\hat{\mathcal{I}}^*_k\mod O(k^{-\infty})\ \ \mbox{on $D_0$},\label{e-gue1374IVabal}
\end{gather}
where $\mathcal{S}_{k}^*$, $\mathcal{G}_{k}^*$, $\mathcal{N}_{k}^*$, $\Td{\mathcal{I}}^*_k$ and $\hat{\mathcal{I}}^*_k$ are the formal adjoints of
$\mathcal{S}_{k}$, $\mathcal{G}_{k}$, $\mathcal{N}_{k}$, $\Td{\mathcal{I}}_k$ and $\hat{\mathcal{I}}_k$ with respect to $(\,\cdot\,|\,\cdot\,)$
respectively and $\Box^{(q)}_{s,k}$ is given by \eqref{e-msmilkVI}.
\end{thm} 

We notice that $\mathcal{S}_{k}$, $\mathcal{S}_{k}^*$, $\mathcal{G}_{k}$, $\mathcal{G}_{k}^*$, $\mathcal{N}_{k}$, $\mathcal{N}_{k}^*$, are all properly supported on $D$. We need

\begin{thm}\label{t-gue13630I}
With the notations and assumptions above, let $\mathcal{S}_k$ be as in Theorem~\ref{t-gue13630}. Then, $\mathcal{S}_{k}$ 
is a smoothing operator and the kernel of $\mathcal{S}_{k}$ satisfies
\begin{equation} \label{e-gue13630IVa}
\mathcal{S}_{k}(x,y)\equiv \int e^{ik\varphi(x,y,s)}a(x,y,s,k)ds\mod
O(k^{-\infty})\ \ \mbox{on $D$}
\end{equation}
with
\begin{equation}  \label{e-gue13630Va}
\begin{split}
&a(x,y,s,k)\in S^{n}_{{\rm loc\,}}\big(1;\Omega,T^{*0,q}X\boxtimes T^{*0,q}X\big)\bigcap C^\infty_0\big(\Omega,T^{*0,q}X\boxtimes T^{*0,q}X\big),\\
&a(x,y,s,k)\sim\sum^\infty_{j=0}a_j(x,y,s)k^{n-j}\text{ in }S^{n}_{{\rm loc\,}}
\big(1;\Omega,T^{*0,q}X\boxtimes T^{*0,q}X\big), \\
&a_j(x,y,s)\in C^\infty_0\big(\Omega,T^{*0,q}X\boxtimes T^{*0,q}X\big),\ \ j=0,1,2,\ldots,\\
&a_0(x,x,s):T^{*0,q}_xX\To\mathcal{N}(x,s,n_-),\ \ \forall (x,x,s)\in\Omega,
\end{split}
\end{equation}
and $\varphi(x,y,s)$ is as in Theorem~\ref{t-dcgewI} and \eqref{e-guew13627}, where $\mathcal{N}(x,s,n_-)$ is as in \eqref{e-gue1373III}, 
\[
\begin{split}
\Omega:=&\{(x,y,s)\in D\times D\times\Real;\, (x,-2{\rm Im\,}\ddbar_b\phi(x)+s\omega_0(x))\in V\bigcap\Sigma,\\
&\quad\mbox{$(y,-2{\rm Im\,}\ddbar_b\phi(y)+s\omega_0(y))\in V\bigcap\Sigma$, $\abs{x-y}<\varepsilon$, for some $\varepsilon>0$}\}.
\end{split}\]
\end{thm}

\begin{proof}
Theorem~\ref{t-gue13630I} essentially follows from the stationary phase
formula of Melin-Sj\"{o}strand~\cite{MS74}. From the definition \eqref{e-geu13629I} of
$\mathcal{S}_{k}$ and Theorem~\ref{t-dcgewI}, we see
that the distribution kernel of $\mathcal{S}_{k}$ is given by
\begin{equation} \label{e-gue13630VIa}
\begin{split}
&\mathcal{S}_{k}(x,y)\equiv\int_{t\geq0}e^{it\Phi(\hat x,\hat y,s)-ikx_{2n}+iky_{2n}}b(\hat x,\hat y,s,t)\chi_1(x_{2n})\chi(y_{2n})dx_{2n}dtdy_{2n}ds\mod O(k^{-\infty})  \\
&\equiv\int_{t\geq0}e^{it\varphi(x,y,s)+i(x_{2n}-y_{2n})(t-k)}b(\hat x,\hat y,s,t)\chi_1(x_{2n})\chi(y_{2n})dx_{2n}dtdy_{2n}ds \mod O(k^{-\infty})\\
&\equiv\int_{\sigma\geq0}e^{ik\bigr(\varphi(x,y,s)\sigma+(x_{2n}-y_{2n})(\sigma-1)\bigr)}kb(\hat x,\hat y,s,k\sigma)\chi_1(x_{2n})\chi(y_{2n})dx_{2n}d\sigma dy_{2n}ds\mod O(k^{-\infty}),
\end{split}
\end{equation}
where the integrals above are defined as oscillatory integrals and $t=k\sigma$. Let $\gamma(\sigma)\in C^\infty_0(\Real_+)$ with
$\gamma(\sigma)=1$ in some small neighbourhood of $1$. We introduce the cut-off functions $\gamma(\sigma)$ and $1-\gamma(\sigma)$ in the integral \eqref{e-gue13630VIa}:
\begin{align}
I_0(x,y):=\int_{\sigma\geq0}e^{ik\bigr(\varphi(x,y,s)\sigma+(x_{2n}-y_{2n})(\sigma-1)\bigr)}\gamma(\sigma)kb(\hat x,\hat y,s,k\sigma)\chi_1(x_{2n} )\chi(y_{2n})dx_{2n}d\sigma dy_{2n}ds  , \label{e-gue13630VIIa}\\
I_1(x,y):=\int_{\sigma\geq0}e^{ik\bigr(\varphi(x,y,s)\sigma+(x_{2n}-y_{2n})(\sigma-1)\bigr)}(1-\gamma(\sigma))kb(\hat x,\hat y,s,k\sigma)\chi_1(x_{2n} )\chi(y_{2n})dx_{2n}d\sigma dy_{2n}ds\,, \label{e-gue13630VIIIa}
\end{align}
so that $\mathcal{S}_k(x,y)\equiv I_0(x,y)+I_1(x,y)\mod O(k^{-\infty})$. First, we study $I_1(x,y)$. Note that when $\sigma\neq1$,
$d_{y_{2n}}\bigr(\varphi(x,y,s)\sigma+(x_{2n}-y_{2n})(\sigma-1)\bigr)=1-\sigma\neq0$. Thus, we can integrate by
parts and get that $I_1$ is smoothing and 
\begin{equation}\label{e-gue1373f}
I_1(x,y)\equiv0\mod O(k^{-\infty}).
\end{equation}

Next, we study the kernel $I_0(x,y)$. From \eqref{e-gue1373VIIa}, we may assume that $b(\hat x,\hat y,s,k\sigma)$ is supported in some small neighbourhood of $\hat x=\hat y$. We want to apply the stationary phase method of Melin and Sj\"{o}strand~(see page 148 of\cite{MS74}) to carry out the $dx_{2n}d\sigma$ integration in \eqref{e-gue13630VIIa}.
Put 
\[\Psi(\hat x,\hat y,\sigma):=\varphi(x,y,s)\sigma+(x_{2n}-y_{2n})(\sigma-1).\]
We first notice that $d_\sigma\Psi(\hat x,\hat y,\sigma)|_{\hat x=\hat y}=0$ and
$d_{x_{2n}}\Psi(\hat x,\hat y,\sigma)|_{\sigma=1}=0$.
Thus, $x=y$ and $\sigma=1$ are real critical points. Moreover, we can
check that the Hessian of $\Psi(\hat x,\hat y,\sigma)$ at $\hat x=\hat y$, $\sigma=1$,
is given by
\[\left(
\begin{array}[c]{cc}
  \Psi''_{\sigma\sigma}(\hat x,\hat x,1)& \Psi''_{x_{2n}\sigma}(\hat x,\hat x,1) \\
  \Psi''_{\sigma x_{2n}}(\hat x,\hat x,1) & \Psi''_{x_{2n}x_{2n}}(\hat x,\hat x,1)
\end{array}\right)=\left(
\begin{array}[c]{cc}
 0 & 1 \\
 1 &0
\end{array}\right).\]
Thus, $\Psi(\hat x,\hat y,\sigma)$ is a non-degenerate complex valued phase function
in the sense of Melin-Sj\"{o}strand~\cite{MS74}. Let
\[\Td\Psi(\Td{\hat x},\Td{\hat y},\Td\sigma):=\Td\varphi(\Td x,\Td y,s)\Td\sigma+(\Td x_{2n}-\Td y_{2n})(\Td\sigma-1) \]
be an almost analytic extension of $\Psi(\hat x,\hat y,\hat\sigma)$, where
$\Td\varphi(\Td x,\Td y,s)$ is an almost analytic extension of $\varphi(x,y,s)$. Here we fix $s$. We can
check that given $y_{2n} $ and $(x,y)$, $\Td x_{2n}=y_{2n}-\varphi(x,y,s)$, $\Td\sigma=1$ are the solutions of 
\[\frac{\pr\Td\Psi}{\pr\Td\sigma}=0\,,\: \frac{\pr\Td\Psi}{\pr\Td x_{2n} }=0.\]
From this and by the stationary phase formula of Melin-Sj\"{o}strand~\cite{MS74}, we
get
\begin{equation} \label{e-gue1373VIII}
I_0(x,y)\equiv\int e^{ik\varphi(x,y,s)}a(x,y,s,k)ds\mod O(k^{-\infty}),
\end{equation}
where $a(x,y,s,k)\in S^{n}_{{\rm loc\,}}\big(1;\Omega,T^{*0,q}X\boxtimes T^{*0,q}X\big)\bigcap C^\infty_0\big(\Omega,T^{*0,q}X\boxtimes T^{*0,q}X\big)$, 
\[\mbox{$a(x,y,s,k)\sim\sum^\infty_{j=0}a_j(x,y,s)k^{n-j}$ in $S^{n}_{{\rm loc\,}}\big(1;\Omega,T^{*0,q}X\boxtimes T^{*0,q}X\big)$},\] 
$a_j(x,y,s)\in C^\infty_0\big(\Omega,T^{*0,q}X\boxtimes T^{*0,q}X\big)$, $j=0,1,2,\ldots$, and
\begin{equation} \label{e-gue1373aI}
a_0(x, y,s)=2\pi\int \Td b_0((x,y_{2n}-\varphi(x,y,s)),\hat y,s)\chi(y_{2n})\Td\chi_1(y_{2n}-\varphi(x,y,s))dy_{2n},
\end{equation}
where $\Td\chi_1$ and $\Td b_0$ are almost analytic extensions of $b_0$ and $\chi_1$ respectively, $b_0$ is as in Theorem~\ref{t-dcgewI}. From \eqref{e-gue1373aI} and notice that $\chi_1=1$ on ${\rm Supp\,}\chi$, $\varphi(x,x,s)=0$, 
we deduce that
\begin{equation} \label{e-gue1373aIII}
a_0(x, x,s)=2\pi\int \Td b_0((x,y_{2n},y,s)\chi(y_{2n})dy_{2n}.
\end{equation}
From \eqref{e-gue1373VII} and \eqref{e-gue1373aIII}, we conclude that 
\begin{equation}\label{e-gue1373aII}
a_0(x,x,s):T^{*0,q}_xX\To\mathcal{N}(x,s,n_-),\ \ \forall (x,x,s)\in\Omega.
\end{equation}

From \eqref{e-gue1373f}, \eqref{e-gue1373VIII} and \eqref{e-gue1373aII}, the theorem follows. 
\end{proof}

We need 

\begin{thm}\label{t-gue1374}
With the notations and assumptions before, we have 
\begin{equation}\label{e-gue1374}
\begin{split}
&\mathcal{S}^*_k\mathcal{S}_k\equiv\Td{\mathcal{I}}^*_k\mathcal{S}_k\mod O(k^{-\infty})\ \ \mbox{on $D_0$},\\
&\mathcal{S}^*_k\mathcal{S}_k\equiv\hat{\mathcal{I}}^*_k\mathcal{S}_k\mod O(k^{-\infty})\ \ \mbox{on $D_0$}
\end{split}
\end{equation}
and $\mathcal{S}^*_{k}\mathcal{S}_k$ 
is a smoothing operator and the kernel of $\mathcal{S}^*_{k}\mathcal{S}_k$ satisfies
\begin{equation} \label{e-gue1374I}
(\mathcal{S}^*_{k}\mathcal{S}_k)(x,y)\equiv\int e^{ik\varphi(x,y,s)}g(x,y,s,k)ds\mod
O(k^{-\infty})\ \ \mbox{on $D_0$}
\end{equation}
with
\begin{equation}  \label{e-gue1374II}
\begin{split}
&g(x,y,s,k)\in S^{n}_{{\rm loc\,}}\big(1;\Omega,T^{*0,q}X\boxtimes T^{*0,q}X\big)\bigcap C^\infty_0\big(\Omega,T^{*0,q}X\boxtimes T^{*0,q}X\big),\\
&g(x,y,s,k)\sim\sum^\infty_{j=0}g_j(x,y,s)k^{n-j}\text{ in }S^{n}_{{\rm loc\,}}
\big(1;\Omega,T^{*0,q}X\boxtimes T^{*0,q}X\big), \\
&g_j(x,y,s)\in C^\infty_0\big(\Omega,T^{*0,q}X\boxtimes T^{*0,q}X\big),\ \ j=0,1,2,\ldots,\\
&g_0(x,x,s)=\alpha^*_0(x,d_x\varphi(x,x,s))a_0(x,x,s),\ \ \forall (x,x,s)\in\Omega,
\end{split}
\end{equation}
where $\Td{\mathcal{I}}^*_k$ is as in \eqref{e-gue130814III}, $\alpha^*_0(x,\eta):T^{*0,q}_xX\To T^{*0,q}_xX$ is the adjoint of $\alpha_0(x,\eta)$ with respect to the Hermitian metric $\langle\,\cdot\,|\,\cdot\,\rangle$ on $T^{*0,q}_xX$, $\alpha_0(x,\eta)$ is as in \eqref{e-gue130814III}.
\end{thm}

\begin{proof}
From \eqref{e-gue1374IV} and \eqref{e-gue1374IVabal}, we have $\Td{\mathcal{I}}^*_k\mathcal{S}_k\equiv(\mathcal{G}^*_k\Box^{(q)}_{s,k}+\mathcal{S}^*_k)\mathcal{S}_k\mod O(k^{-\infty})$ on $D_0$ and $\hat{\mathcal{I}}^*_k\mathcal{S}_k\equiv(\mathcal{N}^*_k\Box^{(q)}_{s,k}+\mathcal{S}^*_k)\mathcal{S}_k\mod O(k^{-\infty})$ on $D_0$. Since $\Box^{(q)}_{s,k}\mathcal{S}_k\equiv0\mod O(k^{-\infty})$ on $D$, \eqref{e-gue1374} follows. 

From Theorem~\ref{t-gue13630I}, \eqref{e-gue1374} and the stationary phase
formula of Melin-Sj\"{o}strand~\cite{MS74}, we get \eqref{e-gue1374I} and \eqref{e-gue1374II}. The theorem follows. 
\end{proof} 

In the rest of this section, we will compute the leading terms $a_0(x,x,s)$ and $g_0(x,x,s)$ in the asymptotic expansions \eqref{e-gue13630Va} and \eqref{e-gue1374II} respectively. 

As before, let $x=(x_1,\ldots,x_{2n-1})$ be local coordinates on $D_0$. We also write $y=(y_1,\ldots,y_{2n-1})$ and $u=(u_1,\ldots,u_{2n-1})$. On $D_0$, we put $dv_X(x)=m(x)dx_1dx_2\ldots dx_{2n-1}=m(x)dx$. From \eqref{e-gue13630IVa}, we can check that 
\[\mathcal{S}^*_k(x,u)=\int_{\sigma\geq0}e^{-ik\ol\varphi(u,x,\sigma)}a^*(u,x,\sigma,k)d\sigma,\]
where $a^*(u,x,\sigma,k):T^{*0,q}_uX\To T^{*0,q}_xX$ is the adjoint of $a(u,x,\sigma,k):T^{*0,q}_xX\To T^{*0,q}_uX$ with 
respect to $\langle\,\cdot\,|\,\cdot\,\rangle$. Thus,
\begin{equation}\label{e-gue13710}
(\mathcal{S}^*_k\circ\mathcal{S}_k)(x, y)\equiv\int_{\sigma\geq0,s\geq0}
e^{ik\bigr(-\ol\varphi(u,x,\sigma)+\varphi(u, y,s)\bigr)}a^*(u,x,\sigma,k)a(u,y,s,k) m(u)dud\sigma ds\mod O(k^{-\infty}).
\end{equation}
Note that $\mathcal{S}^*_k(x,u)$ and $\mathcal{S}_k(u,y)$ are $k$-negligible outside $x=u$ and $u=y$ respectively and
$d_u\bigr(-\ol\varphi(u,x,\sigma)+\varphi(u,y,s)\bigr)\neq0$ if $\sigma\neq s$ and $(x,y)$ is in some small neighbourhood of $x=y$. From this observation, we conclude that for every $\varepsilon>0$, we have
\begin{equation}\label{e-gue13711}
\begin{split}
&(\mathcal{S}^*_k\circ\mathcal{S}_k)(x, y)\\
&\equiv\int_{\sigma\geq0,s\geq0}
e^{ik\bigr(-\ol\varphi(u,x,\sigma)+\varphi(u,y,s)\bigr)}a^*(u,x,\sigma,k)a(u,y,s,k)\mu(\frac{\sigma-s}{\varepsilon})\mu(\frac{\abs{y-u}^2}{\varepsilon})\mu(\frac{\abs{x-u}^2}{\varepsilon})m(u)dud\sigma ds\\
&\mod O(k^{-\infty}),
\end{split}
\end{equation}
where $\mu\in C^\infty_0(]-1,1[)$, $\mu=1$ on $[\frac{1}{2},\frac{1}{2}]$. 
We want to apply the stationary phase method of Melin and Sj\"{o}strand~(see page 148 of\cite{MS74}) to carry out the $dud\sigma$ integration in \eqref{e-gue13711}. Put 
\begin{equation}\label{e-gue13710I}
\Xi(u,x,y,\sigma,s):=-\ol\varphi(u,x,\sigma)+\varphi(u,y,s).
\end{equation} 
From \eqref{e-dgugeIX}, \eqref{e-dugeX}, \eqref{e-dugeXI} and \eqref{e-dgugeXI}, it is easy to see that
\[{\rm Im\,}\Xi(u,x,y,\sigma,s)\geq0,\ \ d_u\Xi(u, x, y,\sigma,s)|_{u=x=y,\sigma=s}=0,\ \ d_\sigma\Xi(u, x, y,\sigma,s)|_{u=x=y,\sigma=s}=0.\]
Thus, $x=y=u$, $\sigma=s$, $x$ is real, are real critical points. Now, we will compute the Hessian of $\Xi$ at $x=y=u$, $\sigma=s$. We write $H_{\Xi}(x,s)$ to denote the
Hessian of $\Xi$ at $x=y=u$, $\sigma=s$. $H_{\Xi}(x,s)$ has the following form:
\begin{equation}\label{e-gue13711I}
H_{\Xi}(x,s)=\left[
\begin{array}[c]{cccc}
  \frac{\pr^2\Xi}{\pr\sigma\pr\sigma}|_{u=x=y,\sigma=s}&\frac{\pr^2\Xi}{\pr\sigma\pr u_1}|_{u=x=y,\sigma=s}&\cdots&\frac{\pr^2\Xi}{\pr\sigma\pr u_{2n-1}}|_{u=x=y,\sigma=s} \\
  \frac{\pr^2\Xi}{\pr u_1\pr\sigma}|_{u=x=y,\sigma=s}&\frac{\pr^2\Xi}{\pr u_1\pr u_1}|_{u=x=y,\sigma=s}&\cdots&\frac{\pr^2\Xi}{\pr u_1\pr u_{2n-1}}|_{u=x=y,\sigma=s}\\
 \vdots&\vdots&\vdots&\vdots\\
  \frac{\pr^2\Xi}{\pr u_{2n-1}\pr\sigma}|_{u=x=y,\sigma=s}&\frac{\pr^2\Xi}{\pr u_{2n-1}\pr u_1}|_{u=x=y,\sigma=s}&\cdots&\frac{\pr^2\Xi}{\pr u_{2n-1}\pr u_{2n-1}}|_{u=x=y,\sigma=s}
\end{array}\right].\end{equation}
We fix $(p,p,s_0)\in\Omega$, $p\in D_0$. Take local coordinates $x=(x_1,\ldots,x_{2n-1})$ so that \eqref{e-geusw13623} hold. It is easy to see that 
\begin{equation}\label{e-gue13711II}
m(p)=2^{n-1}.
\end{equation}
From \eqref{e-guew13627}, it is straightforward to check that 
\begin{equation}\label{e-gue13711III}
\begin{split}
&\frac{\pr^2\Xi}{\pr\sigma\pr\sigma}|_{u=x=y=p,\sigma=s_0}=\frac{\pr^2\Xi}{\pr u_1\pr\sigma}|_{u=x=y,\sigma=s}=\ldots=\frac{\pr^2\Xi}{\pr u_{2n-2}\pr\sigma}|_{u=x=y,\sigma=s}=0,\\
&\frac{\pr^2\Xi}{\pr u_{2n-1}\pr\sigma}|_{u=x=y,\sigma=s}=-1,\\
&\frac{\pr^2\Xi}{\pr u_{2j-1}\pr u_{2j-1}}|_{u=x=y=p,\sigma=s_0}=\frac{\pr^2\Xi}{\pr u_{2j}\pr u_{2j}}|_{u=x=y=p,\sigma=s_0}=2i\abs{\lambda_j(s_0)},\ \ j=1,\ldots,n-1,\\
&\frac{\pr^2\Xi}{\pr u_j\pr u_k}|_{u=x=y=p,\sigma=s_0}=0\ \ \mbox{if $j\neq k$ and $j, k=1,\ldots,2n-2$},
\end{split}
\end{equation}
where $\lambda_1(s_0),\ldots,\lambda_{n-1}(s_0)$ are as in Theorem~\ref{t-gue140121II}. From \eqref{e-gue13711III}, we see that 
\begin{equation}\label{e-gue13711III-I}
H_{\Xi}(p,s_0)=\left[
\begin{array}[c]{ccccccc}
  0&0&0&\cdots&0&0&-1\\
  0&2i\abs{\lambda_1(s_0)}&0&\cdots&0&0&0\\
  0&0&2i\abs{\lambda_1(s_0)}&\cdots&0&0&0\\
  \vdots&\vdots&\vdots&\ddots&\vdots&\vdots&\vdots\\
 0&0&0&\cdots&2i\abs{\lambda_{n-1}(s_0)}&0&0\\
 0&0&0&\cdots&0&2i\abs{\lambda_{n-1}(s_0)}&0\\
 -1&*&*&\cdots&*&*&*
\end{array}\right].\end{equation}
Thus, 
\begin{equation}\label{e-gue13712}
\det\bigr(\frac{H_{\Xi}(p,s_0)}{2\pi i}\bigr)=\frac{1}{4}\pi^{-2n}\abs{\lambda_1(s_0)}^2\abs{\lambda_2(s_0)}^2\cdots\abs{\lambda_{n-1}(s_0)}^2.
\end{equation}
Since $(p,p,s_0)\in\Omega$ is arbitrary, we conclude that $\det\bigr(\frac{H_{\Xi}(p,s_0)}{2\pi i}\bigr)\neq0$, for every $(x,x,s)\in\Omega$. Hence, we can apply the stationary phase method of Melin and Sj\"{o}strand~(see page 148 of\cite{MS74}) to carry out the $dud\sigma$ integration in \eqref{e-gue13711} and obtain 
\begin{equation}\label{e-gue13712I}
(\mathcal{S}^*_{k}\mathcal{S}_k)(x,y)\equiv\int e^{ik\varphi_1(x,y,s)}h(x,y,s,k)ds\mod
O(k^{-\infty})
\end{equation}
with
\begin{equation}  \label{e-gue13712II}
\begin{split}
&h(x,y,s,k)\in S^{n}_{{\rm loc\,}}\big(1;\Omega,T^{*0,q}X\boxtimes T^{*0,q}X\big)\bigcap C^\infty_0\big(\Omega,T^{*0,q}X\boxtimes T^{*0,q}X\big),\\
&h(x,y,s,k)\sim\sum^\infty_{j=0}h_j(x,y,s)k^{n-j}\text{ in }S^{n}_{{\rm loc\,}}
\big(1;\Omega,T^{*0,q}X\boxtimes T^{*0,q}X\big), \\
&h_j(x,y,s)\in C^\infty_0\big(\Omega,T^{*0,q}X\boxtimes T^{*0,q}X\big),\ \ j=0,1,2,\ldots,\\
&h_0(x,x,s)=\Bigr(\det\bigr(\frac{H_{\Xi}(x,s)}{2\pi i}\bigr)\Bigr)^{-\frac{1}{2}}a^*_0(x,x,s)a_0(x,x,s)m(x),\ \ \forall (x,x,s)\in\Omega,
\end{split}
\end{equation}
and 
\begin{equation} \label{e-gue13712III}
\begin{split}
&\varphi_1(x, x,s)=0,\ \ d_x\varphi_1(x, x,s)=d_x\varphi(x, x,s),\ \ d_y\varphi_1(x, x,s)=d_y\varphi(x, x,s),\  \forall (x,x,s)\in\Omega,\\
&{\rm Im\,}\varphi_1(x,y,s)\geq0,\ \ \forall (x,y,s)\in\Omega,
\end{split}
\end{equation}
where $a_0(x,y,s)$ is as in \eqref{e-gue13630Va} and $a^*_0(x,x,s):T^{*0,q}_xX\To T^{*0,q}_xX$ is the adjoint of $a_0(x,x,s):T^{*0,q}_xX\To T^{*0,q}_xX$ with respect to $\langle\,\cdot\,|\,\cdot\,\rangle$. We need 

\begin{lem}\label{l-gue13712}
With the notations above, for every $(x,x,s)\in\Omega$, we have 
\begin{equation}\label{e-gue13712IV}
h_0(x,x,s)=g_0(x,x,s),
\end{equation}
where $g_0(x,y,s)$ is as in \eqref{e-gue1374II}. 
\end{lem}

\begin{proof}
Fix $(x_0,x_0,s_0)$. Suppose that ${\rm Re\,}h_0(x_0,x_0,s_0)\neq{\rm Re\,}g_0(x_0,x_0,s_0)$. We may assume that 
${\rm Re\,}h_0(x_0,x_0,s_0)<{\rm Re\,}g_0(x_0,x_0,s_0)$. Take $\epsilon_0>0$ be a small constant so that 
${\rm Re\,}h_0(x_0,x_0,s)<{\rm Re\,}g_0(x_0,x_0,s)$, for every $\abs{s-s_0}<\epsilon_0$. Let $\Sigma'$ and $V$ be as in \eqref{e-dhmpXIa} and \eqref{e-dhmpXII} respectively. For every $\epsilon>0$, put 
\[\Sigma'_{s_0,\epsilon}:=\set{(x,s\omega_0(x)-2{\rm Im\,}\ddbar_b\phi(x))\in\Sigma';\, \abs{s-s_0}<\epsilon}.\]
Let $r(x,\eta)\in C^\infty_0(V)$ with $r(x,\eta)\geq0$, $r(x,\eta)=1$ on $\Sigma'_{s_0,\frac{\epsilon_0}{2}}$ and 
${\rm Supp\,}r(x,\eta)\bigcap\Sigma\subset\Sigma'_{s_0,\epsilon_0}$. We remind that $\Sigma$ is given by \eqref{e-crmiII}. Consider the classical semi-classical pseudodifferential operator:
\[R=\frac{k^{2n-1}}{(2\pi)^{2n-1}}\int e^{ik<x-y,\eta>}r(x,\eta)d\eta.\]
From \eqref{e-gue13712I}, \eqref{e-gue13712II}, \eqref{e-gue13712III}, \eqref{e-dugeX} and the stationary phase method of Melin and Sj\"{o}strand~(see page 148 of\cite{MS74}), we have 
\begin{equation}\label{e-gue13712V}
\begin{split}
&(R\mathcal{S}^*_k\mathcal{S}_k)(x,x)\sim\sum^\infty_{j=0}k^{-j}\int\vartheta_j(x,x,s)ds\text{ in }S^{n}_{{\rm loc\,}}
\big(1;\Omega,T^{*0,q}X\boxtimes T^{*0,q}X\big), \\
&\vartheta_j(x,y,s)\in C^\infty_0\big(\Omega,T^{*0,q}X\boxtimes T^{*0,q}X\big),\ \ j=0,1,2,\ldots,\\
&\vartheta_0(x,x,s)=r(x,s\omega_0(x)-2{\rm Im\,}\ddbar_b\phi(x))h_0(x,x,s),\ \ \forall (x,x,s)\in\Omega.
\end{split}
\end{equation}
Similarly, from Theorem~\ref{t-gue1374}, we conclude that 
\begin{equation}\label{e-gue13712VI}
\begin{split}
&(R\mathcal{S}^*_k\mathcal{S}_k)(x,x)\sim\sum^\infty_{j=0}k^{-j}\int \zeta_j(x,x,s)ds\text{ in }S^{n}_{{\rm loc\,}}
\big(1;\Omega,T^{*0,q}X\boxtimes T^{*0,q}X\big), \\
&\zeta_j(x,y,s)\in C^\infty_0\big(\Omega,T^{*0,q}X\boxtimes T^{*0,q}X\big),\ \ j=0,1,2,\ldots,\\
&\zeta_0(x,x,s)=r(x,s\omega_0(x)-2{\rm Im\,}\ddbar_b\phi(x))g_0(x,x,s),\ \ \forall (x,x,s)\in\Omega.
\end{split}
\end{equation}
From \eqref{e-gue13712V} and \eqref{e-gue13712VI}, we deduce that 
\begin{equation}\label{e-gue13712VII}
\int r(x_0,s\omega_0(x_0)-2{\rm Im\,}\ddbar_b\phi(x_0)){\rm Re\,}h_0(x_0,x_0,s)ds=\int r(x_0,s\omega_0(x_0)-2{\rm Im\,}\ddbar_b\phi(x_0)){\rm Re\,}g_0(x_0,x_0,s)ds. 
\end{equation}
Since ${\rm Supp\,}r(x,\eta)\bigcap\Sigma\subset\Sigma'_{s_0,\epsilon_0}$, we have 
\[\begin{split}
&\int r(x_0,s\omega_0(x_0)-2{\rm Im\,}\ddbar_b\phi(x_0)){\rm Re\,}h_0(x_0,x_0,s)ds\\
&=\int_{\abs{s-s_0}<\epsilon_0}r(x_0,s\omega_0(x_0)-2{\rm Im\,}\ddbar_b\phi(x_0)){\rm Re\,}h_0(x_0,x_0,s)ds\end{split}\] and 
\[\begin{split}
&\int r(x_0,s\omega_0(x_0)-2{\rm Im\,}\ddbar_b\phi(x_0)){\rm Re\,}g_0(x_0,x_0,s)ds\\
&=\int_{\abs{s-s_0}<\epsilon_0}r(x_0,s\omega_0(x_0)-2{\rm Im\,}\ddbar_b\phi(x_0)){\rm Re\,}g_0(x_0,x_0,s)ds.\end{split}\] 
From this observation and \eqref{e-gue13712VII}, we deduce that
\[\int_{\abs{s-s_0}<\epsilon_0}r(x_0,s\omega_0(x_0)-2{\rm Im\,}\ddbar_b\phi(x_0))\bigr({\rm Re\,}h_0(x_0,x_0,s)-{\rm Re\,}g_0(x_0,x_0,s)\bigr)ds=0.\]
Since ${\rm Re\,}h_0(x_0,x_0,s)<{\rm Re\,}g_0(x_0,x_0,s)$, for every $\abs{s-s_0}<\epsilon_0$, and $r(x_0,s\omega_0(x_0)-2{\rm Im\,}\ddbar_b\phi(x_0))\geq0$, $r(x_0,s\omega_0(x_0)-2{\rm Im\,}\ddbar_b\phi(x_0))$ is not identically to zero function, we get a contradiction. 
Thus, ${\rm Re\,}h_0(x_0,x_0,s_0)={\rm Re\,}g_0(x_0,x_0,s_0)$.

We can repeat the procedure above and conclude that ${\rm Im\,}h_0(x_0,x_0,s_0)={\rm Im\,}g_0(x_0,x_0,s_0)$. Since $(x_0,x_0,s_0)$ is arbitrary, the lemma follows.
\end{proof}

As before, we fix $(p,p,s_0)\in\Omega$ and take local coordinates $x=(x_1,\ldots,x_{2n-1})$ so that \eqref{e-geusw13623} hold. From \eqref{e-gue13711II}, \eqref{e-gue13712}, \eqref{e-gue13712II} Lemma~\ref{l-gue13712} and \eqref{e-gue1374II}, we see that 
\begin{equation}\label{e-gue13716}
\begin{split}
g_0(p,p,s_0)&=(2\pi)^n\abs{\lambda_1(s_0)}^{-1}\abs{\lambda_2(s_0)}^{-1}\cdots\abs{\lambda_{n-1}(s_0)}^{-1}a^*_0(p,p,s_0)a_0(p,p,s_0)\\
&=\alpha^*_0(p,s_0\omega_0(p)-2{\rm Im\,}\ddbar_b\phi(p))a_0(p,p,s_0),
\end{split}
\end{equation}
where $\alpha_0(x,\eta)$ is as in \eqref{e-gue130814III} and $\alpha^*_0(x,\eta):T^{*0,q}_xX\To T^{*0,q}_xX$ is the adjoint of $\alpha_0(x,\eta)$ with respect to the Hermitian metric $\langle\,\cdot\,|\,\cdot\,\rangle$ on $T^{*0,q}_xX$. Let 
$\mathcal{\pi}_{(p,s_0,n_-)}:T^{*0,q}_pX\To\mathcal{N}(p,s_0,n_-)$ be the orthogonal projection with respect to $\langle\,\cdot\,|\,\cdot\,\rangle$. In view of \eqref{e-gue13630Va}, we know that $a_0(p,p,s_0):T^{*0,q}_pX\To\mathcal{N}(p,s_0,n_-)$. From \eqref{e-gue1373III}, we see that ${\rm dim\,}\mathcal{N}(p,s_0,n_-)=1$. From this observation and \eqref{e-gue13716}, it is straightforward to see that 
\[\begin{split}
a_0(p,p,s_0)=&(2\pi)^{-n}\abs{\lambda_1(s_0)}\abs{\lambda_2(s_0)}\cdots\abs{\lambda_{n-1}(s_0)}\mathcal{\pi}_{(p,s_0,n_0)}\alpha(p,s_0\omega_0(p)-2{\rm Im\,}\ddbar_b\phi(p)),\\
g_0(p,p,s_0)=&(2\pi)^{-n}\abs{\lambda_1(s_0)}\abs{\lambda_2(s_0)}\cdots\abs{\lambda_{n-1}(s_0)}\times\\
&\quad\alpha^*_0(p,s_0\omega_0(p)-2{\rm Im\,}\ddbar_b\phi(p))\mathcal{\pi}_{(p,s_0,n_-)}\alpha(p,s_0\omega_0(p)-2{\rm Im\,}\ddbar_b\phi(p)).
\end{split}\]

Summing up, we obtain 

\begin{thm}\label{t-gue13716}
With the same same notations and assumptions as in Theorem~\ref{t-gue13630}, let $a_0(x,y,s)\in C^\infty_0(\Omega,T^{*0,q}X\boxtimes T^{*0,q}X)$ and $g_0(x,y,s)\in C^\infty_0(\Omega,T^{*0,q}X\boxtimes T^{*0,q}X)$
be as in \eqref{e-gue13630Va} and \eqref{e-gue1374II} respectively. Fix $(p,p,s_0)\in\Omega$, $p\in D_0$, and let $\mathcal{\pi}_{(p,s_0,n_-)}:T^{*0,q}_pX\To\mathcal{N}(p,s_0,n_-)$ be the orthogonal projection with respect to $\langle\,\cdot\,|\,\cdot\,\rangle$, where $\mathcal{N}(p,s_0,n_-)$ is given by \eqref{e-gue1373III}. Then,
\begin{equation}\label{e-gue13716I}
\begin{split}
&a_0(p,p,s_0)=(2\pi)^{-n}\abs{\det\bigr(M^\phi_p-2s_0\mathcal{L}_p\bigr)}\mathcal{\pi}_{(p,s_0,n_-)}\alpha_0(p,s_0\omega_0(p)-2{\rm Im\,}\ddbar_b\phi(p)),\\
&g_0(p,p,s_0)\\
&=(2\pi)^{-n}\abs{\det\bigr(M^\phi_p-2s_0\mathcal{L}_p\bigr)}\alpha^*_0(p,s_0\omega_0(p)-2{\rm Im\,}\ddbar_b\phi(p))\mathcal{\pi}_{(p,s_0,n_-)}\alpha_0(p,s_0\omega_0(p)-2{\rm Im\,}\ddbar_b\phi(p)),
\end{split}
\end{equation}
where $\alpha_0(x,\eta)$ is as in \eqref{e-gue130814III}, $\alpha^*_0(x,\eta):T^{*0,q}_xX\To T^{*0,q}_xX$ is the adjoint of $\alpha_0(x,\eta)$ with respect to the Hermitian metric $\langle\,\cdot\,|\,\cdot\,\rangle$ on $T^{*0,q}_xX$ and  $\abs{\det\bigr(M^\phi_p-2s_0\mathcal{L}_p\bigr)}=\abs{\lambda_1(s_0)}\abs{\lambda_2(s_0)}\cdots\abs{\lambda_{n-1}(s_0)}$. Here $\lambda_1(s_0),\ldots,\lambda_{n-1}(s_0)$ are eigenvalues of the Hermitian quadratic form $M^\phi_p-2s_0\mathcal{L}_p$ with respect to $\langle\,\cdot\,|\,\cdot\,\rangle$. 
\end{thm}

Using Theorem~\ref{t-dcmimpI} and repeating the proof of Theorem~\ref{t-gue13630} we conclude that 

\begin{thm}\label{t-gue13717}
Let $s$ be a local trivializing section of $L$ on an open subset $D\subset X$ and $\abs{s}^2_{h^L}=e^{-2\phi}$. We assume that there exist a $\lambda_0\in\Real$ and $x_0\in D$ such that $M^\phi_{x_0}-2\lambda_0\mathcal{L}_{x_0}$ is non-degenerate of constant signature $(n_-,n_+)$. Let $q\neq n_-$. We fix $D_0\Subset D$, $D_0$ open. Let $V$ be as in \eqref{e-dhmpXII}. Let 
\[\mbox{$\hat{\mathcal{I}}_k\equiv\frac{k^{2n-1}}{(2\pi)^{2n-1}}\int e^{ik<x-y,\eta>}\alpha(x,\eta,k)d\eta\mod O(k^{-\infty})$ at $T^*D_0\bigcap\Sigma$}\]
be a classical semi-classical pseudodifferential operator on $D$ of order $0$ from sections of $T^{*0,q}X$ to sections of $T^{*0,q}X$, where $\alpha(x,\eta,k)\in S^0_{{\rm loc\,},{\rm cl\,}}(1;T^*D,T^{*0,q}X\boxtimes T^{*0,q}X)$ with $\alpha(x,\eta,k)=0$ if $\abs{\eta}>M$, for some large $M>0$ and ${\rm Supp\,}\alpha(x,\eta,k)\bigcap T^*D_0\Subset V$. Then, there exists a properly supported continuous operator
\[\mathcal{N}_{k}=O(k^s): H^s_{{\rm comp\,}}(D,T^{*0,q}X)\To H^{s+1}_{{\rm comp\,}}(D,T^{*0,q}X),\ \ \forall s\in\mathbb N_0,\]
such that 
\[\Box^{(q)}_{s,k}\mathcal{N}_{k}\equiv\hat{\mathcal{I}}_k\mod O(k^{-\infty})\]
on $D_0$, where $\Box^{(q)}_{s,k}$ is given by \eqref{e-msmilkVI}.
\end{thm}

\section{Szeg\"{o} kernel asymptotics for lower energy forms} \label{s-safle} 

Let $\lambda\geq0$. We recall that (see \eqref{e-suX}) $H^q_{b,\leq\lambda}(X,L^k)$ denote the spectral space of $\Box^{(q)}_{b,k}$ corresponding to energy less that $\lambda$ and  $\Pi^{(q)}_{k,\leq\lambda}:L^2_{(0,q)}(X,L^k)\To H^q_{b,\leq\lambda}(X,L^k)$ denote the orthogonal projection with respect to $(\,\cdot\,|\,\cdot\,)_{h^{L^k}}$. Fix $N_0\geq1$. In this section, we will study semi-classical asymptotic expansion of $\Pi^{(q)}_{k,\leq k^{-N_0}}$.


\subsection{Asymptotic upper bounds}\label{s-sub}

Fix $N_0\geq 1$. In this section we will give pointwise upper bounds for $u$ and $\pr^\alpha u$, where $u\in H^q_{b,\leq k^{-N_0}}(X,L^k)$. 

Let $s$ be a local trivializing section of $L$ on an open subset $D\subset X$ and $\abs{s}^2_{h^L}=e^{-2\phi}$. Fix $p\in D$, let $U_1(y),\ldots, U_{n-1}(y)$
be an orthonormal frame of $T^{1,0}_yX$ varying smoothly with $y$ in a neighbourhood of $p$,
for which the Levi form is diagonal at $p$. We take local coordinates
$x=(x_1,\ldots,x_{2n-2},x_{2n-1})=(z,x_{2n-1})$, $z_j=x_{2j-1}+ix_{2j}$, $j=1,\ldots,n-1$,
defined on a small neighbourhood of $p$ such that
\begin{equation}\label{e-gue13716II}
\begin{split}
&x(p)=0,\ \ \omega_0(p)=dx_{2n-1},\ \ T(p)=-\frac{\pr}{\pr x_{2n-1}}(p),\\
&\langle\,\frac{\pr}{\pr x_j}(p)\,|\,\frac{\pr}{\pr x_t}(p)\,\rangle=2\delta_{j,t},\ \ j,t=1,\ldots,2n-2,\\
&U_j=\frac{\pr}{\pr z_j}-i\tau_j\ol z_j\frac{\pr}{\pr x_{2n-1}}-c_jx_{2n-1}\frac{\pr}{\pr x_{2n-1}}+O(\abs{x}^2),\ \ j=1,\ldots,n-1,\\
&\phi=\sum^{n-1}_{j=1}(\alpha_j z_j+\ol\alpha_j\ol z_j)+\beta x_{2n-1}+\sum^{n-1}_{j,t=1}(a_{j,t}z_jz_t+\ol a_{j,t}\ol z_j\ol z_t)+\sum^{n-1}_{j,t=1}\mu_{j,\,t}\ol z_jz_t\\
&+O(\abs{z}\abs{x_{2n-1}})+O(\abs{x_{2n-1}}^2)+O(\abs{x}^3),
\end{split}
\end{equation}
where $\tau_j, \beta\in\Real$, $j=1,\ldots,n-1$, $\mu_{j,t}, c_j, \alpha_j, a_{j,t}\in\Complex$, $\mu_{j,t}=\ol\mu_{t,j}$, $j, t=1,\ldots,n-1$ (This is always possible, see~\cite[p.\,157--160]{BG88}). Note that $\tau_1,\ldots,\tau_{n-1}$ are eigenvalues of $\mathcal{L}_p$ with respect to $\langle\,\cdot\,|\,\cdot\,\rangle$.
We assume that this local coordinates are defined on $D$ and
until further notice, we work with this local coordinates and we identify $D$ with some open set in $\Real^{2n-1}$. 
Put
\begin{gather}
R(x)=R(z,x_{2n-1})=\sum^{n-1}_{j=1}\alpha_j z_j+\sum^{n-1}_{j,t=1}a_{j,t}z_jz_t,\label{e-gue13716III}\\
\phi_0=\phi-R(x)-\ol{R(x)} =\beta x_{2n-1}+\sum^{n-1}_{j,t=1}\mu_{j,\,t}\ol z_jz_t+O(\abs{z}\abs{x_{2n-1}})+O(\abs{x_{2n-1}}^2)+O(\abs{x}^3).\label{e-gue13716IV}
\end{gather}
Let $(\,\cdot\,|\,\cdot\,)_{k\phi}$ and $(\,\cdot\,|\,\cdot\,)_{k\phi_0}$ be the inner products on the space
$\Omega^{0,q}_0(D)$ defined as follows:
\[
(f\ |\ g)_{k\phi}=\int_D\!\langle\,f\,|\,g\,\rangle e^{-2k\phi}dv_X\,, \quad (f\ |\ g)_{k\phi_0}=\int_D\!\langle\,f\,|\, g\,\rangle e^{-2k\phi_0}dv_X\,,
\]
where $f, g\in\Omega^{0,q}_0(D)$. We denote by $L^2_{(0,q)}(D, k\phi)$ and $L^2_{(0,q)}(D, k\phi_0)$ the completions of $\Omega^{0,q}_0(D)$ with respect to $(\,\cdot\,|\,\cdot\,)_{k\phi}$ and $(\,\cdot\,|\,\cdot\,)_{k\phi_0}$, respectively. We have the unitary identification
\begin{equation} \label{e-gue13717}
\left\{\begin{aligned}
L^2_{(0,q)}(D, k\phi_0)&\leftrightarrow L^2_{(0,q)}(D, k\phi) \\
u&\rightarrow e^{2kR}u, \\
u=e^{-2kR}v&\leftarrow v.
\end{aligned}
\right.
\end{equation}
Let
$\ddbar^{*,k\phi}_b:\Omega^{0,q+1}(D)\To\Omega^{0,q}(D)$
be the formal adjoint of $\ddbar_b$ with respect to $(\,\cdot\,|\,\cdot\,)_{k\phi}$. Put
\[
\Box^{(q)}_{b,k\phi}=\ddbar_b\ddbar^{*,k\phi}_b+\ddbar^{*,k\phi}_b\ddbar_b:\Omega^{0,q}(D)\To\Omega^{0,q}(D)\,.
\]
Let $u\in\Omega^{0,q}(D, L^k)$. Then there exists $\hat u\in\Omega^{0,q}(D)$ such that $u=s^k\hat u$ and we have $\Box^{(q)}_{b,k}u=s^k\Box^{(q)}_{b,k\phi}\hat u$.
In this section, we identify $u$ with $\hat u$ and $\Box^{(q)}_{b,k}$ with $\Box^{(q)}_{b, k\phi}$. Note that $\abs{u(0)}^2=\abs{\hat u(0)}^2e^{-2k\phi(0)}=\abs{\hat u(0)}^2$. Let $\Td\ddbar_b:\Omega^{0,q}(D)\To\Omega^{0,q+1}(D)$ be the first order partial differential operator given by 
\begin{equation}\label{e-gue13717I}
\begin{split}
&\Td\ddbar_b:\Omega^{0,q}(D)\To\Omega^{0,q+1}(D),\\
&\ddbar_b(e^{2kR}u)=e^{2kR}\Td\ddbar_bu,\ \ \forall u\in\Omega^{0,q}(D).
\end{split}
\end{equation}
Let $\Td\ddbar_b^{*}:\Omega^{0,q+1}(D)\To\Omega^{0,q}(D)$ be the formal adjoint of $\Td\ddbar_b$ with respect to $(\,\cdot\,|\,\cdot\,)_{k\phi_0}$. It is easy to see that 
\begin{equation}\label{e-gue13717II}
\ddbar^{*,k\phi}_b(e^{2kR}u)=e^{2kR}\Td\ddbar_b^{*}u,\ \ \forall u\in\Omega^{0,q+1}(D).
\end{equation}
Put 
\begin{equation}\label{e-gue13717III}
\Td\Box^{(q)}_{b,k\phi}=\Td\ddbar_b\Td\ddbar_b^*+\Td\ddbar^*_b\Td\ddbar_b:\Omega^{0,q}(D)\To\Omega^{0,q}(D).
\end{equation}
From \eqref{e-gue13717I} and \eqref{e-gue13717II}, we have 
\begin{equation}\label{e-gue13717IV}
\Box^{(q)}_{b,k\phi}(e^{2kR}u)=e^{2kR}\Td\Box^{(q)}_{b,k\phi}u,\ \ \forall u\in\Omega^{0,q}(D).
\end{equation}

Until further notice, we fix $q\in\set{0,1,\ldots,n-1}$ and we assume that $Y(q)$ holds at each point of $D$. 
For $r>0$, let
$D_r=\set{x=(x_1,\ldots,x_{2n-1})\in\Real^{2n-1};\, \abs{x_j}<r,\ \ j=1,\ldots,2n-1}$.
Let $F_k$ be the scaling map:
$F_k(z,x_{2n-1})=(\frac{z}{\sqrt{k}}, \frac{x_{2n-1}}{k})$.
From now on, we assume that $k$ is large enough so that $F_k(D_{\log k})\subset D$. Let $(e_j(x))_{j=1,\ldots,n-1}$ denote the basis of $T^{*0,1}_xX$, dual to $(\ol U_j(x))_{j=1,\ldots,n-1}$. Let $J=(j_1,\ldots,j_q)\in\set{1,\ldots,n-1}^q$ be a multiindex. Define 
\[e_J=e_{j_1}\wedge\cdots\wedge e_{j_q},\ \ \mbox{if $1\leq j_1,j_2,\ldots,j_q\leq n-1$}.\]
Then $\set{e_J;\, \mbox{$J\in\set{1,\ldots,n-1}^q$, $J$ strictly increasing}}$ is an orthonormal frame for $T^{*0,q}X$ over $D$.
We define the scaled bundle $F^*_kT^{*0,q}X$ on $D_{\log k}$ to be the bundle whose fiber at $x\in D_{\log k}$ is
\[F^*_kT^{*0,q}_xX:=\Bigr\{\textstyle\sideset{}{'}\sum\limits_{J\in\set{1,\ldots,n-1}^q}a_Je_J(\frac{z}{\sqrt{k}},\frac{x_{2n-1}}{k});\, a_J\in\Complex, \forall J\in\set{1,\ldots,n-1}^q\Bigr\},\]
where $\sum'$ means that the summation is performed only over strictly increasing multiindices.
We take the Hermitian metric $\langle\,\cdot\,|\,\cdot\,\rangle_{F^*_k}$ on $F^*_kT^{*0,q}X$ so that at each point $x\in D_{\log k}$, 
\[\set{e_J\big(\tfrac{z}{\sqrt{k}}\;,\tfrac{x_{2n-1}}{k}\big)\,; \text{$J\in\set{1,\ldots,n-1}^q$, $J$ strictly increasing}}\]
is an orthonormal basis for $F^*_kT^{*0,q}_{x}X$. For $r>0$, let $F^*_k\Omega^{0,q}(D_r)$
denote the space of smooth sections of $F^*_kT^{*0,q}X$ over $D_r$. Let $F^*_k\Omega^{0,q}_0(D_r)$ be the subspace of
$F^*_k\Omega^{0,q}(D_r)$ whose elements have compact support in $D_r$.
Given $f\in\Omega^{0,q}(F_k(D_{\log k}))$ we write
$f=\sideset{}{'}\sum\limits_{\abs{J}=q}f_Je_J$.
We define the scaled form $F_k^*f\in F^*_k\Omega^{0,q}(D_{\log k})$ by:
\[
F_k^*f=\sideset{}{'}\sum_{J\in\set{1,\ldots,n-1}^q}f_J\Big(\frac{z}{\sqrt{k}}\;, \frac{x_{2n-1}}{k}\Big)e_J\Big(\frac{z}{\sqrt{k}}\;,\frac{x_{2n-1}}{k}\Big)\,.
\] 
It is well-known (see section 2.2 in~\cite{HM09}) that there is a scaled Laplacian $\Td\Box^{(q),(k)}_{b,k\phi}:F^*_k\Omega^{0,q}(D_{\log k})\To F^*_k\Omega^{0,q}(D_{\log k})$ such that 
\begin{equation}\label{e-gue13717V}
\Td\Box^{(q),(k)}_{b,k\phi}(F^*_ku)=\frac{1}{k}F^*_k(\Td\Box^{(q)}_{b,k\phi}u),\ \ \forall u\in\Omega^{0,q}(F_k(D_{\log k})),
\end{equation}
and all the derivatives of the coefficients of the operator $\Td\Box^{(q),(k)}_{k\phi,(k)}$ are uniformly bounded in $k$ on $D_{\log k}$. Let $D_r\subset D_{\log k}$ and let $W^s_{kF^*_k\phi_0}(D_r, F^*_kT^{*0,q}X)$,
$s\in\mathbb N_0$, denote the Sobolev space of order $s$ of sections of $F^*_kT^{*0,q}X$
over $D_r$ with respect to the weight $e^{-2kF^*_k\phi_0}$. The Sobolev norm on this space is given by
\begin{equation} \label{e-gue13717VI}
\norm{u}^2_{kF^*_k\phi_0,s,D_r}
=\sideset{}{'}\sum_{\alpha\in\mathbb{N}^{2n-1}_0,\abs{\alpha}\leq s,J\in\set{1,\ldots,n-1}^q}
\int_{D_r}\!\abs{\pr^\alpha_xu_J}^2e^{-2kF^*_k\phi_0}(F^*_km)(x)dx,
\end{equation}
where
$u=\sum'_{J\in\set{1,\ldots,n-1}^q}u_Je_J\big(\frac{z}{\sqrt{k}},\frac{x_{2n-1}}{k}\big)\in W^s_{kF^*_k\phi_0}
(D_r,F^*_kT^{*0,q}X)$ and $m(x)dx$ is the volume form.
If $s=0$, we write $\norm{\cdot}_{kF^*_k\phi_0,D_r}$ to denote $\norm{\cdot}_{kF^*_k\phi_0,0,D_r}$. 

The following is well-known(see section 2.2 in~\cite{HM09})

\begin{prop}\label{p-gue13717}
Let $r>0$ with $D_{2r}\subset D_{\log k}$ and let $s\in\N_0$. Then, there is a constant $C_{r,s}>0$
independent of $k$ and the point $p$, such that for every $u\in\Omega^{0,q}(D_{\log k})$,
\begin{equation} \label{e-gue13717VII}
\norm{u}^2_{kF^*_k\phi_0,s,D_{r}}\leqslant  C_{r,s}\Bigr(\norm{u}^2_{kF^*_k\phi_0,D_{2r}}+\sum^s_{j=1}\big\|(\Td\Box^{(q),(k)}_{b,k\phi})^ju\big\|^2_{kF^*_k\phi_0,D_{2r}}\Bigl)\,.
\end{equation}

Moreover, there exist a semi-norm $P$ on $C^\infty(D_{2r})$ and a strictly positive continuous function $F:\Real\To\Real_+$, where $P$ and $F$ only depend on $r$ and $s$ independent of the point $p$ and $k$, such that if we put
\begin{equation*}
\begin{split}
&A\\
&=\{\mbox{all the coefficients of $\Td\Box^{(q),(k)}_{b,k\phi}$, $\Td\ddbar_b$, $\Td\ddbar^*_b$, $[\ol U_j ,U_t]$, $\ol U_j$, $U_t$, $j,t=1,\ldots,n-1$, and $kF^*_k\phi_0$, $F^*_km$}\} 
\end{split}
\end{equation*}
and $B=\set{\mbox{all the eigenvalues of $\mathcal{L}_p$}}$, then $C_{r,s}$ can be bounded by $\sum\limits_{f\in A}F(P(f))+\sum\limits_{\lambda\in B}F(\lambda)$. 
\end{prop}

We need 

\begin{lem} \label{l-gue13717}
For $k$ large and for every $\alpha\in\mathbb N^{2n-1}_0$, there is a constant $C_{\alpha}>0$ independent of $k$ and the point $p$, such that for all $u\in\Omega^{0,q}(D_{\log k})$ with $\norm{u}_{kF^*_k\phi_0,D_{\log k}}\leq 1$ and $\norm{(\Td\Box^{(q),(k)}_{b,k\phi})^mu}_{kF^*_k\phi_0,D_{\log k}}\leq k^{-m}$, $\forall m\in\N_0$, we have
\begin{equation} \label{e-gue13717VIII}
\abs{(\pr^\alpha_xu)(0)}\leq C_\alpha.
\end{equation}
\end{lem} 

\begin{proof}
Let $u\in\Omega^{0,q}(D_{\log k})$, $\norm{u}_{kF^*_k\phi_0,D_{\log k}}\leq 1$, $\norm{(\Td\Box^{(q),(k)}_{b,k\phi})^mu}_{kF^*_k\phi_0,D_{\log k}}\leq k^{-m}$, $\forall m\in\N_0$.
By using Fourier transform, it is easy to see that (cf.\ Lemma 2.6 in~\cite{HM09})
\begin{equation} \label{e-gue13717aI}
\abs{(\pr^\alpha_xu)(0)}\leq C\norm{u}_{kF^*_k\phi_0,n+\abs{\alpha},D_r},
\end{equation}
for some $r>0$, where $C>0$ only depends on the dimension and the size of $\alpha$. From \eqref{e-gue13717VII}, we see that
\begin{equation} \label{e-gue13717aII}
\begin{split}
\norm{u}^2_{kF^*_k\phi_0,n+\abs{\alpha},D_r}&\leq C_{r,n+\abs{\alpha}}\Bigr(\norm{u}^2_{kF^*_k\phi_0,D_{2r}}+\sum^{n+\abs{\alpha}}_{j=1}\norm{(\Td\Box^{(q),(k)}_{b,k\phi})^ju}_{kF^*_k\phi_0,D_{2r}}\Bigr)\\
&\leq C_{r,n+\abs{\alpha}}\Bigr(1+\sum^\infty_{j=1}k^{-j}\Bigr)\leq\Td C_\alpha,
\end{split}
\end{equation}
where $\Td C_{\alpha}>0$ is independent of $k$ and the point $p$. Combining \eqref{e-gue13717aII} with \eqref{e-gue13717aI}, \eqref{e-gue13717VIII} follows.
\end{proof} 

Now, we can prove

\begin{thm} \label{t-gue13718}
For every $\alpha\in\mathbb N^{2n-1}_0$, $D'\Subset D$, there is a constant $C_{\alpha,D'}>0$ independent of $k$, such that for every $u\in H^q_{b,\leq k^{-N_0}}(X,L^k)$, $u|_D=s^k\Td u$, $\Td u\in\Omega^{0,q}(D)$, we have 
\begin{equation} \label{e-gue13718}
\abs{(\pr^\alpha_x(\Td ue^{-k\phi}))(x)}\leq C_{\alpha,D'}k^{\frac{n}{2}+\abs{\alpha}}\norm{u},\ \ \forall x\in D'.
\end{equation}
\end{thm}

\begin{rem} \label{r-gue13718}
Let $s_1$ be another local frame of $L$ on $D$, $\abs{s_1}^2=e^{-2\phi_1}$. We have $s_1=gs$ for some CR function $g\in C^\infty(D)$, $g\neq0$ on $D$. Let $u\in\Omega^{0,q}(D,L^k)$.
On $D$, we write $u=s^k\Td u=s^k_1\Td v$. Then, we can check that
\begin{equation} \label{e-gue13718I}
\Td ve^{-k\phi_1}=\Td u({\ol g}^{\,1/2}g^{-1/2})^ke^{-k\phi}.
\end{equation}
From \eqref{e-gue13718I}, it is easy to see that if $\Td u$ satisfies \eqref{e-gue13718}, then $\Td v$ also satisfies \eqref{e-gue13718}. Thus, the conclusion of Theorem~\ref{t-gue13718} makes sense.
\end{rem} 

\begin{proof} [Proof of Theorem~\ref{t-gue13718}]
We may assume that $0\in D'$.
Let $u\in H^q_{b,\leq k^{-N_0}}(X,L^k)$, $u|_D=s^k\Td u$, $\Td u\in\Omega^{0,q}(D)$. We may assume that
$F_k(D_{\log k})\subset D$ and consider $\Td u|_{F_k(D_{\log k})}$.
Set 
\[\beta_k:=k^{-\frac{n}{2}}F^*_k(e^{-2kR}\Td u)\in F^*_k\Omega^{0,q}(D_{\log k}).\] 
We recall that $R$ is given by \eqref{e-gue13716III}. (See also \eqref{e-gue13717}.) We can check that
\begin{equation} \label{e-gue13718II}
\norm{\beta_k}_{kF^*_k\phi_0,D_{\log k}}\leq\norm{u}_{h^{L^k}}.
\end{equation}
Since $u\in H^q_{b,\leq k^{-N_0}}(X,L^k)$, we have $\norm{(\Box^{(q)}_{b,k})^mu}_{h^{L^k}}\leq k^{-mN_0}\norm{u}_{h^{L^k}}$ for all $m\in\N$. From this observation, \eqref{e-gue13717V} and \eqref{e-gue13717IV}, we have
\begin{equation} \label{e-gue13718III}
\begin{split}
\norm{(\Td\Box^{(q),(k)}_{b,k\phi})^m\beta_k}_{kF^*_k\phi_0,D_{\log k}}&=\frac{1}{k^{m+\frac{n}{2}}}\norm{F^*_k\bigr((\Td\Box^{(q)}_{b,k\phi})^me^{-2kR}\Td u\bigr)}_{kF^*_k\phi_0,D_{\log k}}\\
&\leq\frac{1}{k^m}\norm{(\Box^{(q)}_{b,k})^mu}_{h^{L^k}}\leq k^{-mN_0-m}\norm{u}_{h^{L^k}}.
\end{split}
\end{equation}
From \eqref{e-gue13718II}, \eqref{e-gue13718III} and Lemma~\ref{l-gue13717}, it is straightforward to see that for every
$\alpha\in\mathbb N^{2n-1}_0$,
\[\abs{k^{-\frac{n}{2}-\frac{\abs{\alpha}}{2}}(\pr^\alpha_x\Td u)(0)}\leq\hat C_\alpha k^{\frac{\abs{\alpha}}{2}}\sum\limits_{\gamma\in\mathbb N^{2n-1}_0,\abs{\gamma}\leq\abs{\alpha}}\abs{(\pr^\gamma_x\beta_k)(0)}\leq\Td C_{\alpha}\norm{u}_{h^{L^k}},\]
where $\hat C_\alpha>0$, $\Td C_{\alpha}>0$ are constants independent of $k$ and the point $p$. Thus, for every $\alpha\in\mathbb N^{2n}_0$, there is a constant $C_{\alpha}>0$ independent of $k$ and the point $p$, such that
\[\abs{(\pr^\alpha_x(\Td ue^{-k\phi}))(0)}\leq C_\alpha k^{\frac{n}{2}+\abs{\alpha}}\norm{u}_{h^{L^k}}.\]

Let $x_0$ be another point of $D'$. We can repeat the procedure above and conclude that for every $\alpha\in\mathbb N^{2n-1}_0$, there is a $C_\alpha(x_0)>0$ independent of $k$ and the point $x_0$, such that
\[\abs{(\pr^\alpha_x(\Td ue^{-k\phi}))(x_0)}\leq C_\alpha(x_0) k^{\frac{n}{2}+\abs{\alpha}}\norm{u}_{h^{L^k}}.\]
The theorem follows.
\end{proof} 

We pause and introduce some notations.  We identify $\Real^{2n-1}$ with the Heisenberg group $H_n:=\Complex^{n-1}\times\Real$. Put
\begin{equation} \label{e-gue131029}
\psi_0(z, \theta)=\beta x_{2n-1}+\sum^{n-1}_{j,t=1}\mu_{j,t}\ol z_jz_t\in C^\infty(H_n,\Real),
\end{equation}
where $\beta$ and $\mu_{j,t}$, $j,t=1,\ldots,n-1$, are as in \eqref{e-gue13716II}. Let $(\,\cdot\,|\,\cdot\, )_{\psi_0}$
be the inner product on $C^\infty_0(H_n)$ defined as follows:
\[
(\,f\,|\,g\,)_{\psi_0}=\int_{H_n}f\ol g e^{-2\psi_0}d\lambda(x)\,, \quad f, g\in C^\infty_0(H_n),\]
where $d\lambda(x)=2^{n-1}dx_1\cdots dx_{2n-1}$. For $f\in C^\infty(H_n)$, we write $\norm{f}^2_{\psi_0}:=(\,f\,|\,f\,)_{\psi_0}$.
Let $u(x)\in C^\infty(H_n)$. Fix $\alpha\in\mathbb N^{2n-1}_0$. Assume that $\norm{\pr^\alpha_xu}_{\psi_0}<\infty$. Put 
\begin{equation}\label{e-gue131029I}
v_\alpha(x)=(\pr^\alpha_xu)(x)e^{-\beta x_{2n-1}}.
\end{equation}
Set 
\begin{equation}\label{e-gue131029II}
\Phi_0=\sum^{n-1}_{j,t=1}\mu_{j,\,t}\ol z_jz_t.
\end{equation}
We have \[\int_{H_n}\!\abs{v_\alpha(x)}^2e^{-2\Phi_0(z)}d\lambda(x)<\infty.\]
Let us denote by $L^2(H_n, \Phi_0)$ the completion of $C^\infty_0(H_n)$
with respect to the norm $\|\cdot\|_{\Phi_0}$, where
\[
\|u\|^2_{\Phi_0}=\int_{H_n}\abs{u}^2e^{-2\Phi_0}d\lambda(x)\,,\quad u\in C^\infty_0(H_n)\,.
\]
Choose $\chi(x_{2n-1})\in C^\infty_0(\Real)$ so that $\chi(x_{2n-1})=1$ when $\abs{x_{2n-1}}<1$ and $\chi(x_{2n-1})=0$ when $\abs{x_{2n-1}}>2$ and set $\chi_j(x_{2n-1})=\chi(x_{2n-1}/j)$, $j\in\mathbb{N}$. Let
\begin{equation} \label{e-gue131029III}
\hat v_{\alpha,j}(z,\eta)=\int_{\Real}v(x)\chi_j(x_{2n-1})e^{-ix_{2n-1}\eta}dx_{2n-1}\in C^\infty(H_n),\ \  j=1,2,\ldots.
\end{equation}
From Parseval's formula, we have
\begin{align*}
&\int_{H_n}\abs{\hat v_{\alpha,j}(z,\eta)-\hat v_{\alpha,t}(z,\eta)}^2e^{-2\Phi_0(z)}d\eta dv(z)\\
&=2\pi\int_{H_n}\abs{v(x)}^2\abs{\chi_j(x_{2n-1})-\chi_t(x_{2n-1})}^2e^{-2\Phi_0(z)}d\eta dv(z)\To0,\  j,t\To\infty,
\end{align*}
where $dv(z)=2^{n-1}dx_1\cdots dx_{2n-2}$. 
Thus, there is $\hat v_\alpha(z, \eta)\in L^2(H_n, \Phi_0)$ such that $\hat v_{\alpha,j}(z,\eta)\To\hat v_\alpha(z, \eta)$ in $L^2(H_n, \Phi_0)$. We call $\hat v(z, \eta)$ the Fourier transform of $v_\alpha(x)$ with respect to $x_{2n-1}$. Formally,
\begin{equation} \label{e-gue131029IV}
\begin{split}
\hat v_\alpha(z, \eta)&=\int_{\Real}e^{-ix_{2n-1}\eta}v_\alpha(x)dx_{2n-1}\\
&=\int_{\Real}e^{-ix_{2n-1}\eta}(\pr^\alpha_xu)(x)e^{-\beta x_{2n-1}}dx_{2n-1}.
\end{split}
\end{equation}

Put 
\begin{equation}\label{e-gue131029V}
\Real_{p}:=\set{\eta\in\Real;\, \mbox{$M^\phi_p-2\eta\mathcal{L}_p$ is positive definite}}. 
\end{equation}
We can check that 
\[\Real_{p}:=\set{\eta\in\Real;\, \mbox{the matrix $\left(\mu_{j,t}\right)^{n-1}_{j,t=1}-\eta\left(\tau_j\delta_{j,t}\right)^{n-1}_{j,t=1}$ is positive definite}}.\]
The following theorem is essentially well-known (see section 2 and section 3 in~\cite{HM09}) 

\begin{thm}\label{t-gue131029}
With the notations used before, fix $N_0\geq1$. Let $\set{k_j}^{\infty}_{j=1}$ be a sequence with $0<k_1<k_2<k_3<\cdots$, $\lim_{j\To\infty}k_j=\infty$. Let $f_{k_j}\in H^0_{b,\leq k^{-N_0}}(X,L^k)$ with $\norm{f_{k_j}}_{h^{L^{k_j}}}=1$, $j=1,2,\ldots$. On $D$, put $f_{k_j}=s^k\Td f_{k_j}$, $\Td f_{k_j}\in C^\infty(D)$, $j=1,2,\ldots$. Let $u_{k_j}=k_j^{-\frac{n}{2}}F^*_{k_j}(e^{-2k_jR}\Td f_{k_j})\in C^\infty(F_{k_j}(D_{\log k_j}))$. Identify $u_{k_j}$ with a function on $H_n$ by extending it with zero, for each $j$.
Then there is a subsequence $\set{u_{k_{j_s}}}$ of $\set{u_{k_j}}$ such that $u_{k_{j_s}}$ converges uniformly with all its derivatives on any compact subset of $H_n$ to a smooth function $u\in C^\infty(H_n)$ with $\norm{\pr^\alpha_xu}_{\psi_0}<\infty$, for every $\alpha\in\mathbb N^{2n-1}_0$. Moreover, fix $\alpha\in\mathbb N^{2n-1}_0$ and let $\hat v_\alpha(z,\eta)\in L^2(H_n,\Phi_0)$ be as in \eqref{e-gue131029IV}. 
Then, for almost everywhere $\eta\in\Real$, 
\begin{equation}\label{e-gue131029VI}
\abs{\hat v_\alpha(z,\eta)}\leq f_\alpha(\eta)g_\alpha(z,\eta)1_{\Real_p}(\eta),\ \ \forall z\in C^\infty(\Complex^n),
\end{equation} 
where $f_\alpha(\eta)$ is a positive measurable function with $\int_\Real\abs{f_\alpha(\eta)}d\eta<C<\infty$, $C>0$ is a constant independent of the sequence $\set{f_{k_j}}^\infty_{j=1}$ and the point $p$, $g_\alpha(z,\eta)\in C^\infty(H_n,\ol\Real_+)$, $1_{\Real_p}(\eta)=1$ if $\eta\in\Real_p$, $1_{\Real_p}(\eta)=0$ if $\eta\notin\Real_p$ and $\abs{g_\alpha(0,\eta)}1_{\Real_p}(\eta)\leq C_1$, $C_1>0$ is a constant independent of the sequence $\set{f_{k_j}}^\infty_{j=1}$ and the point $p$. Thus, fix $z\in\Complex^{n-1}$, $\int\abs{\hat v_\alpha(z,\eta)}d\eta<\infty$. Furthermore, we have 
\begin{equation}\label{e-gue131029VII}
(\pr^\alpha_xu)(x)e^{-\beta x_{2n-1}}=\frac{1}{2\pi}\int e^{ix_{2n-1}\eta}\hat v_\alpha(z,\eta)d\eta
=\frac{1}{2\pi}\int e^{ix_{2n-1}\eta}\hat v_\alpha(z,\eta)1_{\Real_p}(\eta)d\eta,\ \ \forall x\in H_n.
\end{equation}
\end{thm}

Theorem~\ref{t-gue131029} will only be used in section~\ref{s-aket}. 

\subsection{Kernel of the spectral function}\label{s-kotsf}

We first introduce some notations. Let $(e_1,\ldots,e_{n-1})$ be a smooth local orthonormal
frame of $T^{*0,1}_xX$ over an open set $D\subset X$. Then
$(e^J:=e_{j_1}\wedge\cdots\wedge e_{j_q})_{1\leqslant
j_1<j_2<\cdots<j_q\leq n-1}$ is an orthonormal frame of
$T^{*0,q}_xX$ over $D$. For $f\in\Omega^{0,q}(D)$, we may write
$f=\sideset{}{'}\sum\limits_{J\in\set{1,\ldots,n-1}^q} f_Je^J$, with $f_J=\langle\,f\,|\,e^J\,\rangle\in
C^\infty(D)$. We call $f_J$ the component of $f$ along $e^J$. Let
$A:\Omega^{0,q}_0(D)\To\Omega^{0,q}(D)$ be a continuous operator with smooth kernel. We write
\begin{equation} \label{e-gue13719}
A(x,y)=\sideset{}{'}\sum_{I,J\in\set{1,\ldots,n-1}^q}e^I(x)A_{I,J}(x,y)e^J(y),
\end{equation}
where $A_{I,J}\in C^\infty(D\times D)$ for all strictly increasing $I,J\in\set{1,\ldots,n-1}^q$,
in the sense that 
\begin{equation} \label{e-gue13719I}
(Au)(x)=\sideset{}{'}\sum_{I,J\in\set{1,\ldots,n-1}^q}e^I(x)\int_DA_{I,J}(x,y)u_J(y)dv_X(y),
\end{equation}
for all $u=\sideset{}{'}\sum\limits_{J\in\set{1,\ldots,n-1}^q}u_Je^J\in\Omega^{0,q}_0(D)$. Let $A^*$ be the formal adjoint of $A$ with respect to $(\,\cdot\,|\,\cdot\,)$. We can check that
\begin{equation} \label{e-gue13719IIb}
\begin{split}
&A^*(x,y)=\sideset{}{'}\sum_{I,J\in\set{1,\ldots,n-1}^q}e^I(x)A^*_{I,J}(x,y)e^J(y),\\ 
&\mbox{$A^*_{I,J}(x,y)=\ol{A_{J,I}(y,x)}$ for all strictly increasing $I,J\in\set{1,\ldots,n-1}^q$}.
\end{split}
\end{equation} 
Let
\begin{equation*}
\begin{split}
&B:\Omega^{0,q}(D)\To\Omega^{0,q}(D),\ \ \Omega^{0,q}_0(D)\To\Omega^{0,q}_0(D),\\
&B(x,y)=\sideset{}{'}\sum_{I,J\in\set{1,\ldots,n-1}^q}e^I(x)B_{I,J}(x,y)e^J(y),
\end{split}
\end{equation*}
where $B_{I,J}(x,y)\in C^\infty(D\times D)$ for all strictly increasing $I,J\in\set{1,\ldots,n-1}^q$, be a properly supported smoothing operator. We write 
\[(B\circ A)(x,y)=\sideset{}{'}\sum\limits_{I,J\in\set{1,\ldots,n-1}^q}e^I(x)(B\circ A)_{I,J}(x,y)e^J(y)\] 
in the sense of \eqref{e-gue13719I}. It is not
difficult to see that
\begin{equation} \label{e-gue13719II}
(B\circ
A)_{I,J}(x,y)=\sideset{}{'}\sum\limits_{K\in\set{1,\ldots,n-1}^q}\int_DB_{I,K}(x,u)A_{K,J}(u,y)dv_X(u),
\end{equation}
for all strictly increasing $I,J\in\set{1,\ldots,n-1}^q$. 

Now, we return to our situation. Fix $q\in\set{1,2,\ldots,n-1}$. As before, let $s$ be a local trivializing section of $L$ on an open subset $D\subset X$ and $\abs{s}^2_{h^L}=e^{-2\phi}$. Until further notice, we assume that $Y(q)$ holds at each point of $D$. Since $Y(q)$ holds at each point of $D$, by Kohn's $L^2$ estimates (see~\cite{CS01}), we have 
\[\Pi^{(q)}_{k,\leq\lambda}(x,y)\in C^\infty\bigr(D\times D,(T^{*0,q}_yX\otimes L^k_y)\boxtimes(T^{*0,q}_xX\otimes L^k_x)\bigr),\]
for every $\lambda\geq0$. Fix $\lambda\geq0$. On $D\times D$, we write
\[\Pi^{(q)}_{k,\leq\lambda}(x,y)=s(x)^k\Pi^{(q)}_{k,\leq\lambda,s}(x, y)s^*(y)^k,\]
where $\Pi^{(q)}_{k,k^{-N_0},s}(x, y)$ is smooth on $D\times D$,
so that for $x\in D$, $u\in\Omega^{0,q}_0(D,L^k)$,
\begin{equation} \label{e-gue13719III}
\begin{split}
(\Pi^{(q)}_{k,\leq\lambda}u)(x)&=s(x)^k\int_X\Pi^{(q)}_{k,\leq\lambda,s}(x, y)\langle\,u(y)\,,\,s^*(y)^k\,\rangle dv_X(y)\\
&=s(x)^k\int_X\Pi^{(q)}_{k,\leq\lambda,s}(x, y)\Td u(y)dv_X(y),\ \ u=s^k\Td u,\ \ \Td u\in\Omega^{0,q}_0(D).
\end{split}
\end{equation}
For $x=y$, we can check that 
$\Pi^{(q)}_{k,\leq\lambda,s}(x,x)\in C^\infty(D,T^{*0,q}X\boxtimes T^{*0,q}X)$
is independent of the choice of local frame $s$. 

As \eqref{e-gue130820}, we define the localized spectral projection (with respect to the trivializing section $s$) by
\begin{align} \label{e-gue13719IV}
\hat\Pi^{(q)}_{k,\leq\lambda,s}:L^2_{(0,q)}(D)\bigcap\mathscr E'(D,T^{*0,q}X)&\To\Omega^{0,q}(D),\nonumber \\
u&\To e^{-k\phi}s^{-k}\Pi^{(q)}_{k,\leq\lambda}(s^ke^{k\phi}u).
\end{align}
That is, if $\Pi^{(q)}_{k,\leq\lambda}(s^ke^{k\phi}u)=s^kv$ on $D$, then
$\hat\Pi^{(q)}_{k,\leq\lambda,s}u=e^{-k\phi}v$. We notice that
\begin{equation} \label{e-gue13719V}
\hat\Pi^{(q)}_{k,\leq\lambda,s}(x,y)=e^{-k\phi(x)}\Pi^{(q)}_{k,\leq\lambda,s}(x,y)e^{k\phi(y)},
\end{equation}
where $\hat\Pi^{(q)}_{k,\leq\lambda,s}(x,y)$ is the kernel of
$\hat\Pi^{(q)}_{k,\leq\lambda,s}$ with respect to $(\,\cdot\,|\,\cdot\,)$ and $\Pi^{(q)}_{k,\leq\lambda,s}(x,y)$ is as in \eqref{e-gue13719III}. When $\lambda=0$, we call $\hat\Pi^{(q)}_{k,\leq0,s}$ the localized Szeg\"{o} projection and we set
\begin{equation}\label{e-gue130820I}
\hat\Pi^{(q)}_{k,s}:=\hat\Pi^{(q)}_{k,\leq0,s}.
\end{equation}

We write
\begin{equation} \label{e-gue13719VI}
\hat\Pi^{(q)}_{k,\leq\lambda,s}(x,y)=\sideset{}{'}\sum\limits_{I,J\in\set{1,\ldots,n-1}^q}e^I(x)\hat\Pi^{(q)}_{k,\leq\lambda,s,I,J}(x,y)e^J(y)
\end{equation}
in the sense of \eqref{e-gue13719I}, where $\hat\Pi^{(q)}_{k,\leq\lambda,s,I,J}\in C^\infty(D\times D)$, for all strictly increasing $I, J\in\set{1,\ldots,n-1}^q$. Since $\Pi^{(q)}_{k,\leq\lambda}$ is self-adjoint, we have
\begin{equation} \label{e-gue13719VII}
\hat\Pi^{(q)}_{k,\leq\lambda,s,I,J}(x,y)=\ol{\hat\Pi^{(q)}_{k,\leq\lambda,s,J,I}(y,x)},
\end{equation}
for all strictly increasing $I, J\in\set{1,\ldots,n-1}^q$. 

Now, we fix $N_0\geq1$.  Let $\{f_j\}_{j=1}^{d_k}\subset L^2_{(0,q)}(X,L^k)$ be
an orthonormal frame for $H^q_{b,\leq k^{-N_0}}(X,L^k)$, $d_k\in\N_0\bigcup\set{\infty}$. Note that $f_j|_D\in\Omega^{0,q}(D)$, $j=1,\ldots,d_k$. For each $j$,
we write 
\[f_j|_D=\sideset{}{'}\sum\limits_{J\in\set{1,\ldots,n-1}^q}f_{j,J}(x)e^J(x),\ \ \mbox{$f_{j,J}\in C^\infty(D,L^k)$ for all strictly increasing $J\in\set{1,\ldots,n-1}^q$}.\]   
For $j=1,\ldots,d_k$ and strictly increasing $J\in\set{1,\ldots,n-1}^q$ we define $\Td f_{j,J}\in C^\infty(D)$ and $\Td f_j\in\Omega^{0,q}(D)$ by
\begin{equation} \label{e-gue13719VIII}
f_{j,J}=s^k\Td f_{j,J},\ \ \Td f_j=\sideset{}{'}\sum\limits_{J\in\set{1,\ldots,n-1}^q}\Td f_{j,J}(x)e^J(x).
\end{equation}
Then, $f_j|_D=s^k\Td f_j$, $j=1,\ldots,d_k$, and it is not difficult to see that
\begin{equation} \label{e-gue13719aI}
\hat\Pi^{(q)}_{k,\leq k^{-N_0},s,I,J}(x,y)=\sum^{d_k}_{j=1}\Td f_{j,I}(x)\ol{\Td f_{j,J}(y)}e^{-k(\phi(x)+\phi(y))},
\end{equation}
for all strictly increasing $I, J\in\set{1,\ldots,n-1}^q$. Since $\hat\Pi^{(q)}_{k,\leq\lambda,s,I,J}$ is smooth for every strictly increasing $I, J\in\set{1,\ldots,n-1}^q$, we conclude that for all $\alpha\in\N_0^{2n-1}$, 
\begin{equation} \label{e-gue13719aII}
\mbox{$\sum^{d_k}_{j=1}\abs{(\pr^\alpha_x(\Td f_je^{-k\phi}))(x)}^2$ converges at each point of $x\in D$}.
\end{equation} 
Similarly, if $F:\mathscr E'(D,T^{*0,q}X)\To \mathscr E'(D,T^{*0,q}X)$ is a properly supported continuous operator such that for all $s\in\mathbb N_0$, 
\[F:H^s_{{\rm comp\,}}(D,T^{*0,q}X)\To H^{s+s_0}_{{\rm comp\,}}(D,T^{*0,q}X)\] 
is continuous, for some $s_0\in\Real$. Then, we can check that 
\begin{equation} \label{e-gue13719aIII}
\mbox{$\sum^{d_k}_{j=1}\abs{(F(\Td f_je^{-k\phi}))(x)}^2$ converges at each point of $x\in D$}.
\end{equation} 

\begin{prop}\label{p-gue13717I}
With the notations used above, for every $\alpha\in\mathbb N^{2n-1}_0$, $D'\Subset D$, there is a constant 
$C_{\alpha,D'}>0$ independent of $k$, such that
\begin{equation} \label{eIX}
\sum^{d_k}_{j=1}\abs{(\pr^\alpha_x(\Td f_je^{-k\phi}))(x)}^2\leq C_{\alpha,D'}k^{n+2\abs{\alpha}},\ \ \forall x\in D'.
\end{equation}
\end{prop}

\begin{proof}
Fix $\alpha\in\mathbb N^{2n-1}_0$ and $p\in D'$. We may assume that 
\[\sum^{d_k}_{j=1}\abs{(\pr^\alpha_x(\Td f_je^{-k\phi}))(p)}^2\neq0.\] 
Set
\[u(x)=\frac{1}{\sqrt{\sum^{d_k}_{j=1}\abs{(\pr^\alpha_x(\Td f_je^{-k\phi}))(p)}^2}}
\sum^{d_k}_{j=1}f_j(x)\ol{(\pr^\alpha_x(\Td f_je^{-k\phi}))(p)}.\] 
Since $\sum^{d_k}_{j=1}\abs{(\pr^\alpha_x(\Td f_je^{-k\phi}))(p)}^2$ converges, it is easy to
check that 
\[u\in H^q_{b,\leq k^{-N_0}}(X,L^k),\ \ \norm{u}_{h^{L^k}}=1.\] 
On $D$, we write $u=s^k\Td u$, $\Td u\in\Omega^{0,q}(D)$. We can check that
\begin{equation} \label{e-gue13719aIV}
\Td u=\frac{1}{\sqrt{\sum^{d_k}_{j=1}\abs{(\pr^\alpha_x(\Td f_je^{-k\phi}))(p)}^2}}
\sum^{d_k}_{j=1}\Td f_j(x)\ol{(\pr^\alpha_x(\Td f_je^{-k\phi}))(p)}.
\end{equation}
In view of Theorem~\ref{t-gue13718}, we see that $\abs{(\pr^\alpha_x(\Td ue^{-k\phi}))(p)}\leq C_\alpha k^{\frac{n}{2}+\abs{\alpha}}$, with 
$C_\alpha>0$ independent of $k$ and of the point $p$. From \eqref{e-gue13719aIV}, it is straightforward to see that
\[\abs{(\pr^\alpha_x(\Td ue^{-k\phi}))(p)}=\sqrt{\sum^{d_k}_{j=1}\abs{(\pr^\alpha_x(\Td f_je^{-k\phi}))(p)}^2}\leq C_\alpha k^{\frac{n}{2}+\abs{\alpha}}.\]
The proposition follows.
\end{proof} 



Now, we assume that there exist a $\lambda_0\in\Real$ and $x_0\in D$ such that $M^\phi_{x_0}-2\lambda_0\mathcal{L}_{x_0}$ is non-degenerate of constant signature $(n_-,n_+)$. Let $q=n_-$. We fix $D_0\Subset D$, $D_0$ open. Let $V$ be as in \eqref{e-dhmpXII}. Let 
\[\mbox{$\hat{\mathcal{I}}_k\equiv\frac{k^{2n-1}}{(2\pi)^{2n-1}}\int e^{ik<x-y,\eta>}\alpha(x,\eta,k)d\eta\mod O(k^{-\infty})$ at $T^*D_0\bigcap\Sigma$}\]
be a properly supported classical semi-classical pseudodifferential operator on $D$ of order $0$ from sections of $T^{*0,q}X$ to sections of $T^{*0,q}X$, where $\alpha(x,\eta,k)\in S^0_{{\rm loc\,},{\rm cl\,}}(1;T^*D,T^{*0,q}X\boxtimes T^{*0,q}X)$ with $\alpha(x,\eta,k)=0$ if $\abs{\eta}>M$, for some large $M>0$ and ${\rm Supp\,}\alpha(x,\eta,k)\bigcap T^*D_0\Subset V$.  
Let $\mathcal{S}_k$, $\mathcal{N}_k$ be as in Theorem~\ref{t-gue13630}. Let $\Box^{(q)}_{s,k}$ be as in \eqref{e-msmilkVI}. Then,
\[\begin{split}
&\Box^{(q)}_{s,k}\mathcal{N}_k+\mathcal{S}_k=\hat{\mathcal{I}}_k+H_k\ \ \mbox{on $\mathscr D'(D_0,T^{*0,q}X)$},\\
&\mathcal{N}^*_k\Box^{(q)}_{s,k}+\mathcal{S}^*_k=\hat{\mathcal{I}}^*_k+H^*_k\ \ \mbox{on
$\mathscr D'(D_0,T^{*0,q}X)$},
\end{split}\]
where $H_k\equiv0\mod O(k^{-\infty})$ on $D_0$, $\mathcal{N}^*_k$, $\mathcal{S}^*_k$, $\hat{\mathcal{I}}^*_k$ and $H^*_k$ are formal adjoints of $\mathcal{N}_k$, $\mathcal{S}_k$, $\hat{\mathcal{I}}_k$ and $H_k$ with respect to $(\,\cdot\,|\,\cdot\,)$ respectively. 
Now,
\begin{equation} \label{e-gue13719aV}
\hat{\mathcal{I}}^*_k\hat\Pi^{(q)}_{k,\leq k^{-N_0},s}=(\mathcal{N}^*_k\Box^{(q)}_{s,k}-H^*_k+\mathcal{S}^*_k)\hat\Pi^{(q)}_{k,\leq k^{-N_0},s}=R+\mathcal{S}^*_k\hat\Pi^{(q)}_{k,\leq k^{-N_0},s}\ \ \mbox{on $\mathscr E'(D_0,T^{*0,q}X)$},
\end{equation}
where we denote
\[R=(\mathcal{N}^*_k\Box^{(q)}_{s,k}-H^*_k)\hat\Pi^{(q)}_{k,\leq k^{-N_0},s}.\]
We write 
\[R(x,y)=\sideset{}{'}\sum\limits_{I,J\in\set{1,\ldots,n-1}^q}e^I(x)R_{I,J}(x,y)e^J(y)\]
in the sense of \eqref{e-gue13719I}, where $R_{I,J}\in C^\infty(D_0\times D_0)$ for all strictly increasing $I,J\in\set{1,\ldots,n-1}^q$. From \eqref{e-gue13719aI}, it is straightforward to see that
\begin{equation} \label{e-gue13721}
\begin{split}
&R_{I,J}(x,y)=\sum^{d_k}_{j=1}\Td g_{j,I}(x)\ol{\Td f_{j,J}(y)}e^{-k\phi(y)},\\
&\Td g_j=(\mathcal{N}^*_k\Box^{(q)}_{s,k}-H^*_k)(\Td f_je^{-k\phi})(x),\ \ \Td g_j(x)=\sideset{}{'}\sum_{I\in\set{1,\ldots,n-1}^q}\Td g_{j,I}(x)e^I(x),\ \ j=1,\ldots,d_k,
\end{split}
\end{equation}
for all strictly increasing $I, J\in\set{1,\ldots,n-1}^q$. From \eqref{e-gue13719aIII}, we see that for all $\alpha\in\mathbb N^{2n-1}_0$, 
\begin{equation*}
\mbox{$\sum^{d_k}_{j=1}\abs{(\pr^\alpha_x\Td g_j)(x)}^2$ converges at each point of $x\in D_0$}.
\end{equation*} 
To estimate $R_{I,J}(x,y)$, we first need

\begin{lem} \label{l-gue13721}
With the notations and assumptions used above, for every $D'\Subset D_0$, $\alpha\in\mathbb N^{2n-1}_0$, there is a constant $C_{\alpha,D'}>0$ independent of $k$, such that for all $u\in H^q_{b,\leq k^{-N_0}}(X,L^k)$, $\norm{u}_{h^{L^k}}=1$, $u|_{D_0}=s^k\Td u$, $\Td u\in\Omega^{0,q}(D_0)$, if we set $\Td v(x)=(\mathcal{N}^*_k\Box^{(q)}_{s,k}-H^*_k)(\Td ue^{-k\phi})$, then
\[\abs{(\pr^\alpha_x\Td v)(x)}\leq C_{\alpha,D'}k^{\frac{5n}{2}+2\abs{\alpha}-N_0-2},\ \ \forall x\in D'.\]
\end{lem} 

\begin{proof}
Let $u\in H^q_{b,\leq k^{-N_0}}(X,L^k)$, $\norm{u}_{h^{L^k}}=1$, $u|_{D_0}=s^k\Td u$, $\Td u\in\Omega^{0,q}(D)$. Set $\Td v(x)=\mathcal{N}^*_k\Box^{(q)}_{s,k}(\Td ue^{-k\phi})$. We recall that (see \eqref{e-gue13630Ia})
\begin{equation} \label{e-gue13721I}
\mathcal{N}^*_k=O(k^s):H^s_{{\rm comp\,}}(D_0,T^{*0,q}X)\To H^{s+1}_{{\rm loc\,}}(D_0,T^{*0,q}X),\ \ \forall s\in\N_0.
\end{equation}
Let $D'\Subset D''\Subset D_0$. By using Fourier transforms, we see that for all $x\in D'$, we have
\[\abs{(\pr^\alpha_x\Td v)(x)}\leq C_\alpha\norm{\Td v}_{n+\abs{\alpha},D''},\]
where $C_\alpha$ only depends on the dimension and the length of $\alpha$. Here $\norm{.}_{s,D''}$ denotes the usual Sobolev norm of order $s$ on $D''$.
From this observation and \eqref{e-gue13721I}, we see that for every $N>0$, 
\begin{equation} \label{e-gue13721II}
\abs{(\pr^\alpha_x\Td v)(x)}\leq C_\alpha\norm{\Td v}_{n+\abs{\alpha},D''}\leq C'_\alpha k^{n-1+\abs{\alpha}}\norm{\Box^{(q)}_{s,k}(\Td ue^{-k\phi})}_{n-1+\abs{\alpha},D''}+C_Nk^{-N},
\end{equation}
where $C'_\alpha>0$, $C_N>0$ are independent of $k$ and $\Td u$. Let $\Box^{(q)}_{b,k}u=f$, $f|_{D_0}=s^k\Td f$, $\Td f\in\Omega^{0,q}(D_0)$. We can check that $f\in H^q_{b,\leq k^{-N_0}}(X,L^k)$ and $\norm{f}_{h^{L^k}}\leq k^{-N_0}$. From \eqref{s2-emsmilkI}, we see that
\begin{equation} \label{e-gue13721III}
\Box^{(q)}_{s,k}(e^{-k\phi}\Td u)=e^{-k\phi}\Td f.
\end{equation}
In view of Theorem~\ref{t-gue13718}, we know that for all $\beta\in\mathbb N^{2n-1}_0$,
\[\abs{\pr^\beta_x(\Box^{(q)}_{s,k}(e^{-k\phi}\Td u))}=\abs{\pr^\beta_x(e^{-k\phi}\Td f)}\leq C_\beta k^{\frac{n}{2}+\abs{\beta}}\norm{f}_{h^{L^k}}\leq C_\beta k^{\frac{n}{2}+\abs{\beta}-N_0}\ \ \mbox{on $D''$},\]
where $C_\beta>0$ is independent of $k$. Thus,
\begin{equation} \label{e-gue13721IV}
\norm{\Box^{(q)}_{s,k}(e^{-k\phi}\Td u)}_{n-1+\abs{\alpha},D''}\leq\Td C_\alpha k^{\frac{3n}{2}+\abs{\alpha}-N_0-1},
\end{equation}
where $\Td C_\alpha>0$ is independent of $k$. Combining \eqref{e-gue13721IV} with \eqref{e-gue13721II}, the lemma follows.
\end{proof} 

\begin{lem} \label{l-gue13721I}
Let $\Td g_j(x)\in\Omega^{0,q}(D_0)$, $j=1,\ldots,d_k$, be as in \eqref{e-gue13721}. For every $D'\Subset D_0$, $\alpha\in\mathbb N^{2n-1}_0$,
there is a constant $C_\alpha>0$ independent of $k$, such that for all $x\in D'$
\[\sum^{d_k}_{j=1}\abs{(\pr^\alpha_x\Td g_j)(x)}^2\leq C_\alpha k^{5n+4\abs{\alpha}-2N_0-4}\,.\]
\end{lem}

\begin{proof}
Fix $\alpha\in\mathbb N^{2n-1}_0$ and $p\in D'$. We may assume that $\sum^{d_k}_{j=1}\abs{(\pr^\alpha_x\Td g_j)(p)}^2\neq0$. Set
\[h(x)=\frac{1}{\sqrt{\sum^{d_k}_{j=1}\abs{(\pr^\alpha_x\Td g_j)(p)}^2}}
\sum^{d_k}_{j=1}f_j(x)\ol{(\pr^\alpha_x\Td g_j)(p)}.\] 
Since $\sum^{d_k}_{j=1}\abs{(\pr^\alpha_x\Td g_j)(p)}^2$ converges, we
can check that $h\in H^q_{b,\leq k^{-N_0}}(X,L^k)$, $\norm{h}_{h^{L^k}}=1$. On $D_0$, we write $h=s^k\Td h$. We can check that
\[\mathcal{N}^*_k\Box^{(q)}_{s,k}(\Td he^{-k\phi})=\frac{1}{\sqrt{\sum^{d_k}_{j=1}\abs{(\pr^\alpha_x\Td g_j)(p)}^2}}
\sum^{d_k}_{j=1}\Td g_j(x)\ol{(\pr^\alpha_x\Td g_j)(p)}.\]
In view of Lemma~\ref{l-gue13721}, we see that
\[\abs{\pr^\alpha_x(\mathcal{N}^*_k\Box^{(q)}_{s,k}(\Td he^{-k\phi}))(p)}=\sqrt{\sum^{d_k}_{j=1}\abs{(\pr^\alpha_x\Td g_j)(p)}^2}\leq C_\alpha k^{\frac{5n}{2}+2\abs{\alpha}-N_0-2},\]
where $C_\alpha>0$ is independent of $k$ and the point $p$.
The lemma follows.
\end{proof}

Now, we can prove 

\begin{prop} \label{p-gue13721}
With the notations and assumptions used above, for every $D'\Subset D_0$, $\alpha, \beta\in\mathbb N^{2n-1}_0$, there is a constant
$C_{\alpha,\beta}>0$ independent of $k$, such that
\begin{equation} \label{e-gue13721V}
\abs{(\pr^\alpha_x\pr^\beta_yR_{I,J})(x,y)}\leq C_{\alpha,\beta}k^{3n+2\abs{\alpha}+\abs{\beta}-N_0-2},\ \ \forall (x,y)\in D'\times D',
\end{equation}
for all strictly increasing $I, J\in\set{1,\ldots,n-1}^q$, where $R_{I,J}(x,y)$ is as in \eqref{e-gue13721}.
\end{prop} 

\begin{proof}
Fix $p\in D'$ and $J\in\set{1,\ldots,n-1}^q$ strictly increasing. Let $\alpha, \beta\in\mathbb N^{2n-1}_0$. We may assume that
$\sum^{d_k}_{j=1}\abs{(\pr^\beta_y(\Td f_{j,J}e^{-k\phi}))(p)}^2\neq0$. Put
\begin{equation} \label{e-gue13721VI}
u(x)=\frac{1}{\sqrt{\sum^{d_k}_{j=1}\abs{(\pr^\beta_y(\Td f_{j,J}e^{-k\phi}))(p)}^2}}
\sum^{d_k}_{j=1}f_j(x)\ol{(\pr^\beta_y(\Td f_{j,J}e^{-k\phi}))(p)}.
\end{equation}
Then, $u\in H^q_{b,\leq k^{-N_0}}(X,L^k)$, $\norm{u}_{h^{L^k}}=1$. On $D_0$, we write $u=s^k\Td u$, $\Td u=\sideset{}{'}\sum\limits_{I\in\set{1,\ldots,n-1}^q}\Td u_Ie^I$.
Put $\Td v=\mathcal{N}^*_k\Box^{(q)}_s(\Td ue^{-k\phi})=\sideset{}{'}\sum\limits_{I\in\set{1,\ldots,n-1}^q}\Td v_Ie^I\in\Omega^{0,q}(D)$. It is not difficult to
check that
\[\Td v=\frac{1}{\sqrt{\sum^{d_k}_{j=1}\abs{(\pr^\beta_y(\Td f_{j,J}e^{-k\phi}))(p)}^2}}\sum^{d_k}_{j=1}\Td g_j\ol{(\pr^\beta_y(\Td f_{j,J}e^{-k\phi}))(p)},\]
where $\{\Td g_j\}_{j=1}^{d_k}$ are as in \eqref{e-gue13721}. In view of Lemma~\ref{l-gue13721}, there exists $C_\alpha>0$ independent of $k$ and the point $p$ such that
$\abs{(\pr^\alpha_x\Td v)(x)}\leq C_\alpha k^{\frac{5n}{2}+2\abs{\alpha}-N_0-2}$, for all $x\in D'$. In particular,
\begin{equation} \label{e-gue13721VII}
\begin{split}
\abs{(\pr^\alpha_x\Td v_I)(x)}&=\frac{1}{\sqrt{\sum^{d_k}_{j=1}\abs{(\pr^\beta_y(\Td f_{j,J}e^{-k\phi}))(p)}^2}}\abs{\sum^{d_k}_{j=1}(\pr^\alpha_x\Td g_{j,I})(x)\ol{(\pr^\beta_y(\Td f_{j,J}e^{-k\phi}))(p)}}\\
&\leq C_\alpha k^{\frac{5n}{2}+2\abs{\alpha}-N_0-2},\ \ \forall x\in D',
\end{split}
\end{equation}
for all strictly increasing $I\in\set{1,\ldots,n-1}^q$. In view of Proposition~\ref{p-gue13717I}, we see that
\[\sum^{d_k}_{j=1}\abs{(\pr^\beta_y(\Td f_je^{-k\phi}))(p)}^2\leq C_\beta k^{n+2\abs{\beta}},\]
where $C_\beta>0$ is independent of $k$ and the point $p$. From this and \eqref{e-gue13721VII}, we conclude the existence of a constant $C_{\alpha,\beta}>0$ independent of $k$ and the point $p$ with
\[\begin{split}
\abs{(\pr^\alpha_x\pr^\beta_yR_{I,J})(x,p)}&=
\sqrt{\sum^{d_k}_{j=1}\abs{(\pr^\beta_y(\Td f_{j,J}e^{-k\phi}))(p)}^2}\abs{(\pr^\alpha_x\Td v_I)(x)}\\
&\leq C_{\alpha,\beta}k^{3n+2\abs{\alpha}+\abs{\beta}-N_0-2},\end{split}\]
for all $x\in D'$, all strictly increasing $I,J\in\set{1,\ldots,n-1}^q$.
The proposition follows.
\end{proof}

Let $R^*$ be the formal adjoint $R$ with respect to $(\,\cdot\,|\,\cdot\,)$. From \eqref{e-gue13719aV}, we have 
\begin{equation}\label{e-gue13721VIII}
\hat\Pi^{(q)}_{k,\leq k^{-N_0},s}\hat{\mathcal{I}}_k=R^*+\hat\Pi^{(q)}_{k,\leq k^{-N_0},s}\mathcal{S}_k.
\end{equation}
We also write
\begin{equation}\label{e-gue13721aI}
R^*(x,y)=\sideset{}{'}\sum\limits_{I,J\in\set{1,\ldots,n-1}^q}e^I(x)R^*_{I,J}(x,y)e^J(y)
\end{equation}
in the sense of \eqref{e-gue13719I}, where $R^*_{I,J}(x,y)\in C^\infty(D_0\times D_0)$, for all strictly increasing $I,J\in\set{1,\ldots,n-1}^q$. Note that $R^*_{I,J}(x,y)=\ol{R_{J,I}(y,x)}$, for all strictly increasing $I,J\in\set{1,\ldots,n-1}^q$. 
Combining this observation with \eqref{e-gue13721V}, we conclude that for every $D'\Subset D_0$, $\alpha, \beta\in\mathbb N^{2n-1}_0$, there is a constant
$C_{\alpha,\beta}>0$ independent of $k$, such that
\begin{equation} \label{e-gue13721aII}
\abs{(\pr^\alpha_x\pr^\beta_yR^*_{I,J})(x,y)}\leq C_{\alpha,\beta}k^{3n+2\abs{\beta}+\abs{\alpha}-N_0-2},\ \ \forall (x,y)\in D'\times D',
\end{equation}
for all strictly increasing $I, J\in\set{1,\ldots,n-1}^q$. 

We consider $\hat{\mathcal{I}}^*_kR^*$. Note that $\hat{\mathcal{I}}^*_kR^*$ is a smoothing function on $D$ and we write
\[(\hat{\mathcal{I}}^*_kR^*)(x,y)=\sideset{}{'}\sum\limits_{I,J\in\set{1,\ldots,n-1}^q}e^I(x)(\hat{\mathcal{I}}^*_kR^*)_{I,J}(x,y)e^J(y)\]
in the sense of \eqref{e-gue13719I}, where $(\hat{\mathcal{I}}^*_kR^*)_{I,J}\in C^\infty(D\times D)$ for all strictly increasing $I,J\in\set{1,\ldots,n-1}^q$. It is easy to see that 
\begin{equation} \label{e-gue13721aIII}
\hat{\mathcal{I}}^*_{k}=O(k^0): H^s_{{\rm comp\,}}(D,T^{*0,q}X)\To H^{s}_{{\rm comp\,}}(D,T^{*0,q}X),
\end{equation}
for every $s\in\mathbb N_0$. From \eqref{e-gue13721aIII}, we can repeat the proof of Proposition~\ref{p-gue13721} with minor change and conclude that 
\begin{prop} \label{p-gue13721I}
With the notations and assumptions used above, for every $D'\Subset D_0$, $\alpha, \beta\in\mathbb N^{2n-1}_0$, there is a constant
$C_{\alpha,\beta}>0$ independent of $k$, such that
\begin{equation} \label{e-gue13721aIII-I}
\abs{\pr^\alpha_x\pr^\beta_y(\hat{\mathcal{I}}^*_kR^*)_{I,J}(x,y)}\leq C_{\alpha,\beta}k^{3n+2\abs{\alpha}+\abs{\beta}-N_0-2},\ \ \forall (x,y)\in D'\times D',
\end{equation}
for all strictly increasing $I, J\in\set{1,\ldots,n-1}^q$.
\end{prop} 

\subsection{Szeg\"{o} kernel asymptotics for lower energy forms}\label{s-safllI} 

Let $\lambda\geq0$ and let $H^q_{b,>\lambda}(X,L^k)$ and $\Pi^{(q)}_{b,>\lambda}$ be as in \eqref{e-suX} and \eqref{e-suXI-I} respectively.  
It is well-known (see section 2 in~\cite{Dav95}) that for all $\lambda>0$, 
\begin{equation} \label{e-suXII}
L^2_{(0,q)}(X,L^k)=H^q_{b,\leq\lambda}(X, L^k)\oplus H^q_{b,>\lambda}(X, L^k)
\end{equation} 
and 
\begin{equation} \label{e-suXIII}
\norm{u}^2_{h^{L^k}}\leq\frac{1}{\lambda}(\,\Box^{(q)}_{b,k}u\,|\,u\,)_{h^{L^k}},\ \ \forall u\in H^q_{b,>\lambda}(X,L^k)\bigcap{\rm Dom\,}\Box^{(q)}_{b,k}.
\end{equation}

Let $s$ be a local trivializing section of $L$ on an open subset $D\subset X$ and $\abs{s}^2_{h^L}=e^{-2\phi}$. 
Consider the localization
\begin{equation} \label{e-gue13723}
\begin{split}
\hat\Pi^{(q)}_{k,>\lambda,s}:L^2_{(0,q)}(D)\cap\mathscr E'(D,T^{*0,q}X)&\To L^2_{(0,q)}(D),\\
u&\To  e^{-k\phi}s^{-k}\Pi^{(q)}_{k,>\lambda}(s^ke^{k\phi}u).
\end{split}
\end{equation}
From \eqref{e-suXII}, we have the decomposition
\begin{equation} \label{e-gue13723I}
u=\hat\Pi^{(q)}_{k,\leq\lambda,s}u+\hat\Pi^{(q)}_{k,>\lambda,s}u,\ \ u\in\Omega^{0,q}_0(D).
\end{equation}

We work with the same notations and assumptions as in section~\ref{s-kotsf}. 
Let $u\in H^{s_1}_{{\rm comp\,}}(D,T^{*0,q}X)$, $s_1\leq0$, $s_1\in\mathbb Z$.
From \eqref{e-gue13723I}, we have
\begin{equation} \label{e-gue13723II}
\mathcal{S}_{k}u=\hat\Pi^{(q)}_{k,\leq k^{-N_0},s}\,\mathcal{S}_{k}u+\hat\Pi^{(q)}_{k,>k^{-N_0},s}\mathcal{S}_{k}u.
\end{equation}
From \eqref{e-gue13723} and \eqref{e-suXIII}, we can check that
\begin{equation} \label{e-gue13723III}
\begin{split}
\norm{\hat\Pi^{(q)}_{k,>k^{-N_0},s}\mathcal{S}_{k}u}_{D}&\leq\norm{\Pi^{(q)}_{k,>k^{-N_0}}(s^ke^{k\phi}(\mathcal{S}_{k}u))}_{h^{L^k}}\leq k^{N_0}\norm{\Box^{(q)}_{b,k}\Pi^{(q)}_{k,>k^{-N_0}}(s^ke^{k\phi}(\mathcal{S}_{k}u))}_{h^{L^k}}\\
&\leq k^{N_0}\norm{\Box^{(q)}_{b,k}(s^ke^{k\phi}(\mathcal{S}_{k}u))}_{h^{L^k}}=k^{N_0}\norm{\Box^{(q)}_{s,k}(\mathcal{S}_{k}u)}.
\end{split}
\end{equation}
Here we have used \eqref{s2-emsmilkI}. In view of Theorem~\ref{t-gue13630}, we see that $\Box^{(q)}_{s,k}\mathcal{S}_{k}\equiv0\mod O(k^{-\infty})$.
From this observation and \eqref{e-gue13723III}, we conclude that
\begin{equation} \label{e-gue13723IV}
\hat\Pi^{(q)}_{k,>k^{-N_0},s}\,\mathcal{S}_{k}=O(k^{-N}):H^{s_1}_{{\rm comp\,}}(D,T^{*0,q}X)\To H^0_{{\rm loc\,}}(D,T^{*0,q}X),
\end{equation}
locally uniformly on $D$, for all $N\geq0$, $s_1\in\mathbb Z$, $s_1\leq0$. Since $Y(q)$ holds at each point of $D$, we can repeat Kohn's $L^2$ estimates (see~\cite{CS01}) and obtain

\begin{prop}\label{p-gue13723}
Let $u\in{\rm Dom\,}\Box^{(q)}_{b,k}$. If $(\Box^{(q)}_{b,k})^ju\in{\rm Dom\,}\Box^{(q)}_{b,k}$, for all $j=1,2,\ldots$, then $u|_D\in\Omega^{0,q}(D,L^k)$. 

Moreover, for every $m\in\mathbb N_0$ and $D'\Subset D''\Subset D_0$, there are constants $C_m>0$ and $n_m\in\mathbb N$ independent of $k$ such that 
\begin{equation}\label{e-gue13723V}
\norm{u}_{m,D'}\leq C_mk^{n_m}\Bigr(\norm{u}_{D''}+\sum^m_{j=1}\norm{(\Box^{(q)}_{s,k})^ju}_{D''}\Bigr),\ \ \forall u\in\Omega^{0,q}(D_0), 
\end{equation}
where $\norm{\cdot}_{s,D'}$ denote the usual Sobolev norm of order $s$ on $D'$ with respect to $dv_X(x)$ and $\norm{\cdot}_{D''}$ denote the $L^2$ norm on $D''$ with respect to $dv_X(x)$. 
\end{prop} 

Let $u\in H^{s_1}_{{\rm comp\,}}(D,T^{*0,q}X)$, $s_1\leq0$, $s_1\in\mathbb Z$. Since $s^ke^{k\phi}\mathcal{S}_ku\in\Omega^{0,q}_0(D,L^k)$, we have 
\[(\Box^{(q)}_{b,k})^j(s^ke^{k\phi}\mathcal{S}_ku)\in{\rm Dom\,}\Box^{(q)}_{b,k},\ \ \forall j=1,2,\ldots.\] 
Hence, 
\[\bigr(\Box^{(q)}_{b,k}\bigr)^j\Bigr(\Pi^{(q)}_{k,\leq k^{-N_0}}(s^ke^{k\phi}\mathcal{S}_ku)\Bigr)=\Pi^{(q)}_{k,\leq k^{-N_0}}\Bigr(\bigr(\Box^{(q)}_{b,k}\bigr)^j(s^ke^{k\phi}\mathcal{S}_ku)\Bigr)\in{\rm Dom\,}\Box^{(q)}_{b,k},\ \  \forall j=1,2,\ldots.\] 
Since $I=\Pi^{(q)}_{k,\leq k^{-N_0}}+\Pi^{(q)}_{k,>k^{-N_0}}$ on $L^2_{(0,q)}(X,L^k)$, we conclude that \[\bigr(\Box^{(q)}_{b,k}\bigr)^j\Bigr(\Pi^{(q)}_{k,> k^{-N_0}}(s^ke^{k\phi}\mathcal{S}_ku)\bigr)\in{\rm Dom\,}\Box^{(q)}_{b,k},\ \ \forall j=1,2,\ldots.\]
From this and Proposition~\ref{p-gue13723}, we conclude that 
\begin{equation} \label{e-gue13723VI}
\begin{split}
&\Pi^{(q)}_{k,> k^{-N_0}}(s^ke^{k\phi}\mathcal{S}_ku)|_D\in\Omega^{0,q}(D,L^k),\\
&\hat\Pi^{(q)}_{k,> k^{-N_0},s}(\mathcal{S}_ku)|_D\in\Omega^{0,q}(D).
\end{split}
\end{equation}
Moreover, from \eqref{e-gue13723V}, for every $m\in\mathbb N_0$ and $D'\Subset D''\Subset D$, it is straightforward to see that 
\begin{equation} \label{e-gue13723VII}
\begin{split}
\norm{\hat\Pi^{(q)}_{k,>k^{-N_0},s}\mathcal{S}_{k}u}_{m,D'}&\leq C_mk^{n_m}\Bigr(\norm{\hat\Pi^{(q)}_{k,>k^{-N_0},s}(\mathcal{S}_{k}u))}_{D''}+\sum^m_{j=1}\norm{(\Box^{(q)}_{s,k})^j\bigr(\hat\Pi^{(q)}_{k,>k^{-N_0},s}(\mathcal{S}_{k}u)\bigr)}_{D''}\Bigr)\\
&\leq C_mk^{n_m}\Bigr(\norm{\hat\Pi^{(q)}_{k,>k^{-N_0},s}(\mathcal{S}_{k}u))}_{D''}+\sum^m_{j=1}\norm{\Pi^{(q)}_{k,>k^{-N_0}}(\Box^{(q)}_{b,k})^j(e^{k\phi}s^k\mathcal{S}_{k}u)}_{h^{L^k}}\Bigr)\\
&\leq C_mk^{n_m}\Bigr(\norm{\hat\Pi^{(q)}_{k,>k^{-N_0},s}(\mathcal{S}_{k}u))}_{D''}+\sum^m_{j=1}\norm{\Pi^{(q)}_{k,>k^{-N_0}}\bigr(e^{k\phi}s^k(\Box^{(q)}_{s,k})^j\mathcal{S}_{k}u\bigr)}_{h^{L^k}}\Bigr),
\end{split}
\end{equation}
where $C_m>0$ and $n_m\in\mathbb N$ are constants independent of $k$. Here we use the facts 
\[\begin{split}
(\Box^{(q)}_{s,k})^j\bigr(\hat\Pi^{(q)}_{k,>k^{-N_0},s}(\mathcal{S}_{k}u)\bigr)&=s^{-k}e^{-k\phi}(\Box^{(q)}_{b,k})^j\bigr(\Pi^{(q)}_{k,>k^{-N_0}}(e^{k\phi}s^k\mathcal{S}_ku)\bigr)\\
&=s^{-k}e^{-k\phi}\Pi^{(q)}_{k,>k^{-N_0}}\bigr((\Box^{(q)}_{b,k})^j(e^{k\phi}s^k\mathcal{S}_ku)\bigr)\end{split}\] 
and 
\[(\Box^{(q)}_{b,k})^j(e^{k\phi}s^k\mathcal{S}_ku)=e^{k\phi}s^k(\Box^{(q)}_{s,k})^j(\mathcal{S}_ku),\]
for all $j=1,2,\ldots$. 
From $\Box^{(q)}_{s,k}\mathcal{S}_{k}\equiv0\mod O(k^{-\infty})$, \eqref{e-gue13723IV} and \eqref{e-gue13723VII}, we conclude that 
\[\hat\Pi^{(q)}_{k,>k^{-N_0},s}\mathcal{S}_{k}=O(k^{-N}):H^{s_1}_{{\rm comp\,}}(D,T^{*0,q}X)\To H^m_{{\rm loc\,}}(D,T^{*0,q}X),\]
locally uniformly on $D$, for all $N\geq0$, $s_1\in\mathbb Z$, $s_1\leq0$, and $m\in\mathbb N_0$. Thus, 
\begin{equation}\label{e-gue13723VIII}
\hat\Pi^{(q)}_{k,>k^{-N_0},s}\mathcal{S}_{k}\equiv0\mod O(k^{-\infty}). 
\end{equation}
Note that $\mathcal{S}_k=\hat\Pi^{(q)}_{k,\leq k^{N_0},s}\mathcal{S}_k+\hat\Pi^{(q)}_{k,>k^{N_0},s}\mathcal{S}_k$. From this observation, \eqref{e-gue13721VIII} and \eqref{e-gue13723VIII}, we deduce that 
\begin{equation}\label{e-gue13724}
\hat\Pi^{(q)}_{k,\leq k^{-N_0},s}\hat{\mathcal{I}}_k-R^*\equiv\mathcal{S}_k\mod O(k^{-\infty})
\end{equation} 
on $D_0$, where $R^*(x,y)$ satisfies \eqref{e-gue13721aII}. Hence, 
\begin{equation}\label{e-gue13724ho}
\hat{\mathcal{I}}^*_k\hat\Pi^{(q)}_{k,\leq k^{-N_0},s}\hat{\mathcal{I}}_k-\hat{\mathcal{I}}^*_kR^*\equiv\mathcal{S}^*_k\mathcal{S}_k\mod O(k^{-\infty})
\end{equation} 
on $D_0$, where $(\hat{\mathcal{I}}^*_kR^*)(x,y)$ satisfies \eqref{e-gue13721aIII-I}. Here we used \eqref{e-gue1374}. 

Summing up, we obtain one of the main results of this work 

\begin{thm}\label{t-gue130816}
Let $s$ be a local trivializing section of $L$ on an open subset $D\subset X$ and $\abs{s}^2_{h^L}=e^{-2\phi}$. We assume that there exist a $\lambda_0\in\Real$ and $x_0\in D$ such that $M^\phi_{x_0}-2\lambda_0\mathcal{L}_{x_0}$ is non-degenerate of constant signature $(n_-,n_+)$. Let $q=n_-$ and assume that $Y(q)$ holds at each point of $D$. We fix $D_0\Subset D$, $D_0$ open. Let $V$ be as in \eqref{e-dhmpXII}. Let 
\[\mbox{$\hat{\mathcal{I}}_k\equiv\frac{k^{2n-1}}{(2\pi)^{2n-1}}\int e^{ik<x-y,\eta>}\alpha(x,\eta,k)d\eta\mod O(k^{-\infty})$ at $T^*D_0\bigcap\Sigma$}\]
be a properly supported classical semi-classical pseudodifferential operator on $D$ of order $0$ from sections of $T^{*0,q}X$ to sections of $T^{*0,q}X$, where 
\[\begin{split}&\mbox{$\alpha(x,\eta,k)\sim\sum_{j=0}\alpha_j(x,\eta)k^{-j}$ in $S^0_{{\rm loc\,}}(1;T^*D,T^{*0,q}X\boxtimes T^{*0,q}X)$},\\
&\alpha_j(x,\eta)\in C^\infty(T^*D,T^{*0,q}D\boxtimes T^{*0,q}D),\ \ j=0,1,\ldots,
\end{split}\]
with $\alpha(x,\eta,k)=0$ if $\abs{\eta}>M$, for some large $M>0$ and ${\rm Supp\,}\alpha(x,\eta,k)\bigcap T^*D_0\Subset V$. Then for every $N_0\geq1$ and every $D'\Subset D_0$, $\alpha, \beta\in\mathbb N^{2n-1}_0$, there is a constant $C_{D',\alpha,\beta,N_0}>0$ independent of $k$, such that 
\begin{equation} \label{e-gue130819II}
\begin{split}
&\abs{\pr^\alpha_x\pr^\beta_y\big((\hat\Pi^{(q)}_{k,\leq k^{-N_0},s}\hat{\mathcal{I}}_k)(x,y)-\int e^{ik\varphi(x,y,s)}a(x,y,s,k)ds\big)}\\
&\leq C_{D',\alpha,\beta,N_0}k^{3n+2\abs{\beta}+\abs{\alpha}-N_0-2}\ \ \mbox{on $D'\times D'$},\\
&\abs{\pr^\alpha_x\pr^\beta_y\big((\hat{\mathcal{I}}^*_k\hat\Pi^{(q)}_{k,\leq k^{-N_0},s}\hat{\mathcal{I}}_k)(x,y)-(\int e^{ik\varphi(x,y,s)}g(x,y,s,k)ds\big)}\\
&\leq C_{D',\alpha,\beta,N_0}k^{3n+2\abs{\beta}+\abs{\alpha}-N_0-2}\ \ \mbox{on $D'\times D'$},
\end{split}
\end{equation}
where $\hat\Pi^{(q)}_{k,\leq k^{-N_0},s}$ is the localized spectral projection \eqref{e-gue13719IV}, $\varphi(x,y,s)\in C^\infty(\Omega)$ is as in Theorem~\ref{t-dcgewI}, \eqref{e-guew13627} and
\begin{equation}\label{e-gue130819}
\begin{split}
&a(x,y,s,k),\ g(x,y,s,k)\in S^{n}_{{\rm loc\,}}\big(1;\Omega,T^{*0,q}X\boxtimes T^{*0,q}X\big)\bigcap C^\infty_0\big(\Omega,T^{*0,q}X\boxtimes T^{*0,q}X\big),\\
&a(x,y,s,k)\sim\sum^\infty_{j=0}a_j(x,y,s)k^{n-j}\text{ in }S^{n}_{{\rm loc\,}}
\big(1;\Omega,T^{*0,q}X\boxtimes T^{*0,q}X\big), \\
&g(x,y,s,k)\sim\sum^\infty_{j=0}g_j(x,y,s)k^{n-j}\text{ in }S^{n}_{{\rm loc\,}}
\big(1;\Omega,T^{*0,q}X\boxtimes T^{*0,q}X\big), \\
&a_j(x,y,s),\ g_j(x,y,s)\in C^\infty_0\big(\Omega,T^{*0,q}X\boxtimes T^{*0,q}X\big),\ \ j=0,1,2,\ldots,
\end{split}
\end{equation}
with 
\begin{equation}  \label{e-gue130819I}
\begin{split}
&a_0(x,x,s)=(2\pi)^{-n}\abs{\det\bigr(M^\phi_x-2s\mathcal{L}_x\bigr)}\mathcal{\pi}_{(x,s,n_-)}\alpha_0(x,s\omega_0(x)-2{\rm Im\,}\ddbar_b\phi(x)),\\
&g_0(x,x,s)\\
&=(2\pi)^{-n}\abs{\det\bigr(M^\phi_x-2s\mathcal{L}_x\bigr)}\alpha^*_0(x,s\omega_0(x)-2{\rm Im\,}\ddbar_b\phi(x))\mathcal{\pi}_{(x,s,n_-)}\alpha_0(x,s_0\omega_0(x)-2{\rm Im\,}\ddbar_b\phi(x)),
\end{split}
\end{equation}
for every $(x,x,s)\in\Omega$, $x\in D_0$, where 
\[
\begin{split}
\Omega:=&\{(x,y,s)\in D\times D\times\Real;\, (x,-2{\rm Im\,}\ddbar_b\phi(x)+s\omega_0(x))\in V\bigcap\Sigma,\\
&\quad\mbox{$(y,-2{\rm Im\,}\ddbar_b\phi(y)+s\omega_0(y))\in V\bigcap\Sigma$, $\abs{x-y}<\varepsilon$, for some $\varepsilon>0$}\},
\end{split}\]
$\alpha^*_0(x,\eta):T^{*0,q}_xX\To T^{*0,q}_xX$ is the adjoint of $\alpha_0(x,\eta)$ with respect to the Hermitian metric $\langle\,\cdot\,|\,\cdot\,\rangle$ on $T^{*0,q}_xX$, $\mathcal{\pi}_{(x,s,n_-)}:T^{*0,q}_xX\To\mathcal{N}(x,s,n_-)$ is the orthogonal projection with respect to $\langle\,\cdot\,|\,\cdot\,\rangle$,  $\mathcal{N}(x,s,n_-)$ is given by \eqref{e-gue1373III}, $\abs{\det\bigr(M^\phi_x-2s\mathcal{L}_x\bigr)}=\abs{\lambda_1(s)}\abs{\lambda_2(s)}\cdots\abs{\lambda_{n-1}(s)}$, where
$\lambda_1(s),\ldots,\lambda_{n-1}(s)$ are eigenvalues of the Hermitian quadratic form $M^\phi_x-2s\mathcal{L}_x$ with respect to $\langle\,\cdot\,|\,\cdot\,\rangle$.
\end{thm}

By using Theorem~\ref{t-gue13717} and repeat the proof of Theorem~\ref{t-gue130816}, we deduce

\begin{thm}\label{t-gue130819}
Let $s$ be a local trivializing section of $L$ on an open subset $D\subset X$ and $\abs{s}^2_{h^L}=e^{-2\phi}$. We assume that there exist a $\lambda_0\in\Real$ and $x_0\in D$ such that $M^\phi_{x_0}-2\lambda_0\mathcal{L}_{x_0}$ is non-degenerate of constant signature $(n_-,n_+)$. Let $q\neq n_-$ and assume that $Y(q)$ holds at each point of $D$. We fix $D_0\Subset D$, $D_0$ open. Let $V$ be as in \eqref{e-dhmpXII}. Let 
\[\mbox{$\hat{\mathcal{I}}_k\equiv\frac{k^{2n-1}}{(2\pi)^{2n-1}}\int e^{ik<x-y,\eta>}\alpha(x,\eta,k)d\eta\mod O(k^{-\infty})$ at $T^*D_0\bigcap\Sigma$}\]
be a properly supported classical semi-classical pseudodifferential operator on $D$ of order $0$ from sections of $T^{*0,q}X$ to sections of $T^{*0,q}X$, where 
$\alpha(x,\eta,k)\in S^0_{{\rm loc\,}}(1;T^*D,T^{*0,q}X\boxtimes T^{*0,q}X)$
with $\alpha(x,\eta,k)=0$ if $\abs{\eta}>M$, for some large $M>0$ and ${\rm Supp\,}\alpha(x,\eta,k)\bigcap T^*D_0\Subset V$. Then for every $N_0\geq1$ and every $D'\Subset D_0$, $\alpha, \beta\in\mathbb N^{2n-1}_0$, there is a constant $C_{D',\alpha,\beta,N_0}>0$ independent of $k$, such that
\begin{equation} \label{e-gue130819III}
\begin{split}
&\abs{\pr^\alpha_x\pr^\beta_y\bigr((\hat\Pi^{(q)}_{k,\leq k^{-N_0},s}\hat{\mathcal{I}}_k)(x,y)\bigr)}\leq C_{D',\alpha,\beta,N_0}k^{3n+2\abs{\beta}+\abs{\alpha}-N_0-2}\ \ \mbox{on $D'\times D'$},\\
&\abs{\pr^\alpha_x\pr^\beta_y\bigr((\hat{\mathcal{I}}^*_k\hat\Pi^{(q)}_{k,\leq k^{-N_0},s}\hat{\mathcal{I}}_k\bigr)(x,y)}\leq C_{D',\alpha,\beta,N_0}k^{3n+2\abs{\beta}+\abs{\alpha}-N_0-2}\ \ \mbox{on $D'\times D'$},
\end{split}
\end{equation}
where $\hat\Pi^{(q)}_{k,\leq k^{-N_0},s}$ is the localized spectral projection \eqref{e-gue13719IV}.
\end{thm}  

\section{Almost Kodaira embedding Theorems on CR manifolds} \label{s-aket}

In this section, we will use Theorem~\ref{t-gue130816} to establish "Almost Kodaira embedding Theorems on CR manifolds" (see Definition~\ref{d-gue131109}). First, we recall Definition~\ref{d-gue130918} for the definition of "positive CR line bundles".

In this section, we assume that $X$ is compact, $L$ is positive and condition $Y(0)$ holds at each point of $X$. 
Fix $N_0\gg 2n$, $N_0$ large. Let $s$ be a local trivializing section of $L$ on an open subset $D\subset X$, $\abs{s}^2_{h^L}=e^{-2\phi}$. As \eqref{e-dhmpXIa}, we set 
\begin{equation}\label{e-gue130917}
\begin{split}
\Sigma'=&\{(x,\lambda\omega_0(x)-2{\rm Im\,}\ddbar_b\phi(x))\in T^*D\bigcap\Sigma;\,\\ &\quad\mbox{$M^\phi_x-2\lambda\mathcal{L}_x$ is positive definite}\}.
\end{split}
\end{equation}
Fix $p\in D$. Let $x=(x_1,\ldots,x_{2n-1})$, $z_j=x_{2j-1}+ix_{2j}$, $j=1,\ldots,n-1$, be local coordinates of $X$ defined on $D$ such that \eqref{e-gue13716II} hold. On $D$, we write $\omega_0(x)=\sum^{2n-1}_{j=1}\beta_j(x)dx_j$. We take $D$ small enough so that $\beta_{2n-1}(x)\geq\frac{1}{2}$, for every $x\in D$. Let $\eta=(\eta_1,\ldots,\eta_{2n-1})$ denote the dual coordinates of $x$. Let $\Real_p$ be as in \eqref{e-gue131029V}.
Take $\psi(\eta)\in C^\infty_0(\Real_{p},\ol\Real_+)$ with 
\begin{equation}\label{e-gue130918I}
\int\psi(\eta)\Bigr(M^\phi_p-2\eta\mathcal{L}_p\Bigr)d\eta\geq\frac{1}{2}\int_{\Real_p}\Bigr(M^\phi_p-2\eta\mathcal{L}_p\Bigr)d\eta.
\end{equation}
Let $M>0$ be a large constant so that for every $(x,\eta)\in T^*D$ if $\abs{\eta'}>\frac{M}{2}$ then $(x,\eta)\notin\Sigma$, where $\eta'=(\eta_1,\ldots,\eta_{2n-2})$, $\abs{\eta'}=\sqrt{\sum^{2n-2}_{j=1}\abs{\eta_j}^2}$.
Take $D_0\subset D$ be a small open neighbourhood of $p$ so that 
\begin{equation}\label{e-gue130918II}
\mbox{$M^\phi_x-2\langle\,\eta\,|\,\omega_0(x)\,\rangle\mathcal{L}_x$ is positive definite, for every $x\in W$, every $\eta_{2n-1}\in\Gamma$, $\abs{\eta'}<M$},
\end{equation}
where $W\subset D$ and $\Gamma\subset\Real_p$ are small open neighbourhoods of $D_0$ and ${\rm Supp\,}\psi$ respectively. Put 
\begin{equation}\label{e-gue130918III}
V=\set{(x,\eta)\in T^*D;\, x\in W, \eta_{2n-1}\in\Gamma, \abs{\eta'}<M}.
\end{equation}
From \eqref{e-gue130918II}, it is straightforward to see that 
\begin{equation}\label{e-gue130918IV}
V\bigcap\Sigma\subset\Sigma'.
\end{equation} 

Take $\tau$, $\tau_1\in C^\infty_0(D)$, $\tau=1$ on $D_0$ and $\tau_1=1$ on ${\rm Supp\,}\tau$. Let $\chi\in C^\infty_0(]-1,1[)$, $\chi=1$ on $[-\frac{1}{2},\frac{1}{2}]$. Let
\begin{equation}\label{e-gue130918V}
\hat{\mathcal{I}}_k=\frac{k^{2n-1}}{(2\pi)^{2n-1}}\int e^{ik<x-y,\eta>}\tau(x)\chi(\frac{\abs{\eta'}^2}{M^2})\psi(\eta_{2n-1})\tau_1(y)d\eta. 
\end{equation}
It is straightforward to see that $\hat{\mathcal{I}}_k$ satisfies the assumptions in Theorem~\ref{t-gue13630}. Let 
\[\mathcal{S}_k(x,y)=\int e^{ik\varphi(x,y,s)}a(x,y,s,k)ds\in C^\infty(D\times D)\] 
be as in Theorem~\ref{t-gue13630}, Theorem~\ref{t-gue13630I}. From \eqref{e-gue13716I}, it is not difficult to see that 
\begin{equation}\label{e-gue130919}
a_0(p,p,s)=(2\pi)^{-n}\abs{\det\bigr(M^\phi_p-2s\mathcal{L}_p\bigr)}\psi(s),
\end{equation}
for all $(p,p,s)\in\Omega$. We recall that $\Omega$ is as in Theorem~\ref{t-gue13630I}. 
Put $x'=(x_1,\ldots,x_{2n-2})$ and set $\abs{x'}^2=\sum^{2n-2}_{j=1}\abs{x_{j}}^2$. 
Let 
\begin{equation}\label{e-gue131001}
\begin{split}
&v_k=\int e^{ik\varphi(x,p,s)}k^{-\frac{n}{2}}a(x,p,s,k)ds\chi(\frac{k}{(\log k)^2}\abs{x'}^2)\chi(k^{\frac{5}{6}}x_{2n-1})\in C^\infty_0(D_0),\\
&u_k=s^ke^{k\phi}v_k\in C^\infty_0(D_0,L^k).
\end{split}
\end{equation}
Here we assume that $k$ is large enough so that ${\rm Supp\,}v_k\Subset D_0$. First, we need 

\begin{lem}\label{l-gue131001}
With the assumptions and notations above, for every $N>0$ and $m>0$, there is a constant $C_{N,m}>0$ independent of $k$ and the point $p$ such that 
\begin{equation}\label{e-gue131001I}
\sum_{\abs{\alpha}\leq m,\alpha\in\mathbb N_0^{2n-1}}{\rm Sup\,}\set{\abs{\pr^\alpha_x(\Box^{(0)}_{s,k}v_k(x))};\, x\in D_0}\leq C_{N,m}k^{-N},
\end{equation}
where $\Box^{(0)}_{s,k}$ is given by \eqref{s2-emsmilkI} and \eqref{e-msmilkVI}. 

Moreover, there exist $C>1$ and $k_0>0$ independent of $k$ and the point $p$ such that for every $k\geq k_0$, we have 
\begin{equation}\label{e-gue131001II}
\frac{1}{C}\leq\norm{u_k}_{h^{L^k}}\leq C.
\end{equation}
\end{lem} 

\begin{proof}
Put $\Td v_k=\int e^{ik\varphi(x,p,s)}k^{-\frac{n}{2}}a(x,p,s,k)ds\in C^\infty(D)$. In view of \eqref{e-gue13630IIa}, we see that 
\begin{equation}\label{e-gue131001III}
\sum_{\abs{\alpha}\leq m,\alpha\in\mathbb N_0^{2n-1}}{\rm Sup\,}\set{\abs{\pr^\alpha_x(\Box^{(0)}_{s,k}\Td v_k(x))};\, x\in D_0}\leq\Td C_{N,m}k^{-N}
\end{equation}
for every $N>0$, where $\Td C_{N,m}>0$ is a constant independent of $k$. Since $X$ is assumed to be compact, it is straightforward to see that $\Td C_{N,m}>0$ can be taken to be independent of $p$. Now, we claim that 
for every $N>0$ and $m>0$, there is a constant $\hat C_{N,m}>0$ independent of $k$ and the point $p$ such that 
\begin{equation}\label{e-gue131001IV}
\sum_{\abs{\alpha}\leq m,\alpha\in\mathbb N_0^{2n-1}}{\rm Sup\,}\set{\abs{\pr^\alpha_x(\Td v_k(x)-v_k(x))};\, x\in D_0}\leq\hat C_{N,m}k^{-N}.
\end{equation}
Note that 
\[{\rm Supp\,}(\Td v_k(x)-v_k(x))\bigcap D_0\subset\set{x\in D_0;\, \mbox{$\abs{x'}\geq\frac{1}{2}\frac{\log k}{\sqrt{k}}$ or $\abs{x_{2n-1}}\geq\frac{1}{2}k^{-\frac{5}{6}}$}}.\]
When $\abs{x'}\geq\frac{1}{2}\frac{\log k}{\sqrt{k}}$, we use the fact 
$k{\rm Im\,}\varphi(x,p,s)\geq ck\abs{x'}^2\geq c'(\log k)^2$, where $c>0$, $c'>0$ are constants independent of $x$ and $s$, and conclude that $\Td v_k(x)-v_k(x)\equiv0\mod O(k^{-\infty})$. When $\abs{x_{2n-1}}\geq\frac{1}{2}k^{-\frac{5}{6}}$, we have 
\[k\frac{\pr\varphi(x,p,s)}{\pr s}\geq kc_1\abs{x_{2n-1}}\geq c'_1k^{\frac{1}{6}},\]
where $c_1>0$, $c'_1>0$ are constants independent of $x$ and $s$. From this we can integrate by parts in $s$ several times and conclude that $\Td v_k(x)-v_k(x)\equiv0\mod O(k^{-\infty})$. The claim \eqref{e-gue131001IV} follows. 

From \eqref{e-gue131001III} and \eqref{e-gue131001IV}, the claim \eqref{e-gue131001I} follows. 

On $D$, we write $dv_X(x)=m(x)dx$, $m(x)\in C^\infty(D)$. Then, $m(0)=2^{n-1}$. Put 
\[\Upsilon(x)=\chi(\frac{k}{(\log k)^2}\abs{x'}^2)\chi(k^{\frac{5}{6}}x_{2n-1}).\]
We have 
\begin{equation}\label{e-gue131006}
\begin{split}
\norm{u_k}^2_{h^{L^k}}&=\int\abs{\int e^{ik\varphi(x,p,s)}k^{-\frac{n}{2}}a(x,p,s)ds}^2\Upsilon^2(x)m(x)dx\\
&=\int\abs{\int e^{ik\varphi(F^*_kx,p,s)}k^{-n}a(F^*_kx,p,s)ds}^2\Upsilon^2(F^*_kx)m(F^*_kx)dx,
\end{split}
\end{equation}
where $F^*_kx=(\frac{x'}{\sqrt{k}},\frac{x_{2n-1}}{k})$. Put 
\begin{equation}\label{e-gue131006I}
\begin{split}
&\varphi(x,p,s)=-i\sum^{n-1}_{j=1}\alpha_jz_j+i\sum^{n-1}_{j=1}\ol\alpha_j\ol z_j+s(x_{2n-1}-y_{2n-1})\\
&\quad+\sum^{2n-2}_{j,t=1}\beta_{j,t}(s)x_jx_t+O(\abs{x_{2n-1}}\abs{x'})+O(\abs{x_{2n-1}}^2)+O(\abs{x}^3)
\end{split}
\end{equation} 
and set
\begin{equation}\label{e-gue131006II}
\varphi_0(x,s)=\varphi_0(x',s)=\sum^{2n-2}_{j,t=1}\beta_{j,t}(s)x_jx_t.
\end{equation} 
From \eqref{e-dgugeXI}, it is easy to see that 
\begin{equation}\label{e-gue131006III}
{\rm Im\,}\varphi_0(x,s)\geq c\abs{x'}^2,\ \ \forall (x,p,s)\in\Omega,
\end{equation}
where $c>0$ is a constant independent of $(x,p,s)\in\Omega$ and $c$ can be take to be independent of the point $p$. It is easy to see that 
\begin{equation}\label{e-gue131006IV}
{\rm Supp\,}\Upsilon(F^*_kx)\subset\set{x\in\Real^{2n-1};\, \abs{x'}\leq\log k, \abs{x_{2n-1}}\leq k^{\frac{1}{6}}}.
\end{equation}
From \eqref{e-gue131006IV}, it is straightforward to check that on ${\rm Supp\,}\Upsilon(F^*_kx)$, 
\begin{equation}\label{e-gue131006V}
\begin{split}
&k\varphi(F^*_kx,p,s)=\sqrt{k}(-i\sum^{n-1}_{j=1}\alpha_jz_j+i\sum^{n-1}_{j=1}\ol\alpha_j\ol z_j)+sx_{2n-1}+\varphi_0(x',s)+\delta^0_k(x,s),\\
&k^{-n}a(F^*_kx,p,s)=(2\pi)^{-n}\abs{\det\bigr(M^\phi_p-2s\mathcal{L}_p\bigr)}\psi(s)+\delta^1_k(x,s),\\
&m(F^*_kx)=2^{n-1}+\delta^2_k(x,s),
\end{split}
\end{equation}
where $\delta^0_k(x,s), \delta^1_k(x,s), \delta^2_k(x,s)\in C^\infty$ and 
\begin{equation}\label{e-gue131006VI}
{\rm Sup\,}\set{\abs{\delta^0_k(x,s)}+\abs{\delta^1_k(x,s)}+\abs{\delta^2_k(x,s)};\, (x,p,s)\in\Omega}\leq C_0k^{-\frac{1}{3}}\log k,
\end{equation}
where $C_0>0$ is a constant independent of $k$ and the point $p$. From \eqref{e-gue131006VI} and \eqref{e-gue131006}, we have 
\begin{equation}\label{e-gue131007I}
\norm{u_k}^2_{h^{L^k}}=2^{n-1}\int\abs{\int e^{isx_{2n-1}+i\varphi_0(x',s)}(2\pi)^{-n}\psi(s)\abs{\det\bigr(M^\phi_p-2s\mathcal{L}_p\bigr)}\bigr(1+\gamma_k(x,s)\bigr)ds}^2\Upsilon^2(F^*_kx)dx,
\end{equation}
where $\gamma_k(x,s)\in C^\infty$ and 
\begin{equation}\label{e-gue131007II}
{\rm Sup\,}\set{\abs{\gamma_k(x,s)};\, (x,p,s)\in\Omega}\leq C_1k^{-\frac{1}{3}}\log k.
\end{equation}
Here $C_1>0$ is a constant independent of $k$ and the point $p$. From \eqref{e-gue131007II}, we can check that 
\begin{equation}\label{e-gue131007III}
\begin{split}
&2^{n-1}\int\abs{\int e^{isx_{2n-1}+i\varphi_0(x',s)}(2\pi)^{-n}\psi(s)\abs{\det\bigr(M^\phi_p-2s\mathcal{L}_p\bigr)}\abs{\gamma_k(x,s)}\bigr)ds}^2\Upsilon^2(F^*_kx)dx\\
&\leq C_2k^{-\frac{2}{3}}(\log k)^2(\log k)^{2n-2}k^{\frac{1}{6}}=C_2 k^{-\frac{1}{2}}(\log k)^{2n}\To0\ \ \mbox{as $k\To\infty$},
\end{split}
\end{equation}
where $C_2>0$ is a constant independent of $k$ and the point $p$. Put 
\[\begin{split}
&A=2^{n-1}\int_{\abs{x}\leq1}\abs{\int e^{isx_{2n-1}+i\varphi_0(x',s)}(2\pi)^{-n}\psi(s)\abs{\det\bigr(M^\phi_p-2s\mathcal{L}_p\bigr)}\bigr)ds}^2dx,\\
&B=2^{n-1}\int_{\Real^{2n-1}}\abs{\int e^{isx_{2n-1}+i\varphi_0(x',s)}(2\pi)^{-n}\psi(s)\abs{\det\bigr(M^\phi_p-2s\mathcal{L}_p\bigr)}\bigr)ds}^2dx.
\end{split}\]
From \eqref{e-gue131006III} and by using integration by parts with respect to $s$, it is easy to see that $B<\infty$. From this observation, \eqref{e-gue131007III} and \eqref{e-gue131007I}, we conclude that if $k$ large then 
\[\frac{A}{2}\leq\norm{u_k}^2_{h^{L^k}}\leq 2B.\] 
\eqref{e-gue131001II} follows. 
\end{proof} 

Now, we can prove 

\begin{thm}\label{t-gue131007}
With the assumptions and notations above, fix $N_0>>1$, $N_0$ large. There exists $k_0>0$ and $C_0>0$ independent of $k$ and the point $p$ such that for every $k\geq k_0$, there is a $\mu_k\in H^0_{b,\leq k^{-N_0}}(X,L^k)$ with $\norm{\mu_k}_{h^{L^k}}=1$ and 
\begin{equation}\label{e-gue131007IV}
\abs{\mu_k(p)}^2_{h^{L^k}}\geq C_0k^n.
\end{equation}
\end{thm}

\begin{proof}
Let $u_k\in C^\infty(X,L^k)$ be as in Lemma~\ref{l-gue131001}. Put $u^0_k=\Pi^{(0)}_{k,\leq k^{-N_0}}u_k$, $u^1_k=(I-\Pi^{(0)}_{k,\leq k^{-N_0}})u_k$. We have the orthogonal decomposition
\[u_k=u^0_k+u^1_k.\] 
For every $m\in\mathbb N_0$, we have 
\begin{equation}\label{e-gue131007V}
\begin{split}
\norm{(\Box^{(0)}_{b,k})^mu^1_k}^2_{h^{L^k}}&\leq k^{N_0}(\,(\Box^{(0)}_{b,k})^{m+1}u^1_k\,|\,u^1_k)_{h^{L^k}}\\
&\leq k^{N_0}(\,(\Box^{(0)}_{b,k})^{m+1}u_k\,|\,u_k)_{h^{L^k}}.
\end{split}
\end{equation} 
From \eqref{e-gue131001I} and \eqref{e-gue131007V}, we conclude that for every $N>0$ and every $m\in\mathbb N_0$, there is a constant $C_{N,m}>0$ independent of $k$ and the point $p$ such that 
\begin{equation}\label{e-gue131007VI}
\norm{(\Box^{(0)}_{b,k})^mu^1_k}^2_{h^{L^k}}\leq C_{N,m}k^{-N}.
\end{equation} 
From \eqref{e-gue131007VI}, we can repeat the proof of Lemma~\ref{l-gue13717} with minor changes and obtain that for every $N>0$ there is a constant $C_{N}>0$ independent of $k$ and the point $p$ such that 
\begin{equation}\label{e-gue131007VII}
\abs{u^1_k(p)}^2_{h^{L^k}}\leq k^{-N}.
\end{equation} 
Furthermore, from \eqref{e-gue131007VI} and \eqref{e-gue131001II}, we see that there exist $C>1$ and $k_0>0$ independent of $k$ and the point $p$ such that for every $k\geq k_0$, we have 
\begin{equation}\label{e-gue131007VIII}
\frac{1}{C}\leq\norm{u^0_k}_{h^{L^k}}\leq C.
\end{equation} 
Put $\mu_k=\frac{u^0_k}{\norm{u^0_k}_{h^{L^k}}}$. Then, $\mu_k\in H^0_{b,\leq k^{-N_0}}(X,L^k)$ and $\norm{\mu_k}_{h^{L^k}}=1$. Moreover, from \eqref{e-gue131007VI}, \eqref{e-gue131007VII} and \eqref{e-gue131007VIII}, we deduce that 
for every $N>0$ there is a constant $C_{N}>0$ independent of $k$ and the point $p$ such that 
\begin{equation}\label{e-gue131007a}
\abs{\abs{\mu_k(p)}^2_{h^{L^k}}-\frac{1}{\norm{u^0_k}_{h^{L^k}}^2}\abs{v_k(p)}^2}\leq C_Nk^{-N}.
\end{equation} 
From \eqref{e-gue131007a} and notice that $\abs{v_k(p)}^2\geq C_1k^n$, where $C_1>0$ is a constant independent of $k$ and the point $p$, the theorem follows. 
\end{proof} 

From now on,  we fix $N_0>2n+1$. We assume that $k$ is large enough so that the properties in Theorem~\ref{t-gue131007} hold.
Let $\set{f_1\in H^0_{b,\leq k^{-N_0}}(X,L^k) ,\ldots,f_{d_k}\in H^0_{b,\leq k^{-N_0}}(X,L^k)}$ be an orthonormal frame of the space $H^0_{b,\leq k^{-N_0}}(X,L^k)$. From Theorem~\ref{t-gue131007}, we deduce 

\begin{thm}\label{t-gue131007I}
We have 
\begin{equation}\label{e-gue131007aI}
\sum^{d_k}_{j=1}\abs{f_j(x)}^2_{h^{L^k}}\geq C_0k^n,\ \ \forall x\in X,
\end{equation}
where $C_0>0$ is the constant as in \eqref{e-gue131007IV}. In particular, there is a constant $c_0>0$ such that for every $x\in X$, there exists a $j_0\in\set{1,2,\ldots,d_k}$ such that 
\begin{equation}\label{e-gue131019}
\abs{f_{j_0}(x)}^2_{h^{L^k}}\geq c_0.
\end{equation}
\end{thm}

\begin{proof}
We only need to prove \eqref{e-gue131019}. 
It is well-known (see~\cite{HM09}) that there is a constant $C_1>0$ such that 
\begin{equation}\label{e-gue131019I}
{\rm dim\,}H^0_{b,\leq k^{-N}}(X,L^k)=d_k\leq C_1k^n,
\end{equation}
where $C_1>0$ is a constant independent of $k$. From \eqref{e-gue131019I} and \eqref{e-gue131007aI}, we have for every $x\in X$, 
\[\begin{split}
&C_1k^n{\rm Sup\,}\set{\abs{f_j(x)}^2_{h^{L^k}};\, j=1,2,\ldots,d_k}\\
&\geq d_k{\rm Sup\,}\set{\abs{f_j(x)}^2_{h^{L^k}};\, j=1,2,\ldots,d_k}\\
&\geq\sum^{d_k}_{j=1}\abs{f_j(x)}^2_{h^{L^k}}\geq C_0k^n.\end{split}\]
From this, \eqref{e-gue131019} follows. 
\end{proof}

Assume that $X=D_1\bigcup D_2\bigcup\cdots D_N$, where $D_j$ is a small open set of $X$ with the properties as in the beginning of section~\ref{s-aket}, for each $j$. On $D_j$, let $\hat{\mathcal{I}}_{k,j}$ be the operator as in \eqref{e-gue130918V}. Fix $N'\gg1$ be a large constant. 
The (asymptotic) Kodaira map $\Phi_{N_0,k}:X\To\Complex\mathbb P^{d_k-1}$ is given by 
\begin{equation}\label{e-gue131019II}
\begin{split}
\Phi_{N_0,k}:x\in X&\To\mathbb C\mathbb P^{N_k-1},\\
x\in X&\To[f_1(x),\ldots,f_{d_k}(x),k^{-{N'}}\hat{\mathcal{I}}_{k,1}f_1,\ldots,k^{-{N'}}\hat{\mathcal{I}}_{k,N}f_{d_k}]\in\Complex\mathbb P^{N_k-1},
\end{split}
\end{equation}
where $N_k=d_k+Nd_k$. 
In view of \eqref{e-gue131019}, we see that $\Phi_{N_0,k}$ is well-defined as a smooth map from $X$ to $\Complex\mathbb P^{d_k-1}$. Our next goal is to prove 

\begin{thm}\label{t-gue131019}
For $k$ large, the differential map
\[d\Phi_{N_0,k}(x):T_xX\To T_{\Phi_{N_0,k}(x)}\Complex\mathbb P^{d_k-1}\]
is injective, for every $x\in X$. 
\end{thm}

To prove Theorem~\ref{t-gue131019}, we need some preparations. Fix $p\in X$ and let $s$ be a local trivializing section of $L$ on an open neighbourhood $D\subset X$ of $p$. We take local coordinate $x=(x_1,\ldots,x_{2n-1})$, $z_j=x_{2j-1}+ix_{2j}$, $j=1,\ldots,n-1$, and $s$ so that \eqref{e-gue13716II} hold. Let $R(x)=R(z)$ be as in \eqref{e-gue13716III}.
We first need 

\begin{lem}\label{l-gue131019}
With the assumptions and notations above, there exist $v^j_k\in C^\infty_0(D)$, $j=1,\ldots,n$, with 
\begin{equation}\label{e-gue131019III-I}
\frac{1}{C}\leq\norm{s^ke^{k\phi}v^j_k}_{h^{L^k}}\leq C,\ \ j=1,\ldots,n,
\end{equation}
for every $k$, where $C>0$ is a constant independent of $k$ and the point $p$, and 
\begin{equation}\label{e-gue131019IV}
\sum_{\abs{\alpha}\leq m,\alpha\in\mathbb N_0^{2n-1}}{\rm Sup\,}\set{\abs{\pr^\alpha_x(\Box^{(0)}_{s,k}v^j_k(x))};\, x\in D, j=1,\ldots,n}\leq C_{N,m}k^{-N},
\end{equation}
for every $N>0$ and $m\in\mathbb N_0$, where $C_{N,m}$ is a constant independent of $k$ and the point $p$
and $\Box^{(0)}_{s,k}$ is given by \eqref{s2-emsmilkI} and \eqref{e-msmilkVI}, such that
\begin{align}
&v^j_k(0)=0,\ \  j=1,\ldots,n,\label{e-gue131019V}\\
&\pr_{\ol z_t}\bigr(e^{-2kR}v^j_k\bigr)(0)=0,\ \ j=1,\ldots,n,\ \ t=1,\ldots,n-1,\label{e-gue131019VI}\\
&\pr_{z_t}\bigr(e^{-2kR}v^j_k\bigr)(0)=0\ \ \mbox{if $j\neq t$, $j=1,\ldots,n$, $t=1,\ldots,n-1$},\label{e-gue131019VI-I}\\
&\pr_{x_{2n-1}}\bigr(e^{-2kR}v^j_k\bigr)(0)=0,\ \ j=1,\ldots,n-1,\label{e-gue131019VII}\\
&\abs{\pr_{x_{2n-1}}\bigr(e^{-2kR}v^n_k\bigr)(0)}^2\geq c_1k^{n+2},\ \ \abs{\pr_{z_j}\bigr(e^{-2kR}v^j_k\bigr)(0)}^2\geq c_1k^{n+1},\ \ j=1,\ldots,n-1,\label{e-gue131019VIII}
\end{align}  
where $c_1>0$ is a constant independent of $k$ and the point $p$. 
\end{lem}

\begin{proof}
From Borel construction, it is clearly that we can find $z_j+\beta_j(x)\in C^\infty(\Real^{2n-1})$, $j=1,\ldots,n-1$, and $x_{2n-1}+\beta_n(x)\in C^\infty(\Real^{2n-1})$ such that 
\begin{equation}\label{e-gue131022}
\begin{split}
&\mbox{$\ddbar_b(z_j+\beta_j(x))$ vanishes to infinite order at $p$, $j=1,\ldots,n-1$},\\
&\mbox{$\ddbar_b(x_{2n-1}+\beta_n(x))$ vanishes to infinite order at $p$},
\end{split}
\end{equation}
where $\beta_j(x)=O(\abs{x}^2)$, $j=1,\ldots,n$. Let $v_k(x)\in C^\infty_0(D_0)$ be as in \eqref{e-gue131001}. Recall that $D_0\Subset D$ be as in the discussion after \eqref{e-gue130918I}. Put 
\begin{equation}\label{e-gue131022I}
\begin{split}
&v^j_k(x)=\sqrt{k}(z_j+\beta_j(x))v_k(x)\in C^\infty_0(D_0),\ \ j=1,\ldots,n-1,\\
&v^n_k(x)=k(x_{2n-1}+\beta_n(x))v_k(x)\in C^\infty_0(D_0).
\end{split}
\end{equation}
We can repeat the proof of \eqref{e-gue131001II} with minor changes and deduce \eqref{e-gue131019III-I}. Moreover, from \eqref{e-gue131001I}, \eqref{e-gue131022} and \eqref{e-gue1373VIIa}, we obtain \eqref{e-gue131019IV}. Finally, from the constructions of $v_k$ and $v^j_k(x)$, $j=1,\ldots,n$, we get \eqref{e-gue131019V}, \eqref{e-gue131019VI}, \eqref{e-gue131019VI-I}, \eqref{e-gue131019VII} and \eqref{e-gue131019VIII}. The lemma follows. 
\end{proof}

We also need 

\begin{lem}\label{l-gue131022}
With the assumptions and notations above, fix $N_0>2n+1$. There exist 
\[\mu^j_k\in H^0_{b,\leq k^{-N_0}}(X,L^k),\ \ j=1,\ldots,n,\] 
with $\norm{\mu^j_k}_{h^{L^k}}=1$, $j=1,\ldots,n$, such that if we put $\mu^j_k=s^k\Td\mu^j_k$ on $D$, $j=1,\ldots,n$, then 
\begin{equation}\label{e-gue131022II}
\abs{\pr_{x_{2n-1}}\bigr(e^{-2kR}\Td\mu^n_k\bigr)(0)}^2\geq c_2k^{n+2},\ \ \abs{\pr_{z_j}\bigr(e^{-2kR}\Td\mu^j_k\bigr)(0)}^2\geq c_2k^{n+1},\ \ j=1,\ldots,n-1,
\end{equation}  
where $c_2>0$ is a constant independent of $k$ and the point $p$ and for every $N>0$ there is a $C_N>0$ independent of $k$ and the point $p$ such that
\begin{equation}\label{e-gue131022III}
\begin{split}
{\rm Sup\,}&\{\abs{\bigr(e^{-2kR}\Td\mu^j_k\bigr)(0)}, \abs{\pr_{\ol z_t}\bigr(e^{-2kR}\Td\mu^j_k\bigr)(0)}, \abs{(1-\delta_{j,t})\pr_{z_t}\bigr(e^{-2kR}\Td\mu^j_k\bigr)(0)}, \abs{\pr_{x_{2n-1}}\bigr(e^{-2kR}\Td\mu^t_k\bigr)(0)};\,\\
 &\quad j=1,\ldots,n, t=1,\ldots,n-1\}\leq C_Nk^{-N}.
 \end{split}
\end{equation}  
\end{lem}

\begin{proof}
Fix $j=1,\ldots,n$. Let $v^j_k\in C^\infty_0(D)$ be as in Lemma~\ref{l-gue131019}. Put $u^j_k=s^ke^{k\phi}v^j_k$ and set $\beta^j_k=\Pi^{(0)}_{k,\leq k^{-N_0}}u^j_k$, $\gamma^j_k=(I-\Pi^{(0)}_{k,\leq k^{-N_0}})u^j_k$. We have the orthogonal decomposition
\[u^j_k=\beta^j_k+\gamma^j_k.\] 
For every $m\in\mathbb N_0$, we have 
\begin{equation}\label{e-gue131022IV}
\begin{split}
\norm{(\Box^{(0)}_{b,k})^m\gamma^j_k}^2_{h^{L^k}}&\leq k^{N_0}(\,(\Box^{(0)}_{b,k})^{m+1}\gamma^j_k\,|\,\gamma^j_k)_{h^{L^k}}\\
&\leq k^{N_0}(\,(\Box^{(0)}_{b,k})^{m+1}u^j_k\,|\,u^j_k)_{h^{L^k}}.
\end{split}
\end{equation} 
From \eqref{e-gue131019IV} and \eqref{e-gue131022IV}, we conclude that for every $N>0$ and every $m\in\mathbb N_0$, there is a constant $C_{N,m}>0$ independent of $k$ and the point $p$ such that 
\begin{equation}\label{e-gue131022V}
\norm{(\Box^{(0)}_{b,k})^m\gamma^j_k}^2_{h^{L^k}}\leq C_{N,m}k^{-N}.
\end{equation} 
From \eqref{e-gue131022V}, we can repeat the proof of Lemma~\ref{l-gue13717} with minor changes and obtain that for every $N>0$ and every $\alpha\in\mathbb N_0^{2n-1}$, there is a constant $C_{N,\alpha}>0$ independent of $k$ and the point $p$ such that 
\begin{equation}\label{e-gue131022VI}
\abs{\pr^\alpha_x\bigr(e^{-2kR}\Td\gamma^j_k\bigr)(p)}^2\leq k^{-N},
\end{equation} 
where $\gamma^j_k=s^k\Td\gamma^j_k$ on $D$. 
Furthermore, from \eqref{e-gue131022V} and \eqref{e-gue131019III-I}, we see that there exist $C>1$ and $k_0>0$ independent of $k$ and the point $p$ such that for every $k\geq k_0$, we have 
\begin{equation}\label{e-gue131022VII}
\frac{1}{C}\leq\norm{\beta^j_k}_{h^{L^k}}\leq C.
\end{equation} 
Put $\mu^j_k=\frac{\beta^j_k}{\norm{\beta^j_k}_{h^{L^k}}}$. Then, $\mu^j_k\in H^0_{b,\leq k^{-N_0}}(X,L^k)$ and $\norm{\mu^j_k}_{h^{L^k}}=1$. On $D$, put $\mu^j_k=s^k\Td\mu^j_k$. From \eqref{e-gue131022VI} and \eqref{e-gue131022VII}, we deduce that 
for every $N>0$  and every $\alpha\in\mathbb N_0^{2n-1}$ there is a constant $C_{N,\alpha}>0$ independent of $k$ and the point $p$ such that 
\begin{equation}\label{e-gue131022VIII}
\abs{\abs{\pr^\alpha_x\bigr(e^{-2kR}\Td\mu^j_k\bigr)(p)}^2-\frac{1}{\norm{\beta^j_k}^2_{h^{L^k}}}\abs{\pr^\alpha_x\bigr(e^{-2kR}v^j_k\bigr)(p)}^2}\leq C_{N,\alpha}k^{-N}.
\end{equation} 
From \eqref{e-gue131022VIII} and Lemma~\ref{l-gue131019}, the lemma follows. 
\end{proof}

From Lemma~\ref{l-gue131022} and the Gram-Schmidt process, we deduce

\begin{prop}\label{p-gue131022}
With the assumptions and notations above, fix $N_0>2n+1$. There exist $g^j_k\in H^0_{b,\leq k^{-N_0}}(X,L^k)$, $j=1,\ldots,n$, with $(g^j_k\,|\,g^t_k)_{h^{L^k}}=\delta_{j,t}$, $j,t=1,\ldots,n$, such that if we put $g^j_k=s^k\Td g^j_k$ on $D$, $j=1,\ldots,n$, then 
\begin{equation}\label{e-gue131022aII}
\begin{split}
&\abs{\pr_{z_t}\bigr(e^{-2kR}\Td g^t_k\bigr)(0)}^2\geq c_2k^{n+1},\ \ t=1,\ldots,n-1,\\
&\abs{\pr_{x_{2n-1}}\bigr(e^{-2kR}\Td g^n_k\bigr)(0)}^2\geq c_2k^{n+2},
\end{split}
\end{equation}  
where $c_2>0$ is a constant independent of $k$ and the point $p$ and for every $N>0$ there is a $C_N>0$ independent of $k$ and the point $p$ such that
\begin{equation}\label{e-gue131022aIII}
\begin{split}
&{\rm Sup\,}\{\abs{\pr_{x_{2n-1}}\bigr(e^{-2kR}\Td g^t_k\bigr)(0)}, \abs{\pr_{z_s}\bigr(e^{-2kR}\Td g^t_k\bigr)(0)};\, s,t=1,\ldots,n-1, s>t\}\\
&\quad+{\rm Sup\,}\{\abs{\bigr(e^{-2kR}\Td g^t_k\bigr)(0)}, \abs{\pr_{\ol z_s}\bigr(e^{-2kR}\Td g^t_k\bigr)(0)};\, s=1,\ldots,n-1, t=1,\ldots,n\}\leq C_Nk^{-N}.
\end{split}
\end{equation}  
\end{prop}

From Proposition~\ref{p-gue131022} and some straightforward but elementary linear algebra argument, we obtain the 
following(we omit the proof)

\begin{prop}\label{p-gue131022I}
With the assumptions and notations above, fix $N_0>2n+1$. Let 
\[g^j_k\in H^0_{b,\leq k^{-N_0}}(X,L^k),\ \ j=1,\ldots,n,\] 
be as in Proposition~\ref{p-gue131022}. We put 
\begin{equation}\label{e-gue131022aI}
\begin{split}
&\mbox{$g^j_k=s^k\Td g^j_k$ on $D$, $j=1,\ldots,n$},\\
&\mbox{$\Td g^j_k=h^{2j-1}_k+ih^{2j}_k$, $j=1,\ldots,n$, where $h^{2j-1}_k={\rm Re\,}\Td g^j_k$, $h^{2j}_k={\rm Im\,}\Td g^j_k$, $j=1,\ldots,n$}.
\end{split}
\end{equation} 
There is a $k_0>0$ independent of the point $p$ such that for every $k\geq k_0$, the matrix
\[\begin{split}H_k:&=\left[
\begin{array}[c]{cccc}
  \pr_{x_1}\bigr(e^{-2kR}h^1_k\bigr)(p)&\pr_{x_2}\bigr(e^{-2kR}h^1_k\bigr)(p)&\cdots&\pr_{x_{2n-1}}\bigr(e^{-2kR}h^{1}_k\bigr)(p)\\
  \pr_{x_1}\bigr(e^{-2kR}h^2_k\bigr)(p)&\pr_{x_2}\bigr(e^{-2kR}h^2_k\bigr)(p)&\cdots&\pr_{x_{2n-1}}\bigr(e^{-2kR}h^{2}_k\bigr)(p)\\
 \vdots&\vdots&\vdots&\vdots\\
\pr_{x_1}\bigr(e^{-2kR}h^{2n}_k\bigr)(p)&\pr_{x_2}\bigr(e^{-2kR}h^{2n}_k\bigr)(p)&\cdots&\pr_{x_{2n-1}}\bigr(e^{-2kR}h^{2n}_k\bigr)(p)
\end{array}\right],\\
H_k:&\Real^{2n-1}\To\Real^{2n},\end{split}\]
is injective.
\end{prop}


From the proofs of Lemma~\ref{l-gue13717} and Proposition~\ref{p-gue13717I}, we conclude that 

\begin{lem}\label{l-gue131023}
With the assumptions and notations above, fix $N_0>2n+1$. Let 
\[x_k=(x^1_k,\ldots,x^{2n-1}_k)\in\Real^{2n-1},\ \ y_k=(y^1_k,\ldots,y^{2n-1}_k)\in\Real^{2n-1}\] 
with $\lim_{k\To\infty}(\sqrt{k}\sum^{2n-2}_{j=1}\abs{x^j_k}+k\abs{x^{2n-1}_k})=0$, $\lim_{k\To\infty}(\sqrt{k}\sum^{2n-2}_{j=1}\abs{y^j_k}+k\abs{y^{2n-1}_k})=0$. Then, for every $\alpha=(\alpha_1,\ldots,\alpha_{2n-1})\in\mathbb N_0^{2n-1}$, $\beta=(\beta_1,\ldots,\beta_{2n-1})$, there are constants $C_\alpha>0$, $C_{\alpha,\beta}>0$ independent of $k$ and the point $p$ such that for every $u\in H^0_{b,\leq k^{-N_0}}(X,L^k)$ with $\norm{u}_{h^{L^k}}=1$ we have 
\begin{equation}\label{e-gue131023II}
\abs{\pr^\alpha_x\bigr(e^{-2kR}\Td u\bigr)(x_k)}^2\leq C_{\alpha}k^{n+\abs{\alpha'}+2\alpha_{2n-1}}
\end{equation}
and 
\begin{equation}\label{e-gue131102}
\abs{\pr^\alpha_x\pr^\beta_y\Bigr(e^{-kR(x)+k\ol R(x)}\hat\Pi^{(0)}_{k,\leq k^{-N_0},s}(x,y)e^{-k\ol R(y)+kR(y)}\Bigr)(x_k,y_k)}\leq C_{\alpha,\beta}k^{n+\frac{\abs{\alpha'}}{2}+\frac{\abs{\beta'}}{2}+\alpha_{2n-1}+\beta_{2n-1}},
\end{equation}
where $u=s^k\Td u$ on $D$, $\hat\Pi^{(0)}_{k,\leq k^{-N_0},s}(x,y)$ is the localized Szeg\"{o} projection (see \eqref{e-gue13719IV} and \eqref{e-gue13719aI}) and $\abs{\alpha'}=\sum^{2n-2}_{j=1}\alpha_j$, $\abs{\beta'}=\sum^{2n-2}_{j=1}\beta_j$
\end{lem}

\begin{proof}[Proof of Theorem~\ref{t-gue131019}]
We are going to prove that if $k$ is large then the map
\[d\Phi_{N_0,k}(x):T_xX\To T_{\Phi_{N_0,k}(x)}\Complex\mathbb P^{d_k-1}\]
is injective. Fix $p\in X$ and let $s$ be a local trivializing section of $L$ on an open neighbourhood $D\subset X$ of $p$. We take local coordinates $x=(x_1,\ldots,x_{2n-1})$, $z_j=x_{2j-1}+ix_{2j}$, $j=1,\ldots,n-1$, and $s$ so that \eqref{e-gue13716II} hold. We shall use the same notations as above. From Theorem~\ref{t-gue131007I}, we may assume that 
\begin{equation}\label{e-gue131023II-I}
\abs{\bigr(e^{-2kR}\Td f_1\bigr)(p)}^2\geq c_0,
\end{equation}
where $c_0>0$ is a constant independent of $k$ and the point $p$. Let $g^1_k,\ldots,g^n_k\in H^0_{b,\leq k^{-N_0}}(X,L^k)$ be as in Proposition~\ref{p-gue131022}. In view of \eqref{e-gue131022aIII}, we may assume that 
\begin{equation}\label{e-gue131023III}
\sum^n_{j=1}\abs{\bigr(e^{-2kR}\Td g^j_k\bigr)(p)}^2\leq\frac{c_0}{2}. 
\end{equation}
Now, we claim that $f_1,g^1_k,\ldots,g^n_k$ are linearly independent over $\Complex$. If $f_1,g^1_k,\ldots,g^n_k$ are linearly dependent then we have $f_1=\sum^n_{j=1}\lambda_jg^j_k$, where $\lambda_j\in\Complex$, $j=1,\ldots,n$. Since $\norm{f_1}_{h^{L^k}}=1$, we have $\sum^n_{j=1}\abs{\lambda_j}^2=1$. Thus, 
\[\abs{\bigr(e^{-2kR}\Td f_1\bigr)(p)}^2\leq\Bigr(\sum^n_{j=1}\abs{\lambda_j}^2\Bigr)\Bigr(\sum^n_{j=1}\abs{\bigr(e^{-2kR}\Td g^j_k\bigr)(p)}^2\Bigr)\leq\frac{c_0}{2}.\]
We get a contradiction. Thus, $f_1,g^1_k,\ldots,g^n_k$ are linearly independent. Put 
\begin{equation}\label{e-gue131023IV}
\begin{split}
&p^j_k=\frac{e^{-2kR}\Td g^j_k}{e^{-2kR}\Td f_1},\ \ j=1,\ldots,n-1,\\
&p^j_k=\alpha^{2j-1}_k+i\alpha^{2j}_k,\ \ \alpha^{2j-1}_k={\rm Re\,}p^j_k,\ \ \alpha^{2j}_k={\rm Im\,}p^j_k,\ \ j=1,\ldots,n-1.
\end{split}
\end{equation}
From \eqref{e-gue131022aII}, \eqref{e-gue131022aIII}, Lemma~\ref{l-gue131023} and \eqref{e-gue131023II-I}, it is not difficult to see that 
\begin{equation}\label{e-gue131024}
\abs{\pr_{z_t}p^t_k(p)}^2\geq c_2k^{n+1},\ \ t=1,\ldots,2n-1,\ \ \abs{\pr_{x_{2n-1}}p^n_k(p)}^2\geq c_2k^{n+2},
\end{equation}  
where $c_2>0$ is a constant independent of $k$ and the point $p$ and for every $N>0$ there is a $C_N>0$ independent of $k$ and the point $p$ such that
\begin{equation}\label{e-gue131024I}
\begin{split}
&{\rm Sup\,}\{\abs{\pr_{x_{2n-1}}p^t_k(p)}, \abs{\pr_{z_s}p^t_k(p)};\, s,t=1,\ldots,n-1, s>t\}\\
&\quad+{\rm Sup\,}\{\abs{p^t_k(p)}, \abs{\pr_{\ol z_s}p^t_k(p)};\, s=1,\ldots,n-1, t=1,\ldots,n\}\leq C_Nk^{-N}.
\end{split}
\end{equation}
From \eqref{e-gue131024}, \eqref{e-gue131024I} and some elementary linear algebra argument, we conclude that 
there is a $k_0>0$ independent of the point $p$ such that for every $k\geq k_0$, the matrix
\[\begin{split}
A_k&:=\left[
\begin{array}[c]{cccc}
  \pr_{x_1}\bigr(e^{-2kR}\alpha^1_k\bigr)(p)&\pr_{x_2}\bigr(e^{-2kR}\alpha^1_k\bigr)(p)&\cdots&\pr_{x_{2n}}\bigr(e^{-2kR}\alpha^{1}_k\bigr)(p)\\
  \pr_{x_1}\bigr(e^{-2kR}\alpha^2_k\bigr)(p)&\pr_{x_2}\bigr(e^{-2kR}\alpha^2_k\bigr)(p)&\cdots&\pr_{x_{2n}}\bigr(e^{-2kR}\alpha^{2}_k\bigr)(p)\\
 \vdots&\vdots&\vdots&\vdots\\
\pr_{x_1}\bigr(e^{-2kR}\alpha^{2n-1}_k\bigr)(p)&\pr_{x_2}\bigr(e^{-2kR}\alpha^{2n}_k\bigr)(p)&\cdots&\pr_{x_{2n}}\bigr(e^{-2kR}\alpha^{2n}_k\bigr)(p)
\end{array}\right],\\
A_k&:\Real^{2n-1}\To\Real^{2n}\end{split}\]
is injective. Hence the differential of the map $x\in X\To(\frac{g^1_k}{f_1}(x),\ldots,\frac{g^n_k}{f_1}(x))\in\Complex^n$ at $p$ is injective if $k\geq k_0$. From this and some elementary linear algebra arguments, we conclude that 
the differential of the map $x\in X\To(\frac{f_2}{f_1}(x),\ldots,\frac{f_{d_k}}{f_1}(x))\in\Complex^{d_k}$ at $p$ is injective if $k\geq k_0$. Theorem~\ref{t-gue131019} follows. 
\end{proof}

Our last goal in this section is to prove that for $k$ large, the map $\Phi_{N_0,k}:X\To\Complex\mathbb P^{d_k-1}$ is injective. 

\begin{thm}\label{t-gue131103}
With the assumptions and notations above, fix $N_0>2n+1$. 
For $k$ large, the map $\Phi_{N_0,k}:X\To\Complex\mathbb P^{d_k-1}$ is injective. 
\end{thm}

\begin{proof}
We assume that the claim of the theorem is not true. We can find $x_{k_j}, y_{k_j}\in X$, $x_{k_j}\neq y_{k_j}$, $0<k_1<k_2<\cdots$, $\lim_{j\To\infty}k_j=\infty$, such that $\Phi_{N_0,k_j}(x_{k_j})=\Phi_{N_0,k_j}(y_{k_j})$, for each $j$. We may suppose that there are $x_{k}, y_k\in X$, $x_k\neq y_{k}$, such that $\Phi_{N_0,k}(x_k)=\Phi_{N_0,k}(y_k)$, for each $k$. We may assume that $x_k\To p\in X$, $y_k\To q\in X$, as $k\To\infty$. If $p\neq q$. Then, for $k$ large, we have ${\rm dist\,}(x_k,y_k)\geq\frac{1}{2}{\rm dist\,}(p,q)$. In view of the proof of Theorem~\ref{t-gue131007}, it is not difficult to see that we can find $u_k, v_k\in H^0_{b,\leq k^{-N_0}}(X,L^k)$ such that for $k$ large, we have 
\begin{equation}\label{e-gue131103IV}
\abs{u_k(x_k)}^2_{h^{L^k}}\geq C_0k^n,\ \ \abs{u_k(y_k)}^2_{h^{L^k}}\leq\frac{C_0}{2}k^n,
\end{equation}
and 
\begin{equation}\label{e-gue131103V}
\abs{v_k(y_k)}^2_{h^{L^k}}\geq C_0k^n,\ \ \abs{v_k(x_k)}^2_{h^{L^k}}\leq\frac{C_0}{2}k^n,
\end{equation}
where $C_0>0$ is a constant independent of $k$. Now,  $\Phi_{N_0,k}(x_k)=\Phi_{N_0,k}(y_k)$ implies that
\[\abs{u_k(x_k)}^2_{h^{L^k}}=r_k\abs{u_k(y_k)}^2_{h^{L^k}},\ \ \abs{v_k(x_k)}^2_{h^{L^k}}=r_k\abs{v_k(y_k)}^2_{h^{L^k}},\]
where $r_k\in\Real_+$, for each $k$. \eqref{e-gue131103IV} implies that $r_k\geq 2$, for $k$ large. But  \eqref{e-gue131103V} implies that $r_k\leq\frac{1}{2}$, for $k$ large. We get a contradiction. Thus, we must have $p=q$. 

Let $X=D_1\bigcup D_2\bigcup\cdots D_N$, where $D_j$ is an open set as in the discussion before \eqref{e-gue131019II}. We assume that $p\in D_1=:D$. 
Let $s$ be a local trivializing section of $L$ on an open subset $D\subset X$ of $p$, $\abs{s}^2_{h^L}=e^{-2\phi}$. Let $x=(x_1,\ldots,x_{2n-1})$, $z_j=x_{2j-1}+ix_{2j}$, $j=1,\ldots,n-1$, be local coordinates of $X$ defined on $D$. For simplicity, we assume that \eqref{e-gue13716II} hold. We shall use the same notations as before. We write $x_k=(x^1_k,\ldots,x^{2n-1}_k)\in\Real^{2n-1}$, $y_k=(y^1_k,\ldots,y^{2n-1}_k)\in\Real^{2n-1}$. 

{\rm Case I\,}: $\limsup_{k\To\infty}\sqrt{k}\sum^{2n-2}_{j=1}\abs{x^j_k-y^j_k}=M>0$ ($M$ can be $\infty$). \\
For simplicity, we may assume that 
\begin{equation}\label{e-gue131103VIb}
\lim_{k\To\infty}\sqrt{k}\sum^{2n-2}_{j=1}\abs{x^j_k-y^j_k}=M,\ \ M\in]0,\infty].
\end{equation}
Now, $\Phi_{N_0,k}(x_k)=\Phi_{N_0,k}(y_k)$ implies that we can find a sequence $\lambda_k\in\Complex$ such that for each $k$, 
\begin{equation}\label{e-gue131103VI}
e^{-k\phi(x_k)}\Td u_k(x_k)=\lambda_ke^{-k\phi(y_k)}\Td u_k(y_k),
\end{equation} 
for every $u_k\in H^0_{b,\leq k^{-N_0}}(X,L^k)$, $u_k=s^k\Td u_k$ on $D$. We may assume that 
\begin{equation}\label{e-gue131103VII}
\limsup_{k\To\infty}\abs{\lambda_k}\geq 1. 
\end{equation}
Let $\hat{\mathcal{I}}_k:=\hat{\mathcal{I}}_{k,1}$ be as in \eqref{e-gue130918V} and \eqref{e-gue131019II}. Let 
\[h_k=\sum^{d_k}_{j=1}f_j\bigr(\hat{\mathcal{I}}_ke^{-k\phi}\Td f_j\bigr)(y_k)\in H^0_{b,\leq k^{-N_0}}(X,L^k),\]
where $f_j=s^k\Td f_j$ on $D$, $j=1,\ldots,d_k$. On $D$, we write $h_k=s^k\Td h_k$. Then, it is easy to see that $(\hat\Pi^{(0)}_{k,\leq k^{-N_0},s}\hat{\mathcal{I}}_k)(x,y_k)=e^{-k\phi(x)}\Td h_k(x)$. From this observation and Theorem~\ref{t-gue130816}, it is straightforward to check that 
\begin{equation}\label{e-gue131103VIII}
e^{-k\phi(x)}\Td h_k(x)=\int e^{ik\varphi(x,y_k,s)}a(x,y_k,s,k)ds+R_k(x),
\end{equation}
where $\varphi(x,y,s)\in C^\infty(\Omega)$ is as in Theorem~\ref{t-gue130816}, $\Omega$ is as in the discussion after \eqref{e-gue130819I}, $a(x,y,s,k)\sim\sum^\infty_{j=0}a_j(x,y,s)k^{n-j}$ in 
$S^n_{{\rm loc\,}}(1;\Omega)$, 
\begin{equation}\label{e-gue131103aI}
a_0(p,p,s)=(2\pi)^{-n}\abs{\det\bigr(M^\phi_p-2s\mathcal{L}_p\bigr)}\psi(s),
\end{equation}
and $R_k(x)$ is a smooth function on $D$ such that for every $D'\Subset D$, there is a constant $C_{D'}$ independent of $k$ such that 
\begin{equation}\label{e-gue131103aII}
\abs{R_k(x)}\leq C_{D'}k^{3n-N_0-2}\leq C_{D'}k^{n-2}.
\end{equation}
From \eqref{e-gue131103VIb}, \eqref{e-gue131103aI}, \eqref{e-gue131103aII} and \eqref{e-dgugeXI}, we have
\begin{equation}\label{e-gue131103aIII}
\begin{split}
&\limsup_{k\To\infty}k^{-n}\abs{e^{-k\phi(x_k)}\Td h_k(x_k)}\\
&\leq\limsup_{k\To\infty}k^{-n}\Bigr(\int e^{-{\rm Im\,}\varphi(x_k,y_k,s)}\abs{a(x_k,y_k,s)}ds+R_k(x_k)\Bigr)\\
&\leq e^{-cM^2}(2\pi)^{-n}\int\abs{\det\bigr(M^\phi_p-2s\mathcal{L}_p\bigr)}\psi(s)ds,
\end{split}
\end{equation}
where $c>0$ is a constant independent of $k$, and 
\begin{equation}\label{e-gue131103aIV}
\limsup_{k\To\infty}k^{-n}\abs{e^{-k\phi(y_k)}\Td h_k(y_k)}=(2\pi)^{-n}\int\abs{\det\bigr(M^\phi_p-2s\mathcal{L}_p\bigr)}\psi(s)ds.
\end{equation}
From \eqref{e-gue131103VI} and \eqref{e-gue131103VII}, we conclude that 
\[\limsup_{k\To\infty}k^{-n}\abs{e^{-k\phi(x_k)}\Td h_k(x_k)}\geq\limsup_{k\To\infty}k^{-n}\abs{e^{-k\phi(y_k)}\Td h_k(y_k)}.\] 
From this and \eqref{e-gue131103aIII}, \eqref{e-gue131103aIV}, we deduce that 
\[e^{-cM^2}(2\pi)^{-n}\int\abs{\det\bigr(M^\phi_p-2s\mathcal{L}_p\bigr)}\psi(s)ds\geq(2\pi)^{-n}\int\abs{\det\bigr(M^\phi_p-2s\mathcal{L}_p\bigr)}\psi(s)ds.\]
But this is impossible. We get a contradiction. 

{\rm Case II\,}: $\limsup_{k\To\infty}\sqrt{k}\sum^{2n-2}_{j=1}\abs{x^j_k-y^j_k}=0$,  $\limsup_{k\To\infty}k\abs{\langle\,\omega_0(x_k),y_k-x_k\,\rangle}=M>0$ ($M$ can be $\infty$). \\
For simplicity, we may assume that 
\begin{equation}\label{e-gue131103aV}
\lim_{k\To\infty}k\abs{\langle\,\omega_0(x_k),y_k-x_k\,\rangle}=M,\ \ M\in]0,\infty].
\end{equation}
Now, $\Phi_{N_0,k}(x_k)=\Phi_{N_0,k}(y_k)$ implies that we can find a sequence $\lambda_k\in\Complex$ such that for each $k$, 
\begin{equation}\label{e-gue131103aVI}
e^{-k\phi(x_k)}\Td u_k(x_k)=\lambda_ke^{-k\phi(y_k)}\Td u_k(y_k),
\end{equation} 
for every $u_k\in H^0_{b,\leq k^{-N_0}}(X,L^k)$, $u_k=s^k\Td u_k$ on $D$. We may assume that 
\begin{equation}\label{e-gue131103aVII}
\limsup_{k\To\infty}\abs{\lambda_k}\geq 1. 
\end{equation}
We fist assume that $M=\infty$. Let $h_k$ be as above. From the fact that $\abs{\frac{\pr\varphi(x,y,s)}{\pr s}}_{x=x_k,y=y_k}\geq c\abs{\langle\,\omega_0(x_k),y_k-x_k\,\rangle}$, where $c>0$ is a constant independent of $k$, we can integrate by parts with respect to $s$ and conclude that 
\[\limsup_{k\To\infty}k^{-n}\abs{e^{-k\phi(x_k)}\Td h_k(x_k)}=0.\]
But from \eqref{e-gue131103aVII}, we have 
\[\begin{split}
&0=\limsup_{k\To\infty}k^{-n}\abs{e^{-k\phi(x_k)}\Td h_k(x_k)}\\
&\geq\limsup_{k\To\infty}k^{-n}\abs{e^{-k\phi(y_k)}\Td h_k(y_k)}
=(2\pi)^{-n}\int\abs{\det\bigr(M^\phi_p-2s\mathcal{L}_p\bigr)}\psi(s)ds.\end{split}\]
This is impossible. We get a contradiction. Now, we assume that $M<\infty$. From \eqref{e-gue131103VIII} and \eqref{e-gue131103aII}, it is not difficult to see that 
\begin{equation}\label{e-gue131107}
\begin{split}
\lim_{k\To\infty}k^{-n}\abs{e^{-k\phi(x_k)}\Td h_k(x_k)}&=(2\pi)^{-n}\abs{\int e^{ikMs}\abs{\det\Bigr(M^\phi_p-2s\mathcal{L}_p\Bigr)}\psi(s)ds}\\
&<(2\pi)^{-n}\int\abs{\det\Bigr(M^\phi_p-2s\mathcal{L}_p\Bigr)}\psi(s)ds=\lim_{k\To\infty}k^{-n}\abs{e^{-k\phi(y_k)}\Td h_k(y_k)}.
\end{split}
\end{equation}
We get a contradiction. 

{\rm Case III\,}: $\limsup_{k\To\infty}\sqrt{k}\sum^{2n-2}_{j=1}\abs{x^j_k-y^j_k}=0$,  $\limsup_{k\To\infty}k\abs{\langle\,\omega_0(x_k),y_k-x_k\,\rangle}=0$.\\
Let $g_j=\hat{\mathcal{I}}_kf_j$ and set $g_j=s^k\Td g_j$ on $D$, $j=1,2,\ldots,d_k$.
Put 
\begin{equation}\label{e-gue131107I}
\begin{split}
\alpha_k(t)=&e^{-k\phi(tx_k+(1-t)y_k)-k\phi(y_k)}\times\\
&\quad\sum^{d_k}_{j=1}\Td g_j(tx_k+(1-t)y_k)\ol{\Td g_j}(y_k)e^{-kR(tx_k+(1-t)y_k)+k\ol R(tx_k+(1-t)y_k)-k\ol R(y_k)+kR(y_k)},\\
A_k(t)=&\abs{\alpha_k(t)}^2,\\
B_k(t)=&e^{-2k\phi(tx_k+(1-t)y_k)-2k\phi(y_k)}\times\\
&\quad\sum^{d_k}_{j=1}\abs{\Td g_j(tx_k+(1-t)y_k)e^{-kR(tx_k+(1-t)y_k)+k\ol R(tx_k+(1-t)y_k)}}^2\sum^{d_k}_{j=1}\abs{\Td g_j(y_k)e^{-kR(y_k)+k\ol R(y_k)}}^2,
\end{split}
\end{equation}
where $t\in[0,1]$ and $R$ is as in \eqref{e-gue13716III}. Put $H_k(t)=\frac{A_k(t)}{B_k(t)}$. $H_k(t)$ is a smooth function of $t\in[0,1]$ since $B_k(t)>0$ for every $t\in[0,1]$. Moreover, we can check that $0\leq H_k(t)\leq 1$ and $H_k(1)=H_k(0)=1$. Thus, for each $k$, there is a $t_k\in[0,1]$ such that 
\begin{equation}\label{e-gue131107II}
H''_k(t_k)\geq 0.
\end{equation}
We now calculate $H''_k(t)$. We first calculate $A''_k(t)$. In view of Theorem~\ref{t-gue130816}, it is not difficult to see that
\begin{equation}\label{e-gue131107IV}
\begin{split}
\alpha_{k}(t)=&\int e^{ik\varphi(tx_k+(1-t)y_k,y_k,s)}e^{-kR(tx_k+(1-t)y_k)+k\ol R(tx_k+(1-t)y_k)-k\ol R(y_k)+kR(y_k)}\\
&\quad\quad\quad\times a(tx_k+(1-t)y_k,y_k,s,k)ds+\epsilon_{k}(tx_k+(1-t)y_k,y_k),
\end{split}
\end{equation}
where $a(x,y,s,k)\sim\sum^\infty_{j=0}k^{n-j}a_j(x,y,s)$ in $S^n_{{\rm loc\,}}(1;\Omega)$, $a_j(x,y,s)\in C^\infty_0(\Omega)$, $j=0,1,\ldots$, $\Omega$ is as in the discussion after \eqref{e-gue130819I}, 
\begin{equation}\label{e-gue131108}
a_{0}(p,p,s)=(2\pi)^{-n}\abs{\det\Bigr(M^\phi_p-2s\mathcal{L}_p\Bigr)}\abs{\psi(s)}^2,
\end{equation}
and $\epsilon_{k,\delta}(x,y)$ is a smooth function on $D\times D$ such that for every $D'\Subset D$ and every $\alpha, \beta\in\mathbb N^{2n-1}_0$, there is a constant $C_{D',\alpha,\beta,\delta}$ independent of $k$ such that 
\begin{equation}\label{e-gue131108I}
\abs{\pr^\alpha_x\pr^\beta_y\epsilon_{k,\delta}(x,y)}\leq C_{D',\alpha,\beta,\delta}k^{3n-N_0-2+2\abs{\alpha}+2\abs{\beta}}.
\end{equation}
We can calculate that 
\begin{equation}\label{e-gue131107V}
\begin{split}
A''_k(t)=&2\abs{\alpha'_{k}(t)}^2+\alpha''_{k}(t)\ol\alpha_{k}(t)+\ol\alpha''_{k}(t)\alpha_{k}(t).
\end{split}
\end{equation}
From \eqref{e-guew13627}, \eqref{e-gue131107IV}, \eqref{e-gue131108} and \eqref{e-gue131108I}, it is straightforward to see that(we omit the computations) 
\begin{equation}\label{e-gue131108II}
\begin{split}
&2\abs{\alpha'_{k}(t_k)}^2+\alpha''_{k}(t_k)\ol\alpha_{k}(t_k)+\ol\alpha''_{k}(t_k)\alpha_{k}(t_k)\\
&=2(2\pi)^{-2n}k^{2n+2}\Bigr(\bigr(\int s\abs{\det\Bigr(M^\phi_p-2s\mathcal{L}_p\Bigr)}\abs{\psi(s)}^2ds\bigr)^2\\
&-\int s^2\abs{\det\Bigr(M^\phi_p-2s\mathcal{L}_p\Bigr)}\abs{\psi(s)}^2ds\int\abs{\det\Bigr(M^\phi_p-2s\mathcal{L}_p\Bigr)}\abs{\psi(s)}^2ds\Bigr)(\abs{\langle\,\omega_0(x_k),y_k-x_k\,\rangle})^2\\
&-2(2\pi)^{-2n}k^{2n+1}\int(\sum^{2n-2}_{j,l=1}\frac{\pr^2{\rm Im\,}\varphi}{\pr x_j\pr x_l}(p,p,s)(x^j_k-y^j_k)(x^l_k-y^l_k))\abs{\det\Bigr(M^\phi_p-2s\mathcal{L}_p\Bigr)}\abs{\psi(s)}^2ds\\
&+o(k^{2n})O\Bigr(\bigr(\sqrt{k}\sum^{2n-2}_{j=1}\abs{x^j_k-y^j_k}+k\abs{\langle\,\omega_0(x_k),y_k-x_k\,\rangle}\bigr)^2\Bigr).
\end{split}
\end{equation}
Since 
\[\begin{split}
&\Bigr(\bigr(\int s\abs{\det\Bigr(M^\phi_p-2s\mathcal{L}_p\Bigr)}\abs{\psi(s)}^2ds\bigr)^2\\
&\quad-\int s^2\abs{\det\Bigr(M^\phi_p-2s\mathcal{L}_p\Bigr)}\abs{\psi(s)}^2ds\int\abs{\det\Bigr(M^\phi_p-2s\mathcal{L}_p\Bigr)}\abs{\psi(s)}^2ds\Bigr)<0,\end{split}\]
there is a constant $C_1>0$ independent of $k$ such that 
\begin{equation}\label{e-gue131108III}
\begin{split}
&\Bigr(\bigr(\int s\abs{\det\Bigr(M^\phi_p-2s\mathcal{L}_p\Bigr)}\abs{\psi(s)}^2ds\bigr)^2\\
&-\int s^2\abs{\det\Bigr(M^\phi_p-2s\mathcal{L}_p\Bigr)}\abs{\psi(s)}^2ds\int\abs{\det\Bigr(M^\phi_p-2s\mathcal{L}_p\Bigr)}\abs{\psi(s)}^2ds\Bigr)\langle\,\omega_0(x_k),y_k-x_k\,\rangle^2\\
&\leq -C_1\abs{\langle\,\omega_0(x_k),y_k-x_k\,\rangle}^2.
\end{split}
\end{equation}
Moreover, from \eqref{e-dgugeXI}, we can check that there is a constant $C_2>0$ independent of $k$ such that 
\begin{equation}\label{e-gue131108IV}
\begin{split}
&-\int(\sum^{2n-2}_{j,l=1}\frac{\pr^2{\rm Im\,}\varphi}{\pr x_j\pr x_l}(p,p,s)(x^j_k-y^j_k)(x^t_k-y^t_k))\abs{\det\Bigr(M^\phi_p-2s\mathcal{L}_p\Bigr)}\abs{\psi(s)}^2ds\\
&\leq -C_2\sum^{2n-2}_{j=1}(x^j_k-y^j_k)^2.
\end{split}
\end{equation}
From \eqref{e-gue131108II}, \eqref{e-gue131108III} and \eqref{e-gue131108IV}, we deduce that 
\begin{equation}\label{e-gue131108VII}
\limsup_{k\To\infty}k^{-2n}\bigr(\sqrt{k}\sum^{2n-2}_{j=1}\abs{x^j_k-y^j_k}+k\abs{\langle\,\omega_0(x_k),y_k-x_k\,\rangle}\bigr)^{-2}A''_k(t_k)\leq -C<0,
\end{equation}
where $C>0$ is a constant. 

Now, we have 
\begin{equation}\label{e-gue131108VIIIa}
H''_k(t_k)=\frac{A''_k(t_k)}{B_k(t_k)}-2\frac{A'_k(t_k)}{B^2_k(t_k)}B'_k(t_k)-\frac{A_k(t_k)}{B^2_k(t_k)}B''_k(t_k)+2\frac{A_k(t_k)(B'_k(t_k))^2}{B^3_k(t_k)}.
\end{equation}
From \eqref{e-gue130819IIm}, it is easy to see that
\begin{equation}\label{e-gue131108VIIIb}
\begin{split}
&\limsup_{k\To\infty}\Bigr(\bigr(\sqrt{k}\sum^{2n-2}_{j=1}\abs{x^j_k-y^j_k}+k\abs{\langle\,\omega_0(x_k),y_k-x_k\,\rangle}\bigr)^{-2}\times\\
&\quad\bigr(-2\frac{A'_k(t_k)}{B^2_k(t_k)}B'_k(t_k)-\frac{A_k(t_k)}{B^2_k(t_k)}B''_k(t_k)+2\frac{A_k(t_k)(B'_k(t_k))^2}{B^3_k(t_k)}\bigr)\Bigr)=0.
\end{split}
\end{equation}
From \eqref{e-gue131108VIIIb}, \eqref{e-gue131108VIIIa} and \eqref{e-gue131108VII}, we deduce that 
\begin{equation}\label{e-gue131108VIIIab}
\begin{split}
&\limsup_{k\To\infty}(\bigr(\sqrt{k}\sum^{2n-2}_{j=1}\abs{x^j_k-y^j_k}+k\abs{\langle\,\omega_0(x_k),y_k-x_k\,\rangle}\bigr)^{-2}H''_k(t_k)\\
&=\limsup_{k\To\infty}\bigr(\sqrt{k}\sum^{2n-2}_{j=1}\abs{x^j_k-y^j_k}+k\abs{\langle\,\omega_0(x_k),y_k-x_k\,\rangle}\bigr)^{-2}\frac{A''_k(t_k)}{B_k(t_k)}\leq -C_0<0,
\end{split}
\end{equation}
where $C_0>0$ is a constant. But from \eqref{e-gue131107II}, we see that 
\[\limsup_{k\To\infty}(\bigr(\sqrt{k}\sum^{2n-2}_{j=1}\abs{x^j_k-y^j_k}+k\abs{\langle\,\omega_0(x_k),y_k-x_k\,\rangle}\bigr)^{-2}H''_k(t_k)\geq0.\]
We get a contradiction. The theorem follows. 
\end{proof} 

Summing up, we obtain one of the main results of this work 

\begin{thm}\label{t-gue131108}
Let $N_0>2n+1$. Then, for $k$ large, the Kodaira map $\Phi_{N_0,k}:X\To\Complex\mathbb P^{d_k-1}$ is an embedding. 
\end{thm}

From Theorem~\ref{t-gue131108}, we deduce Theorem~\ref{t-gue131109}. 

\section{Asymptotic expansion of the Szeg\"{o} kernel}\label{s-gue130820} 

We recall some notations we used before. 
Let $s$ be a local trivializing section of $L$ on an open subset $D\Subset X$ and $\abs{s}^2_{h^L}=e^{-2\phi}$.  Let $A_k:L^2_{(0,q)}(X,L^k)\To L^2_{(0,q)}(X,L^k)$ be a continuous operator. Let 
\[\hat A_{k,s}:L^2_{(0,q)}(D)\bigcap\mathscr E'(D,T^{*0,q}X)\To L^2_{(0,q)}(D)\]
be the localized operator (with respect to the trivializing section $s$) of $A_k$ given by \eqref{e-gue130820}. 
We write $A_k\equiv0\mod O(k^{-\infty})$ on $D$ if $\hat A_{k,s}\equiv0\mod O(k^{-\infty})$ on $D$.
Until further notice, we assume that $Y(q)$ holds on $D$.
First, we need 

\begin{defn}\label{d-gue130820}
Fix $q\in\set{0,1,\ldots,n-1}$. Let $A_k:L^2_{(0,q)}(X,L^k)\To L^2_{(0,q)}(X,L^k)$ be a continuous operator. Let $D\Subset X$. We say that $\Box^{(q)}_{b,k}$ has $O(k^{-n_0})$ small spectral gap on $D$ with respect to $A_k$ if for every $D'\Subset D$, there exist constants $C_{D'}>0$,  $n_0, p\in\mathbb N$, $k_0\in\mathbb N$, such that for all $k\geq k_0$ and $u\in\Omega^{0,q}_0(D',L^k)$, we have  
\[\norm{A_k(I-\Pi^{(q)}_k)u}_{h^{L^k}}\leq C_{D'}\,k^{n_0}\sqrt{(\,(\Box^{(q)}_{b,k})^pu\,|\,u\,)_{h^{L^k}}}.\] 
\end{defn}

\begin{defn}\label{d-gue131205}
Let $A_k:L^2_{(0,q)}(X,L^k)\To L^2_{(0,q)}(X,L^k)$ be a continuous operator.
We say that $\Pi^{(q)}_k$ is $k$-negligible away the diagonal with respect to $A_k$ 
on $D$ if for any $\chi, \chi_1\in C^\infty_0(D)$ with $\chi_1=1$ on some neighbourhood of ${\rm Supp\,}\chi$, we have 
\[\Bigr(\chi A_k(1-\chi_1)\Bigr)\Pi^{(q)}_k\Bigr(\chi A_k(1-\chi_1)\Bigr)^*\equiv0\mod O(k^{-\infty})\ \ \mbox{ on $D$},\]
where 
\[\Bigr(\chi A_k(1-\chi_1)\Bigr)^*:L^2_{(0,q)}(X,L^k)\To L^2_{(0,q)}(X,L^k)\]
is the Hilbert space adjoint of $\chi A_k(1-\chi_1)$ with respect to $(\,\cdot\,|\,\cdot\,)_{h^{L^k}}$. 
\end{defn}

It is easy to see that if $\Pi^{(q)}_k$ is $k$-negligible away the diagonal with respect to $A_k$ on $D$,
then for any $\chi, \chi_1\in C^\infty_0(D)$ with $\chi_1=1$ on some neighborhood of ${\rm Supp\,}\chi$, we have 
\[\Bigr(\chi A_k(1-\chi_1)\Bigr)\Pi^{(q)}_k\equiv0\mod O(k^{-\infty})\ \ \mbox{on $D$}.\]

\begin{defn}\label{d-gue131205I}
Let $A_k:L^2_{(0,q)}(X,L^k)\To L^2_{(0,q)}(X,L^k)$ be a continuous operator.
We say that $A_k$ is a global classical semi-classical pseudodifferential operator of order $m$ on $X$ if for every 
local trivializing section $s$ of $L$ on an open subset $D\subset X$, the localized operator $\hat A_{k,s}$ is a classical semi-classical pseudodifferential operator of order $m$ on $D$.
\end{defn}

\begin{prop}\label{p-gue131205}
Let $A_k:L^2_{(0,q)}(X,L^k)\To L^2_{(0,q)}(X,L^k)$ be a global classical semi-classical pseudodifferential operator on $X$ of order $0$. If $X$ is compact and $Y(q)$ holds on $X$ then $\Pi^{(q)}_k$ is $k$-negligible away the diagonal with respect to $A_k$ on every local trivialization $D\Subset X$.

Furthermore, if $X$ is non-compact and $A_k$ is properly supported on $D\Subset X$ and $Y(q)$ holds on $D$, where $D$ is a local trivialization of $X$,  then $\Pi^{(q)}_k$ is $k$-negligible away the diagonal with respect to $A_k$ on $D$.
\end{prop}

\begin{proof}
Let $s$ be a local trivializing section of $L$ on a local trivialization  $D\subset X$. 
From Theorem~\ref{t-gue13718}, we can repeat the proof of Proposition~\ref{p-gue13717I} with minor change and conclude that for every $\alpha, \beta\in\mathbb N^{2n-1}_0$, and $D'\Subset D$, there is a constant $C_{\alpha,\beta,D'}>0$ independent of $k$ such that 
\begin{equation}\label{e-gue131205}
\pr^\alpha_x\pr^\beta_y\Bigr(\hat\Pi^{(q)}_{k,s}(x,y)\Bigr)\leq C_{\alpha,\beta,D'}k^{n+\abs{\alpha}+\abs{\beta}}\ \ \mbox{on $D'\times D'$}.
\end{equation}
From \eqref{e-gue131205} and by using integration by parts, the proposition can be deduced . We omit the details. 
\end{proof}

Now, we can prove 

\begin{thm}\label{t-gue130820}
Let $s$ be a local trivializing section of $L$ on an open subset $D\subset X$ and $\abs{s}^2_{h^L}=e^{-2\phi}$. We assume that there exist a $\lambda_0\in\Real$ and $x_0\in D$ such that $M^\phi_{x_0}-2\lambda_0\mathcal{L}_{x_0}$ is non-degenerate of constant signature $(n_-,n_+)$. Let $q=n_-$ 
and assume that $Y(q)$ holds at each point of $D$. Let $F_k:L^2_{(0,q)}(X,L^k)\To L^2_{(0,q)}(X,L^k)$ be a continuous operator and let $F^*_k:L^2_{(0,q)}(X,L^k)\To L^2_{(0,q)}(X,L^k)$ be the Hilbert space adjoint of $F_k$ with respect to $(\,\cdot\,|\,\cdot\,)_{h^{L^k}}$. Let
$\hat F_{k,s}$ and $\hat F^*_{k,s}$ be the localized operators of $F_{k,s}$ and $F^*_{k,s}$ respectively.
We fix $D_0\Subset D$, $D_0$ open. Let $V$ be as in \eqref{e-dhmpXII}. Assume that 
\[\hat F_{k,s}-A_k=O(k^{-\infty}):H^s_{{\rm comp\,}}(D,T^{*0,q}X)\To H^s_{{\rm loc\,}}(D,T^{*0,q}X),\  \ \forall s\in\mathbb N_0,\]
where
\[\mbox{$A_k\equiv\frac{k^{2n-1}}{(2\pi)^{2n-1}}\int e^{ik<x-y,\eta>}\alpha(x,\eta,k)d\eta\mod O(k^{-\infty})$ at $T^*D_0\bigcap\Sigma$}\]
is a classical semi-classical pseudodifferential operator on $D$ of order $0$ from sections of $T^{*0,q}X$ to sections of $T^{*0,q}X$, where
\[\begin{split}&\mbox{$\alpha(x,\eta,k)\sim\sum_{j=0}\alpha_j(x,\eta)k^{-j}$ in $S^0_{{\rm loc\,}}(1;T^*D,T^{*0,q}X\boxtimes T^{*0,q}X)$},\\
&\alpha_j(x,\eta)\in C^\infty(T^*D,T^{*0,q}D\boxtimes T^{*0,q}D),\ \ j=0,1,\ldots,
\end{split}\]
with $\alpha(x,\eta,k)=0$ if $\abs{\eta}>M$, for some large $M>0$ and ${\rm Supp\,}\alpha(x,\eta,k)\bigcap T^*D_0\Subset V$. Put $P_k:=F_k\Pi^{(q)}_kF^*_k$ and let $\hat P_{k,s}$ be the localized operator of $P_k$. If $\Box^{(q)}_{b,k}$ has $O(k^{-n_0})$ small spectral gap on $D$ with respect to $F_k$ and $\Pi^{(q)}_k$ is $k$-negligible away the diagonal with respect to $F_k$ on $D$, then
\begin{equation}\label{e-gue130820II}
\hat P_{k,s}(x,y)\equiv\int e^{ik\varphi(x,y,s)}g(x,y,s,k)ds\mod O(k^{-\infty})
\end{equation}
on $D_0$, where $\varphi(x,y,s)\in C^\infty(\Omega)$ is as in Theorem~\ref{t-dcgewI}, \eqref{e-guew13627},
\begin{equation}\label{e-gue130820III}
\begin{split}
&g(x,y,s,k)\in S^{n}_{{\rm loc\,}}\big(1;\Omega,T^{*0,q}X\boxtimes T^{*0,q}X\big)\bigcap C^\infty_0\big(\Omega,T^{*0,q}X\boxtimes T^{*0,q}X\big),\\
&g(x,y,s,k)\sim\sum^\infty_{j=0}g_j(x,y,s)k^{n-j}\text{ in }S^{n}_{{\rm loc\,}}
\big(1;\Omega,T^{*0,q}X\boxtimes T^{*0,q}X\big), \\
&g_j(x,y,s)\in C^\infty_0\big(\Omega,T^{*0,q}X\boxtimes T^{*0,q}X\big),\ \ j=0,1,2,\ldots,
\end{split}
\end{equation}
and for every $(x,x,s)\in\Omega$, $x\in D_0$,
\begin{equation}\label{e-gue130820IV}
\begin{split}
&g_0(x,x,s)\\
&=(2\pi)^{-n}\abs{\det\bigr(M^\phi_x-2s\mathcal{L}_x\bigr)}\alpha_0(x,s\omega_0(x)-2{\rm Im\,}\ddbar_b\phi(x))\mathcal{\pi}_{(x,s,n_-)}\alpha^*_0(x,s\omega_0(x)-2{\rm Im\,}\ddbar_b\phi(x)).
\end{split}
\end{equation}
Here
\[
\begin{split}
\Omega:=&\{(x,y,s)\in D\times D\times\Real;\, (x,-2{\rm Im\,}\ddbar_b\phi(x)+s\omega_0(x))\in V\bigcap\Sigma,\\
&\quad\mbox{$(y,-2{\rm Im\,}\ddbar_b\phi(y)+s\omega_0(y))\in V\bigcap\Sigma$, $\abs{x-y}<\varepsilon$, for some $\varepsilon>0$}\},
\end{split}\]
$\alpha^*_0(x,\eta):T^{*0,q}_xX\To T^{*0,q}_xX$ is the adjoint of $\alpha_0(x,\eta)$ with respect to the Hermitian metric $\langle\,\cdot\,|\,\cdot\,\rangle$ on $T^{*0,q}_xX$, $\mathcal{\pi}_{(x,s,n_-)}:T^{*0,q}_pX\To\mathcal{N}(x,s,n_-)$ is the orthogonal projection with respect to $\langle\,\cdot\,|\,\cdot\,\rangle$,  $\mathcal{N}(x,s,n_-)$ is given by \eqref{e-gue1373III}, 
\[\abs{\det\bigr(M^\phi_x-2s\mathcal{L}_x\bigr)}=\abs{\lambda_1(s)}\abs{\lambda_2(s)}\cdots\abs{\lambda_{n-1}(s)},\] 
$\lambda_1(s),\ldots,\lambda_{n-1}(s)$ are eigenvalues of the Hermitian quadratic form $M^\phi_x-2s_0\mathcal{L}_x$ with respect to $\langle\,\cdot\,|\,\cdot\,\rangle$.
\end{thm}

\begin{proof}
For simplicity, we assume that $A_k$ is properly supported on $D$. 
Take $\chi, \chi_1\in C^\infty_0(D)$ with $\chi=1$ on $D_0$ and $\chi_1=1$ on some neighbourhood of ${\rm Supp\,}\chi$. Put 
\[G_k=\chi F_k\chi_1,\ \ H_k=\chi F_k(1-\chi_1),\ \ B_k=F_k\Pi^{(q)}_k,\ \ R_k=H_k\Pi^{(q)}_k,\] 
and let $G^*_k, H^*_k:L^2_{(0,q)}(X,L^k)\To L^2_{(0,q)}(X,L^k)$ be the Hilbert adjoints of $G_k$ and $H_k$ respectively. 
Let $\hat G^*_{k,s}$, $\hat G_{k,s}$, $\hat H^*_{k,s}$, $\hat H_{k,s}$, $\hat B_{k,s}$, $\hat R_{k,s}$ be the localized operators of $G^*_k$, $G_k$, $H^*_k$, $H_k$, $B_k$, $R_k$ respectively. Since $\Pi^{(q)}_k$ is $k$-negligible away the diagonal with respect to $F_k$ on $D$, it is not difficult to see that 
\begin{equation}\label{e-gue130901III} 
\begin{split}
\mbox{$\hat P_{k,s}\equiv\hat G_{k,s}\hat\Pi^{(q)}_{k,s}\hat G^*_{k,s}\mod O(k^{-\infty})$ on $D_0$},\\
\mbox{$\hat P_{k,s}\equiv A_k\hat\Pi^{(q)}_{k,s}A^*_{k}\mod O(k^{-\infty})$ on $D_0$},
\end{split}
\end{equation}
where $A^*_k$ is the formal adjont of $A_k$. Let $\mathcal{S}_k$ and $\mathcal{N}_k$ be as in Theorem~\ref{t-gue13630}. Here we let $\hat{\mathcal{I}}_k=A^*_{k}$ in Theorem~\ref{t-gue13630}. Let $\Box^{(q)}_{s,k}$ be as in \eqref{e-msmilkVI}. Then,
\begin{equation}\label{e-gue130820V}
\begin{split}
&\Box^{(q)}_{s,k}\mathcal{N}_k+\mathcal{S}_k=A^*_{k}+h_k\ \ \mbox{on $\mathscr D'(D_0,T^{*0,q}X)$},\\
&\mathcal{N}^*_k\Box^{(q)}_{s,k}+\mathcal{S}^*_k=A_{k}+h^*_k\ \ \mbox{on
$\mathscr D'(D_0,T^{*0,q}X)$},
\end{split}\end{equation}
where $h_k\equiv0\mod O(k^{-\infty})$, $\mathcal{N}^*_k$, $\mathcal{S}^*_k$ and $h^*_k$ are formal adjoints of $\mathcal{N}_k$, $\mathcal{S}_k$ and $h_k$ with respect to $(\,\cdot\,|\,\cdot\,)$ respectively. From \eqref{e-gue130820V} and notice that $\Box^{(q)}_{s,k}\hat\Pi^{(q)}_{k,s}=0$, it is not difficult to see that 
\begin{equation}\label{e-gue130820VI}
\begin{split}
&\mathcal{S}^*_k\hat\Pi^{(q)}_{k,s}=(A_k+h^*_k)\hat\Pi^{(q)}_{k,s}\ \ \mbox{on
$\mathscr E'(D_0,T^{*0,q}X)$},\\
&\hat\Pi^{(q)}_{k,s}\mathcal{S}_k=\hat\Pi^{(q)}_{k,s}(A^*_k+h_k)\ \ \mbox{on
$\mathscr E'(D_0,T^{*0,q}X)$}.
\end{split}\end{equation} 

Let $u\in H^m_{{\rm comp\,}}(D_0,T^{*0,q}X)$, $m\leq0$, $m\in\mathbb Z$. 
We consider
\[v=s^ke^{k\phi}\mathcal{S}_{k}u-\Pi^{(q)}_{k}(s^ke^{k\phi}\mathcal{S}_{k}u).\]
Since $Y(q)$ holds on $D$ and $\mathcal{S}_{k}$ is a smoothing operator, we conclude that 
$v\in L^2_{(0,q)}(X,L^k)\bigcap\Omega^{0,q}(D)$. Moreover, from \eqref{s2-emsmilkI}, we have 
\begin{equation} \label{e-gue130820VII}
\Box^{(q)}_{b,k}v=s^ke^{k\phi}\Box^{(q)}_{s,k}\mathcal{S}_{k}u.
\end{equation}
In view of Theorem~\ref{t-gue13630}, we see that $\Box^{(q)}_{s,k}\mathcal{S}_{k}\equiv0\mod O(k^{-\infty})$. Combining this with \eqref{e-gue130820VII}, we obtain for every $p\in\mathbb N$,
\begin{equation}\label{e-gue130820VIII}
\norm{(\Box^{(q)}_{b,k})^pv}_{h^{L^k}}\leq C_{N,p}k^{-N}\norm{u}_m,
\end{equation}
for every $N>0$, where $C_{N,p}>0$ is independent of $k$, $u$ and $\norm{\cdot}_m$ denotes the usual Sobolev norm of order $m$ on $D_0$ with respect to $(\,\cdot\,|\,\cdot\,)$. Moreover, from the explicit formula of the kernel of $\mathcal{S}_k$ (see \eqref{e-gue13630IVa}), it is straightforward to see that 
\begin{equation}\label{e-gue130827}
\norm{v}_{h^{L^k}}\leq Ck^{n+m}\norm{u}_m,
\end{equation}
where $C>0$ is a constant independent of $k$ and $u$.
Note that $\Box^{(q)}_{b,k}$ has $O(k^{-n_0})$ small spectral gap on $D$ with respect to $F_k$ and $\Pi^{(q)}_kv=0$. From this observation, \eqref{e-gue130820VIII} and \eqref{e-gue130827},  we conclude that
$\norm{F_kv}_{h^{L^k}}\leq\Td C_Nk^{-N}\norm{u}_m$, 
for every $N>0$, where $\Td C_N>0$ is independent of $k$. Thus,
\begin{equation}\label{e-gue130901}
\hat F_{k,s}\mathcal{S}_{k}-\hat B_{k,s}\mathcal{S}_{k}=O(k^{-N}):H^m_{{\rm comp\,}}(D_0,T^{*0,q}X)\To L^2_{(0,q)}(D_0),
\end{equation}
for all $N>0$, $m\in\mathbb Z$, $m\leq0$. Since $\Pi^{(q)}_k$ is $k$-negligible away the diagonal with respect to $F_k$ on $D$ and notice that $\hat H_{k,s}\mathcal{S}_k\equiv0\mod O(k^{-\infty})$, we conclude that 
\begin{equation}\label{e-gue130901I}
\mbox{$\hat H_{k,s}\mathcal{S}_{k}-\hat R_{k,s}\mathcal{S}_{k}\equiv0\mod O(k^{-\infty})$ on $D_0$}
\end{equation} 
and hence 
\begin{equation}\label{e-gue130901Ia}
\mbox{$\hat F_{k,s}\mathcal{S}_{k}-\hat B_{k,s}\mathcal{S}_{k}\equiv\hat G_{k,s}\mathcal{S}_{k}-\hat G_{k,s}\hat\Pi^{(q)}_{k,s}\mathcal{S}_{k}\mod O(k^{-\infty})$ on $D_0$}.
\end{equation} 
From \eqref{e-gue130901Ia} and \eqref{e-gue130901}, we obtain 
\begin{equation}\label{e-gue130820Ia}
\begin{split}
&\hat G_{k,s}\mathcal{S}_{k}-\hat G_{k,s}\hat\Pi^{(q)}_{k,s}\mathcal{S}_{k}=O(k^{-N}):H^m_{{\rm comp\,}}(D_0,T^{*0,q}X)\To L^2_{(0,q)}(D_0),\\
&A_k\mathcal{S}_{k}-A_k\hat\Pi^{(q)}_{k,s}\mathcal{S}_{k}=O(k^{-N}):H^m_{{\rm comp\,}}(D_0,T^{*0,q}X)\To L^2_{(0,q)}(D_0),
\end{split}
\end{equation}
for all $N>0$, $m\in\mathbb Z$, $m\leq0$. Put
\begin{equation}\label{e-gue130820Iab}\begin{split}
&A_k\equiv\Td{\mathcal{I}}_k+\Td{\mathcal{I}}_k^1\mod O(k^{-\infty})\ \ \mbox{on $D_0$},\\
&\Td{\mathcal{I}}_k\equiv\frac{k^{2n-1}}{(2\pi)^{2n-1}}\int e^{ik<x-y,\eta>}\alpha(x,\eta,k)d\eta\mod O(k^{-\infty})\ \ \mbox{on $D_0$},\\
&\Td{\mathcal{I}}_k^1\equiv\frac{k^{2n-1}}{(2\pi)^{2n-1}}\int e^{ik<x-y,\eta>}\beta(x,y,\eta,k)d\eta\ \ \mbox{on $D_0$},
\end{split}\end{equation}
where $\Td{\mathcal{I}}_k$ and $\Td{\mathcal{I}}^1_k$ are properly supported on $D_0$, $\beta(x,y,\eta,k)\in S^0_{{\rm loc\,},{\rm cl\,}}(1;T^*D,T^{*0,q}X\boxtimes T^{*0,q}X)$ and there is a small neighbourhood $\Gamma$ of $T^*D_0\bigcap\Sigma$ such that $\beta(x,y,\eta,k)=0$ if $(x,\eta)\in\Gamma$. Since $\beta(x,y,\eta,k)=0$ if $(x,\eta)$ near $T^*D_0$ and notice that $\mathcal{F}\hat\Pi^{(q)}_{k,s}\equiv0\mod O(k^{-\infty})$ on $D_0$ if $\mathcal{F}$ is a properly supported $k$-negligible operator on $D_0$, we deduce that $\Td{\mathcal{I}}_k^1\hat\Pi^{(q)}_{k,s}\equiv0\mod O(k^{-\infty})$ on $D_0$. Moreover, it is not difficult to see that $\Td{\mathcal{I}}^1_k\mathcal{S}_k\equiv0\mod O(k^{-\infty})$ on $D_0$. Combining these with \eqref{e-gue130820Iab}, we obtain 
\begin{equation}\label{e-gue130820IIa}
\begin{split}
&A_k\mathcal{S}_{k}\equiv\Td{\mathcal{I}}_k\mathcal{S}_{k}\mod O(k^{-\infty})\ \ \mbox{on $D_0$},\\
&A_k\hat\Pi^{(q)}_{k,s}\equiv\Td{\mathcal{I}}_k\hat\Pi^{(q)}_{k,s}\mod O(k^{-\infty})\ \ \mbox{on $D_0$},\\
&A_k\mathcal{S}_{k}-A_k\hat\Pi^{(q)}_{k,s}\mathcal{S}_{k}\equiv\Td{\mathcal{I}}_k\mathcal{S}_{k}-\Td{\mathcal{I}}_k\hat\Pi^{(q)}_{k,s}\mathcal{S}_{k}\mod O(k^{-\infty})\ \ \mbox{on $D_0$}.
\end{split}
\end{equation}
From \eqref{e-gue130820IIa} and \eqref{e-gue130820Ia}, we deduce 
\begin{equation}\label{e-gue130820IIIa}
\Td{\mathcal{I}}_k\mathcal{S}_{k}-\Td{\mathcal{I}}_k\hat\Pi^{(q)}_{k,s}\mathcal{S}_{k}=O(k^{-N}):H^m_{{\rm comp\,}}(D_0,T^{*0,q}X)\To L^2_{(0,q)}(D_0),
\end{equation}
for all $N>0$, $m\in\mathbb Z$, $m\leq0$. Take 
\[\gamma(x,\eta,k)\in S^0_{{\rm loc\,},{\rm cl\,}}(1;T^*D,T^{*0,q}X\boxtimes T^{*0,q}X)\bigcap C^\infty_0(V,T^{*0,q}X\boxtimes T^{*0,q}X)\] 
so that $\gamma(x,\eta,k)=1$ on ${\rm Supp\,}\alpha(x,\eta,k)\bigcap T^*D_0$ and let
$\Gamma_k\equiv\int e^{i<x-y,\eta>}\gamma(x,\eta,k)d\eta\mod O(k^{-\infty})$ on $D_0$ be a properly supported classical semi-classical pseudodifferential operator on $D$ of order $0$ from sections of $T^{*0,q}X$ to sections of $T^{*0,q}X$. Since $\gamma(x,\eta,k)=1$ on ${\rm Supp\,}\alpha(x,\eta,k)\bigcap T^*D_0$, we have 
\begin{equation}\label{e-gue130820Va}
\Gamma_k\Td{\mathcal{I}}_k\equiv \Td{\mathcal{I}}_k\mod O(k^{-\infty})\ \ \mbox{on $D_0$}
\end{equation}
and hence 
\begin{equation}\label{e-gue130820VIa}
\Td{\mathcal{I}}_k\mathcal{S}_{k}-\Td{\mathcal{I}}_k\hat\Pi^{(q)}_{k,s}\mathcal{S}_{k}\equiv\Gamma_k\bigr(\Td{\mathcal{I}}_k\mathcal{S}_{k}-\Td{\mathcal{I}}_k\hat\Pi^{(q)}_{k,s}\mathcal{S}_{k}\bigr)\mod O(k^{-\infty})\ \ \mbox{on $D_0$}.
\end{equation}
Since ${\rm Supp\,}\gamma(x,\eta,k)\Subset V$, $\Gamma_k$ is a smoothing operator and we can check that 
\begin{equation}\label{e-gue130820IVa}
\Gamma_k=O(k^s):H^0_{{\rm loc\,}}(D_0,T^{*0,q}X)\To H^s_{{\rm loc\,}}(D_0,T^{*0,q}X),
\end{equation}
for every $s\in\mathbb N_0$. Combining \eqref{e-gue130820IVa}, \eqref{e-gue130820VIa} with \eqref{e-gue130820IIIa}, we deduce that 
\begin{equation}\label{e-gue130820VIIa}
\Td{\mathcal{I}}_k\mathcal{S}_{k}-\Td{\mathcal{I}}_k\hat\Pi^{(q)}_{k,s}\mathcal{S}_{k}\equiv0\mod O(k^{-\infty})\ \ \mbox{on $D_0$}.
\end{equation}
From \eqref{e-gue130820VIIa}, \eqref{e-gue130820IIa}, \eqref{e-gue130820VI} and note that $\hat\Pi^{(q)}_{k,s}h_k\equiv0\mod O(k^{-\infty})$, we get
\begin{equation}\label{e-gue130820VIIIa}
A_k\mathcal{S}_{k}-A_k\hat\Pi^{(q)}_{k,s}A^*_k\equiv0\mod O(k^{-\infty})\ \ \mbox{on $D_0$}.
\end{equation}
From \eqref{e-gue130820V}, we have $\mathcal{S}^*_k\mathcal{S}_k\equiv A_k\mathcal{S}_{k}\mod O(k^{-\infty})$ on $D_0$. From this, \eqref{e-gue1374I}, \eqref{e-gue13716I}, \eqref{e-gue130820VIIIa} and \eqref{e-gue130901III}, the theorem follows. 
\end{proof}

By using Theorem~\ref{t-gue13717} and repeat the proof of Theorem~\ref{t-gue130820}, we deduce 

\begin{thm}\label{t-gue130821}
Let $s$ be a local trivializing section of $L$ on an open subset $D\subset X$ and $\abs{s}^2_{h^L}=e^{-2\phi}$. We assume that there exist a $\lambda_0\in\Real$ and $x_0\in D$ such that $M^\phi_{x_0}-2\lambda_0\mathcal{L}_{x_0}$ is non-degenerate of constant signature $(n_-,n_+)$. Let $q\neq n_-$ and assume that $Y(q)$ holds at each point of $D$. Let $F_k:L^2_{(0,q)}(X,L^k)\To L^2_{(0,q)}(X,L^k)$ be a continuous operator and let $F^*_k:L^2_{(0,q)}(X,L^k)\To L^2_{(0,q)}(X,L^k)$ be the Hilbert space adjoint of $F_k$ with respect to $(\,\cdot\,|\,\cdot\,)_{h^{L^k}}$. Let
$\hat F_{k,s}$ and $\hat F^*_{k,s}$ be the localized operators of $F_{k,s}$ and $F^*_{k,s}$ respectively.
We fix $D_0\Subset D$, $D_0$ open. Let $V$ be as in \eqref{e-dhmpXII}. Assume that
\[\hat F_{k,s}-A_k=O(k^{-\infty}):H^s_{{\rm comp\,}}(D,T^{*0,q}X)\To H^s(D,T^{*0,q}X),\ \ \forall s\in\mathbb N_0,\]
where
\[\mbox{$A_k\equiv\frac{k^{2n-1}}{(2\pi)^{2n-1}}\int e^{ik<x-y,\eta>}\alpha(x,\eta,k)d\eta\mod O(k^{-\infty})$ at $T^*D_0\bigcap\Sigma$}\]
is a classical semi-classical pseudodifferential operator on $D$ of order $0$ from sections of $T^{*0,q}X$ to sections of $T^{*0,q}X$, where
\[\begin{split}&\mbox{$\alpha(x,\eta,k)\sim\sum_{j=0}\alpha_j(x,\eta)k^{-j}$ in $S^0_{{\rm loc\,}}(1;T^*D,T^{*0,q}X\boxtimes T^{*0,q}X)$},\\
&\alpha_j(x,\eta)\in C^\infty(T^*D,T^{*0,q}D\boxtimes T^{*0,q}D),\ \ j=0,1,\ldots,
\end{split}\]
with $\alpha(x,\eta,k)=0$ if $\abs{\eta}>M$, for some large $M>0$ and ${\rm Supp\,}\alpha(x,\eta,k)\bigcap T^*D_0\Subset V$. Put $P_k:=F_k\Pi^{(q)}_kF^*_k$ and let $\hat P_{k,s}$ be the localized operator of $P_k$. If $\Box^{(q)}_{b,k}$ has $O(k^{-n_0})$ small spectral gap on $D$ with respect to $F_k$ and $\Pi^{(q)}_k$ is $k$-negligible away the diagonal with respect to $F_k$ on $D$, then
\begin{equation}\label{e-gue130821}
\hat P_{k,s}\equiv0\mod O(k^{-\infty})\ \ \mbox{on $D_0$}.
\end{equation}
\end{thm}

\section{Sezg\"o kernel asymptotics and Kodairan embedding Theorems on CR manifolds with transversal CR $S^1$ actions} \label{s-saak}

In this section, we will offer some special classes of CR manifolds and CR line bundles such that the conditions in Theorem~\ref{t-gue130820} hold. 


Let $(X, T^{1,0}X)$ be a CR manifold. We assume that $X$ admits a $S^1$ action: $S^1\times X\To X$. We write $e^{i\theta}$, $0\leq\theta<2\pi$, to denote the $S^1$ action.
Let $T\in C^\infty(X,TX)$ be the real vector field given by 
\begin{equation}\label{e-gue131205I}
Tu=\frac{\pr}{\pr\theta}(u(e^{i\theta}x))|_{\theta=0},\ \ u\in C^\infty(X).
\end{equation}
We call $T$ the global vector field induced by the $S^1$ action. Note that we don't assume that this $S^1$ action is globally free.

\begin{defn}\label{d-gue131205II}
We say that the $S^1$ action $e^{i\theta}$, $0\leq\theta<2\pi$, is CR if 
\[[T,C^\infty(X,T^{1,0}X)]\subset C^\infty(X,T^{1,0}X).\]
Furthermore, we say that the $S^1$ action $e^{i\theta}$, $0\leq\theta<2\pi$, is transversal if for every point $x\in X$, 
\[T(x)\oplus T^{1,0}_xX\oplus T^{0,1}_xX=\Complex T_xX.\]
\end{defn}

Until further notice, we assume that $(X, T^{1,0}X)$ is a CR manifold with a transversal CR $S^1$ action
$e^{i\theta}$, $0\leq\theta<2\pi$ and we let $T$ be the global vector field induced by the $S^1$ action.

Fix $\theta_0\in[0,2\pi[$. Let 
\[de^{i\theta_0}:\Complex T_xX\To\Complex T_{e^{i\theta_0}x}X\]
denote the differential of the map $e^{i\theta_0}:X\To X$. 

\begin{defn}\label{d-gue131205III}
Let $U\subset X$ be an open set and let $V\in C^\infty(U,\Complex TX)$ be a vector field on $U$. We say that $V$ is $T$-rigid if
\[de^{i\theta_0}V(x)=V(x),\ \ \forall x\in e^{i\theta_0}U\bigcap U,\]
for every $\theta_0\in[0,2\pi[$ with $e^{i\theta_0}U\bigcap U\neq\emptyset$. 
\end{defn}

We also need 

\begin{defn}\label{d-gue131206}
Let $\langle\,\cdot\,|\,\cdot\,\rangle$ be a Hermitian metric on $\Complex TX$. We say that $\langle\,\cdot\,|\,\cdot\,\rangle$ is $T$-rigid if for $T$-rigid vector fields $V$ and $W$ on $U$, where $U\subset X$ is any open set, we have 
\[\langle\,V(x)\,|\,W(x)\,\rangle=\langle\,de^{i\theta_0}V(e^{i\theta_0}x)\,|\,de^{i\theta_0}W(e^{i\theta_0}x)\,\rangle,\ \ \forall x\in U, \theta_0\in[0,2\pi[.\]
\end{defn}


We are going to show that there exists a $T$-rigid Hermitian metric on $\Complex TX$. We need the following result due to Baouendi-Rothschild-Treves~\cite[section1]{BRT85} 

\begin{thm}\label{t-gue131206}
For every point $x_0\in X$, there exists local coordinates $x=(x_1,\ldots,x_{2n-1})=(z,\theta)=(z_1,\ldots,z_{n-1},\theta)$, 
$z_j=x_{2j-1}+ix_{2j}$, $j=1,\ldots,n-1$, $\theta=x_{2n-1}$, defined in some small neighbourhood $U$ of $x_0$ such that 
\begin{equation}\label{e-gue131206}
\begin{split}
&T=\frac{\pr}{\pr\theta},\\
&Z_j=\frac{\pr}{\pr z_j}+i\frac{\pr\varphi}{\pr z_j}(z)\frac{\pr}{\pr\theta},\ \ j=1,\ldots,n-1,
\end{split}
\end{equation}
where $Z_j(x)$, $j=1,\ldots,n-1$, form a basis of $T^{1,0}_xX$, for each $x\in U$, and $\varphi(z)\in C^\infty(U,\Real)$
independent of $\theta$.
\end{thm} 

Let $x$ and $U$ be as in Theorem~\ref{t-gue131206}. We call $x$ canonical coordinates and $U$ canonical coordinate patch. 

\begin{thm}\label{t-gue131206I}
There is a $T$-rigid Hermitian metric $\langle\,\cdot\,|\,\cdot\,\rangle$ on $\Complex TX$ such that $T^{1,0}X\perp T^{0,1}X$, $T\perp (T^{1,0}X\oplus T^{0,1}X)$, $\langle\,T\,|\,T\,\rangle=1$ and $\langle\,u\,|v\,\rangle$ is real if $u, v$ are real tangent vectors. 
\end{thm} 

\begin{proof}
Let  $\langle\,\cdot\,,\,\cdot\,\rangle$ be any Hermitian metric on $\Complex TX$ such that $T^{1,0}X\perp T^{0,1}X$, $T\perp (T^{1,0}X\oplus T^{0,1}X)$, $\langle\,T\,,\,T\,\rangle=1$ and $\ol{\langle u\,,\,v\,\rangle}=\langle\,\ol u\,,\,\ol v\,\rangle$, for all $u, v\in\Complex TX$. Let $x$ and $U$ be as in Theorem~\ref{t-gue131206}. On $U$, define 
\begin{equation}\label{e-gue131206I}
\langle\,Z_j\,|\,Z_t\,\rangle:=\int^{2\pi}_0\langle\,de^{i\theta}Z_j\,,\,de^{i\theta}Z_t\,\rangle d\theta,\ \ j,t=1,\ldots,n-1,
\end{equation}
where $Z_j$, $j=1,\ldots,n-1$ are as in \eqref{e-gue131206}. Since $Z_j(x)$, $j=1,\ldots,n-1$, form a basis of $T^{1,0}_xX$, for each $x\in U$, On $U$, \eqref{e-gue131206I} defines a $T$-rigid Hermitian metric
$\langle\,\cdot\,|\,\cdot\,\rangle$ on $T^{1,0}X$. We claim that the definition above is independent of the choice of canonical coordinates. 
Let 
$y=(y_1,\ldots,y_{2n-1})=(w,\gamma)$, $w_j=y_{2j-1}+iy_{2j}$, $j=1,\ldots,n-1$, $\gamma=y_{2n-1}$, be another canonical coordinates on $U$. Then, 
\begin{equation}\label{e-gue131206II}\begin{split}
&T=\frac{\pr}{\pr\gamma},\\
&\Td Z_j=\frac{\pr}{\pr w_j}+i\frac{\pr\Td\varphi}{\pr w_j}(w)\frac{\pr}{\pr\gamma},\ \ j=1,\ldots,n-1,
\end{split}
\end{equation}
where $\Td Z_j(y)$, $j=1,\ldots,n-1$, form a basis of $T^{1,0}_yX$, for each $y\in U$, and $\Td\varphi(w)\in C^\infty(U,\Real)$ independent of $\gamma$. As \eqref{e-gue131206I}, on $U$, we define 
\begin{equation}\label{e-gue131206III}
\langle\,\Td Z_j\,|\,\Td Z_t\,\rangle_1:=\int^{2\pi}_0\langle\,de^{i\theta}\Td Z_j\,,\,de^{i\theta}\Td Z_t\,\rangle d\theta,\ \ j,t=1,\ldots,n-1.
\end{equation}
On $U$, \eqref{e-gue131206III} defines a $T$-rigid Hermitian metric $\langle\,\cdot\,|\,\cdot\,\rangle_1$ on $T^{1,0}X$. We claim that $\langle\,\cdot\,|\,\cdot\,\rangle_1=\langle\,\cdot\,|\,\cdot\,\rangle$. 
From \eqref{e-gue131206II} and \eqref{e-gue131206}, it is not difficult to see that
\begin{equation}\label{e-gue131206IV}
\begin{split}
&w=(w_1,\ldots,w_{n-1})=(H_1(z),\ldots,H_{n-1}(z))=H(z),\ \ H_j(z)\in C^\infty,\ \ \forall j,\\
&\gamma=\theta+G(z),\ \ G(z)\in C^\infty,
\end{split}
\end{equation}
where for each $j=1,\ldots,n-1$, $H_j(z)$ is holomorphic. From \eqref{e-gue131206}, \eqref{e-gue131206II} and \eqref{e-gue131206IV}, it is not difficult to see that 
\begin{equation}\label{e-gue131206VI}
\begin{split}
&\Td Z_j=\sum^{n-1}_{t=1}c_{j,t}(x)Z_t,\ \ c_{j,t}\in C^\infty(U),\ \ j,t=1,\ldots,n-1,\\
&\mbox{$\left(c_{j,t}(x)\right)^{n-1}_{j,t=1}$ is invertible at every $x\in U$},\\
&Tc_{j,t}=0,\ \ j,t=1,\ldots,n-1.
\end{split}
\end{equation}

Let $\Gamma, \Lambda\in C^\infty(U,T^{1,0}X)$. We write 
\begin{equation}\label{e-gue131206V}
\begin{split}
&\Gamma=\sum^{n-1}_{j=1}a_j(x)Z_j=\sum^{n-1}_{j=1}\Td a_j(y)\Td Z_j,\ \ a_j, \Td a_j\in C^\infty(U),\ \ j=1,\ldots,n-1,\\
&\Lambda=\sum^{n-1}_{j=1}b_j(x)Z_j=\sum^{n-1}_{j=1}\Td b_j(y)\Td Z_j,\ \ b_j, \Td b_j\in C^\infty(U),\ \ j=1,\ldots,n-1.
\end{split}
\end{equation}
From \eqref{e-gue131206V} and \eqref{e-gue131206VI}, we can check that 
\begin{equation}\label{e-gue131206VII}
\begin{split}
&a_t=\sum^{n-1}_{j=1}\Td a_jc_{j,t},\ \ t=1,\ldots,n-1,\\
&b_t=\sum^{n-1}_{j=1}\Td b_jc_{j,t},\ \ t=1,\ldots,n-1.
\end{split}
\end{equation}
Now, by definition, 
\begin{equation}\label{e-gue131206VIII}
\begin{split}
\langle\,\Gamma\,|\,\Lambda\,\rangle_1&=\sum^{n-1}_{j,t=1}\Td a_j\ol{\Td b_t}\int^{2\pi}_0\langle\,de^{i\theta}\Td Z_j\,,\,de^{i\theta}\Td Z_t\,\rangle d\theta\\
&=\sum^{n-1}_{j,t=1}\Td a_j\ol{\Td b_t}\int^{2\pi}_0\langle\,de^{i\theta}(\sum^{n-1}_{s=1}c_{j,s}Z_s)\,,\,de^{i\theta}(\sum^{n-1}_{u=1}c_{t,u}Z_u)\,\rangle d\theta\\
&=\sum^{n-1}_{j,t=1}\sum^{n-1}_{s,u=1}c_{j,s}\ol c_{t,u}\Td a_j\ol{\Td b_t}\int^{2\pi}_0\langle\,de^{i\theta}Z_s\,,\,de^{i\theta}Z_u\,\rangle d\theta\\
&=\sum^{n-1}_{s,u=1}a_s\ol{b_u}\int^{2\pi}_0\langle\,de^{i\theta}Z_s\,,\,de^{i\theta}Z_u\,\rangle d\theta\\
&=\langle\,\Gamma\,|\,\Lambda\,\rangle.
\end{split}
\end{equation}
Here we used \eqref{e-gue131206VI}, \eqref{e-gue131206VII} and $de^{i\theta}(\sum^{n-1}_{s=1}c_{j,s}Z_s)=\sum^{n-1}_{s=1}c_{j,s}de^{i\theta}Z_s$, $s=1,\ldots,n-1$, since $Tc_{j,s}=0$, $j,s=1,\ldots,n-1$. 
Thus, \eqref{e-gue131206I} defines a $T$-rigid Hermitian metric on $T^{1,0}X$. We extend $\langle\,\cdot\,|\,\cdot\,\rangle$ to a $T$-rigid Hermitian metric on $\Complex TX$ by 
\[\begin{split}
&\langle\,u\,|\,v\,\rangle=\ol{\langle\,\ol u\,|\,\ol v\,\rangle},\ \ u, v\in T^{0,1}X,\\
&T\perp(T^{1,0}X\oplus T^{0,1}X),\ \ \langle\,T\,|\,T\,\rangle=1.
\end{split}\]
The theorem follows. 
\end{proof}

Until further notice, we fix a $T$-rigid Hermitian metric $\langle\,\cdot\,|\,\cdot\,\rangle$ on $\Complex TX$ such that $T^{1,0}X\perp T^{0,1}X$, $T\perp (T^{1,0}X\oplus T^{0,1}X)$, $\langle\,T\,|\,T\,\rangle=1$ and $\langle\,u\,|v\,\rangle$ is real if $u, v$ are real tangent vectors. The Hermitian metric $\langle\,\cdot\,|\,\cdot\,\rangle$ induces,
by duality, a Hermitian metric on $\Complex T^*X$ and also on the bundles of $(0, q)$ forms $T^{*0,q}X$, $q=0,1,\ldots,n-1$. As before, we denote all these induced metrics by $\langle\,\cdot\,|\,\cdot\,\rangle$.

\begin{defn}\label{d-gue131206I}
Let $U$ be an open subset of $X$. A function $u\in C^\infty (U)$ is said to be a $T$-rigid CR function on $U$ if $Tu=0$ and $Zu=0$ for all $Z\in C^\infty(U,T^{0,1}X)$.
\end{defn} 

\begin{defn}\label{d-gue131206II}
Let $L$ be a CR line bundle over $(X,T^{1,0}X)$. We say that $L$ is a $T$-rigid CR line bundle over $(X,T^{1,0}X)$ if $X$ can be covered with open sets $U_j$ with trivializing sections $s_j$, $j=1,2,\ldots$, such that the corresponding transition functions are $T$-rigid CR functions.
\end{defn}
 
Until further notice, we assume that $L$ is a $T$-rigid CR line bundle over $(X,T^{1,0}X)$. Then, by definition, $X$ can be covered with open sets $U_j$ with trivializing sections $s_j$, $j=1,2,\ldots$, such that the corresponding transition functions are $T$-rigid CR functions. In this section, when trivializing sections $s$ are used, we will assume that they are of this special form. 

Fix a Hermitian fiber metric $h^L$ on $L$ and we will denote by
$\phi$ the local weights of the Hermitian metric $h^L$ as \eqref{e-suV}. Since the transition functions are $T$-rigid CR functions, we can check that $T\phi$ is a well-defined global smooth function on $X$ and the Hermitian quadratic form $M^\phi_x$ is globally defined for every $x\in X$(see Definition~\ref{d-suIII-I} and Proposition~\ref{p-suI}). 

\begin{defn} \label{d-gue131207}
$h^L$ is said to be a $T$-rigid Hermitian fiber metric on $L$ if $T\phi=0$.
\end{defn} 

Until further notice, we assume that $h^L$ is a $T$-rigid Hermitian fiber metric on $L$ and $X$ is compact. For $k>0$, as before, we shall consider $(L^k,h^{L^k})$ and we will use the same notations as before. Since the transition functions are $T$-rigid CR functions, $Tu$ is well-defined, for every $u\in\Omega^{0,q}(X,L^k)$. For $m\in\mathbb Z$, put 
\begin{equation}\label{e-gue131207I}
A^{0,q}_m(X,L^k):=\set{u\in\Omega^{0,q}(X,L^k);\, Tu=imu}
\end{equation}
and let $\mathcal{A}^{0,q}_m(X,L^k)\subset L^2_{(0,q)}(X,L^k)$ be the completion of $A^{0,q}_m(X,L^k)$ with respect to $(\,\cdot\,|\,\cdot\,)_{h^{L^k}}$. It is easy to see that for any $m, m'\in\mathbb Z$, $m\neq m'$, 
\begin{equation}\label{e-gue131207II}
(\,u\,|\,v)_{h^{L^k}}=0,\ \ \forall u\in\mathcal{A}^{0,q}_m(X,L^k), v\in\mathcal{A}^{0,q}_{m'}(X,L^k).
\end{equation}
For $m\in\mathbb Z$, let 
\begin{equation}\label{e-gue131207III}
Q^{(q)}_{m,k}:L^2_{(0,q)}(X,L^k)\To\mathcal{A}^{0,q}_m(X,L^k)
\end{equation}
be the orthogonal projection with respect to $(\,\cdot\,|\,\cdot\,)_{h^{L^k}}$. Fix $\delta>0$. Take $\tau_\delta(x)\in C^\infty_0(]-\delta,\delta[)$, 
$0\leq\tau_\delta\leq1$ and $\tau_\delta=1$ on $[-\frac{\delta}{2},\frac{\delta}{2}]$. Let $F^{(q)}_{\delta,k}:L^2_{(0,q)}(X,L^k)\To L^2_{(0,q)}(X,L^k)$ be the continuous map given by 
\begin{equation}\label{e-gue131207IV}
\begin{split}
F^{(q)}_{\delta,k}:L^2_{(0,q)}(X,L^k)&\To L^2_{(0,q)}(X,L^k),\\
u&\To\sum_{m\in\mathbb Z}\tau_\delta(\frac{m}{k})(Q^{(q)}_{m,k}u).
\end{split}
\end{equation}
It is easy to see that $F^{(q)}_{\delta,k}$ is well-defined.  Moreover, it is not difficult to see that for every $m\in\mathbb Z$, we have
\begin{equation}\label{e-gue131207V}
\begin{split}
&\norm{TQ^{(q)}_{m,k}u}_{h^{L^k}}=\abs{m}\norm{Q^{(q)}_{m,k}u}_{h^{L^k}},\ \ \forall u\in L^2_{(0,q)}(X,L^k),\\
&\norm{TF^{(q)}_{\delta,k}u}_{h^{L^k}}\leq k\delta\norm{F^{(q)}_{\delta,k}u}_{h^{L^k}},\ \ \forall u\in L^2_{(0,q)}(X,L^k),
\end{split}
\end{equation}
and 
\begin{equation}\label{e-gue131207V-I}
\begin{split}
&Q^{(q)}_{m,k}:\Omega^{0,q}(X,L^k)\To A^{0,q}_m(X,L^k),\\
&F^{(q)}_{\delta,k}:\Omega^{0,q}(X,L^k)\To\bigcup_{-k\delta\leq m\leq k\delta}A^{0,q}_m(X,L^k).
\end{split}
\end{equation}

Since the Hermitian metric $\langle\,\cdot\,|\,\cdot\,\rangle$ and $h^{L^k}$ are all $T$-rigid, it is straightforward to see that (see section 5 in~\cite{Hsiao12})
\begin{equation}\label{e-gue131207VI}
\begin{split}
&\mbox{$\Box^{(q)}_{b,k}Q^{(q)}_{m,k}=Q^{(q)}_{m,k}\Box^{(q)}_{b,k}$ on $\Omega^{0,q}(X,L^k)$, $\forall m\in\mathbb Z$},\\
&\mbox{$\Box^{(q)}_{b,k}F^{(q)}_{\delta,k}=F^{(q)}_{\delta,k}\Box^{(q)}_{b,k}$ on $\Omega^{0,q}(X,L^k)$},\\
&\mbox{$\ddbar_{b,k}Q^{(q)}_{m,k}=Q^{(q+1)}_{m,k}\ddbar_{b,k}$ on $\Omega^{0,q}(X,L^k)$, $\forall m\in\mathbb Z$, $q=0,1,\ldots,n-2$},\\
&\mbox{$\ddbar_{b,k}F^{(q)}_{\delta,k}=F^{(q+1)}_{\delta,k}\ddbar_{b,k}$ on $\Omega^{0,q}(X,L^k)$, $q=0,1,\ldots,n-2$},\\
&\mbox{$\ddbar^*_{b,k}Q^{(q)}_{m,k}=Q^{(q-1)}_{m,k}\ddbar^*_{b,k}$ on $\Omega^{0,q}(X,L^k)$, $\forall m\in\mathbb Z$, $q=1,\ldots,n-1$},\\
&\mbox{$\ddbar^*_{b,k}F^{(q)}_{\delta,k}=F^{(q-1)}_{\delta,k}\ddbar^*_{b,k}$ on $\Omega^{0,q}(X,L^k)$, $q=1,\ldots,n-1$}.
\end{split}
\end{equation}

By using elementary Fourier analysis, it is straightforward to see that for every $u\in\Omega^{0,q}(X,L^k)$, 
\begin{equation}\label{e-gue131207VII}
\begin{split}
&\mbox{$\lim\limits_{N\To\infty}\sum\limits^N_{m=-N}Q^{(q)}_{m,k}u\To u$ in $C^\infty$ Topology},\\
&\sum^N_{m=-N}\norm{Q^{(q)}_{m,k}u}^2_{h^{L^k}}\leq\norm{u}^2_{h^{L^k}},\ \ \forall N\in\mathbb N_0.
\end{split}
\end{equation}
Thus, for every $u\in L^2_{(0,q)}(X,L^k)$, 
\begin{equation}\label{e-gue131207VIII}
\begin{split}
&\mbox{$\lim\limits_{N\To\infty}\sum\limits^N_{m=-N}Q^{(q)}_{m,k}u\To u$ in $L^2_{(0,q)}(X,L^k)$},\\
&\sum^N_{m=-N}\norm{Q^{(q)}_{m,k}u}^2_{h^{L^k}}\leq\norm{u}^2_{h^{L^k}},\ \ \forall N\in\mathbb N_0.
\end{split}
\end{equation}

Now, we assume that $M^\phi_x$ is non-degenerate of constant signature $(n_-,n_+)$, for every $x\in X$. 
The following is essentially follows from Kohn's $L^2$ estimates (see Chen-Shaw~\cite{CS01}). We omit the proof. 

\begin{thm}\label{t-gue131207}
With the assumptions and notations above, let $q\neq n_-$. For every $u\in\Omega^{0,q}(X,L^k)$, we have 
\begin{equation}\label{e-gue131207a}
\norm{\Box^{(q)}_{b,k}u}^2_{h^{L^k}}+k\abs{(\,Tu\,|\,u\,)_{h^{L^k}}}\geq ck^2\norm{u}^2_{h^{L^k}},
\end{equation}
where $c>0$ is a constant independent of $k$ and $u$. 
\end{thm}

From \eqref{e-gue131207V} and \eqref{e-gue131207a}, we deduce 

\begin{thm}\label{t-gue131207I}
With the assumptions and notations above, let $q\neq n_-$. If $\delta>0$ is small enough, then for every $u\in\Omega^{0,q}(X,L^k)$, we have 
\begin{equation}\label{e-gue131207aI}
\norm{\Box^{(q)}_{b,k}(F^{(q)}_{\delta,k}u)}^2_{h^{L^k}}\geq c_1k^2\norm{F^{(q)}_{\delta,k} u}^2_{h^{L^k}},
\end{equation}
where $c_1>0$ is a constant independent of $k$ and $u$. 
\end{thm}

Now, we assume that $Y(q)$ holds at each point of $X$. Since $X$ is compact, by the classical result of Kohn~\cite[Props.\,8.4.8-9]{CS01}, condition $Y(q)$ implies that
$\Box^{(q)}_{b,k}$ is hypoelliptic, has compact resolvent and the strong Hodge decomposition holds. Let ${\rm Spec\,}\Box^{(q)}_{b,k}$ denote the spectrum of $\Box^{(q)}_{b,k}$. Then ${\rm Spec\,}\Box^{(q)}_{b,k}$ is a discrete subset of $[0,\infty[$, ${\rm Spec\,}\Box^{(q)}_{b,k}$ is the set of all eigenvalues of $\Box^{(q)}_{b,k}$. For $\mu\in{\rm Spec\,}\Box^{(q)}_{b,k}$, put 
\begin{equation}\label{e-gue131207aII}
H^q_{b,\mu}(X,L^k)=\set{u\in L^2_{(0,q)}(X,L^k);\, \Box^{(q)}_{b,k}u=\mu u}
\end{equation} 
and let
\begin{equation}\label{e-gue131207aIIb}
\Pi^{(q)}_{k,\mu}:L^2_{(0,q)}(X,L^k)\To H^q_{b,\mu}(X,L^k)
\end{equation} 
be the orthogonal projection. 

We notice that $H^q_{b,\mu}(X,L^k)\subset\Omega^{0,q}(X,L^k)$, $\forall\mu\in{\rm Spec\,}\Box^{(q)}_{b,k}$ and for every $\lambda\geq0$,
\begin{equation}\label{e-gue131207aIII}
H^q_{b,\leq\lambda}(X,L^k)=\oplus_{\mu\in{\rm Spec\,}\Box^{(q)}_{b,k}, 0\leq\mu\leq\lambda}H^q_{b,\mu}(X,L^k). 
\end{equation}

\begin{thm}\label{t-gue131208}
With the assumptions and notations above, let $q=n_-$. If $\delta>0$ is small enough, then for every $u\in\Omega^{0,q}(X,L^k)$, we have 
\begin{equation}\label{e-gue131216}
F^{(q)}_{\delta,k}\Pi^{(q)}_{\mu,k}u=0,\ \ \forall\mu\in{\rm Spec\,}\Box^{(q)}_{b,k},\ 0<\mu\leq k\delta,
\end{equation}
and
\begin{equation}\label{e-gue131208}
\norm{F^{(q)}_{\delta,k}(I-\Pi^{(q)}_{k})u}_{h^{L^k}}\leq\frac{1}{k\delta}\norm{\Box^{(q)}_{b,k}u}_{h^{L^k}}.
\end{equation}

In particular, if $\delta>0$ is small enough then for every $D\Subset X$, $\Box^{(q)}_{b,k}$ has $O(k^{-n_0})$ small spectral gap on $D$ with respect to $F^{(q)}_{\delta,k}$ in the sense of Definition~\ref{d-gue130820}. 
\end{thm}

\begin{proof}
Let $\delta>0$ be a small constant. For $u\in\Omega^{0,q}(X,L^k)$, we have 
\begin{equation}\label{e-gue131208I}
(I-\Pi^{(q)}_k)u=\sum_{\mu\in{\rm Spec\,}\Box^{(q)}_k,0<\mu\leq k\delta}\Pi^{(q)}_{k,\mu}u+\Pi^{(q)}_{k,>k\delta}u.
\end{equation}
We claim that for every $\mu\in{\rm Spec\,}\Box^{(q)}_k,0<\mu\leq k\delta$ and every $u\in\Omega^{0,q}(X,L^k)$, 
\begin{equation}\label{e-gue131208II}
F^{(q)}_{\delta,k}\Pi^{(q)}_{k,\mu}u=0
\end{equation}
if $\delta>0$ is small enough. Fix $\mu\in{\rm Spec\,}\Box^{(q)}_k,0<\mu\leq k\delta$ and $u\in\Omega^{0,q}(X,L^k)$. Since $q+1\neq n_-$, from \eqref{e-gue131207VI} and \eqref{e-gue131207aI}, we have 
\begin{equation}\label{e-gue131208III}
\norm{\Box^{(q+1)}_{b,k}F^{(q+1)}_{\delta,k}\ddbar_{b,k}\Pi^{(q)}_{k,\mu}u}^2_{h^{L^k}}\geq c_1k^2\norm{F^{(q+1)}_{\delta,k}\ddbar_{b,k}\Pi^{(q)}_{k,\mu}u}^2_{h^{L^k}},\end{equation}
where $c_1>0$ is a constant independent of $k$ and $u$. It is easy to see that 
\[\Box^{(q+1)}_{b,k}F^{(q+1)}_{\delta,k}\ddbar_{b,k}\Pi^{(q)}_{k,\mu}u=\mu F^{(q+1)}_{\delta,k}\ddbar_{b,k}\Pi^{(q)}_{k,\mu}u.\]
Thus, 
\begin{equation}\label{e-gue131208IV}
\norm{\Box^{(q+1)}_{b,k}F^{(q+1)}_{\delta,k}\ddbar_{b,k}\Pi^{(q)}_{k,\mu}u}^2_{h^{L^k}}\leq k^2\delta^2\norm{F^{(q+1)}_{\delta,k}\ddbar_{b,k}\Pi^{(q)}_{k,\mu}u}^2_{h^{L^k}}.\end{equation}
From \eqref{e-gue131208III} and \eqref{e-gue131208IV}, we conclude that if $\delta>0$ is small enough then
\[F^{(q+1)}_{\delta,k}\ddbar_{b,k}\Pi^{(q)}_{k,\mu}u=\ddbar_{b,k}F^{(q)}_{\delta,k}\Pi^{(q)}_{k,\mu}u=0.\]
Similarly, we have 
\[F^{(q-1)}_{\delta,k}\ddbar^*_{b,k}\Pi^{(q)}_{k,\mu}u=\ddbar^*_{b,k}F^{(q)}_{\delta,k}\Pi^{(q)}_{k,\mu}u=0.\]
Hence, 
\begin{equation}\label{e-gue131208V}
F^{(q)}_{\delta,k}\Pi^{(q)}_{k,\mu}u=\frac{1}{\mu}\Box^{(q)}_{b,k}F^{(q)}_{\delta,k}\Pi^{(q)}_{k,\mu}u=0. 
\end{equation}
From \eqref{e-gue131208V}, the claim \eqref{e-gue131208II} follows. 

Now, from \eqref{e-gue131208I} and \eqref{e-gue131208II}, if $\delta>0$ is small enough, then
\begin{equation}\label{e-gue131208VI}
\begin{split}
\norm{F^{(q)}_{\delta,k}(I-\Pi^{(q)}_k)u)}_{h^{L^k}}&=\norm{F^{(q)}_{\delta,k}\Pi^{(q)}_{k,>k\delta}u}_{h^{L^k}}\leq\norm{\Pi^{(q)}_{k,>k\delta}u}_{h^{L^k}}\\
&\leq\frac{1}{k\delta}\norm{\Box^{(q)}_{b,k}\Pi^{(q)}_{k,>k\delta}u}_{h^{L^k}}=\frac{1}{k\delta}\norm{\Pi^{(q)}_{k,>k\delta}\Box^{(q)}_{b,k}u}_{h^{L^k}}\leq\frac{1}{k\delta}\norm{\Box^{(q)}_{b,k}u}_{h^{L^k}},
\end{split}
\end{equation}
for every $u\in\Omega^{0,q}(X,L^k)$. From \eqref{e-gue131208VI}, \eqref{e-gue131208} follows. 
\end{proof} 

Until further notice, we fix $\delta>0$ and we assume that $\delta>0$ is small enough so that \eqref{e-gue131216}, \eqref{e-gue131208} hold and 
\begin{equation}\label{e-gue131208VII}
\mbox{$M^\phi_x-2\lambda\mathcal{L}_x$ is non-degenerate of constant signature, for every $\lambda\in]-\delta,\delta[$ and $x\in X$. }
\end{equation}

Let $D\subset X$ be a canonical coordinate patch and let $x=(x_1,\ldots,x_{2n-1})$ be a canonical coordinates on $D$ as in Theorem~\ref{t-gue131206}. We identify $D$ with $W\times]-\varepsilon,\varepsilon[\subset\Real^{2n-1}$, where $W$ is some open set in $\Real^{2n-2}$ and $\varepsilon>0$. Until further notice, we work with canonical coordinates $x=(x_1,\ldots,x_{2n-1})$. Let $\eta=(\eta_1,\ldots,\eta_{2n-1})$ be the dual coordinates of $x$. Let $s$ be a local trivializing section of $L$ on $D$, $\abs{s}^2_{h^L}=e^{-2\phi}$. Let $M>0$ be a large constant so that for every $(x,\eta)\in T^*D$ if $\abs{\eta'}>\frac{M}{2}$ then $(x,\eta)\notin\Sigma$, where $\eta'=(\eta_1,\ldots,\eta_{2n-2})$, $\abs{\eta'}=\sqrt{\sum^{2n-2}_{j=1}\abs{\eta_j}^2}$. 
Fix $D_0\Subset D$. Let $D'\Subset D$ be an open neighbourhood of $D_0$. Put 
\begin{equation}\label{e-gue131209b}
V:=\set{(x,\eta)\in T^*D';\, \abs{\eta'}<M, \abs{\eta_{2n-1}}<\delta}.
\end{equation}
Then $\ol V\subset T^*D$ and $\ol V\bigcap\Sigma\subset\Sigma'$, where $\Sigma'$ is given by \eqref{e-dhmpXIa}. 
Put 
\begin{equation}\label{e-gue131209}
\hat B_{k,s}=\frac{k^{2n-1}}{(2\pi)^{2n-1}}\int e^{ik<x-y,\eta>}\tau_\delta(\eta_{2n-1})d\eta.
\end{equation}
Let 
$\hat B^*_{k,s}$ be the adjoint of $\hat B_{k,s}$ with respect to $(\,\cdot\,|\,\cdot\,)$. Then, 
\begin{equation}\label{e-gue131209I}
\hat B^*_{k,s}=\frac{k^{2n-1}}{(2\pi)^{2n-1}}\int e^{ik<x-y,\eta>}\tau_\delta(\eta_{2n-1})d\eta.
\end{equation}
It is clearly that 
\[\mbox{$\hat B^*_{k,s}\equiv\frac{k^{2n-1}}{(2\pi)^{2n-1}}\int e^{ik<x-y,\eta>}\alpha(x,\eta,k)d\eta\mod O(k^{-\infty})$ at $T^*D_0\bigcap\Sigma$}\]
is a classical semi-classical pseudodifferential operator on $D$ of order $0$ from sections of $T^{*0,q}X$ to sections of $T^{*0,q}X$, where
\[\begin{split}&\mbox{$\alpha(x,\eta,k)\sim\sum_{j=0}\alpha_j(x,\eta)k^{-j}$ in $S^0_{{\rm loc\,}}(1;T^*D,T^{*0,q}X\boxtimes T^{*0,q}X)$},\\
&\alpha_j(x,\eta)\in C^\infty(T^*D,T^{*0,q}D\boxtimes T^{*0,q}D),\ \ j=0,1,\ldots,
\end{split}\]
with $\alpha(x,\eta,k)=0$ if $\abs{\eta}>M$, for some large $M>0$ and ${\rm Supp\,}\alpha(x,\eta,k)\bigcap T^*D_0\Subset V$.

Let $\hat F^{(q)}_{\delta,k,s}$ be the localized operator of $F^{(q)}_{\delta,k}$. 

\begin{lem}\label{l-gue131209}
$\hat F^{(q)}_{\delta,k,s}=\hat B^1_{k,s}+\hat B_{k,s}$ on $D$, 
where 
\[\hat B^1_{k,s}=O(k^{-\infty}):H^s_{{\rm comp\,}}(D,T^{*0,q}X)\To H^s_{{\rm loc\,}}(D,T^{*0,q}X),\ \ \forall s\in\mathbb N_0.\]
\end{lem}

\begin{proof}
Let $u\in\Omega^{0,q}_0(D,L^k)$, $u=s^k\Td u$, $\Td u\in\Omega^{0,q}(D)$. We also write $y=(y_1,\ldots,y_{2n-1})$ to denote the canonical coordinates $x$. It is easy to see that on $D$, 
\begin{equation}\label{e-gue131217}
\hat F^{(q)}_{\delta,k,s}u(y)=\frac{1}{2\pi}\sum_{m\in\mathbb Z}\tau_\delta(\frac{m}{k})e^{imy_{2n-1}}\int^{\pi}_{-\pi}e^{-imt}u(e^{it}\circ y')dt,\ \ \forall u\in\Omega^{0,q}_0(D),
\end{equation}
where $y'=(y_1,\ldots,y_{2n-2})$. Fix $D'\Subset D$ and let $\chi(y_{2n-1})\in C^\infty_0(]-\pi,\pi[)$ such that $\chi(y_{2n-1})=1$ for every $(y',y_{2n-1})\in D'$. Let $\hat B^1_{k,s}:\Omega^{0,q}_0(D')\To\Omega^{0,q}(D')$ be the continuous operator given by 
\begin{equation}\label{e-gue131217I}
\begin{split}
\hat B^1_{k,s}:\Omega^{0,q}_0(D')&\To\Omega^{0,q}(D'),\\
u&\To\frac{1}{(2\pi)^2}\sum_{m\in\mathbb Z}
\int_{\abs{t}\leq\pi}e^{i<x_{2n-1}-y_{2n-1},\eta_{2n-1}>}\tau_\delta(\frac{\eta_{2n-1}}{k})\\
&\quad\quad\times (1-\chi(y_{2n-1}))e^{imy_{2n-1}}e^{-imt}u(e^{it}\circ y')dtd\eta_{2n-1}dy_{2n-1}.
\end{split}
\end{equation}
By using integration by parts with respect to $\eta_{2n-1}$, it is easy to see that the integral \eqref{e-gue131217I} is well-defined. Moreover, we can integrate by parts with respect to $\eta_{2n-1}$ and $y_{2n-1}$ several times and conclude that 
\begin{equation}\label{e-gue131217II}
\hat B^1_{k,s}=O(k^{-\infty}):H^s_{{\rm comp\,}}(D,T^{*0,q}X)\To H^s_{{\rm loc\,}}(D,T^{*0,q}X),\ \ \forall s\in\mathbb N_0.
\end{equation}
Now, we claim that 
\begin{equation}\label{e-gue131217III}
\hat B_{k,s}+\hat B^1_{k,s}=\hat F^{(q)}_{\delta,k,s}\ \ \mbox{on $\Omega^{0,q}_0(D')$}.
\end{equation} 
Let $u\in\Omega^{0,q}_0(D')$. From \eqref{e-gue131209} and Fourier inversion formula, it is straightforward to see that 
\begin{equation}\label{e-gue131217IV}
\begin{split}
\hat B_{k,s}u(x)&=\frac{1}{(2\pi)^2}\sum_{m\in\mathbb Z}
\int_{\abs{t}\leq\pi}e^{i<x_{2n-1}-y_{2n-1},\eta_{2n-1}>}\tau_\delta(\frac{\eta_{2n-1}}{k})\\
&\quad\times\chi(y_{2n-1})e^{imy_{2n-1}}e^{-imt}u(e^{it}\circ x')dtd\eta_{2n-1}dy_{2n-1}.
\end{split}
\end{equation}
From \eqref{e-gue131217IV} and \eqref{e-gue131217I}, we have
\begin{equation}\label{e-gue131217V}
\begin{split}
&(\hat B_{k,s}+\hat B^1_{k,s})u(x)\\
&=\frac{1}{(2\pi)^2}\sum_{m\in\mathbb Z}
\int_{\abs{t}\leq\pi}e^{i<x_{2n-1}-y_{2n-1},\eta_{2n-1}>}\tau_\delta(\frac{\eta_{2n-1}}{k})e^{imy_{2n-1}}e^{-imt}u(e^{it}\circ x')dtd\eta_{2n-1}dy_{2n-1}.
\end{split}\end{equation}
From Fourier inversion formula and notice that for every $m\in\mathbb Z$,
\[\int e^{imy_{2n-1}}e^{-iy_{2n-1}\eta_{2n-1}}dy_{2n-1}=2\pi\delta_m(\eta_{2n-1}),\] 
where the integral above is defined as an oscillatory integral and $\delta_m$ is the Dirac measure at $m$(see Chapter 7.2 in H\"ormander~\cite{Hor03}), \eqref{e-gue131217V} becomes 
\begin{equation}\label{e-gue131217VI}
\begin{split}
&(\hat B_{k,s}+\hat B^1_{k,s})u(x)\\
&=\frac{1}{2\pi}\sum_{m\in\mathbb Z}\tau_\delta(\frac{m}{k})e^{ix_{2n-1}m}
\int_{\abs{t}\leq\pi}e^{-imt}u(e^{it}\circ x')dt\\
&=\hat F^{(q)}_{\delta,k,s}u(x).
\end{split}\end{equation}
Here we used \eqref{e-gue131217}. 

From \eqref{e-gue131217VI}, the claim \eqref{e-gue131217III} follows. From \eqref{e-gue131217III} and \eqref{e-gue131217II}, the lemma follows. 
\end{proof}

We need 

\begin{lem}\label{l-gue150813}
Let $D\subset X$ be a canonical coordinate patch of $X$. Then, $\Pi^{(q)}_k$ is $k$-negligible away the diagonal with respect to $F^{(q)}_{\delta,k}$ on $D$. 
\end{lem}

\begin{proof}
Let $\chi, \chi_1\in C^\infty_0(D)$, $\chi_1=1$ on some neighbourhood of ${\rm Supp\,}\chi$. Let $u\in H^q_b(X,L^k)$ with $\norm{u}_{h^{L^k}}=1$. In view of Theorem~\ref{t-gue13718}, we see that there is a constant $C>0$ independent of $k$ and $u$ such that 
\begin{equation}\label{e-gue150813}
\abs{u(x)}^2_{h^{L^k}}\leq Ck^n,\ \ \forall x\in X.
\end{equation}
Let $x=(x_1,\ldots,x_{2n-1})=(x',x_{2n-1})$ be canonical coordinates on $D$. Put $v=(1-\chi_1)u$. It is straightforward to see that on $D$, 
\begin{equation}\label{e-gue150813I}
\begin{split}
\chi F^{(q)}_{\delta,k}(1-\chi_1)u(x)=\frac{1}{(2\pi)^2}\sum_{m\in\mathbb Z,\abs{m}\leq 2k\delta}&\int_{\abs{t}\leq\pi}e^{i<x_{2n-1}-y_{2n-1},\eta_{2n-1}>}\chi(x)\tau_\delta(\frac{\eta_{2n-1}}{k})e^{imy_{2n-1}}\\
&\quad\quad\times e^{-imt}v(e^{it}\circ x')dtd\eta_{2n-1}dy_{2n-1}.
\end{split}
\end{equation}
Let $\varepsilon>0$ be a small constant so that for every $(x_1,\ldots,x_{2n-1})\in{\rm Supp\,}\chi$, we have 
\begin{equation}\label{e-gue150813II}
(x_1,\ldots,x_{2n-2},y_{2n-1})\in\set{x\in D;\, \chi_1(x)=1},\ \ \forall\abs{y_{2n-1}-x_{2n-1}}<\varepsilon.
\end{equation}
Let $\psi\in C^\infty_0(]-1,1[)$, $\psi=1$ on $[\frac{1}{2},\frac{1}{2}]$. Put 
\begin{equation}\label{e-gue150813III}
\begin{split}
I_0(x)=\frac{1}{(2\pi)^2}\sum_{m\in\mathbb Z,\abs{m}\leq 2k\delta}&\int_{\abs{t}\leq\pi}e^{i<x_{2n-1}-y_{2n-1},\eta_{2n-1}>}(1-\psi(\frac{x_{2n-1}-y_{2n-1}}{\varepsilon}))\chi(x)\tau_\delta(\frac{\eta_{2n-1}}{k})e^{imy_{2n-1}}\\
&\quad\quad\times e^{-imt}v(e^{it}\circ x')dtd\eta_{2n-1}dy_{2n-1}, 
\end{split}
\end{equation}
\begin{equation}\label{e-gue150813IV}
\begin{split}
I_1(x)=\frac{1}{(2\pi)^2}\sum_{m\in\mathbb Z}&\int_{\abs{t}\leq\pi}e^{i<x_{2n-1}-y_{2n-1},\eta_{2n-1}>}\psi(\frac{x_{2n-1}-y_{2n-1}}{\varepsilon})\chi(x)\tau_\delta(\frac{\eta_{2n-1}}{k})e^{imy_{2n-1}}\\
&\quad\quad\times e^{-imt}v(e^{it}\circ x')dtd\eta_{2n-1}dy_{2n-1}, 
\end{split}
\end{equation}
and 
\begin{equation}\label{e-gue150813V}
\begin{split}
I_2(x)=\frac{1}{(2\pi)^2}\sum_{m\in\mathbb Z,\abs{m}>2k\delta}&\int_{\abs{t}\leq\pi}e^{i<x_{2n-1}-y_{2n-1},\eta_{2n-1}>}\psi(\frac{x_{2n-1}-y_{2n-1}}{\varepsilon})\chi(x)\tau_\delta(\frac{\eta_{2n-1}}{k})e^{imy_{2n-1}}\\
&\quad\quad\times e^{-imt}v(e^{it}\circ x')dtd\eta_{2n-1}dy_{2n-1}.
\end{split}
\end{equation}
It is clearly that on $D$, 
\begin{equation}\label{e-gue150813VI}
\chi F^{(q)}_{\delta,k}(1-\chi_1)u(x)=I_0(x)+I_1(x)-I_2(x).
\end{equation}
By using integration by parts with respect to $\eta_{2n-1}$ several times and \eqref{e-gue150813}, we conclude that for every $N>0$ and $m\in\mathbb N$, there is a constant $C_{N,m}>0$ independent of $u$ and $k$ such that 
\begin{equation}\label{e-gue150813VII}
\norm{I_0(x)}_{C^m(D)}\leq C_{N,m}k^{-N}.
\end{equation}
Similarly, by using integration by parts with respect to $y_{2n-1}$ several times and \eqref{e-gue150813}, we conclude that for every $N>0$ and $m\in\mathbb N$, there is a constant $\Td C_{N,m}>0$ independent of $u$ and $k$ such that 
\begin{equation}\label{e-gue150813VIII}
\norm{I_2(x)}_{C^m(D)}\leq\Td C_{N,m}k^{-N}.
\end{equation}
We can check that 
\begin{equation}\label{e-gue150813aVIII}
\begin{split}
I_1(x)=\frac{1}{2\pi}\int e^{i<x_{2n-1}-y_{2n-1},\eta_{2n-1}>}\psi(\frac{x_{2n-1}-y_{2n-1}}{\varepsilon})\chi(x)\tau_\delta(\frac{\eta_{2n-1}}{k})v(x',y_{2n-1})d\eta_{2n-1}dy_{2n-1}. 
\end{split}
\end{equation}
From \eqref{e-gue150813II} and \eqref{e-gue150813aVIII}, we deduce that 
\begin{equation}\label{e-gue150813bVIII}
\mbox{$I_1(x)=0$ on $D$.}
\end{equation}
From \eqref{e-gue150813VI}, \eqref{e-gue150813VII}, \eqref{e-gue150813VIII} and \eqref{e-gue150813bVIII}, we conclude that for every $N>0$ and $m\in\mathbb N$, there is a constant $\hat C_{N,m}>0$ independent of $u$ and $k$ such that 
\begin{equation}\label{e-gue150813gVIII}
\norm{\chi F^{(q)}_{\delta,k}(1-\chi_1)u(x)}_{C^m(D)}\leq\hat C_{N,m}k^{-N}. 
\end{equation}
From \eqref{e-gue150813} and \eqref{e-gue150813gVIII}, it is not difficult to see that 
\begin{equation}\label{e-gue150813vgVIII}
\mbox{$\sum\limits^{d_k}_{j=1}\abs{\chi F^{(q)}_{\delta,k}(1-\chi_1)f_j(x)}^2_{h^{L^k}}\equiv0\mod O(k^{-\infty})$ on $D$},
\end{equation}
where $\set{f_1,\ldots,f_{d_k}}$ is an orthonormal basis for $H^q_b(X,L^k)$. From \eqref{e-gue150813vgVIII}, the lemma follows. 
\end{proof}

From Lemma~\ref{l-gue150813}, Lemma~\ref{l-gue131209} and Theorem~\ref{t-gue131208}, we see that the operator $F^{(q)}_{\delta,k}:L^2_{(0,q)}(X,L^k)\To L^2_{(0,q)}(X,L^k)$ satisfies all the conditions in Theorem~\ref{t-gue130820}. Summing up, we obtain one of the main results of this work 

\begin{thm}\label{t-gue131209}
Let $(X, T^{1,0}X)$ be a compact CR manifold with a transversal CR $S^1$ action and let $T$ be the global vector field induced by the $S^1$ action. Let $L$ be a $T$-rigid  CR line bundle over $X$ with a $T$-rigid Hermitian fiber metric $h^L$. We assume that $Y(q)$ holds at each point of $X$ and $M^{\phi}_x$ is non-degenerate of constant signature $(n_-,n_+)$, for every $x\in X$.   
Let $s$ be a local trivializing section of $L$ on an  canonical coordinate patch $D\subset X$, $\abs{s}^2_{h^L}=e^{-2\phi}$. Fix $D_0\Subset D$. Let $F^{(q)}_{\delta,k}:L^2_{(0,q)}(X,L^k)\To L^2_{(0,q)}(X,L^k)$ be the continuous operator given by \eqref{e-gue131207IV} and let $F^{(q),*}_{\delta,k}:L^2_{(0,q)}(X,L^k)\To L^2_{(0,q)}(X,L^k)$ be the adjoint of $F^{(q)}_{\delta,k}$ with respect to $(\,\cdot\,|\,\cdot\,)_{h^{L^k}}$. Put $P_k:=F^{(q)}_{\delta,k}\Pi^{(q)}_kF^{(q),*}_{\delta,k}$ and let $\hat P_{k,s}$ be the localized operator of $P_k$. If $q\neq n_-$, then $\hat P_{k,s}\equiv0\mod O(k^{-\infty})$ on $D_0$. If $q=n_-$, then
\begin{equation}\label{e-gue131209aII}
\hat P_{k,s}(x,y)\equiv\int e^{ik\varphi(x,y,s)}g(x,y,s,k)ds\mod O(k^{-\infty})
\end{equation}
on $D_0$, where $\varphi(x,y,s)\in C^\infty(\Omega)$ is as in Theorem~\ref{t-dcgewI}, \eqref{e-guew13627},
\begin{equation}\label{e-gue131209aIII}
\begin{split}
&g(x,y,s,k)\in S^{n}_{{\rm loc\,}}\big(1;\Omega,T^{*0,q}X\boxtimes T^{*0,q}X\big)\bigcap C^\infty_0\big(\Omega,T^{*0,q}X\boxtimes T^{*0,q}X\big),\\
&g(x,y,s,k)\sim\sum^\infty_{j=0}g_j(x,y,s)k^{n-j}\text{ in }S^{n}_{{\rm loc\,}}
\big(1;\Omega,T^{*0,q}X\boxtimes T^{*0,q}X\big), \\
&g_j(x,y,s)\in C^\infty_0\big(\Omega,T^{*0,q}X\boxtimes T^{*0,q}X\big),\ \ j=0,1,2,\ldots,
\end{split}
\end{equation}
and for every $(x,x,s)\in\Omega$, 
\begin{equation}\label{e-gue131209aIV}
\begin{split}
&g_0(x,x,s)\\
&=(2\pi)^{-n}\abs{\det\bigr(M^\phi_x-2s\mathcal{L}_x\bigr)}\abs{\tau_\delta(s)}^2\mathcal{\pi}_{(x,s,n_-)}.
\end{split}
\end{equation}
Here
\[
\begin{split}
\Omega:=&\{(x,y,s)\in D\times D\times\Real;\, (x,-2{\rm Im\,}\ddbar_b\phi(x)+s\omega_0(x))\in V\bigcap\Sigma,\\
&\quad\mbox{$(y,-2{\rm Im\,}\ddbar_b\phi(y)+s\omega_0(y))\in V\bigcap\Sigma$, $\abs{x-y}<\varepsilon$, for some $\varepsilon>0$}\},
\end{split}\]
$V$ is given by \eqref{e-gue131209b}, $\mathcal{\pi}_{(x,s,n_-)}:T^{*0,q}_pX\To\mathcal{N}(x,s,n_-)$ is the orthogonal projection with respect to $\langle\,\cdot\,|\,\cdot\,\rangle$,  $\mathcal{N}(x,s,n_-)$ is given by \eqref{e-gue1373III}, 
\[\abs{\det\bigr(M^\phi_x-2s\mathcal{L}_x\bigr)}=\abs{\lambda_1(s)}\abs{\lambda_2(s)}\cdots\abs{\lambda_{n-1}(s)},\] 
$\lambda_1(s),\ldots,\lambda_{n-1}(s)$ are eigenvalues of the Hermitian quadratic form $M^\phi_x-2s_0\mathcal{L}_x$ with respect to $\langle\,\cdot\,|\,\cdot\,\rangle$.
\end{thm}

We recall "$T$-rigid positive CR line bundle"(see Definition~\ref{d-gue131209})

\begin{thm}\label{t-gue131209I}
Let $(X, T^{1,0}X)$ be a compact CR manifold with a transversal CR $S^1$ action and let $T$ be the global vector field induced by the $S^1$ action. If there is a $T$-rigid positive CR line bundle over $X$, then $X$ can be CR embedded into $\mathbb C\mathbb P^N$, for some $N\in\mathbb N$. 
\end{thm}

\begin{proof}
The proof is essentially the same as the the proof of Theorem~\ref{t-gue131109}. We only give the outline of the proof. 

Fix $p\in X$. Let $u_k\in C^\infty(X,L^k)$ be as in Lemma~\ref{l-gue131001} and put $u^0_k=\Pi^{(q)}_{k,\leq k^{-N_0}}u_k$. From the proof Theorem~\ref{t-gue131007}, we see that $u_k\equiv u^0_k\mod O(k^{-\infty})$ and hence 
\begin{equation}\label{e-gue131218}
F^{(q)}_{\delta,k}u_k\equiv F^{(q)}_{\delta,k}u^0_k\mod O(k^{-\infty}).
\end{equation}
From \eqref{e-gue131216} and \eqref{e-gue131207VI}, we see that 
\begin{equation}\label{e-gue131218I}
F^{(q)}_{\delta,k}u^0_k\in H^{(0)}_b(X,L^k).
\end{equation}
Moreover, from the construction of $u_k$(see \eqref{e-gue131001}) and \eqref{e-gue131218}, it is straightforward to see that there exist $C>1$ and $k_0>0$ independent of $k$ and the point $p$ such that for every $k\geq k_0$, we have 
\begin{equation}\label{e-gue131218II}
\begin{split}
&\frac{1}{C}\leq\norm{F^{(q)}_{\delta,k}u^0_k}_{h^{L^k}}\leq C,\\
&\abs{(F^{(q)}_{\delta,k}u^0_k)(p)}^2_{h^{L^k}}\geq\frac{1}{C}k^n.
\end{split}
\end{equation} 
From \eqref{e-gue131218II}, we can repeat the procedure in section~\ref{s-aket} and conclude that for $k$ large, the Kodaira map
\[\Phi_k:X\To\Complex\mathbb P^{d_k-1}\]
is well-defined as a smooth map. Here $\Phi_k$ is defined as follows. Let $f_1,\ldots,f_{d_k}$ be orthonormal frame for $H^0_b(X,L^k)$. For $x_0\in X$, let $s$ be a local trivializing section of $L$ on an open neighbourhood $D\subset X$ of $x_0$, $\abs{s(x)}^2_{h^{L^k}}=e^{-2\phi}$. On $D$, put $f_j(x)=s^k\Td f_j(x)$, $\Td f_j(x)\in C^\infty(D)$, $j=1,\ldots,d_k$. Then, 
\[\Phi_{k}(x_0)=[\Td f_1(x_0),\ldots,\Td f_{d_k}(x_0)]\in\Complex\mathbb P^{d_k-1}.\]
Moreover, with similar modifications, we can repeat the proof of Theorem~\ref{t-gue131019} and conclude that for $k$ large, the differential map
\[d\Phi_{k}(x):T_xX\To T_{\Phi_{k}(x)}\Complex\mathbb P^{d_k-1}\]
is injective, for every $x\in X$. 

Finally, by using Theorem~\ref{t-gue131209}, we can repeat the proof of Theorem~\ref{t-gue131103} and deduce that for $k$ large, $\Phi_k$ is injective.  
\end{proof}

We now offer some examples of "$T$-rigid CR line bundles over CR manifolds with transversal CR $S^1$ actions". 

\subsection{CR manifolds in projective spaces}\label{s-cmips}

We consider $\Complex\mathbb P^{N-1}$, $N\geq4$. Let $[z]=[z_1,\ldots,z_N]$ be the homogeneous coordinates of $\Complex\mathbb P^{N-1}$. Put 
\[X:=\set{[z_1,\ldots,z_N]\in\Complex\mathbb P^{N-1};\, \lambda_1\abs{z_1}^2+\cdots+\lambda_m\abs{z_m}^2+\lambda_{m+1}\abs{z_{m+1}}^2+\cdots+\lambda_N\abs{z_N}^2=0},\]
where $m\in\mathbb N$ and $\lambda_j\in\Real$, $j=1,\ldots,N$. Then, $X$ is a compact CR manifold of dimension $2(N-1)-1$ with CR structure $T^{1,0}X:=T^{1,0}\Complex\mathbb P^{N-1}\bigcap\Complex TX$. Now, we assume that $\lambda_1<0,\lambda_2<0,\ldots,\lambda_m<0$, $\lambda_{m+1}>0,\lambda_{m+2}>0,\ldots,\lambda_N>0$, where $m\geq2$, $N-m\geq2$. Then, it is easy to see that the Levi form has at least one negative and one positive eigenvalues at each point of $X$. Thus, $Y(0)$ holds at each point of $X$. $X$ admits a $S^1$ action: 
\begin{equation}\label{e-gue131218III}
\begin{split}
S^1\times X&\To X,\\
e^{i\theta}\circ[z_1,\ldots,z_m,z_{m+1},\ldots,z_N]&\To[e^{i\theta}z_1,\ldots,e^{i\theta}z_m,z_{m+1},\ldots,z_N],\ \ \theta\in[-\pi,\pi[.
\end{split}
\end{equation}
Since $(z_1,\ldots,z_m)\neq0$ on $X$, this $S^1$ action is well-defined. Moreover, it is straightforward to check that this $S^1$ action is CR and transversal. Let $T$ be the global vector field induced by the $S^1$ action. 

Let $E\To\Complex\mathbb P^{N-1}$ be the canonical line bundle with respect to the Fubini-Study metric. For $j=1,2,\ldots,N$, put $W_j=\set{[z_1,\ldots,z_N]\in\Complex\mathbb P^{N-1};\, z_j\neq0}$. Then, $E$ is trivial on $W_j$, $j=1,\ldots,N$, and we can find local trivializing section $e_j$ of $E$ on $W_j$, $j=1,\ldots,N$, such that for every $j, t=1,\ldots,N$,
\begin{equation}\label{e-gue131218III-I}
e_j(z)=\frac{z_j}{z_t}e_t(z)\ \ \mbox{on $W_j\bigcap W_t$},\ \ z=[z_1,\ldots,z_N]\in W_j\bigcap W_t.
\end{equation}
Consider $L:=E|_X$. Then, $L$ is a CR line bundle over $(X,T^{1,0}X)$. It is easy to see that $X$ can be covered with open sets $U_j:=W_j|_X$, $j=1,2,\ldots,m$, with trivializing sections $s_j:=e_j|_X$, $j=1,2,\ldots,m$, such that the corresponding transition functions are $T$-rigid CR functions. Thus, $L$ is a $T$-rigid CR line bundle over $(X,T^{1,0}X)$. Let $h^L$ be the Hermitian fiber metric on $L$ given by 
\[\abs{s_j(z_1,\ldots,z_N)}^2_{h^L}:=e^{-\log\bigr(\frac{\abs{z_1}^2+\cdots+\abs{z_N}^2}{\abs{z_j}^2}\bigr)},\ \ j=1,\ldots,m.\]
It is not difficult to check that $h^L$ is well-defined and $h^L$ is a $T$-rigid positive CR line bindle.

\subsection{Compact Heisenberg groups} \label{s-chg}

Let $\lambda_1,\ldots,\lambda_{n-1}$ be given non-zero integers.
Let $\mathscr CH_n=(\Complex^{n-1}\times\Real)/_\sim$\,, where
$(z, t)\sim(\Td z, \Td t)$ if
\[\begin{split}
&\Td z-z=(\alpha_1,\ldots,\alpha_{n-1})\in\sqrt{2\pi}\mathbb Z^{n-1}+i\sqrt{2\pi}\mathbb Z^{n-1},\\
&\Td t-t-i\sum^{n-1}_{j=1}\lambda_j(z_j\ol\alpha_j-\ol z_j\alpha_j)\in 2\pi\mathbb Z.
\end{split}\]
We can check that $\sim$ is an equivalence relation
and $\mathscr CH_n$ is a compact manifold of dimension $2n-1$. The equivalence class of $(z,t)\in\Complex^{n-1}\times\Real$ is denoted by
$[(z, t)]$. For a given point $p=[(z, t)]$, we define
$T^{1,0}_p\mathscr CH_n$ to be the space spanned by
\[
\textstyle
\big\{\frac{\pr}{\pr z_j}+i\lambda_j\ol z_j\frac{\pr}{\pr t},\ \ j=1,\ldots,n-1\big\}.
\]
It is easy to see that the definition above is independent of the choice of a representative $(z,t)$ for $[(z,t)]$.
Moreover, we can check that $T^{1,0}\mathscr CH_n$ is a CR structure. $\mathscr CH_n$ admits the natural $S^1$ action: $e^{i\theta}\circ [z,t]\To [z,t+\theta]$, $0\leq\theta<2\pi$. Let $T$ be the global vector field induced by this $S^1$ action. We can check that this $S^1$ action is CR and transversal and  
$T=\frac{\pr}{\pr t}$.  We take a Hermitian metric $\langle\,\cdot\,|\,\cdot\,\rangle$ on the complexified tangent bundle $\Complex T\mathscr CH_n$ such that
\[
\Big\lbrace
\tfrac{\pr}{\pr z_j}+i\lambda_j\ol z_j\tfrac{\pr}{\pr t}\,, \tfrac{\pr}{\pr\ol z_j}-i\lambda_jz_j\tfrac{\pr}{\pr t}\,, -\tfrac{\pr}{\pr t}\,;\, j=1,\ldots,n-1\Big\rbrace
\]
 is an orthonormal basis. The dual basis of the complexified cotangent bundle is
\[
\Big\lbrace
dz_j\,,\, d\ol z_j\,,\, \omega_0:=-dt+\textstyle\sum^{n-1}_{j=1}(i\lambda_j\ol z_jdz_j-i\lambda_jz_jd\ol z_j); j=1,\ldots,n-1
\Big\rbrace\,.
\]
The Levi form $\mathcal{L}_p$ of $\mathscr CH_n$ at $p\in\mathscr CH_n$ is given by
$\mathcal{L}_p=\sum^{n-1}_{j=1}\lambda_jdz_j\wedge d\ol z_j$. 

Now, we construct a $T$-rigid CR line bundle $L$ over $\mathscr CH_n$. Let $L=(\Complex^{n-1}\times\Real\times\Complex)/_\equiv$ where $(z,\theta,\eta)\equiv(\Td z, \Td\theta, \Td\eta)$ if
\[\begin{split}
&(z,\theta)\sim(\Td z, \Td\theta),\\
&\Td\eta=\eta\exp(\sum^{n-1}_{j,t=1}\mu_{j,t}(z_j\ol\alpha_t+\frac{1}{2}\alpha_j\ol\alpha_t)),
\end{split}\]
where $\alpha=(\alpha_1,\ldots,\alpha_{n-1})=\Td z-z$, $\mu_{j,t}=\mu_{t,j}$, $j, t=1,\ldots,n-1$, are given integers. We can check that $\equiv$ is an equivalence relation and 
$L$ is a $T$-rigid CR line bundle over $\mathscr CH_n$. For $(z, \theta, \eta)\in\Complex^{n-1}\times\Real\times\Complex$, we denote
$[(z, \theta, \eta)]$ its equivalence class.
It is straightforward to see that the pointwise norm
\[
\big\lvert[(z, \theta, \eta)]\big\rvert^2_{h^L}:=\abs{\eta}^2\exp\big(-\textstyle\sum^{n-1}_{j,t=1}\mu_{j,t}z_j\ol z_t\big)
\]
is well-defined. In local coordinates $(z, \theta, \eta)$, the weight function of this metric is
\[\phi=\frac{1}{2}\sum^{n-1}_{j,t=1}\mu_{j,t}z_j\ol z_t.\] 
Thus, $L$ is a $T$-rigid CR line bundle over $\mathscr CH_n$ with $T$-rigid Hermitian metric $h^L$. 
Note that
\[
\textstyle\ddbar_b=\sum^{n-1}_{j=1}d\ol z_j\wedge(\frac{\pr}{\pr\ol z_j}-i\lambda_jz_j\frac{\pr}{\pr\theta})\,,\quad
\pr_b=\sum^{n-1}_{j=1}dz_j\wedge(\frac{\pr}{\pr z_j}+i\lambda_j\ol z_j\frac{\pr}{\pr\theta}).
\]
Thus
$d(\ddbar_b\phi-\pr_b\phi)=\sum^{n-1}_{j,t=1}\mu_{j,t}dz_j\wedge d\ol z_t$ and for any $p\in\mathscr CH_n$, 
\[M^\phi_p=\sum^{n-1}_{j,t=1}\mu_{j,t}dz_j\wedge d\ol z_t.\]
Thus, if $\left(\mu_{j,t}\right)^{n-1}_{j,t=1}$ is positive definite, then $L$ is a $T$-rigid positive CR line bundle. From this and Theorem~\ref{t-gue131209I}, we conclude that 

\begin{thm}\label{t-gue131221}
Assume that $\lambda_1<0$ and $\lambda_2>0$ (then $Y(0)$ holds on $\mathscr CH_n$). Then, $\mathscr CH_n$ can be CR embedded into $\Complex\mathbb P^N$, for some $N\in\mathbb N$. 
\end{thm}

\subsection{Holomorphic line bundles over a complex torus}\label{s-hlbo}

Let
\[T_n:=\Complex^n/(\sqrt{2\pi}\mathbb Z^n+i\sqrt{2\pi}\mathbb Z^n)\]
be the flat torus. Let $\lambda=\left(\lambda_{j,t}\right)^{n}_{j,t=1}$, where $\lambda_{j,t}=\lambda_{t,j}$, 
$j, t=1,\ldots,n$, are given integers. Let $L_\lambda$ be the holomorphic
line bundle over $T_n$
with curvature the $(1,1)$-form
$\Theta_\lambda=\sum^n_{j,t=1}\lambda_{j,t}dz_j\wedge d\ol z_t$.
More precisely, $L_\lambda:=(\Complex^n\times\Complex)/_\sim$\,, where
$(z, \theta)\sim(\Td z, \Td\theta)$ if
\[
\Td z-z=(\alpha_1,\ldots,\alpha_n)\in \sqrt{2\pi}\mathbb Z^n+i\sqrt{2\pi}\mathbb Z^n\,,\quad
\Td\theta=\textstyle\exp\big(\sum^n_{j,t=1}\lambda_{j,t}(z_j\ol\alpha_t+\tfrac{1}{2}\alpha_j\ol\alpha_t\,)\big)\theta\,.
\]
We can check that $\sim$ is an equivalence relation and $L_\lambda$ is a holomorphic line bundle over $T_n$.
For $[(z, \theta)]\in L_\lambda$
we define the Hermitian metric by
\[
\big\vert[(z, \theta)]\big\vert^2:=\abs{\theta}^2\textstyle\exp(-\sum^n_{j,t=1}\lambda_{j,t}z_j\ol z_t)
\]
and it is easy to see that this definition is independent of the choice of a representative $(z, \theta)$ of $[(z, \theta)]$. We denote by $\phi_\lambda(z)$ the weight of this Hermitian fiber metric. Note that $\frac{1}{2}\pr\ddbar\phi_\lambda=\Theta_\lambda$.

Let $L^*_\lambda$ be the
dual bundle of $L_\lambda$ and let $\norm{\,\cdot\,}_{L^*_\lambda}$ be the norm of $L^*_\lambda$ induced by the Hermitian fiber metric on $L_\lambda$. Consider the compact CR manifold of dimension $2n+1$: $X=\{v\in L^*_\lambda;\, \norm{v}_{L^*_\lambda}=1\}$; this is the boundary of the Grauert tube associated to $L^*_\lambda$. The manifold $X$ is equipped with a natural $S^1$-action.
Locally $X$ can be represented in local holomorphic coordinates $(z,\eta)$, where $\eta$ is the fiber coordinate, as the set of all $(z,\eta)$ such that $\abs{\eta}^2e^{2\phi_\lambda(z)}=1$.
The $S^1$-action on $X$ is given by $e^{i\theta}\circ (z,\eta)=(z,e^{i\theta}\eta)$, $e^{i\theta}\in S^1$, $(z,\eta)\in X$. Let $T$ be the global vector field on $X$ induced by this $S^1$ action. We can check that this $S^1$ action is CR and transversal. 

Let $\pi:L^*_\lambda\To T_n$
be the natural projection from $L^*_\lambda$ onto $T_n$. Let $\mu=\left(\mu_{j,t}\right)^{n}_{j,t=1}$, where $\mu_{j,t}=\mu_{t,j}$, $j, t=1,\ldots,n$, are given integers. Let $L_\mu$ be another holomorphic
line bundle over $T_n$ determined by the constant curvature form
$\Theta_\mu=\sum^n_{j,t=1}\mu_{j,t}dz_j\wedge d\ol z_t$ as above. 
The pullback line bundle $\pi^*L_\mu$ is a holomorphic line bundle over $L^*_\lambda$. If we restrict $\pi^*L_\mu$ on $X$, then we can check that $\pi^*L_\mu$ is a $T$-rigid CR line bundle over $X$. 

The Hermitian fiber metric on $L_\mu$ induced by $\phi_\mu$ induces a Hermitian fiber metric on $\pi^*L_\mu$ that we shall denote by $h^{\pi^*L_\mu}$. We let $\psi$ to denote the weight of $h^{\pi^*L_\mu}$. 
The part of $X$ that lies over a fundamental domain of $T_n$ can be represented in local holomorphic coordinates
$(z, \xi)$, where $\xi$ is the fiber coordinate, as the set of all $(z, \xi)$ such that
$r(z, \xi):=\abs{\xi}^2\exp(\sum^n_{j,t=1}\lambda_{j,t}z_j\ol z_t)-1=0$
and the weight $\psi$ may be written as $\psi(z, \xi)=\frac{1}{2}\sum^n_{j,t=1}\mu_{j,t}z_j\ol z_t$. 
From this we see that $\pi^*L_\mu$ is a $T$-rigid CR line bundle over $X$ with $T$-rigid Hermitian fiber metric $h^{\pi^*L_\mu}$. It is straightforward to check that for any $p\in X$, we have 
$M^\psi_p=d(\ddbar_b\psi-\pr_b\psi)(p)|_{T^{1, 0}X}=\sum^n_{j,t=1}\mu_{j,t}dz_j\wedge d\ol z_t$.
Thus, if $\left(\mu_{j,t}\right)^{n-1}_{j,t=1}$ is positive definite, then $L$ is a $T$-rigid positive CR line bundle. From this and Theorem~\ref{t-gue131209I}, we conclude that 

\begin{thm}\label{t-gue131221I}
Assume that the matrix $\lambda=\left(\lambda_{j,t}\right)^n_{j,t=1}$ has at least one negative and one positive eigenvalues. Then, $X=\{v\in L^*_\lambda;\, \norm{v}_{L^*_\lambda}=1\}$ can be CR embedded into $\Complex\mathbb P^N$, for some $N\in\mathbb N$. 
\end{thm} 

\section{Szeg\"o kernel asymptotics on some non-compact CR manifolds} \label{s-skaohg}

By using Theorem~\ref{t-gue130820}, we will establish Szeg\"o kernel asymptotics on some non-compact CR manifolds. Let $\Gamma$ be an open set in $\Complex^{n-1}$, $n\geq 2$. Consider $X:=\Gamma\times\Real$. Let $(z,t)$ be the coordinates of $X$, where $z=(z_1,\ldots,z_{n-1})$ denote the
coordinates of $\Complex^{n-1}$ and $t$ is the coordinate of $\Real$. We write $z_j=x_{2j-1}+ix_{2j}$, $j=1,\ldots,n-1$. We also write $(z,t)=x=(x_1,\ldots,x_{2n-1})$ and let $\eta=(\eta_1,\ldots,\eta_{2n-1})$ be the dual variables of $x$. Let $\mu(z)\in C^\infty(\Gamma,\Real)$. We define
$T^{1,0}X$ to be the space spanned by
\[
\textstyle
\big\{\frac{\pr}{\pr z_j}+i\frac{\pr\mu}{\pr z_j}\frac{\pr}{\pr t},\ \ j=1,\ldots,n-1\big\}.
\]
Then $(X,T^{1,0}X)$ is a non-compact CR manifold of dimension $2n-1$. We take a Hermitian metric $\langle\,\cdot\,|\,\cdot\,\rangle$ on the complexified tangent bundle $\Complex TX$ such that
\[
\Big\lbrace
\tfrac{\pr}{\pr z_j}+i\frac{\pr\mu}{\pr z_j}\tfrac{\pr}{\pr t}\,, \tfrac{\pr}{\pr\ol z_j}-i\frac{\pr\mu}{\pr\ol z_j}\tfrac{\pr}{\pr t}\,, T:=\tfrac{\pr}{\pr t}\,;\, j=1,\ldots,n-1\Big\rbrace
\]
 is an orthonormal basis. The dual basis of the complexified cotangent bundle $\Complex T^*X$ is
\[
\Big\lbrace
dz_j\,,\, d\ol z_j\,,\, -\omega_0:=dt+\textstyle\sum^{n-1}_{j=1}(-i\frac{\pr\mu}{\pr z_j}dz_j+i\frac{\pr\mu}{\pr\ol z_j}d\ol z_j); j=1,\ldots,n-1
\Big\rbrace\,.
\]
The Hermitian metric $\langle\,\cdot\,|\,\cdot\,\rangle$ on the $\Complex TX$ induces Hermitian metrics on the bundle of $(0,q)$ forms $T^{*0,q}X$, $q=1,\ldots,n-1$, we shall also denote these Hermitian metrics by $\langle\,\cdot\,|\,\cdot\,\rangle$. For $V\in T^{*0,q}X$, we write $\abs{V}^2:=\langle\,V\,|\,V\,\rangle$. 

The Levi form $\mathcal{L}_p$ of $X$ at $p\in X$ is given by
\[\mathcal{L}_p=\sum^{n-1}_{j,\ell=1}\frac{\pr^2\mu}{\pr z_j\pr\ol z_{\ell}}(p)dz_j\wedge d\ol z_{\ell}.\]

Let $L$ be the trivial line bundle over $X$ with non-trivial Hermitian fiber metric $\abs{1}^2_{h^L}=e^{-2\phi}$, where $\phi=\phi(z)\in C^\infty(\Gamma)$ is a real valued function. Then, we can check that
\begin{equation}\label{e-gue131222}
M^\phi_p=2\sum^{n-1}_{j,\ell=1}\frac{\pr^2\phi(z)}{\pr z_j\pr\ol z_{\ell}}dz_j\wedge d\ol z_{\ell},\ \ p=(z,t)\in X.
\end{equation}

We shall consider the $k$-th power of $L$ and we will use the same notations as before. Let $(\,\cdot\,|\,\cdot\,)_{h^{L^k}}$ be the $L^2$ inner product on $\Omega^{0,q}_0(X)$ given by 
\[(\,u\,|\,v\,)_{h^{L^k}}=\int_{X}\langle\,u\,|\,v\,\rangle e^{-2k\phi(z)}d\lambda(z)dt,\ \ u, v\in\Omega^{0,q}_0(X),\]
where $d\lambda(z)dt=m(z)dx_1\cdots dx_{2n-2}dt$, $m(z)\in C^\infty(\Gamma)$, is the induced volume form. Let $\norm{\cdot}_{h^{L^k}}$ denote the corresponding $L^2$ norm. Let $L^2_{(0,q)}(X,L^k)$ be the completion of $\Omega^{0,q}_0(X)$ with respect to $(\,\cdot\,|\,\cdot\,)_{h^{L^k}}$ and let 
\[\Box^{(q)}_{b,k}:{\rm Dom\,}\Box^{(q)}_{b,k}\subset L^2_{(0,q)}(X,L^k)\To L^2_{(0,q)}(X,L^k)\] 
be the Gaffney extension of Kohn Laplacian with respect to $(\,\cdot\,|\,\cdot\,)_{h^{L^k}}$ (see \eqref{e-suIX}).

\subsection{The partial Fourier transform and the operator $F^{(q)}_{\delta,k}$}\label{s-tof}

Let $u\in\Omega^{0,q}_0(X,L^k)$. Put
\begin{equation} \label{e-gue131222tI}
(\mathcal{F}u)(z, \eta)=\int_{\Real}e^{-i\eta t}u(z,t)dt.
\end{equation}
From Parseval's formula, we have
\begin{equation}\label{e-gue131225}
\begin{split}
&\norm{\mathcal{F}u}^2_{h^{L^k}}=\int_{X}\abs{(\mathcal{F}u)(z,\eta)}^2e^{-2k\phi(z)}d\eta d\lambda(z)\\
&=(2\pi)\int_{X}\abs{u(z,t)}^2e^{-2k\phi(z)}dtd\lambda(z)=(2\pi)\norm{u}^2_{h^{L^k}}.
\end{split}
\end{equation}
Thus, we can extend the operator $\mathcal{F}$ to $L^2_{(0,q)}(X,L^k)$ and 
\begin{equation}\label{e-gue131225I-I}
\begin{split}
&\mbox{$\mathcal{F}:L^2_{(0,q)}(X,L^k)\To L^2_{(0,q)}(X,L^k)$ is continuous},\\
&\norm{\mathcal{F}u}_{h^{L^k}}=\sqrt{2\pi}\norm{u}_{h^{L^k}},\ \ \forall u\in L^2_{(0,q)}(X,L^k).
\end{split}
\end{equation}
For $u\in L^2_{(0,q)}(X,L^k)$, we call $\mathcal{F}u$ the partial Fourier transform of $u$ with respect to $t$. 

Fix $\delta>0$. Take $\tau_\delta(x)\in C^\infty_0(]-\delta,\delta[)$, 
$0\leq\tau_\delta\leq1$ and $\tau_\delta=1$ on $[-\frac{\delta}{2},\frac{\delta}{2}]$. We also write $\theta$ to denote the $t$ variable. 
Let $F^{(q)}_{\delta,k}:\Omega^{0,q}_0(X,L^k)\To\Omega^{0,q}(X,L^k)$ be the operator given by 
\begin{equation}\label{e-gue131222t}
F^{(q)}_{\delta,k}u(z,t):=\frac{1}{2\pi}\int e^{i<t-\theta,\eta>}u(z,\theta)\tau_\delta(\frac{\eta}{k})d\eta d\theta\in\Omega^{0,q}(X,L^k),\ \ u(z,t)\in\Omega^{0,q}_0(X,L^k).
\end{equation}
From Parseval's formula and \eqref{e-gue131225}, we have
\begin{equation}\label{e-gue131222tIII}
\begin{split}
\norm{F^{(q)}_{\delta,k}u}^2_{h^{L^k}}&=\frac{1}{4\pi^2}\int_{X}\abs{\int e^{i<t-\theta,\eta>}u(z,\theta)\tau_\delta(\frac{\eta}{k})d\eta d\theta}^2e^{-2k\phi(z)}d\lambda(z)dt\\
&=\frac{1}{4\pi^2}\int_{X}\abs{\int e^{i<t,\eta>}(\mathcal{F}u)(z,\eta)\tau_\delta(\frac{\eta}{k})d\eta}^2e^{-2k\phi(z)}d\lambda(z)dt\\
&=\frac{1}{2\pi}\int\abs{(\mathcal{F}u)(z,\eta)}^2\abs{\tau_\delta(\frac{\eta}{k})}^2e^{-2k\phi(z)}d\eta d\lambda(z)\\
&\leq\frac{1}{2\pi}\int\abs{(\mathcal{F}u)(z,\eta)}^2e^{-2k\phi(z)}d\eta d\lambda(z)=\norm{u}^2_{h^{L^k}},
\end{split}
\end{equation}
where $u\in\Omega^{0,q}_0(X,L^k)$.
Thus, we can extend $F^{(q)}_{\delta,k}$ to $L^2_{(0,q)}(X,L^k)$ and 
\begin{equation}\label{e-gue131222tIV}
\begin{split}
&\mbox{$F^{(q)}_{\delta,k}:L^2_{(0,q)}(X,L^k)\To L^2_{(0,q)}(X,L^k)$ is continuous},\\
&\norm{F^{(q)}_{\delta,k}u}_{h^{L^k}}\leq\norm{u}_{h^{L^k}},\ \ \forall u\in L^2_{(0,q)}(X,L^k).
\end{split}
\end{equation}
We need 

\begin{lem}\label{l-gue131225}
Let $u\in L^2_{(0,q)}(X,L^k)$. Then, 
\begin{equation}\label{e-gue131225I}
(\mathcal{F}F^{(q)}_{\delta,k}u)(z,\eta)=(\mathcal{F}u)(z,\eta)\tau_\delta(\frac{\eta}{k}).
\end{equation}
\end{lem}

\begin{proof}
Let $u_j\in\Omega^{0,q}_0(X,L^k)$, $j=1,2,\ldots$, with 
$\lim_{j\To\infty}\norm{u_j-u}_{h^{L^k}}=0$. From \eqref{e-gue131222tIV} and \eqref{e-gue131225I-I}, we see that 
\begin{equation}\label{e-gue131225II}
\mbox{$\mathcal{F}F^{(q)}_{\delta,k}u_j\To\mathcal{F}F^{(q)}_{\delta,k}u$ in $L^2_{(0,q)}(X,L^k)$ as $j\To\infty$}.
\end{equation}
From Fourier inversion formula, we have
\begin{equation}\label{e-gue131225III}
(\mathcal{F}F^{(q)}_{\delta,k}u_j)(z,\eta)=(\mathcal{F}u_j)(z,\eta)\tau_\delta(\frac{\eta}{k}),\ \ j=1,\ldots.
\end{equation}
Note that  $(\mathcal{F}u_j)(z,\eta)\tau_\delta(\frac{\eta}{k})\To(\mathcal{F}u)(z,\eta)\tau_\delta(\frac{\eta}{k})$ in $L^2_{(0,q)}(X,L^k)$ as $j\To\infty$. From this observation, \eqref{e-gue131225III} and \eqref{e-gue131225II}, we obtain \eqref{e-gue131225I}. 
\end{proof} 

The following is straightforward. We omit the proofs. 

\begin{lem}\label{l-gue131225I}
We have
\begin{equation}\label{e-gue131225IV}
\begin{split}
&F^{(q)}_{\delta,k}:{\rm Dom\,}\ddbar_{b,k}\To{\rm Dom\,}\ddbar_{b,k},\ \ q=0,1,\ldots,n-2,\\
&\mbox{$F^{(q+1)}_{\delta,k}\ddbar_{b,k}=\ddbar_{b,k}F^{(q)}_{\delta,k}$ on ${\rm Dom\,}\ddbar_{b,k}$},\ \ q=0,1,\ldots,n-2,
\end{split}
\end{equation} 
and 
\begin{equation}\label{e-gue131225V}
\mbox{$F^{(q)}_{\delta,k}\Pi^{(q)}_k=\Pi^{(q)}_kF^{(q)}_{\delta,k}$ on $L^2_{(0,q)}(X,L^k)$, $q=0,1,\ldots,n-1$}.
\end{equation}

Moreover, for $u\in C^\infty_0(X,L^k)$, we have 
\begin{equation}\label{e-gue131225VI}
\ddbar_z\bigr((\mathcal{F}u)(z,\eta)e^{\eta\mu(z)}\bigr)e^{-\eta\mu(z)}=(\mathcal{F}\ddbar_{b,k}u)(z,\eta),
\end{equation}
where $\mu\in C^\infty(\Gamma,\Real)$ is as in the beginning of section~\ref{s-skaohg}
\end{lem}

\subsection{The small spectral gap property for $\Box^{(0)}_{b,k}$ with respect to $F^{(0)}_{\delta,k}$}\label{s-tssgtp}

We pause and introduce some notations. Let $\Omega^{0,q}(\Gamma)$ be the space of all smooth $(0,q)$ forms on $\Gamma$ and let $\Omega^{0,q}_0(\Gamma)$ be the subspace of $\Omega^{0,q}(\Gamma)$ whose elements have compact support in $\Gamma$. We take the Hermitian metric $\langle\,\cdot\,|\,\cdot\,\rangle$ on $T^{*0,q}\Gamma$ the bundle of $(0, q)$ forms of $\Gamma$ so that 
\[\{d\ol z_{j_1}\wedge d\ol z_{j_2}\wedge\cdots\wedge d\ol z_{j_q}; 1\leq j_1<j_2\cdots<j_q\leq n-1\}\] 
is an orthonormal basis. Let $\Upsilon\in C^\infty(\Gamma,\Real)$ and let $(\,\cdot\,|\,\cdot)_{\Upsilon}$ be the $L^2$ inner product on $\Omega^{0,q}_0(\Gamma)$ given by 
\[(f\,|\,g\,)_{\Upsilon}=\int\langle\,f\,|\,g\,\rangle e^{-2\Upsilon(z)}d\lambda(z),\ \ f, g\in\Omega^{0,q}_0(\Gamma).\]
Let $L^2_{(0,q)}(\Gamma,\Upsilon)$ denote the completion of $\Omega^{0,q}_0(\Gamma)$ with respect to the inner product $(\cdot\,|\,\cdot\,)_{\Upsilon}$. We write $L^2(\Gamma,\Upsilon):=L^2_{(0,0)}(\Gamma,\Upsilon)$. Put 
\[H^0(\Gamma,\Upsilon):=\set{f\in L^2(\Gamma,\Upsilon);\, \ddbar f=0}.\]

Now, we return to our situation. We first consider $\Gamma=\Complex^{n-1}$.

\begin{thm}\label{t-gue131226}
Let $\Gamma=\Complex^{n-1}$. We assume that there are constants $C_0\geq 1$ and $\epsilon_0>0$ such that
\begin{equation}\label{e-gue131226}
\sum^{n-1}_{j,\ell=1}\frac{\pr^2(\phi+\eta\mu)}{\pr z_j\pr\ol z_\ell}(z)w_j\ol w_\ell\geq\frac{1}{C_0}\sum^{n-1}_{j=1}\abs{w_j}^2,\ \ \forall (w_1,\ldots,w_{n-1})\in\Complex^{n-1},\ \ z\in\Complex^{n-1},\ \ \abs{\eta}\leq\epsilon_0,
\end{equation}
and
\begin{equation}\label{e-gue131226I}
\phi(z)+\eta\mu(z)\geq\frac{1}{C_0}\abs{z}^2,\ \ \forall\abs{z}\geq M,\ \ \abs{\eta}\leq\epsilon_0,
\end{equation}
where $\abs{z}^2=\sum^{n-1}_{j=1}\abs{z_j}^2$ and $M>0$ is a constant independent of $\eta$. 
Then, for every $0<\delta\leq\epsilon_0$, we have
\begin{equation}\label{e-gue131226II}
\norm{F^{(0)}_{\delta,k}(I-\Pi^{(0)}_k)u}^2_{h^{L^k}}\leq\frac{C}{k}\norm{\ddbar_{b,k}u}^2_{h^{L^k}},\ \ \forall u\in C^\infty_0(X,L^k),
\end{equation}
where $C>0$ is a constant independent of $k$, $\delta$ and $u$.

In particular, $\Box^{(0)}_{b,k}$ has small spectral gap on $X$. 
\end{thm}

\begin{proof}
Let $u\in C^\infty_0(X,L^k)$. We consider $F^{(0)}_{\delta,k}(I-\Pi^{(0)}_k)u$. In view of \eqref{e-gue131225V}, we see that $F^{(0)}_{\delta,k}(I-\Pi^{(0)}_k)u=(I-\Pi^{(0)}_k)F^{(0)}_{\delta,k}u$. Put 
\begin{equation}\label{e-gue131226aI}
v(z,\eta)=\mathcal{F}F^{(0)}_{\delta,k}(I-\Pi^{(0)}_k)u(z,\eta)e^{\eta\mu(z)}.
\end{equation}
From \eqref{e-gue131222tIV}, \eqref{e-gue131225I-I} and \eqref{e-gue131225I}, we see that $\int\abs{v(z,\eta)}^2e^{-2\eta\mu(z)-2k\phi(z)}d\lambda(z)d\eta<\infty$ and $v(z,\eta)=0$ if $\eta\notin{\rm Supp\,}\tau_\delta(\frac{\eta}{k})$. From Fubini's Theorem and some elementary real analysis, we know that for every $\eta\in\Real$, $v(z,\eta)$ is a measurable function of $z$ and for almost every $\eta\in\Real$, $v(z,\eta)\in L^2(\Complex^{n-1},\eta\mu+k\phi)$ and for every $z\in\Complex^{n-1}$, $v(z,\eta)$ is a measurable function of $\eta$ and for almost every $z\in\Complex^{n-1}$, $\int\abs{v(z,\eta)}^2d\eta<\infty$. Moreover, let $\beta\in L^2(\Complex^{n-1},\eta\mu+k\phi)$, then the function
\[f(\eta):=\eta\To\int v(z,\eta)\ol\beta(z)e^{-2\eta\mu(z)-2k\phi(z)}d\lambda(z)\]
is measurable and  $f(\eta)$ is finite for almost every $\eta\in\Real$, $f(\eta)=0$ if $\eta\notin{\rm Supp\,}\tau_\delta(\frac{\eta}{k})$ and $f(\eta)\in L^1(\Real)\bigcap L^2(\Real)$. We claim that 
\begin{equation}\label{e-gue131229}
\begin{split}
&\mbox{if $\delta\leq\epsilon_0$, then for almost every $\eta\in\Real$, $v(z,\eta)\in L^2(\Complex^{n-1},\eta\mu+k\phi)$ and}\\
&(\,v(z,\eta)\,|\,\beta\,)_{\eta\mu+k\phi}=0,\ \ \forall\beta\in  H^0(\Complex^{n-1},\eta\mu+k\phi).
\end{split}
\end{equation}
From the discussion after \eqref{e-gue131226aI}, we know that there is a measurable set $A_0$ in $\Real$ with $\abs{A_0}=0$ such that for every $\eta\notin A_0$, 
$v(z,\eta)\in L^2(\Complex^{n-1},\eta\mu+k\phi)$, where $\abs{A_0}$ denote the Lebesgue measure of $A_0$. 
From \eqref{e-gue131226I}, we see that $\set{z^\alpha;\, \alpha\in\mathbb N^{n-1}_0}$ is a basis for 
$H^0(\Complex^{n-1},\eta\mu+k\phi)$, for every $\abs{\eta}\leq\epsilon_0$, where $\epsilon_0$ is as in \eqref{e-gue131226I}. 
Let $\eta\notin A_0$. Fix $\alpha\in\mathbb N^{n-1}_0$. We consider 
\[f_\alpha(\eta)=\int v(z,\eta)\ol z^\alpha e^{-2\eta\mu(z)-2k\phi(z)}d\lambda(z).\]
From the discussion after \eqref{e-gue131226aI}, we know that $f_\alpha(\eta)\in L^1(\Real)\bigcap L^2(\Real)$. 
We consider the Fourier transform 
\[\hat f_\alpha(\xi)=\int e^{-i\xi\eta}f_\alpha(\eta)d\eta\]
of $f_\alpha(\eta)$. Let $g_\ell\in C^\infty_0(X,L^k)$, $\ell=1,2,\ldots$, such that 
\[\mbox{$g_\ell\To(I-\Pi^{(0)}_k)u$ in $L^2(X,L^k)$ as $\ell\To\infty$.}\] From \eqref{e-gue131225I-I}, \eqref{e-gue131222tIV} and \eqref{e-gue131225I}, we see that for every $\xi\in\Real$, 
\begin{equation}\label{e-gue131226b}
\lim_{\ell\To\infty}\int\mathcal{F}F^{(0)}_{\delta,k}g_\ell(z,\eta)\ol z^\alpha e^{-\eta\mu(z)-2k\phi(z)}e^{-i\xi\eta}d\lambda(z)d\eta\To\hat f_\alpha(\xi).
\end{equation}
From \eqref{e-gue131225I} and Parseval's formula, we can check that 
\begin{equation}\label{e-gue131226bI}
\begin{split}
&\int\mathcal{F}F^{(0)}_{\delta,k}g_\ell(z,\eta)\ol z^\alpha e^{-\eta\mu(z)-2k\phi(z)}e^{-i\xi\eta}d\lambda(z)d\eta\\
&=\int\mathcal{F}g_\ell(z,\eta)\tau_\delta(\frac{\eta}{k})\ol z^\alpha e^{-\eta\mu(z)-2k\phi(z)}e^{-i\xi\eta}d\lambda(z)d\eta\\
&=\int g_\ell(z,t)(\int\ol z^\alpha\tau_\delta(\frac{\eta}{k})e^{-\eta\mu(z)-i\xi\eta-i\eta t}d\eta)e^{-2k\phi(z)}d\lambda(z)dt\\
&\To\int(I-\Pi^{(0)}_k)u(z,t)(\int \ol z^\alpha\tau_\delta(\frac{\eta}{k})e^{-\eta\mu(z)-i\xi\eta-i\eta t}d\eta)e^{-2k\phi(z)}d\lambda(z)dt\ \ \mbox{as $\ell\To\infty$}.
\end{split}
\end{equation}
It is straightforward to check that the function 
\[\int z^\alpha\tau_\delta(\frac{\eta}{k})e^{-\eta\mu(z)+i\xi\eta+i\eta t}d\eta\in{\rm Ker\,}\ddbar_{b,k}\bigcap L^2(X,L^k).\]
Thus, 
\begin{equation}\label{e-gue131226bII}
\int(I-\Pi^{(0)}_k)u(z,t)(\int \ol z^\alpha\tau_\delta(\frac{\eta}{k})e^{-\eta\mu(z)-i\xi\eta-i\eta t}d\eta)e^{-2k\phi(z)}d\lambda(z)dt=0.
\end{equation}
From \eqref{e-gue131226bII}, \eqref{e-gue131226bI} and \eqref{e-gue131226b}, we conclude that for every $\xi\in\Real$, $\hat f_\alpha(\xi)=0$. By $L^2$ Fourier inversion formula, we conclude that $f_\alpha(\eta)=0$ almost everywhere. Thus, there is a measurable set $A_\alpha\supset A_0$ in $\Real$ with $\abs{A_\alpha}=0$ such that for every $\eta\notin A_\alpha$, 
\[(\,v(z,\eta)\,|\,z^\alpha\,)_{\eta\mu+k\phi}=0.\]

Put $A=\bigcup_{\alpha\in\mathbb N^{n-1}_0}A_\alpha$. Then, $\abs{A}=0$. Note that $\set{z^\alpha;\, \alpha\in\mathbb N^{n-1}_0}$ is a basis for $H^0(\Complex^{n-1},\eta\mu+k\phi)$ if $\abs{\eta}\leq\epsilon_0$. From this observation, we conclude that for every  $\eta\notin A$, 
\[(\,v(z,\eta)\,|\,\beta\,)_{\eta\mu+k\phi}=0,\ \ \forall \beta\in  H^0(\Complex^{n-1},\eta\mu+k\phi).\]
The claim \eqref{e-gue131229} follows. 

Now, we can prove Theorem~\ref{t-gue131226}. We assume that $\delta\leq\epsilon_0$. Let $u\in C^\infty_0(X,L^k)$. From \eqref{e-gue131225IV} and \eqref{e-gue131225VI}, we have 
\begin{equation}\label{e-gue131226III}
\begin{split}
&\ddbar_{b,k}F^{(0)}_{\delta,k}(I-\Pi^{(0)}_k)u=F^{(1)}_{\delta,k}\ddbar_{b,k}u,\\
&(\mathcal{F}F^{(1)}_{\delta,k}\ddbar_{b,k}u)(z,\eta)=\ddbar_z(\mathcal{F}F^{(0)}_{\delta,k}u(z,\eta)e^{\eta\mu(z)})e^{-\eta\mu(z)}.
\end{split}
\end{equation}
As before, we put $v(z,\eta)=\bigr(\mathcal{F}F^{(0)}_{\delta,k}(I-\Pi^{(0)}_k)u\bigr)(z,\eta)e^{\eta\mu(z)}$ and set
\[\ddbar_z\Bigr(\bigr(\mathcal{F}F^{(0)}_{\delta,k}(I-\Pi^{(0)}_k)u\bigr)(z,\eta)e^{\eta\mu(z)}\Bigr)=\ddbar_zv(z,\eta):=g(z,\eta).\] 
It is easy to see that
\begin{equation}\label{e-gue131226IV}
\begin{split}
&\ddbar_zg(z,\eta)=0,\\
&\mbox{$g(z,\eta)=0$ if $\eta\notin{\rm Supp\,}\tau_\delta(\frac{\eta}{k})$},\\
&\int\abs{g(z,\eta)}^2e^{-2\eta\mu(z)-2k\phi(z)}d\lambda(z)<\infty,\ \ \forall \eta\in{\rm Supp\,}\tau_\delta(\frac{\eta}{k}). 
\end{split}
\end{equation}
From \eqref{e-gue131226}, we see that there is a $C>0$ such that 
\begin{equation}\label{e-gue131226V}
\sum^{n-1}_{j,\ell=1}\frac{\pr^2(k\phi+\eta\mu)}{\pr z_j\pr\ol z_\ell}(z)w_j\ol w_\ell\geq\frac{k}{C}\sum^{n-1}_{j=1}\abs{w_j}^2,\ \ \forall (w_1,\ldots,w_{n-1})\in\Complex^{n-1}, z\in\Complex^{n-1}, \eta\in
{\rm Supp\,}\tau_\delta(\frac{\eta}{k}). 
\end{equation}

From \eqref{e-gue131226V} and H\"ormander's $L^2$ estimates (see Lemma 4.4.1. in H\"ormander~\cite{Hor90}), we conclude that for every $\eta\in{\rm Supp\,}\tau_\delta(\frac{\eta}{k})$, we can find a $\beta_\eta(z)\in L^2_{(0,1)}(\Complex^{n-1},\eta\tau+k\phi)$ such that 
\begin{equation}\label{e-gue131226VI}
\ddbar_z\beta_\eta(z)=g(z,\eta)
\end{equation}
and 
\begin{equation}\label{e-gue131226VII}
\int\abs{\beta_\eta(z)}^2e^{-2\eta\mu(z)-2k\phi(z)}d\lambda(z)\leq\frac{C}{k}\int\abs{g(z,\eta)}^2e^{-2\eta\mu(z)-2k\phi(z)}d\lambda(z).
\end{equation}
In view of \eqref{e-gue131229}, we see that there is a measurable set $A$ in $\Real$ with Lebesgue measure zero in $\Real$ such that for every $\eta\notin A$, $v(z,\eta)\perp H^0(\Complex^{n-1},\eta\mu+k\phi)$. Thus, for every $\eta\notin A$,  $v(z,\eta)$ has the minimum $L^2$ norm with respect to $(\,|\,)_{\eta\mu+k\phi}$ of the solutions $\ddbar\alpha=\ddbar_zv(z,\eta)$. From this observation and \eqref{e-gue131226VII}, we conclude that 
\begin{equation}\label{e-gue131227}
\int\abs{v(z,\eta)}^2e^{-2\eta\mu(z)-2k\phi(z)}d\lambda(z)\leq\frac{C}{k}\int\abs{\ddbar_zv(z,\eta)}^2e^{-2\eta\mu(z)-2k\phi(z)}d\lambda(z),\ \ \forall\eta\notin A. 
\end{equation}
Thus, 
\begin{equation}\label{e-gue131227I}
\int\abs{v(z,\eta)}^2e^{-2\eta\mu(z)-2k\phi(z)}d\lambda(z)d\eta\leq\frac{C}{k}\int\abs{\ddbar_zv(z,\eta)}^2e^{-2\eta\mu(z)-2k\phi(z)}d\lambda(z)d\eta.
\end{equation}
From the definition of $v(z,\eta)$, \eqref{e-gue131225I-I}, \eqref{e-gue131226III} and \eqref{e-gue131222tIV}, it is straightforward to see that 
\begin{equation}\label{e-gue131227II}
\int\abs{v(z,\eta)}^2e^{-2\eta\mu(z)-2k\phi(z)}d\lambda(z)d\eta
=(2\pi)\int\abs{F^{(0)}_{\delta,k}(I-\Pi^{(0)}_k)u(z,t)}^2e^{-2k\phi(z)}d\lambda(z)dt
\end{equation}
and 
\begin{equation}\label{e-gue131227III}
\begin{split}
&\int\abs{\ddbar_zv(z,\eta)}^2e^{-2\eta\mu(z)-2k\phi(z)}d\lambda(z)d\eta\\
&=(2\pi)\int\abs{F^{(1)}_{\delta,k}\ddbar_{b,k}u(z,t)}^2e^{-2k\phi(z)}d\lambda(z)dt\\
&\leq(2\pi)\int\abs{\ddbar_{b,k}u(z,t)}^2e^{-2k\phi(z)}d\lambda(z)dt.
\end{split}
\end{equation}
From \eqref{e-gue131227I}, \eqref{e-gue131227II} and \eqref{e-gue131227III}, we conclude that 
\[\begin{split}
&\norm{F^{(0)}_{\delta,k}(I-\Pi^{(0)}_k)u}^2_{h^{L^k}}=
\int\abs{F^{(0)}_{\delta,k}(I-\Pi^{(0)}_k)u(z,t)}^2e^{-2k\phi(z)}d\lambda(z)dt\\
&\leq\frac{C}{k}\int\abs{\ddbar_{b,k}u(z,t)}^2e^{-2k\phi(z)}d\lambda(z)dt=\frac{C}{k}\norm{\ddbar_{b,k}u}^2_{h^{L^k}}.\end{split}\]
Theorem~\ref{t-gue131226} follows. 
\end{proof}

Now, we consider $\Gamma$ is a bounded strongly pseudoconvex domain in $\Complex^{n-1}$. 

\begin{thm}\label{t-gue131229}
Let $\Gamma$ be a bounded strongly pseudoconvex domain in $\Complex^{n-1}$. Assume that $\mu, \phi\in C^\infty(\Td\Gamma,\Real)$, where $\Td\Gamma$ is an open neighbourhood of $\ol\Gamma$. Suppose that 
$\left(\frac{\pr^2\phi}{\pr z_j\pr\ol z_\ell}(z)\right)^{n-1}_{j,\ell=1}$ is positive definite at each point $z$ of $\Td\Gamma$. If $\delta$ is small enough then 
\begin{equation}\label{e-gue131229I}
\norm{F^{(0)}_{\delta,k}(I-\Pi^{(0)}_k)u}^2_{h^{L^k}}\leq\frac{C}{k}\norm{\ddbar_{b,k}u}^2_{h^{L^k}},\ \ \forall u\in C^\infty_0(X,L^k),
\end{equation}
where $C>0$ is a constant independent of $k$, $\delta$ and $u$.

In particular, $\Box^{(0)}_{b,k}$ has small spectral gap on $X$. 
\end{thm}

\begin{proof}
Let $u\in C^\infty_0(X,L^k)$. We have 
\begin{equation}\label{e-gue131229II}
\begin{split}
&\ddbar_{b,k}F^{(0)}_{\delta,k}(I-\Pi^{(0)}_k)u=F^{(1)}_{\delta,k}\ddbar_{b,k}u,\\
&(\mathcal{F}F^{(1)}_{\delta,k}\ddbar_{b,k}u)(z,\eta)=\ddbar_z(\mathcal{F}F^{(0)}_{\delta,k}u(z,\eta)e^{\eta\mu(z)})e^{-\eta\mu(z)}.
\end{split}
\end{equation}
As before, we put $v(z,\eta)=\bigr(\mathcal{F}F^{(0)}_{\delta,k}(I-\Pi^{(0)}_k)u\bigr)(z,\eta)e^{\eta\mu(z)}$ and set
\[\ddbar_z(\mathcal{F}F^{(0)}_{\delta,k}u(z,\eta)e^{\eta\mu(z)})=\ddbar_zv(z,\eta):=g(z,\eta).\] 
Then,
\begin{equation}\label{e-gue131229III}
\begin{split}
&\ddbar_zg(z,\eta)=0,\\
&\mbox{$g(z,\eta)=0$ if $\eta\notin[-k\delta,k\delta]$},\\
&\int\abs{g(z,\eta)}^2e^{-2\eta\mu(z)-2k\phi(z)}d\lambda(z)<\infty,\ \ \forall \eta\in[-k\delta,k\delta]. 
\end{split}
\end{equation}
Since $\left(\frac{\pr^2\phi}{\pr z_j\pr\ol z_\ell}(z)\right)^{n-1}_{j,\ell=1}$ is positive definite at each point $z$ of $\Td\Gamma$, if $\delta>0$ is small enough then there is a $C>0$ such that 
\begin{equation}\label{e-gue131229IV}
\sum^{n-1}_{j,\ell=1}\frac{\pr^2(k\phi+\eta\mu)}{\pr z_j\pr\ol z_\ell}(z)w_j\ol w_\ell\geq\frac{k}{C}\sum^{n-1}_{j=1}\abs{w_j}^2,\ \ \forall (w_1,\ldots,w_{n-1})\in\Complex^{n-1}, z\in\Gamma, \eta\in
{\rm Supp\,}\tau_\delta(\frac{\eta}{k}). 
\end{equation}
We assume that $\delta>0$ is small enough so that \eqref{e-gue131229IV} holds. From \eqref{e-gue131229IV} and H\"ormander's $L^2$ estimates, we conclude that for every $\eta\in[-k\delta,k\delta]$, we can find a $\beta_\eta(z)\in L^2_{(0,1)}(\Gamma,\eta\mu+k\phi)$ such that 
\begin{equation}\label{e-gue131229V}
\ddbar_z\beta_\eta(z)=g(z,\eta)
\end{equation}
and 
\begin{equation}\label{e-gue131229VI}
\int\abs{\beta_\eta(z)}^2e^{-2\eta\mu(z)-2k\phi(z)}d\lambda(z)\leq\frac{C}{k}\int\abs{g(z,\eta)}^2e^{-2\eta\mu(z)-2k\phi(z)}d\lambda(z).
\end{equation}
Moreover, since $g(z,\eta)$ is smooth, it is well-known that $\beta_\eta(z)$ can be taken to be dependent smoothly on $\eta$ and $z$ (see the proof of Lemma 2.1 in Berndtsson~\cite{Bo06}). Put 
\[\alpha(z,t)=\frac{1}{2\pi}\int\beta_\eta(z)e^{-\eta\mu(z)}e^{i\eta t}1_{[-k\delta,k\delta]}(\eta)d\eta.\]
Since  $\beta_\eta(z)$ is smooth with respect to $\eta$, $\alpha(z,t)$ is well-defined and $\alpha(z,t)\in C^\infty(X,L^k)$. Moreover, from \eqref{e-gue131225I-I}, \eqref{e-gue131222tIV}, \eqref{e-gue131229VI}, \eqref{e-gue131229II} and Parseval's formula, we can check that 
\begin{equation}\label{e-gue131229VII}
\begin{split}
\norm{\alpha}^2_{h^{L^k}}&=\int\abs{\alpha(z,t)}^2e^{-2k\phi(z)}d\lambda(z)dt\\
&=\int\abs{\frac{1}{2\pi}\int\beta_\eta(z)e^{-\eta\mu(z)}e^{i\eta t}1_{[-k\delta,k\delta]}(\eta)d\eta}^2e^{-2\phi(z)}dtd\lambda(z)\\
&=\frac{1}{2\pi}\int\abs{\beta_\eta(z)}^2e^{-2\eta\mu(z)-2k\phi(z)}d\eta d\lambda(z)\\
&\leq\frac{C}{(2\pi)k}\int\abs{g(z,\eta)}^2e^{-2\eta\mu(z)-2k\phi(z)}d\eta d\lambda(z)\\
&=\frac{C}{(2\pi)k}\int\abs{\mathcal{F}F^{(1)}_{\delta,k}\ddbar_{b,k}u(z,\eta)}^2e^{-2k\phi(z)}d\eta d\lambda(z)\\
&=\frac{C}{k}\int\abs{F^{(1)}_{\delta,k}\ddbar_{b,k}u(z,t)}^2e^{-2k\phi(z)}dtd\lambda(z)\\
&\leq \frac{C}{k}\int\abs{\ddbar_{b,k}u(z,t)}^2e^{-2k\phi(z)}dtd\lambda(z)= \frac{C}{k}\norm{\ddbar_{b,k}u}^2_{h^{L^k}}<\infty.
\end{split}
\end{equation}
Furthermore, it is straightforward to see that 
\begin{equation}\label{e-gue131229VIII}
\begin{split}
\ddbar_{b,k}\alpha(z,t)&=\frac{1}{2\pi}\int g(z,\eta)e^{i\eta t}1_{[-k\delta,k\delta]}(\eta)e^{-\eta\mu(z)}d\eta \\
&=\frac{1}{2\pi}\int\mathcal{F}F^{(1)}_{\delta,k}\ddbar_{b,k}u(z,\eta)e^{i\eta t}d\eta \\
&=F^{(1)}_{\delta,k}\ddbar_{b,k}u(z,t).
\end{split}
\end{equation}
From \eqref{e-gue131229VII} and \eqref{e-gue131229VIII}, we conclude that 
$\ddbar_{b,k}\alpha(z,t)=F^{(1)}_{\delta,k}\ddbar_{b,k}u(z,t)$ and $\norm{\alpha}^2_{h^{L^k}}\leq\frac{C}{k}\norm{\ddbar_{b,k}u}^2_{h^{L^k}}$. Since $(I-\Pi^{(0)}_k)F^{(0)}_{\delta,k}u$ has the minimum $L^2$ norm of the solutions of $\ddbar_{b,k}f=F^{(1)}_{\delta,k}\ddbar_{b,k}u(z,t)$, we conclude that 
\[\norm{(I-\Pi^{(0)}_k)F^{(0)}_{\delta,k}u}^2_{h^{L^k}}=\norm{F^{(0)}_{\delta,k}(I-\Pi^{(0)}_k)u}^2_{h^{L^k}}\leq\norm{\alpha}^2_{h^{L^k}}\leq\frac{C}{k}\norm{\ddbar_{b,k}u}^2_{h^{L^k}}.\]
The theorem follows. 
\end{proof}

\subsection{Szeg\"o kernel asymptotics on $\Gamma\times\Real$, where $\Gamma=\Complex^{n-1}$ or $\Gamma$ is a bounded strongly pseudoconvex domain in $\Complex^{n-1}$}\label{s-saoh}

We fix $0<\delta$, $\delta$ is small. 
Let $D\Subset X$ be any open set. Let $M>0$ be a large constant so that for every $(x,\eta)\in T^*D$ if $\abs{\eta'}>\frac{M}{2}$ then $(x,\eta)\notin\Sigma$, where $\eta'=(\eta_1,\ldots,\eta_{2n-2})$, $\abs{\eta'}=\sqrt{\sum^{2n-2}_{j=1}\abs{\eta_j}^2}$. 
Fix $D_0\Subset D$. Let $D'\Subset D$ be an open neighbourhood of $D_0$. Put 
\begin{equation}\label{e-gue131227f}
V:=\set{(x,\eta)\in T^*D';\, \abs{\eta'}<M, \abs{\eta_{2n-1}}<\delta}.
\end{equation}
Then $\ol V\subset T^*D$. Moreover, if $\delta>0$ is small enough, then $\ol V\bigcap\Sigma\subset\Sigma'$, where $\Sigma'$ is given by \eqref{e-dhmpXIa}. Let $F^{(0)}_{\delta,k}:L^2(X,L^k)\To L^2(X,L^k)$ be as in \eqref{e-gue131222t} and \eqref{e-gue131222tIV}. It is clearly that 
\[\mbox{$F^{(0)}_{\delta,k}\equiv\frac{k^{2n-1}}{(2\pi)^{2n-1}}\int e^{ik<x-y,\eta>}\alpha(x,\eta,k)d\eta\mod O(k^{-\infty})$ at $T^*D_0\bigcap\Sigma$}\]
is a classical semi-classical pseudodifferential operator on $D$ of order $0$, where 
\[\begin{split}&\mbox{$\alpha(x,\eta,k)\sim\sum_{j=0}\alpha_j(x,\eta)k^{-j}$ in $S^0_{{\rm loc\,}}(1;T^*D)$},\\
&\alpha_j(x,\eta)\in C^\infty(T^*D),\ \ j=0,1,\ldots,\\
&\alpha_0(x,\eta)=\tau_\delta(\eta_{2n-1}),
\end{split}\]
with $\alpha(x,\eta,k)=0$ if $\abs{\eta}>M$, for some large $M>0$ and ${\rm Supp\,}\alpha(x,\eta,k)\bigcap T^*D_0\Subset V$. 
Now, we assume that $\pr\ddbar\phi$ and $\pr\ddbar\mu$ are uniformly bounded on $\Gamma$, that is, there is a 
constant $C_0>0$ such that $\abs{\frac{\pr^2\phi}{\pr z_j\pr\ol z_\ell}(z)}\leq C_0$, $\abs{\frac{\pr^2\mu}{\pr z_j\pr\ol z_\ell}(z)}\leq C_0$, $\forall z\in\Gamma$, $j, \ell=1,\ldots,n-1$. By using the technique in section~\ref{s-sub} and~\cite{HM09}, we can show that for every $\alpha, \beta\in\mathbb N^{2n-1}_0$, there is a constant $C_{\alpha,\beta}>0$ independent of $k$ such that 
\begin{equation}\label{e-gue140115a}
\abs{\pr^\alpha_x\pr^\beta_y\bigr(e^{-k\phi(x)}\Pi^{(0)}_k(x,y)e^{k\phi(y)}\bigr)}\leq C_{\alpha,\beta}k^{n+\abs{\alpha}+\abs{\beta}},\ \ \forall (x,y)\in(\Gamma\times\Real)\times(\Gamma\times\Real).
\end{equation}
By using \eqref{e-gue140115a} and integration by parts, it is not difficult to see that $\Pi^{(0)}_k$ is $k$-negligible away the diagonal with respect to $F^{(0)}_{\delta,k}$ on $\Gamma\times\Real$. 

From the discussion above, Theorem~\ref{t-gue131226}, Theorem~\ref{t-gue131229} and Theorem~\ref{t-gue130820}, we obtain one of the main results of this work

\begin{thm}\label{t-gue131227}
With the notations above, we assume that $\Gamma=\Complex^{n-1}$ or $\Gamma$ is a bounded strongly pseudoconvex domain in $\Complex^{n-1}$ and condition $Y(0)$ holds on $X$. When $\Gamma=\Complex^{n-1}$, we assume that there are constants $C_0\geq 1$ and $\epsilon_0>0$ such that
\[\begin{split}
&\sum^{n-1}_{j,\ell=1}\frac{\pr^2(\phi+\eta\mu)}{\pr z_j\pr\ol z_\ell}(z)w_j\ol w_\ell\geq\frac{1}{C_0}\sum^{n-1}_{j=1}\abs{w_j}^2,\ \ \forall (w_1,\ldots,w_{n-1})\in\Complex^{n-1},\ \ z\in\Complex^{n-1},\ \ \abs{\eta}\leq\epsilon_0,\\
&\abs{\frac{\pr^2\phi}{\pr z_j\pr\ol z_\ell}(z)}\leq C_0,\ \ \abs{\frac{\pr^2\mu}{\pr z_j\pr\ol z_\ell}(z)}\leq C_0,\ \ \forall z\in\Gamma,\ \ j, \ell=1,\ldots,n-1,\end{split}\]
and
\[\phi(z)+\eta\mu(z)\geq\frac{1}{C_0}\abs{z}^2,\ \ \forall \abs{z}\geq M,\ \ \abs{\eta}\leq\epsilon_0,\]
where $M>0$ is a constant independent of $\eta$. When $\Gamma$ is a bounded strongly pseudoconvex domain in $\Complex^{n-1}$, we assume that $\mu, \phi\in C^\infty(\Td\Gamma,\Real)$ and $\left(\frac{\pr^2\phi}{\pr z_j\pr\ol z_\ell}(z)\right)^{n-1}_{j,\ell=1}$ is positive definite at each point $z$ of $\Td\Gamma$, where $\Td\Gamma$ is an open neighbourhood of $\ol\Gamma$. 
Let $F^{(0)}_{\delta,k}:L^2(X,L^k)\To L^2(X,L^k)$ be the continuous operator given by \eqref{e-gue131222t} and \eqref{e-gue131222tIV} and let $F^{(0),*}_{\delta,k}:L^2(X,L^k)\To L^2(X,L^k)$ be the adjoint of $F^{(0)}_{\delta,k}$ with respect to $(\,\cdot\,|\,\cdot\,)_{h^{L^k}}$.
Let $D\Subset X$ be any open set and we fix any $D_0\Subset D$. Put $P_k:=F^{(q)}_{\delta,k}\Pi^{(q)}_kF^{(q),*}_{\delta,k}$. If $\delta$ is small, then 
\[
\hat P_{k,s}(x,y)\equiv\int e^{ik\varphi(x,y,s)}g(x,y,s,k)ds\mod O(k^{-\infty})\]
on $D_0$, where $\varphi(x,y,s)\in C^\infty(\Omega)$ is as in Theorem~\ref{t-dcgewI}, \eqref{e-guew13627},
\[
\begin{split}
&g(x,y,s,k)\in S^{n}_{{\rm loc\,}}(1;\Omega)\bigcap C^\infty_0(\Omega),\\
&g(x,y,s,k)\sim\sum^\infty_{j=0}g_j(x,y,s)k^{n-j}\text{ in }S^{n}_{{\rm loc\,}}
(1;\Omega), \\
&g_j(x,y,s)\in C^\infty_0(\Omega),\ \ j=0,1,2,\ldots,
\end{split}
\]
and for every $(x,x,s)\in\Omega$, $x\in D_0$, $x=(x_1,\ldots,x_{2n-2})=(z,t)$,
\[g_0(x,x,s)=(2\pi)^{-n}2^{n-1}\abs{\det\Bigr(\left(\frac{\pr^2(\phi-s\mu)}{\pr z_j\pr\ol z_\ell}(z)\right)^{n-1}_{j,\ell=1}\Bigr)}\abs{\tau_\delta(s)}^2.\]
Here
\[
\begin{split}
\Omega:=&\{(x,y,s)\in D\times D\times\Real;\, (x,-2{\rm Im\,}\ddbar_b\phi(x)+s\omega_0(x))\in V\bigcap\Sigma,\\
&\quad\mbox{$(y,-2{\rm Im\,}\ddbar_b\phi(y)+s\omega_0(y))\in V\bigcap\Sigma$, $\abs{x-y}<\varepsilon$, for some $\varepsilon>0$}\},
\end{split}\]
$V$ is given by \eqref{e-gue131227f}, 
\end{thm}


\section{The proof of Theorem~\ref{t-dgudmgeaI}} \label{sub-pf}

We now prove Theorem~\ref{t-dgudmgeaI}. We will use the same notations and assumptions as section~\ref{s-mhdt}. Fix $(\hat x_0,\hat x_0,s_0)\in\hat\Omega$, $t_0>0$. Put 
\[\begin{split}
&\hat x_0=(x_0,x_{n,0}),\ \ x_0\in\Real^{2n-1},\ \ (x_0,x_0,s_0)\in\Omega,\\
&(\hat x_0,\hat\xi_0)=(\hat x_0,t_0\frac{\pr\Phi}{\pr\hat x}(\hat x_0,\hat x_0,s_0))=(\hat x_0,t_0\frac{\pr\Phi_1}{\pr\hat x}(\hat x_0,\hat x_0,s_0))\in\Bigr(U\bigcap\hat\Sigma\Bigr)\bigcap T^*\hat D.
\end{split}\] 
From Lemma~\ref{l-gue130804}, we know that 
there are $\hat\varphi(x,y,s)$, $\hat\varphi_1(x,y,s)\in C^\infty(\Lambda)$, where $\Lambda$ is a small neighbourhood of $(x_0,x_0,s_0)$, such that  $\hat\varphi(x,y,s)$ and $\hat\varphi_1(x,y,s)$
satisfy \eqref{e-dgugeIX}, \eqref{e-dugeX}, \eqref{e-dugeXI}, \eqref{e-dgugeXI} and \eqref{e-gue1373VIIa} and $\frac{\pr\hat\varphi}{\pr y_{2n-1}}(x,y,s)-(\alpha_{2n-1}(y)+s\beta_{2n-1}(y))$, $\frac{\pr\hat\varphi_1}{\pr y_{2n-1}}(x,y,s)-(\alpha_{2n-1}(y)+s\beta_{2n-1}(y))$ vanish to infinite order at $x=y$, and $t\hat\Phi(\hat x,\hat y,s):=t(x_{2n}-y_{2n}+\hat\varphi(x,y,s))$ and $t\Phi(\hat x,\hat y,s)$ are equivalent for classical symbols at every point of 
\[{\rm diag\,}'\Bigr((U\bigcap\hat\Sigma)\times(U\bigcap\hat\Sigma)\Bigr)\bigcap\set{(\hat x,\hat x,td_{\hat x}\Phi(\hat x,\hat x,s),-td_{\hat x}\Phi(\hat x,\hat x,s))\in T^*\hat D;\, (x,x,s)\in\Lambda, t>0},\]
$t\hat\Phi_1(\hat x,\hat y,s):=t(x_{2n}-y_{2n}+\hat\varphi_1(x,y,s))$ and $t\Phi_1(\hat x,\hat y,s)$ are equivalent for classical symbols at every point of 
\[{\rm diag\,}'\Bigr((U\bigcap\hat\Sigma)\times(U\bigcap\hat\Sigma)\Bigr)\bigcap\set{(\hat x,\hat x,td_{\hat x}\Phi_1(\hat x,\hat x,s),-td_{\hat x}\Phi_1(\hat x,\hat x,s))\in T^*\hat D;\, (x,x,s)\in\Lambda, t>0}.\] 

Let $\hat W$ be a small neighbourhood of $(\hat x_0,\hat x_0,s_0)$ and let $I_0$ be a neighbourhood of $t_0$ in $\Real_+$. Put 
\begin{equation}\label{e-gmedmdhIV}
\begin{split}
\Lambda_{\Td{t\hat\Phi}}:=&\{(\Td{\hat x},\Td{\hat y},\Td t\frac{\pr\Td{\hat\Phi}}{\pr\Td x}(\Td{\hat x},\Td{\hat y},\Td s),\Td t\frac{\pr\Td{\hat\Phi}}{\pr\Td y}(\Td{\hat x},\Td{\hat y},\Td s))\in\Complex^{2n}\times\Complex^{2n}\times\Complex^{2n}\times\Complex^{2n};\, \\
&\quad\Td{\hat\Phi}(\Td{\hat x},\Td{\hat y},\Td s)=0, \frac{\pr\Td{\hat\Phi}}{\pr\Td s}(\Td{\hat x},\Td{\hat y},\Td s)=0, (\Td{\hat x},\Td{\hat y},\Td s)\in\hat W^\Complex,  \Td t\in I_0^\Complex\},\\
\Lambda_{\Td{t\hat\Phi}_1}:=&\{(\Td{\hat x},\Td{\hat y},\Td t\frac{\pr\Td{\hat\Phi}_1}{\pr\Td x}(\Td{\hat x},\Td{\hat y},\Td s),\Td t\frac{\pr\Td{\hat\Phi}_1}{\pr\Td y}(\Td{\hat x},\Td{\hat y},\Td s))\in\Complex^{2n}\times\Complex^{2n}\times\Complex^{2n}\times\Complex^{2n};\, \\
&\quad\Td{\hat\Phi}_1(\Td{\hat x},\Td{\hat y},\Td s)=0, \frac{\pr\Td{\hat\Phi}_1}{\pr\Td s}(\Td{\hat x},\Td{\hat y},\Td s)=0, (\Td{\hat x},\Td{\hat y},\Td s)\in\hat W^\Complex,  \Td t\in I_0^\Complex\}.\\
\end{split}
\end{equation} 
From global theory of complex Fourier integral operators of Melin-Sj\"ostrand~\cite{MS74}, we know that $t\hat\Phi(\hat x,\hat y,s)$ and $t\hat\Phi_1(\hat x,\hat y,s)$ are equivalent for classical symbols at $(\hat x_0,\hat x_0,\hat\xi_0,-\hat\xi_0)\in{\rm diag\,}'\Bigr((U\bigcap\hat\Sigma)\times(U\bigcap\hat\Sigma)\Bigr)$ in the sense of Melin-Sj\"{o}strand~\cite{MS74} if and only if $\Lambda_{\Td{t\hat\Phi}}$ and $\Lambda_{\Td{t\hat\Phi}_1}$ are equivalent in the sense that there is a neighbourhood $Q$ of $(\hat x_0,\hat x_0,\hat\xi_0,-\hat\xi_0)$ in $\Complex^{2n}\times\Complex^{2n}\times\Complex^{2n}\times\Complex^{2n}$, such that for every $N>0$, we have 
\begin{equation}\label{e-gmedmdhV-b}
\begin{split}
&{\rm dist\,}(z,\Lambda_{\Td{t\hat\Phi}_1})\leq C_N\abs{{\rm Im\,}z}^N,\ \ \forall z\in Q\bigcap\Lambda_{\Td{t\hat\Phi}},\\
&{\rm dist\,}(z_1,\Lambda_{\Td{t\hat\Phi}})\leq C_N\abs{{\rm Im\,}z_1}^N,\ \ \forall z_1\in Q\bigcap\Lambda_{\Td{t\hat\Phi}_1},
\end{split}
\end{equation}
where $C_N>0$ is a constant independent of $z$ and $z_1$. 

We first assume that $\varphi(x,y,s)$ and $\varphi_1(x,y,s)$ are equivalent at each point of $\Omega$ in the sense 
of Definition~\ref{d-geudhdc13619}. 
We take almost analytic extensions of $t\hat\Phi$ and $t\hat\Phi_1$ such that
\begin{equation}\label{e-gmedmdhIII}
\begin{split}
&\Td{t\hat\Phi}(\Td{\hat x},\Td{\hat y},\Td s)=\Td t\Td{\hat\Phi}(\Td{\hat x},\Td{\hat y},\Td s)=\Td t(\Td{x}_{2n}-\Td{y}_{2n})+\Td t\Td{\hat\varphi}(\Td x,\Td y,\Td s),\\
&\Td{t{\hat\Phi}_1}(\Td{\hat x},\Td{\hat y},\Td s)=\Td t\Td{\hat\Phi}_1(\Td{\hat x},\Td{\hat y},\Td s)=\Td t(\Td{x}_{2n}-\Td{y}_{2n})+\Td t\Td{\hat\varphi}_1(\Td x,\Td y,\Td s)
\end{split}
\end{equation}
on $\hat\Lambda^\Complex$. Here $\hat\Lambda=\Lambda\times I_1$, $I_1$ is a small open neighbourhood of $x_{n,0}$. 
We are going to prove that $\Lambda_{\Td{t\hat\Phi}}$ and $\Lambda_{\Td{t\hat\Phi}_1}$ are equivalent in the sense of \eqref{e-gmedmdhV-b}.

Since $\frac{\pr^2\hat\varphi}{\pr s\pr x_{2n-1}}(x,x,s)\neq0$, $\frac{\pr^2\hat\varphi_1}{\pr s\pr x_{2n-1}}(x,x,s)\neq0$(see \eqref{e-gmedmdhI}), from Malgrange preparation theorem (see Theorem 7.57 in~\cite{Hor03}), we have
\begin{equation}\label{e-gmedmdhII}
\begin{split}
&\frac{\pr\hat\varphi}{\pr s}(x,y,s)=(x_{2n-1}-\beta(x',y,s))g(x,y,s),\\
&\frac{\pr\hat\varphi_1}{\pr s}(x,y,s)=(x_{2n-1}-\beta_1(x',y,s))g_1(x,y,s)
\end{split}
\end{equation}
in some small neighbourhood of $(x_0,x_0,s_0)$, where $\beta, \beta_1, g, g_1\in C^\infty$, $\beta(x',x,s)=\beta_1(x',x,s)=x_{2n-1}$, $g(x,x,s)\neq0$, $g_1(x,x,s)\neq0$. We may take $\Lambda$ small enough so that \eqref{e-gmedmdhII} hold on $\Lambda$. It is easy to check that 
$\frac{\pr^2\hat\varphi}{\pr x_{2n-1}\pr s}(x,x,s)=g(x,x,s)$, 
$\frac{\pr^2\hat\varphi}{\pr x_j\pr s}(x,x,s)=-\frac{\pr\beta}{\pr x_j}(x',x,s)g(x,x,s)$, $j=1,\ldots,2n-2$.
From this observation and notice that $\frac{\pr^2\hat\varphi}{\pr x_j\pr s}(x,x,s)$ is real, $j=1,\ldots,2n-1$, we conclude that 
\begin{equation}\label{e-gmedmdhII-II}
\begin{split}
&{\rm Re\,}g(x,x,s)\neq0,\ \ {\rm Im\,}g(x,x,s)=0,\\
&-\frac{\pr{\rm Im\,}\beta}{\pr x_j}(x',x,s)=0,\ \ j=1,\ldots,2n-2.
\end{split}
\end{equation}

From \eqref{e-gmedmdhII}, we conclude that for every $N>0$, there is a constant $C_N>0$ such that 
for all $(\Td x,\Td y,\Td s)\in\hat W^\Complex$, 
\begin{equation}\label{e-gmedmdhIII-I}
\begin{split}
&\abs{\frac{\pr\Td{\hat\varphi}}{\pr\Td s}(\Td x,\Td y,\Td s)-\bigr((\Td{x}_{2n-1}-\Td\beta(\Td{x}',\Td y,\Td s))\Td g(\Td x,\Td y,\Td s)\bigr)}\leq C_N\abs{{\rm Im\,}(\Td x,\Td y,\Td s)}^N,\\
&\abs{\frac{\pr\Td{\hat\varphi}_1}{\pr\Td s}(\Td x,\Td y,\Td s)-\bigr((\Td{x}_{2n-1}-\Td\beta_1(\Td{x}',\Td y,\Td s))\Td g_1(\Td x,\Td y,\Td s)\bigr)}\leq C_N\abs{{\rm Im\,}(\Td x,\Td y,\Td s)}^N.
\end{split}
\end{equation}

In view of \eqref{e-gmedmdhIII}, we may take $\hat\Lambda$ and $\hat\Lambda^\Complex$ small enough so that on $\hat\Lambda^\Complex$, $\frac{\pr^2\Td{\hat\varphi}}{\pr\Td x_{2n-1}\pr\Td s}\neq0$, $\frac{\pr^2\Td{\hat\varphi}_1}{\pr\Td x_{2n-1}\pr\Td s}\neq0$, and 
there are $\delta(\Td x',\Td y,\Td s)\in C^\infty(\hat\Lambda^\Complex)$, $\delta_1(\Td x',\Td y,\Td s)\in C^\infty(\hat\Lambda^\Complex)$ such that 
\begin{equation}\label{e-gmedmdhV}
\frac{\pr\Td{\hat\varphi}}{\pr\Td s}((\Td x',\delta(\Td x',\Td y,\Td s)),\Td y,\Td s)=\frac{\pr\Td{\hat\varphi}_1}{\pr\Td s}((\Td x', \delta_1(\Td x',\Td y,\Td s)),\Td y,\Td s)=0\ \ \mbox{on $\hat\Lambda^\Complex$}.
\end{equation}
From \eqref{e-gmedmdhV} and \eqref{e-gmedmdhIII-I}, we see that for every $N>0$ there is a constant $D_N>0$ such that for all $(\Td x',\Td y,\Td s)\in\hat\Lambda^\Complex$, 
\begin{equation}\label{e-gmedmdhVI}
\begin{split}
&\abs{\delta(\Td x',\Td y,\Td s)-\Td\beta(\Td x',\Td y,\Td s)}\leq D_N\abs{{\rm Im\,}(\Td x',\delta(\Td x',\Td y,\Td s),\Td y,\Td s)}^N,\\
&\abs{\delta_1(\Td x',\Td y,\Td s)-\Td\beta_1(\Td x',\Td y,\Td s)}\leq D_N\abs{{\rm Im\,}(\Td x',\delta_1(\Td x',\Td y,\Td s),\Td y,\Td s)}^N.
\end{split}
\end{equation}
Since $\frac{\pr\Td{\hat\varphi}}{\pr\Td s}(y',y_{2n-1},y,s)=\frac{\pr\Td{\hat\varphi}_1}{\pr\Td s}(y',y_{2n-1},y,s)=0$, $\delta_1(y',y,s)=\delta(y',y,s)=y_{2n-1}$. Hence there is a constant $C_0>0$ such that for all $(\Td x',\Td y,\Td s)\in\hat\Lambda^\Complex$,
\begin{equation}\label{e-gmedmdhVI-I}
\begin{split}
&\abs{{\rm Im\,}\delta(\Td x',\Td y,\Td s)}\leq C_0(\abs{{\rm Im\,}(\Td x',\Td y,\Td s)}+\abs{x'-y'}),\\
&\abs{{\rm Im\,}\delta_1(\Td x',\Td y,\Td s)}\leq C_0(\abs{{\rm Im\,}(\Td x',\Td y,\Td s)}+\abs{x'-y'}).
\end{split}
\end{equation}

From \eqref{e-gmedmdhIII} and \eqref{e-gmedmdhV}, we can check that 
\begin{equation}\label{e-gmedmdhVII}
\begin{split}
\Lambda_{\Td{t\hat\Phi}}=&\{(\Td{\hat x},\Td{\hat y},\Td t\frac{\pr\Td{\hat\Phi}}{\pr\Td x}(\Td{\hat x},\Td{\hat y},\Td s),\Td t\frac{\pr\Td{\Phi}}{\pr\Td y}(\Td{\hat x},\Td{\hat y},\Td s);\, \\
&\Td x_{2n-1}=\delta(\Td x',\Td y,\Td s), \Td x_{2n}=\Td y_{2n}-\Td{\hat\varphi}(\Td x',\delta(\Td x',\Td y,\Td s),\Td y,\Td s), (\Td{\hat x},\Td{\hat y},\Td s)\in\hat\Lambda^\Complex,  \Td t\in I_0^\Complex\},\\
\Lambda_{\Td{t\hat\Phi}_1}=&\{(\Td{\hat x},\Td{\hat y},\Td t\frac{\pr\Td{\hat\Phi}_1}{\pr\Td x}(\Td{\hat x},\Td{\hat y},\Td s),\Td t\frac{\pr\Td{\hat\Phi}_1}{\pr\Td y}(\Td{\hat x},\Td{\hat y},\Td s);\, \\
&\Td x_{2n-1}=\delta_1(\Td x',\Td y,\Td s), \Td x_{2n}=\Td y_{2n}-\Td{\hat\varphi}_1(\Td x',\delta_1(\Td x',\Td y,\Td s),\Td y,\Td s), (\Td{\hat x},\Td{\hat y},\Td s)\in\hat\Lambda^\Complex,  \Td t\in I_0^\Complex\}.
\end{split}
\end{equation}
From Definition~\ref{d-geudhdc13619} and \eqref{e-gmedmdhII}, it is easy to see that
\begin{equation}\label{e-gmedmdhVIII}
\begin{split}
&\beta(x',y,s)-\beta_1(x',y,s),\ \ \Td{\hat\varphi}(x',\beta(x',y,s),y,s)-\Td{\hat\varphi}_1(x',\beta(x',y,s),y,s),\\  
&\frac{\pr\Td{\hat\varphi}}{\pr x_j}(x',\beta(x',y,s),y,s)-\frac{\pr\Td{\hat\varphi}_1}{\pr x_j}(x',\beta_1(x',y,s),y,s),\ \ j=1,\ldots,2n-1,\\
&\frac{\pr\Td{\hat\varphi}}{\pr y_j}(x',\beta(x',y,s),y,s)-\frac{\pr\Td{\hat\varphi}_1}{\pr y_j}(x',\beta_1(x',y,s),y,s),\ \ j=1,\ldots,2n-1
\end{split}
\end{equation}
vanish to infinite order at $x'=y'$. Thus, if $\hat\Lambda^\Complex$ is small, then for every $N>0$, there is a constant $E_N>0$ such that
\begin{equation}\label{e-gmedmdhIX}
\begin{split}
&\abs{\Td\beta(\Td x',\Td y,\Td s)-\Td\beta_1(\Td x',\Td y,\Td s)}\leq E_N(\abs{{\rm Im\,}(\Td x',\Td y,\Td s)}^N+\abs{{\rm Re\,}\Td x'-{\rm Re\,}\Td y'}^N),\\ 
&\abs{\Td{\hat\varphi}(\Td x',\Td\beta(\Td x',\Td y,\Td s),\Td y,\Td s)-\Td{\hat\varphi}_1(\Td x',\Td\beta_1(\Td x',\Td y,\Td s),\Td y,\Td s)}\\
&\quad\leq E_N(\abs{{\rm Im\,}(\Td x',\Td y,\Td s)}^N+\abs{{\rm Re\,}\Td x'-{\rm Re\,}\Td y'}^N),\\ 
&\abs{\pr_{\Td x}\Td{\hat\varphi}(\Td x',\Td\beta(\Td x',\Td y,\Td s),\Td y,\Td s)-\pr_{\Td x}\Td{\hat\varphi}_1(\Td x',\Td\beta_1(\Td x',\Td y,\Td s),\Td y,\Td s)}\\
&\quad\leq E_N(\abs{{\rm Im\,}(\Td x',\Td y,\Td s)}^N+\abs{{\rm Re\,}\Td x'-{\rm Re\,}\Td y'}^N),\\
&\abs{\pr_{\Td y}\Td{\hat\varphi}(\Td x',\Td\beta(\Td x',\Td y,\Td s),\Td y,\Td s)-\pr_{\Td y}\Td{\hat\varphi}_1(\Td x',\Td \beta_1(\Td x',\Td y,\Td s),\Td y,\Td s)}\\
&\quad\leq E_N(\abs{{\rm Im\,}(\Td x',\Td y,\Td s)}^N+\abs{{\rm Re\,}\Td x'-{\rm Re\,}\Td y'}^N),
\end{split}
\end{equation}
where $\pr_{\Td x}=(\frac{\pr}{\pr\Td x_1},\ldots,\frac{\pr}{\pr\Td x_{2n-1}})$, $\pr_{\Td y}=(\frac{\pr}{\pr\Td y_1},\ldots,\frac{\pr}{\pr\Td y_{2n-1}})$. From \eqref{e-gmedmdhVI}, \eqref{e-gmedmdhVI-I}, \eqref{e-gmedmdhVII} and \eqref{e-gmedmdhIX}, we see that there is a small neighbourhood $Q$ of $(\hat x_0,\hat x_0,\hat\xi_0,-\hat\xi_0)$ in $\Complex^{2n}\times\Complex^{2n}\times\Complex^{2n}\times\Complex^{2n}$ such that for all $N>0$ and every $z\in Q\bigcap\Lambda_{\Td{t\hat\Phi}}$, $z=(\Td{\hat x},\Td{\hat y},\Td t\frac{\pr\Td{\hat\Phi}}{\pr\Td x}(\Td{\hat x},\Td{\hat y},\Td s),\Td t\frac{\pr\Td{\hat\Phi}}{\pr\Td y}(\Td{\hat x},\Td{\hat y},\Td s))$, $\Td x_{2n-1}=\Td\delta(\Td x',\Td y,\Td s)$, $\Td x_{2n}=\Td y_{2n}-\Td{\hat\varphi}(\Td x',\Td\delta(\Td x',\Td y,\Td s),\Td y,\Td s)$, we have 
\begin{equation}\label{e-gmedmdhX}
{\rm dist\,}(z,\Lambda_{\Td{t\hat\Phi}_1})\leq F_N(\abs{{\rm Im\,}(\Td x',\Td y,\Td s)}^N+\abs{{\rm Re\,}\Td x'-{\rm Re\,}\Td y'}^N),
\end{equation}
where $F_N>0$ is a constant independent of $z\in Q$. We claim that if $Q$ is small enough then there is a constant $C_1>0$ independent of $z\in Q$ such that 
\begin{equation}\label{e-gmedmdhXI}
\begin{split}
&\abs{{\rm Im\,}(\Td x',\Td y,\Td s)}^2+\abs{{\rm Re\,}\Td x'-{\rm Re\,}\Td y'}^2\leq C_1\abs{{\rm Im\,}z},\\ 
&\quad\forall z=(\Td{\hat x},\Td{\hat y},\Td t\frac{\pr\Td{\hat\Phi}}{\pr\Td x}(\Td{\hat x},\Td{\hat y},\Td s),\Td t\frac{\pr\Td{\hat\Phi}}{\pr\Td y}(\Td{\hat x},\Td{\hat y},\Td s))\in Q\bigcap\Lambda_{\Td{t\hat\Phi}}. 
\end{split}
\end{equation}
From Taylor expansion, it is easy to see that
\[\begin{split}
&{\rm Im\,}\Td{\hat\varphi}(x',\delta(x',y,s),y,s)\\
&={\rm Im\,}\hat\varphi(x',{\rm Re\,}\delta(x',y,s),y,s)+\frac{\pr{\rm Im\,}\hat\varphi}{\pr x_{2n-1}}(x',{\rm Re\,}\delta(x',y,s),y,s){\rm Im\,}\delta(x',y,s)+(\abs{{\rm Im\,}\delta(x',y,s)}^2).
\end{split}\]
From this and note that ${\rm Im\,}d_x\hat\varphi(x',{\rm Re\,}\delta(x',y,s),y,s)|_{x'=y'}=0$, $\delta(x',y,s)=y_{2n-1}$ if $y=(x',y_{2n-1})$, we see that there is a constant $c_0>0$ such that 
\begin{equation}\label{e-gelkmidhmpI-b}
c_0\abs{x'-y'}^2\leq\abs{{\rm Im\,}\Td{\hat\varphi}(x',\delta(x',y,s),y,s)}+\abs{{\rm Im\,}\delta(x',y,s)}\leq\frac{1}{c_0}\abs{x'-y'}^2,
\end{equation}
for all $(\hat x,\hat y,s)$ in a small neighbourhood of $(\hat x_0,\hat y_0,\hat s_0)$. Thus, if $\hat\Lambda$ and $\hat \Lambda^\Complex$ are small, then
\begin{equation}\label{e-gelkmidhmpI}
\begin{split}
&\abs{{\rm Im\,}\Td y_{2n}-{\rm Im\,}\Td{\hat\varphi}(\Td x',\delta(\Td x',\Td y,\Td s),\Td y,\Td s)}+\abs{{\rm Im\,}\Td y_{2n}}+\abs{{\rm Im\,}\delta(\Td x',\Td y,\Td s)}+\abs{{\rm Im\,}(\Td x',\Td y,\Td s)}\\
&\geq c_1\abs{{\rm Re\,}\Td x'-{\rm Re\,}\Td y'}^2\ \ \mbox{on $\hat\Lambda^\Complex$},
\end{split}
\end{equation}
where $c_1>0$ is a constant. 

We consider Taylor expansion of $t\frac{\pr\Td{\hat\Phi}}{\pr\Td x_{2n-1}}(y',\delta(y',y,\Td s),y,\Td s)$ at ${\rm Im\,}s=0$: 
\begin{equation}\label{e-gelkmidhmpII}
\begin{split}
&t\frac{\pr\Td{\hat\Phi}}{\pr\Td x_{2n-1}}(y',\delta(y',y,\Td s),y,\Td s)=t\frac{\pr\Td{\hat\Phi}}{\pr\Td x_{2n-1}}(y',\delta(y',y,{\rm Re\,}\Td s),y,{\rm Re\,}\Td s)\\
&\quad+t\Bigr(\frac{\pr^2\Td{\hat\Phi}}{\pr\Td x_{2n-1}\pr\Td s}(y',\delta(y',y,{\rm Re\,}\Td s),y,{\rm Re\,}\Td s)\Bigr)i{\rm Im\,}\Td s+O(\abs{{\rm Im\,}\Td s}^2).
\end{split}
\end{equation}
Here we used the fact that $\frac{\pr\delta}{\pr\Td s}(y',y,{\rm Re\,}\Td s)=0$ since $\delta(y',y,{\rm Re\,}\Td s)=y_{2n-1}$. Since 
\[t\frac{\pr\Td{\hat\Phi}}{\pr\Td x_{2n-1}}(y',\delta(y',y,{\rm Re\,}\Td s),y,{\rm Re\,}\Td s),\ \  t\frac{\pr^2\Td{\hat\Phi}}{\pr\Td x_{2n-1}\pr\Td s}(y',\delta(y',y,{\rm Re\,}\Td s),y,{\rm Re\,}\Td s)\] 
are real and $t\frac{\pr^2\Td\phi}{\pr\Td x_{2n-1}\pr\Td s}(y',\delta(y',y,{\rm Re\,}\Td s),y,{\rm Re\,}\Td s)\neq0$, we conclude that there is a constant $c_1>0$ such that 
\[\abs{{\rm Im\,}\Bigr(t\frac{\pr\Td{\hat\Phi}}{\pr\Td x_{2n-1}}(y',\delta(y',y,\Td s),y,\Td s)\Bigr)}\geq c_1\abs{{\rm Im\,}\Td s}\]
for $\abs{{\rm Im\,}\Td s}$ is small. Thus, if $\hat\Lambda$ and $\hat\Lambda^\Complex$ are small, then
\begin{equation}\label{e-gelkmidhmpIII}
\begin{split}
&\abs{{\rm Im\,}\Bigr(\Td t\frac{\pr\Td{\hat\Phi}}{\pr\Td x_{2n-1}}(\Td x',\delta(\Td x',\Td y,\Td s),\Td y,\Td s)\Bigr)}+C_0\abs{{\rm Im\,}\Bigr(\Td t\frac{\pr\Td{\hat\Phi}}{\pr\Td x_{2n}}(\Td x',\delta(\Td x',\Td y,\Td s),\Td y,\Td s)\Bigr)}\\
&\quad+\abs{{\rm Im\,}(\Td x',\Td y)}+\abs{x'-y'}
\geq c_2\abs{{\rm Im\,}\Td s}\ \ \mbox{on $\hat\Lambda^\Complex$},
\end{split}
\end{equation}
where $C_0>0$ and $c_2>0$ are constants. 
Note that ${\rm Im\,}\Bigr(\Td t\frac{\pr\Td{\hat\Phi}}{\pr\Td x_{2n}}(\Td x',\delta(\Td x',\Td y,\Td s),\Td y,\Td s)\Bigr)={\rm Im\,}\Td t$. 

From \eqref{e-gmedmdhVI-I}, \eqref{e-gmedmdhVII}, \eqref{e-gelkmidhmpI} and \eqref{e-gelkmidhmpIII}, the claim \eqref{e-gmedmdhXI} follows. 
From \eqref{e-gmedmdhXI} and \eqref{e-gmedmdhX}, we conclude that there is a neighbourhood $Q$ of $(\hat x_0,\hat x_0,\hat\xi_0,-\hat\xi_0)$ 
in $\Complex^{2n}\times\Complex^{2n}\times\Complex^{2n}\times\Complex^{2n}$, such that for every $N$, 
we have ${\rm dist\,}(z,\Lambda_{\Td{t\hat\Phi}_1})\leq C_N\abs{{\rm Im\,}z}^N$,
for all $z\in Q\bigcap\Lambda_{\Td{t\hat\Phi}}$. We can repeat the procedure above and conclude that there is a neighbourhood $Q_1$ of $(\hat x_0,\hat x_0,\hat\xi_0,-\hat\xi_0)$ in $\Complex^{2n}\times\Complex^{2n}\times\Complex^{2n}\times\Complex^{2n}$, such that for every $N$, we have ${\rm dist\,}(z_1,\Lambda_{\Td{t\hat\Phi}})\leq C_N\abs{{\rm Im\,}z_1}^N$, $\forall z_1\in Q_1\bigcap\Lambda_{\Td{t\hat\Phi}_1}$. We 
obtain that $t\hat\Phi(\hat x,\hat y,s)$ and $t\hat\Phi_1(\hat x,\hat y,s)$ are equivalent for classical symbols at $(\hat x_0,\hat x_0,\hat\xi_0,-\hat\xi_0)$ in the sense of Melin-Sj\"{o}strand~\cite{MS74}.

Now, we assume that $t\Phi(\hat x,\hat y,s)$ and $t\Phi_1(\hat x,\hat y,s)$ are equivalent for classical symbols at 
each point of 
\[{\rm diag\,}'\Bigr((U\bigcap\hat\Sigma)\times (U\bigcap\hat\Sigma)\Bigr)\bigcap\set{(\hat x,\hat x,\hat\xi,-\hat\xi);\, (\hat x,\hat\xi)\in T^*\hat D}.\] 
From global theory of complex Fourier integral operators of Melin-Sj\"ostrand~\cite{MS74}, we see that $\Lambda_{\Td{t\hat\Phi}}$ and $\Lambda_{\Td{t\Phi}_1}$ are equivalent in the sense of \eqref{e-gmedmdhV-b}. 

Since $\frac{\pr^2\hat\varphi}{\pr s\pr x_{2n-1}}(x,x,s)\neq0$, $\frac{\pr^2\hat\varphi_1}{\pr s\pr x_{2n-1}}(x,x,s)\neq0$, we may assume that \eqref{e-gmedmdhII} hold on $\Lambda$. We may take $\hat\Lambda$ and $\hat\Lambda^\Complex$ small enough so that \eqref{e-gmedmdhV},
\eqref{e-gmedmdhVI}, \eqref{e-gmedmdhVI-I} and \eqref{e-gmedmdhVII} hold on $\hat\Lambda^\Complex$. Since $\frac{\pr\hat\varphi}{\pr y_{2n-1}}(x,y,s)-(\alpha_{2n-1}(y)+s\beta_{2n-1}(y))$ and $\frac{\pr\hat\varphi_1}{\pr y_{2n-1}}(x,y,s)-(\alpha_{2n-1}(y)+s\beta_{2n-1}(y))$ vanish to infinite order at $x=y$, it is straightforward to see that if $\hat\Lambda$ is small enough then for every $N>0$, every $z\in\hat\Lambda\bigcap\Lambda_{\Td{t\hat\Phi}}$, $z=(\hat x,\hat y,t\frac{\pr\Td{\hat\Phi}}{\pr x}(\hat x,\hat y,s),t\frac{\pr\Td{\Phi}}{\pr y}(\hat x,\hat y,s))\in\Lambda_{\Td{t\hat\Phi}}$, $x_{2n-1}=\delta(x',y,s)$, $x_{2n}=y_{2n}-\Td{\hat\varphi}(x',\delta(x',y,s),y,s)$, there is a constant $c_N>0$ such that 
\begin{equation} \label{e-gueama13620I}
\begin{split}
&{\rm dist\,}(z,\Lambda_{\Td{t\hat\Phi}_1})+\abs{x'-y'}^N \\&\quad\geq c_N\Bigr(\abs{\delta(x',y,s)-\delta_1(x',y,s)}+\abs{\Td{\hat\varphi}(x',\delta(x',y,s),y,s)-\Td{\hat\varphi}_1(x',\delta_1(x',y,s),y,s)}\\
&\quad+\abs{d_x\Bigr(\Td{\hat\varphi}(x',\delta(x',y,s),y,s)-\Td{\hat\varphi}_1(x',\delta_1(x',y,s),y,s)\Bigr)}\\
&\quad+\abs{d_y\Bigr(\Td{\hat\varphi}(x',\delta(x',y,s),y,s)-\Td{\hat\varphi}_1(x',\delta_1(x',y,s),y,s)\Bigr)}\Bigr).
\end{split}
\end{equation}

It is easy to see that there is a constant $c>0$ such that for every 
\[\begin{split}
&z=(\hat x,\hat y,t\frac{\pr\Td{\hat\Phi}}{\pr y}(\hat x,\hat y,s),t\frac{\pr\Td{\Phi}}{\pr y}(\hat x,\hat y,s))\in\Lambda_{\Td{t\hat\Phi}},\\
&x_{2n-1}=\delta(x',y,s),\ \ x_{2n}=y_{2n}-\Td{\hat\varphi}(x',\delta(x',y,s),y,s),\end{split}\]
$(\hat x,\hat y,s)$ is in a small real neighbourhood of $(\hat x_0,\hat x_0,s_0)$, $t$ is in a small real neighbourhood of $t_0$, 
we have
\begin{equation}\label{e-gelkmidhmpIII-I}
\abs{{\rm Im\,}z}\leq\frac{1}{c_0}\abs{x'-y'}.
\end{equation}
 
From \eqref{e-gmedmdhV-b}, \eqref{e-gmedmdhVI}, \eqref{e-gmedmdhVI-I}, \eqref{e-gmedmdhVIII}, \eqref{e-gueama13620I} and \eqref{e-gelkmidhmpIII-I}, we see that if $\Lambda$ and $\hat\Lambda$ are small enough then for every $N>0$, there is a constant $B_N>0$ such that on $\hat\Lambda$, 
\begin{equation}\label{e-gelkmidhmpIII-II}
\begin{split}
&\abs{\beta(x',y,s)-\beta_1(x',y,s)}\leq B_N\abs{x'-y'}^N,\\
&\abs{\Td{\hat\varphi}(x',\beta(x',y,s),y,s)-\Td{\hat\varphi}_1(x',\beta_1(x',y,s),y,s)}\leq B_N\abs{x'-y'}^N,\\
&\abs{d_x\Td{\hat\varphi}(x',\beta(x',y,s),y,s)-d_x\Td{\hat\varphi}_1(x',\beta_1(x',y,s),y,s)}\leq B_N\abs{x'-y'}^N,\\
&\abs{d_y\Td{\hat\varphi}(x',\beta(x',y,s),y,s)-d_y\Td{\hat\varphi}_1(x',\beta_1(x',y,s),y,s)}\leq B_N\abs{x'-y'}^N.
\end{split}
\end{equation}
From \eqref{e-gmedmdhII} and \eqref{e-gelkmidhmpIII-II}, it is not difficult to see that $\varphi(x,y,s)$ and $\varphi_1(x,y,s)$ are equivalent at each point of $\Omega$ in the sense of Definition~\ref{d-geudhdc13619}. 
Theorem~\ref{t-dgudmgeaI} follows.

\end{document}